\documentclass[reqno]{amsart}

\usepackage{color}
\usepackage[colorlinks=true]{hyperref}

\newcommand{\nn}{ {\nabla}  }

\newcommand{\pp}{ {\partial} }

\newcommand{\A}{\alpha }

\newcommand{\JJ}{{\mathcal J}}

\newcommand{\R} {\mathbb R}
\newcommand{\cuad}{{\sqcap\kern-.68em\sqcup}}

\newcommand{\I}{{\mathcal I}}

\newcommand{\dist}{{\rm dist}\, }
\newcommand{\foral}{\quad\mbox{for all}\quad}
\newcommand{\ve}{\varepsilon}

\newcommand{\be}{\begin{equation}}
\newcommand{\ee}{\end{equation}}

\newcommand{\equ}[1]{(\ref{#1})}

\newtheorem{lemma}{Lemma}[section]
\newtheorem{coro}{Corollary}[section]

\newtheorem{teo}{Theorem}

\newtheorem{prop}{Proposition}[section]
\newtheorem{corollary}{Corollary}[section]
\newtheorem{remark}{Remark}[section]
\newcommand{\bremark}{\begin{remark} \em}
\newcommand{\eremark}{\end{remark} }

\numberwithin{equation}{section}

\long\def\comment#1{\marginpar{\raggedright\small$\bullet$\ #1}}

\def\cb{\color{blue}}
\def\cb{}

\definecolor{g2}{rgb}{0,0.6,0}
\definecolor{r2}{rgb}{0.8,0,0}
\newcommand{\cat}{{\mathcal C}}
\newcommand{\FF}{\mathcal{F}}
\newcommand{\pd}[2]{\frac{\partial#1}{\partial#2}}

\begin{document}

\title{ Nonlocal $s$-minimal surfaces and Lawson cones }

\author{Juan D\'avila}
\address{\noindent J. D\'avila -
Departamento de Ingenier\'{\i}a Matem\'atica and CMM, Universidad
de Chile, Casilla 170 Correo 3, Santiago, Chile.}
\email{jdavila@dim.uchile.cl}

\author{Manuel del Pino}
\address{\noindent M. del Pino- Departamento de
Ingenier\'{\i}a  Matem\'atica and CMM, Universidad de Chile, Casilla
170 Correo 3, Santiago, Chile.} \email{delpino@dim.uchile.cl}

\author{Juncheng Wei}
\address{\noindent J. Wei -Department of Mathematics, University of British Columbia, Vancouver, B.C., Canada, V6T 1Z2 and Department of Mathematics, Chinese University of Hong Kong, Shatin, NT, Hong Kong.} \email{jcwei@math.ubc.ca}

\thanks{J. Wei is partially supported by NSERC of Canada. J. D\'avila and M. del Pino have been supported by Fondecyt and Fondo Basal CMM grants.
We would like to thank Alessio Figalli,  Jean-Michel Roquejoffre and   Enrico Valdinoci for useful discussions during the preparation of this paper. Part of this work was
concluded while J. D\'avila and M. del Pino  were visiting the PIMS center at UBC. They are grateful for the hospitality received }

\begin{abstract}
The nonlocal $s$-fractional minimal surface equation for  $\Sigma= \pp E$ where $E$ is an open set in $\R^N$  is given by
$$
H_\Sigma^ s (p) := \int_{\R^N}  \frac {\chi_E(x) - \chi_{E^c}(x)} {|x-p|^{N+s}}\, dx \ =\ 0 \foral p\in  \Sigma.
$$
Here $0<s<1$, $\chi$ designates characteristic function,  and the integral is understood in the principal value sense.
The classical notion of minimal surface is recovered by letting $s\to 1$. In this  paper we exhibit the first concrete examples (beyond the plane) of
 nonlocal $s-$minimal surfaces.  When $s$ is close to $1$, we first  construct a connected embedded $s$-minimal surface of revolution in $\R^3$, the {\bf nonlocal catenoid}, an analog of the standard catenoid $|x_3| = \log (r + \sqrt{r^2 -1})$. Rather than eventual logarithmic growth, this surface becomes asymptotic to the cone $|x_3|= r\sqrt{1-s}$. We also find a two-sheet embedded $s$-minimal surface asymptotic to the same cone, an analog to the simple union of two parallel planes.

 On the other hand, for any $0<s<1$, $n,m\ge 1$, $s-$minimal Lawson   cones  $|v|=\alpha|u|$,  $(u,v)\in \R^n\times \R^m$,  are found to exist. In sharp contrast with the classical case,  we prove their stability for small $s$ and $n+m=7$, which suggests that unlike the classical theory (or the case $s$ close to 1), the regularity of $s$-area minimizing surfaces may not hold true in dimension $7$.

\end{abstract}

\maketitle

\section{Introduction}

\subsection{Fractional minimal surfaces}

Phase transition models where the motion of the interface region is driven by curvature type flows arise in many applications. The standard flow by mean curvature of surfaces $\Sigma(t)$ in $\R^N$  is that in which the normal speed of each point $x\in \Sigma(t)$ is proportional to its mean curvature $H_{\Sigma(t)} (x) = \sum_{i=1}^{N-1} k_i(x)$ where the $k_i$'s designate the
principal curvatures, namely the eigenvalues of the second fundamental form. Evans \cite{evans} showed  that standard mean curvature flow for level surfaces of a function can be recovered as the limit of a discretization scheme in time where heat flow $u_t -\Delta u=0$ of suitable initial data  is used to connect consecutive time steps, which was introduced in \cite{mbo}. When standard diffusion is replaced by that of the fractional Laplacian   $u_t + (-\Delta)^{\frac s2}u =0$ in order
 to describe long range, nonlocal interactions between points in the two distinct phases by a Levy process,  Caffarelli and Souganidis \cite{cs}, see also Imbert \cite{imbert}, found that for $1\le s<2$
 flow by mean curvature is still recovered, while for $0<s<1$, the stronger nonlocal effect makes the surfaces evolve in normal velocity according to their {\em fractional mean curvature},
 defined for a surface $\Sigma =\pp E$ where $E$ is an open subset of $\R^N$ as
\begin{align}
\label{1}
 H_{\Sigma}^s (p):=
\int_{\R^N} \frac{\chi_E (x)-\chi_{ E^c}(x)} {|x-p|^{N+s}} \, d x \quad \hbox{for } p\in \Sigma.
\quad
\end{align}
Here $\chi$ denotes characteristic function, $E^c = \R^N \setminus E$  and the integral is understood in the principal value sense,
$$
 H_{\Sigma}^s (p)=\lim_{\delta\to 0} \int_{\R^N\setminus B_\delta(p)} \frac{\chi_E (x)-\chi_{E^c}(x)} {|x-p|^{N+s}} \, d x.
\quad
$$
This quantity is  well-defined  provided that $\Sigma$ is  regular near $p$. It agrees with usual mean curvature in the limit $s\to 1$ by the relation
\begin{equation}
\label{mc}
\lim_{s\to 1} (1-s)\,H_{\Sigma}^s (p)=   c_N H_{\Sigma} (p),
\end{equation}
see \cite{imbert}.
Stationary surfaces for the fractional mean curvature flow are naturally called fractional minimal surfaces.
We say that  $\Sigma$ is an {\em s-minimal surface} in an open set $\Omega$, if the surface $\Sigma\cap \Omega$ is sufficiently regular, and it
satisfies the {\em nonlocal minimal surface equation}
\begin{equation}
\label{ms1}
H_{\Sigma}^s (p)=0 \quad
\hbox{for all $p\in \Sigma\cap \Omega$.}
\end{equation}

For instance, it is clear by symmetry and definition \equ{1} that a hyperplane is a $s$-minimal surface in $\R^N$ for all $0<s<1$. Similarly,
the {\em Simons cone}
\begin{equation*}
C_m^m= \{(u,v)\in \R^m\times\R^m \ /\  |v|=|u|\}
\end{equation*}
is a $s$-minimal surface in $\R^{2m}\setminus\{0\}$.
As far as we know, no other explicit  minimal surfaces in $\R^N$ have been found in the literature. The purpose of this paper
is to exhibit a new class of non-trivial examples. The hyperplane is not just a minimal surface but also established in \cite{caffarelli-roquejoffre-savin} to be {\em locally area minimizing} in a sense that we describe next.

\medskip
Caffarelli,  Roquejoffre and Savin introduced in \cite{caffarelli-roquejoffre-savin} a nonlocal notion of surface area of $\Sigma =\pp E$ whose Euler-Lagrange equation corresponds to equation \equ{ms1}.  For $0<s<1$,
the $s$-perimeter of a measurable set $E \subset\R^N$ is defined as
$$
\I_s(E) = \int_E \int_{ E^c}\frac {dx\, dy}{ |x-y|^{N+ s} }.
$$
The above  quantity corresponds to a total interaction between points of $E$ and $E^c$, where the interaction density $1/{|x-y|^{N+ s} }$
is largest possible when the points $x\in E$ and $y\in E^c$ are both close to a given point of the boundary. $\I_s(E) $ has a simple representation in terms of the
usual semi-norm  in the fractional Sobolev space $H^{\frac s2}(\R^ N)$.
In fact,
\be
\I_s(E) = [ \chi_E ]_{H^{\frac s2}(\R^N)}  := \int_{\R^N} \int_{\R^N}  \frac { (\chi_E(x) - \chi_E(y) )^2 }{ |x-y|^{N+ s} }dx\, dy.
\label{pers2}\ee
Alternatively,
we can also write
$$
\I_s(E)   =  [ \chi_E ]_{W^{s,1}(\R^N)} =  \int_{\R^N} \int_{\R^N}  \frac { |\chi_E(x) - \chi_E(y) | }{ |x-y|^{N+ s} }dx\, dy.
$$
If $E$ is an open set and $\Sigma =\pp E$ is a smooth bounded surface
we have that
$$(1-s)\I_s(E) \to   c_N {\mathcal H}^{N-1} (\Sigma) = \int_{\R^N}  |\nn \chi_E|  $$
where the latter equality is classically understood in the sense of functions of bounded variation. $\I_s$ can also be achieved as the $\Gamma$-limit as $\ve \to 0$
of the nonlocal Allen-Cahn phase transition
functional $\int \frac \ve 2 |\nn^{\frac s2} u|^2 + \frac 1 {4\ve} (1-u^2)^2$ along functions that $\ve$-regularize  $\chi_E -\chi_{E^c}$.  See \cite{sv2,v2}.

\medskip
This nonlocal notion of perimeter  is localized to a bounded open set $\Omega$ by taking away the contribution of points of $E$ and $E^c$ outside $\Omega$, formally setting
 $$\I_s(E,\Omega) = \int_E \int_{ E^c}\frac {dx\, dy}{ |x-y|^{N+ s} }  - \int_{E\cap\,\Omega^c} \int_{ E^c \cap\,\Omega^c}  \frac {dx\, dy}{ |x-y|^{N+ s} }. $$
 This quantity makes sense, even if the last two terms above are infinite, by rewriting it in the form
$$
\I_s(E,\Omega) =  \int_{E\cap\,\Omega} \int_{E^c} \frac {dx\, dy}{ |x-y|^{N+ s} }  +  \int_{E\cap\, \Omega^c}
\int_{ E^c\cap\, \Omega} \frac {dx\, dy}{ |x-y|^{N+ s} }  .
$$
Again, if  $E$ is an open set with $\Sigma\cap \Omega$ smooth, $\Sigma=\pp E $. The usual
notion of perimeter is recovered by the relation
\begin{equation*}
\lim_{s\to 1}  (1-s) \I_s(E,\Omega) =  c_N {\mathcal H}^{N-1} (\Sigma \cap \Omega),
\end{equation*}
see \cite{savin-valdinoci}.  Let $h$ be a smooth function on $\Sigma$ supported in $\Omega$, and $\nu$ a normal vector field to $\Sigma$ exterior to $E$. For a sufficiently small number
$t$ we let  $E_{th}$ be the set whose boundary  $\pp E_{th} $ is parametrized as
$$
\pp E_{th} =  \{ x+ th(x) \nu(x) \ /\ x\in \pp E \}.
$$
 The first  variation of the perimeter along these normal perturbations yields precisely
$$
\frac d{dt} \I_s( E_{th}, \Omega) \Big |_{t=0}   =  - \int_{\Sigma} H^s_{\Sigma} h
 $$
and this quantity vanishes for all such $h$
if and only if \equ{ms1} holds. Thus $\Sigma = \pp E$ is an $s$-minimal surface in $\Omega$ if the first variation of perimeter for normal perturbations of $E$ inside $\Omega$
is identically equal to zero.

 \medskip
If $\Sigma=\pp E$ is a nonlocal minimal surface the
 {\em second variation} of the $s$-perimeter in $\Omega$ can be computed as
\begin{equation}
\label{j1}
\frac {d^2}{dt^2} Per_{s} (E_{th},\Omega) \Big|_{t=0} \ =\ -2 \int_\Sigma  \JJ^s_\Sigma [h]\, h.
\end{equation}
We call  $\JJ^s_\Sigma [h]$ the {\em fractional Jacobi operator}. It is explicitly computed as

\begin{align}
\label{nonlocal jac}
 \JJ^s_\Sigma [h](p)
=
\int_{\Sigma}
\frac{h(x)-h(p)}{|p-x|^{N+s}} dx
+
h(p)
\int_\Sigma
\frac{\langle \nu (p)- \nu(x), \nu (p) \rangle }{|p-x|^{N+s}} dx,
\quad  p \in \Sigma ,
\end{align}
where the first integral is understood in a pricipal value sense.
In agreement with formula \eqref{j1}, we say that an  $s$-minimal surface $\Sigma$ is {\em stable} in $\Omega$ if
\begin{equation*}
-\int_\Sigma  \JJ^s_\Sigma [h]\, h\, \ge \, 0\foral h\in C_0^\infty(\Sigma\cap\Omega).
\end{equation*}
Naturally we get the correspondence between this nonlocal operator and the usual Jacobi operator
\begin{align}
\label{asym jac}
\lim_{s\to 1} (1-s)\JJ^s_\Sigma [h]  = c_N  \JJ_\Sigma [h], \quad \JJ_\Sigma [h] = \Delta_\Sigma h  + |A_\Sigma|^2 h
\end{align}
where $\Delta_\Sigma$ is the Laplace-Beltrami operator and $|A_\Sigma|^2 =\sum_{i=1}^{N-1} k_i^2 $ where the $k_i$ are the principal curvatures.

\medskip
A basic example of a stable fractional minimal surface $\Sigma =\pp E$  is a {\em fractional minimizing surface}.
In \cite{caffarelli-roquejoffre-savin} the existence of fractional perimeter-minimizing sets is proven in the following sense:
let $\Omega$ be a bounded domain with Lipschitz boundary, and $E_0\subset \Omega^c$ a given set. Let $\mathcal F$ be the class of all sets $F$ with $F\cap\Omega^c = E_0$.  Then
  there exists a set $E\in \mathcal F$ with
$$
\I_s(E,\Omega)\ = \ \inf_{F\in \mathcal F} \I_s(F,\Omega) .
 $$
Moreover, $\pp E\cap \Omega$ is a $(N-1)$-dimensional set, which is a surface of class $C^{1,\alpha}$ except possibly on a singular set of Hausdorff dimension at most $N-2$.
Minimizers $E$ are proven to satisfy in a viscosity sense the { fractional minimal surface equation} \equ{ms1}. In fact, a hyperplane is minimizing in the above sense inside any bounded set.
No other example of embedded smooth fractional minimal surface in $\R^N$ (minimizing or not) is known.

\medskip

\subsection{Axially symmetric $s$-minimal surfaces}
After a plane, next in complexity  in $\R^3$  is the {\em axially symmetric case}, namely the case of a surface of revolution around the $x_3$-axis. In the classical case, the minimal surface equation reduces to  a simple ODE from which the catenoid $C_1$ is obtained:
$$
C_1= \{ (x_1,x_2,x_3)\ /\
|x_3| = \log ( r+ \sqrt{r^2 -1}), \quad r=\sqrt{x_1^2 + x_2^2}> 1\}.
$$
A main purpose of
this paper is the construction of an axially symmetric $s$-minimal surface $C_s$ for $s$close to 1 in such a way that $C_s\to C_1$ as $s\to 1$ on bounded sets. We call this surface
the {\em fractional catenoid}. A striking feature of the surface of revolution $C_s$ is that it becomes at main order as $r\to\infty$ a cone with small slope rather than having logarithmic growth. It is precisely in this feature where the strength of the nonlocal effect is felt.

The usual catenoid $C_1$ cannot be obtained by an area minimization scheme in expanding domains since it is linearly unstable, hence non-minimizing, inside any sufficiently large domain. It is unlikely that $C_s$ can be captured with a scheme based on the results in  \cite{caffarelli-roquejoffre-savin}. In fact, even worse, this is a highly unstable object compared with the classical case: there are
elements in an approximate kernel of its $s$-Jacobi operator that change sign infinitely many times. The Morse index of $C_s$ is infinite in any reasonable sense (unlike the standard
catenoid, whose Morse index is one).

\begin{teo}\label{teo1}{\em  \bf (The fractional catenoid)}
For all $0<s< 1$ sufficiently close to 1 there exists a connected surface of revolution $C_s$ such that if we set $\ve= (1-s)$ then
 $$ \sup_{x\in C_s\cap B(0,2)}  \dist(x,C_1) \le  c \frac { \sqrt{\ve}} { |\log\ve|}, $$
\medskip
and, for $r=\sqrt{x_1^2+ x_2^2}>2$ the set $C_s$ can be described as
$|x_3| = f(r),\quad  $ where
$$ f(r) = \left \{ \begin{matrix}  \log (r + \sqrt{r^2-1}) + O \left (\frac {r \sqrt{\ve}} { |\log\ve|}\right )    & \hbox{ if }\quad  r< \frac 1 {\sqrt{\ve}}\\   r \sqrt{\ve} + O(|\log\ve|) +
O \left (\frac {r  \sqrt{\ve} } { |\log\ve|}\right )  & \hbox{ if }\quad r> \frac 1 {\sqrt{\ve}}. \end{matrix} \right .$$

\end{teo}

\begin{figure}
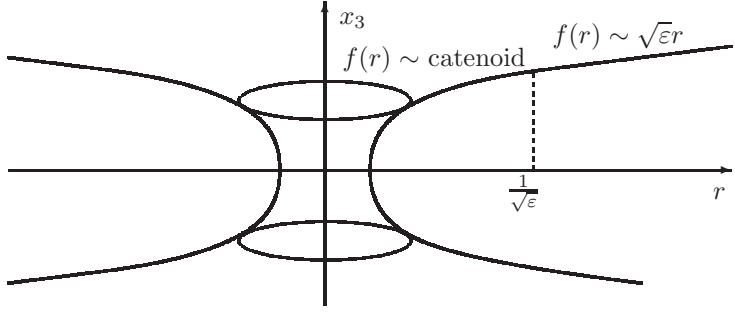

\def\JPicScale{0.6}
\include{fig1}
\label{fig1}
\caption{Fractional catenoid}
\end{figure}

As we have mentioned, a plane is an $s$-minimal surface for any $0<s<1$. In the classical scenario, so is the union of two parallel planes,
say $x_3=1$ and $x_3= -1$. This is no longer the case when $0<s<1$ since the nonlocal interaction between the two components deforms them and in fact equilibria
 is reached when the two components diverge becoming cones. Our second results states the existence
of a {\em two-sheet} nontrivial $s$-minimal surface $D_s$ for $s$ close to 1 where the components eventually become at main order the cone $x_3= \pm r\sqrt{\ve}$.
As in the $s$-catenoid, this is a highly unstable object.

\begin{teo}\label{teo2}{\em  \bf (The two-sheet $s$-minimal surface)}
For all $0<s< 1$ sufficiently close to 1 there exists a two-component surface of revolution $D_s = D_s^+\cup D_s^- $ such that if we set $\ve= (1-s)$ then
$D_s^\pm$ is the graph of the radial functions $x_3 =\pm f(r)$ where $f$ is a positive  function of class $C^2$ with $f(0)=1$, $f'(0)=0$, and

$$ f(r) = \left \{ \begin{matrix}  1+ \frac{\ve}4 r^2  + O \left (\ve r \right )    & \hbox{ if }\quad  r< \frac 1 {\sqrt{\ve}}\\   r \sqrt{\ve} + O(1) +
O \left (\ve r \right )  & \hbox{ if }\quad r> \frac 1 {\sqrt{\ve}}. \end{matrix} \right .$$

\end{teo}

\begin{figure}
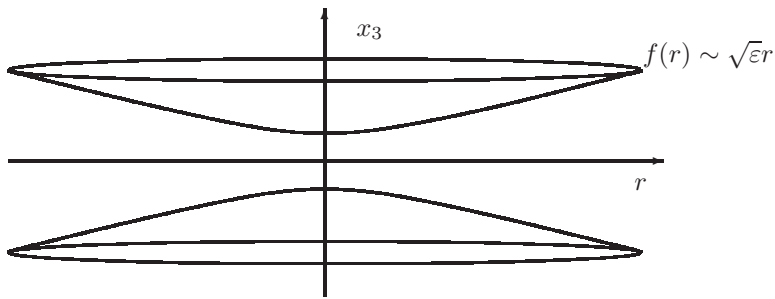

\def\JPicScale{0.6}
\include{fig4}
\label{fig2}
\caption{Two-sheet $s$-minimal surface}
\end{figure}

As we shall discuss later,
Theorem \ref{teo2} can be generalized to the existence of a $k$-sheet axially symmetric $s$-minimal surface constituted by the union of the graphs of $k$ radial functions
$x_3 = f_j(r)$, $j=1,\ldots,k$,
with
$$ f_1>f_2> \cdots > f_k $$
where asymptotically we have
\be\label{ass}
f_j(r) = a_j r \sqrt{\ve}  + O(\ve r) \quad\hbox{as } r\to +\infty. \ee
Here the constants $a_i$ are required to satisfy the constraints
\be
a_1> a_2>\cdots > a_k, \quad   \sum_{i=1}^ k a_i = 0 
\label{const}\ee
and the balancing conditions
\begin{align}
\label{system ai}
a_i = 2 \sum_{j\not=i} \frac{(-1)^{i+j+1}}{a_i-a_j}, \quad  \; \foral i = 1,\ldots,k.
\end{align}

A solution of the system \eqref{system ai} can be obtained by minimization of
$$
E(a_1,\ldots,a_k) = \frac12 \sum_{i=1}^k a_i^2 + \sum_{i\not= j} (-1)^{i+j}\log(|a_i-a_j|)
$$
in the set of $k$-tuples $a=(a_1,\ldots,a_k)$ that satisfy \equ{const}. If this minimizer or, more generally, a critical point $a$ of $E$ constrained to \equ{const}
is non-degenerate, in the sense that $D^2E(a)$ is non-singular, then an $s$-minimal surface with the required properties \equ{ass} can indeed be found. This condition is evidently satisfied
by $a=(1,-1)$ when $k=2$.

\medskip
The method for the proofs of the above results relies in a simple idea of obtaining a good initial approximation $\Sigma_0$ to a solution of the equation
$H_\Sigma =0 $ Then we consider the surface perturbed normally by a small function $h$, $\Sigma_h$. As we will see, regardless that $\Sigma_0$ is a minimal surface or not,
we can expand
$$ H_{\Sigma_h} =  H_{\Sigma_0}  +  J_s[h]  + N(h)$$
where
$N(h)$ is at main order quadratic in $h$. In the classical case, $N(h)$ depends on first and second derivatives of $h$ with various terms that can be qualitatively
described (see \cite{kapouleas}).
We shall see that if the approximation $\Sigma_0$ is properly chosen, in particular so that the error $ H_{\Sigma_0} $ is  small in $\ve =1-s$ and has suitable decay along the manifold,
 then this equation can be solved by a fixed point argument. To do so, we need to identify the functional spaces to set up the problem, that take into account the delicate issues of  non-compactness and strong long range interactions. These spaces should be such that a left inverse of $J_s$ can be found with good transformation properties, and  $N(h)$ has a small Lipschitz dependence for the corresponding norms. The latter issue is especially delicate, for $N(h)$ is made out of various pieces, all strongly singular integral nonlinear operators involving fractional derivatives up to the nearly second order. The transformation properties of these nonlinear terms have suitable analogs with to those found by Kapouleas \cite{kapouleas}, but the proofs in the current situation are considerably harder.

\medskip
 The procedure we set up in this paper, and the associated  computations,  apply in large generality, not just to the axially symmetric case. For instance most of the calculations actually apply to a general setting of finding as $s\to 1$ a connected surface with multiple ends that are eventually conic and satisfy relations \equ{const}, where the starting point is a
  multiple-logarithmic-end minimal surface. This paper sets the basis of the gluing arguments for the construction of fractional minimal surfaces, in a way similar that  the paper
  \cite{kapouleas} did for the construction by gluing methods of classical minimal and CMC surfaces. The fractional scenario makes the analysis considerably harder.

\subsection{Fractional Lawson cones}
The pictures associated to Theorems \ref{teo1} and \ref{teo2} resemble that of "one-sheet" and "two-sheet" revolution hyperboloids, asymptotic to a cone $|x_3| = r\sqrt{1-s}$. It is reasonable to believe that a cone of this form, with aperture close to $\sqrt{1-s}$ is a fractional minimal surface with a singularity at the origin.
We consider, more in general, for given $n,m\ge 1$, and $0<s<1$ the problem of finding a value $\alpha >0$ such that
the {\em Lawson cone}
\be\label{cmn}C_\alpha= \{(u,v)\in \R^m\times\R^n \ /\  |v|=\alpha |u|\}\ee
is a $s$-minimal surface in $\R^{m+n}\setminus\{0\}$. For the classical case $s=1$ this is easy:  since  $\Sigma=C_\alpha$ is the zero level set of the function
 $g(u,v) = |v|-\alpha |u|$
then $(u,v)\in C_\A$ we have
$$
H_\Sigma (u,v) = {\rm div}\, \left ( \frac{ \nn g }{|\nn g|}  \right ) =  \frac 1{\sqrt{1+ \A^2} }\, \left [\, \frac{n-1}{|v|} - \alpha \frac{m-1}{|u|}\right],
$$
and the latter quantity is equal to zero on $\Sigma$ if and only if $n=m=1$ and $\alpha =1$ or
$$  n\ge 2, \ m\ge 2, \quad \alpha = \sqrt{\frac{n-1}{m-1}}.$$ Following \cite{lawson}, we call this one the minimal Lawson cone $C_m^n$. For the fractional situation  we have the
following result.

\begin{teo}\label{teo3}
{\em (Existence of $s$-Lawson cones) }
For any given $m\ge 1$, $n\ge 1$, $0<s<1$, there is a unique $\alpha=\alpha(s,m,n)>0$ such that the cone $C_\alpha$ given by $\equ{cmn}$  is an $s$-fractional minimal surface.
We call this $C_m^n (s)$ the $s$-Lawson cone.

\end{teo}

A notable different between classical and nonlocal cases is that in the latter, a nontrivial minimal cone in $\R^n$
$$C_1^{n-1}(s)\, =\, \left \{ (x',x_{n}) \in \R^{n}\ /\ |x_{n}|  = \alpha_n(s) |x'|\, \right \}, $$
with $n\ge 3$ does exist. This is not true in the classical case. The bottomline is that when aperture becomes very large ($\alpha$ small), in the standard case mean curvature
approaches 0, while the nonlocal interaction between the two pieces of the cone makes its fractional mean curvature go to $-\infty$.
For $n=2$, $C_1^2(s)$  is precisely the $s$-minimal cone that represents at main order the asymptotic behavior of the revolution $s$-minimal surfaces of Theorems \ref{teo1} and \ref{teo2}. Letting
$\ve = 1-s\to 0$,
we have, as suspected $$\alpha_2(s)= \sqrt{\ve} + O(\ve), $$ so that the two halves of the minimal cone become planes. In the opposite limit, $s\to 0$, there is no collapsing.
In fact, if $n\le m$ we have
$$
\lim_{s\to 0} \alpha(s,m,n) = \alpha_0
$$
where $\alpha_0>0 $ is the unique number $\alpha $ such that
$$
\int_{\alpha}^\infty
\frac{t^{n-1}}
{(1 + t^2  )^\frac{m+n}2}
d t - \int_0^{\alpha}
\frac{t^{n-1}}
{(1 + t^2  )^\frac{m+n}2} dt = 0.
$$
An interesting analysis of asymptotics for the fractional perimeter $\I_s$ and assoaciated  $s$-minimizing surfaces as $s\to 0$ is contained in \cite{figalli}.

\medskip
Minimal cones are important objects  in the regularity theory of classical minimal surfaces and Bernstein type results for minimal graphs.  Simons \cite{simons} proved that {\em no stable minimal cone exists in dimension $N\le 7$, except for hyperplanes}.  This result  implies that locally area minimizing surfaces must be smooth outside a closed set of Hausdorff dimension at most $N-8$. He also proved that the cone $C_4^4$ (Simons' cone) was stable, and conjectured its minimizing character. This  was proved in a deep work by Bombieri, De Giorgi and Giusti \cite{bdg}.

 \medskip
Savin and Valdinoci \cite{savin-valdinoci} proved the nonexistence of  fractional minimizing cones in $\R^2$, which implies regularity of fractional minimizing surfaces except
for a set of Hausdorff dimension at most $N-3$, thus improving the original result in \cite{caffarelli-roquejoffre-savin}. Figalli and Valdinoci \cite{fv} prove that, in every dimension, Lipschitz nonlocal minimal surfaces are smooth, see also \cite{bfv}. Also, They extend to the nonlocal setting a famous theorem of De Giorgi stating that the validity of Bernstein's theorem as a consequence of the nonexistence of singular minimal cones in one dimension less.

\medskip
In \cite{caffarelli-valdinoci}, Caffarelli and Valdinoci  proved that regularity of non-local minimizers holds up to a $(N-8)$-dimensional set, whenever $s$ is sufficiently close to $1$. Thus,  there remains a conspicuous gap between the best general regularity result found so far and the case $s$ close to 1.
Our second results concerns this issue. Its most interesting feature is that, in strong contrast with the classical case, when $s$ is sufficiently close to zero,  Lawson cones {\bf are all  stable} in dimension $N=7$, which suggests that a regularity theory up to a $(N-7)$-dimensonal set should be the best possible for general $s$.


\begin{teo} {\em (Stability of $s$-Lawson cones) }
\label{thm stability}
There is a $s_0>0$ such that  for each $s\in (0,s_0)$, all minimal cones $C_m^n(s)$ are  unstable if $N= m+ n\leq 6$ and
stable if  $N= 7$.
\end{teo}



\medskip
Besides the reults in \cite{simons,bdg}, we remark that
for $N>8$ the cones $C_m^n$ are all area minimizing. For $N=8$ they are area minimizing if and only if $|m-n|\le 2$.
These facts were established by Lawson \cite{lawson}  and Simoes \cite{simoes},
{\cb see also \cite{miranda,concus-miranda,benarros-miranda,davini}.}

\medskip
The rest of this paper will be devoted to the proofs of Theorems \ref{teo1}--\ref{thm stability}.
The proof of Theorem~\ref{teo2} is actuallly a simpler variation of that of Theorem~\ref{teo1}. We will just concentrate in the proof of Theorem~\ref{teo1}, whose  scheme we explain in Section~\ref{sect scheme}. There we  shall
isolate the main steps in the form of intermediate results  which we prove in the subsequent sections.
The proofs of Theorems~\ref{teo3} and \ref{thm stability} rely on explicit computations of singular integral quantities, and are carried out in Sections~\ref{sect exist unique} and \ref{sect stability}.

We leave for the Appendix self contained proofs of asymptotic formulas \eqref{mc}, \eqref{asym jac} in Section~\ref{sect asymptotics},  and the computation of first and second variations of the $s$-perimeter in Section~\ref{sect jacobi}.

\medskip




\section{Scheme of the proof of Theorem \ref{teo1}}
\label{sect scheme}

In this section we shall outline the proof of Theorem 1, isolating the main steps whose proofs are delayed to later sections.
We look for a set $E \subseteq \R^3$ with smooth $\Sigma=\pp E $ such that
\be
H_\Sigma^s (x) :=
\int_{\R^3} \frac{\chi_E(y) - \chi_{E^c}(y)}{|x-y|^{3+s}}\, d y = 0 ,
\foral x \in \Sigma
\label{xx}\ee
where $0<s<1$, $1-s$ is small and the integral is understood in a principal value sense sense.

We look for $E$ in the form of a solid of revolution around the $x_3$-axis.
More precisely,  let us represent points in space by
$x= (x',x_3)$ with $x'\in \R^2$,  and denote $r=|x'|$.
We shall construct a first approximation for $E$  of the form
\begin{align}
\label{e1}
E_0 = \{  \ x = (x',x_3) \in \R^2 \times \R: |x'| <R \text{ or } |x'|\geq R, |x_3|>f(x) \ \},
\end{align}
where $f$ is a positive and increasing function on $[R,\infty)$.


From now on we let $\ve = 1-s$.
As we will demonstrate later, for an appropriate class of sets $E$ equation \equ{xx}  formally resembles
\begin{align}
\label{xx1}
- 2H_{\Sigma}(x) +  \frac {\ve}  {|x_3|} =0 .
\end{align}

We will obtain the surface $\Sigma$ and the corresponding set $E$ by first constructing an initial surface $\Sigma_0 = \partial E_0$ that is an approximate solution of  \equ{xx1} and then perturbing it.

For the construction of $\Sigma_0$ we take the standard catenoid parametrized as
$$
|x_3| = f_C(r) ,\quad r=|x'| \geq 1,
$$
where
\begin{align}
\label{catenoid}
f_C(r) = \log ( r + \sqrt{r^2-1}), \quad r\geq 1.
\end{align}
If we describe $\Sigma= \partial E$ with $E$ as in \eqref{e1}
and assume that for $r$ large  $f'(r)$ is small, then equation \eqref{xx1} is approximated by
\begin{align}
\label{eqf}
\Delta f = \frac{\ve}{f} .
\end{align}
This motivates us to define $f_\ve (r)$ as solution of the initial value problem
\begin{align}
\label{f eps}
\left\{
\begin{aligned}
& f_\ve'' + \frac1r f_\ve'  =  \frac {\ve}  {f_\ve} , \quad r > \ve^{-\frac 12}
\\
&f_\ve (\ve^{-\frac 12}) =  f_C(\ve^{-\frac 12}), \quad f_\ve' (\ve^{-\frac 12}) =  f_C'(\ve^{-\frac 12}).
\end{aligned}
\right.
\end{align}
Let
\begin{align}
\label{def Feps}
F_\ve(r) := f_C(r) +  \eta( r- \ve^{-\frac 12} ) ( f_\ve (r) - f_C(r)) ,
\quad r\geq 1,
\end{align}
where $\eta \in C^\infty(\R)$ is a cut-off function with
\begin{align}
\label{def eta1}
\eta(t) = 0 \quad\text{for } t< 0, \quad
\eta(t) = 1 \quad\text{for } t>1.
\end{align}
We define  the  surface $\Sigma_0$ by
\begin{align}
\label{def sigma0}
\Sigma_0 = \{ |x_3| =   F_\ve(r), r\geq 1 \}.
\end{align}
Then
$$
\Sigma_0 = \partial E_0, \quad
E_0 = \{ r<1, \text{ or } r\geq 1 \text{ and } |x_3|\geq F_\ve(r) \} .
$$

Next we perturb the surface $\Sigma_0$ in the normal direction. For this, let $\nu_{\Sigma_0}(x)$ be the unit normal vector field on $\Sigma_0$  such that $\nu_3(x)x_3 \ge 0$. We consider a  function $h$ defined on $\Sigma_0$, and define

$$
\Sigma_h = \{ x +  h(x)\nu_{\Sigma_0}(x)\ /\ x\in \Sigma_0\} .
$$
If $h$ is small in a suitable norm, then $\Sigma_h$ is an embedded surface that can be written as $\Sigma_h = \partial E_h$ for a set $E_h$ that is close to $E_0$.
We can expand, for a point $x\in \Sigma_0$ and $x_h = x + h (x) \nu_{\Sigma_0}(x)$:
\begin{align}
\label{eq29}
H^s_{\Sigma_h}(x_h) = H^s_{\Sigma_0}(x) + 2 \JJ^s_{\Sigma_0}(h)(x) + N(h)(x) ,
\end{align}
where $\JJ^s_{\Sigma_0}$ is the {\em nonlocal Jacobi operator} given by
\begin{align*}
\JJ^s_{\Sigma_0}(h)(x)
=
\int_{\Sigma_0}
\frac{h(y)-h(x)}{|x-y|^{3+s}} dy
+ h(x)
\int_{\Sigma_0}
\frac{\langle \nu_{\Sigma_0} (x)- \nu_{\Sigma_0}(y), \nu_{\Sigma_0} (x) \rangle }{|x-y|^{3+s}} dy,
\end{align*}
for $ x \in \Sigma_0$, and $N(h)$ is defined by equality \eqref{eq29}.

The objective is then to find $h$ such that
\begin{align}
\label{eq30}
H^s_{\Sigma_0} + 2 \JJ^s_{\Sigma_0}(h) + N(h) = 0 .
\end{align}
We note that, assuming $h$ is smooth and bounded,
$$
\text{p.v.}
\int_{\Sigma_0}
\frac{h(y)-h(x)}{|x-y|^{3+s}} dy
= \frac1\ve \frac\pi2 \Delta_{\Sigma_0}h(x) + O(1)
$$
as $\ve\to0$, where $\Delta_{\Sigma_0}$ is the Laplace-Beltrami operator on $\Sigma_0$ (see Lemma~\ref{conv lapl}).
Therefore it is more convenient to rewrite \eqref{eq30} as
$$
\ve H^s_{\Sigma_0} + 2 \ve \JJ^s_{\Sigma_0}(h) + \ve N(h) = 0 \quad\text{in }\Sigma_0.
$$

It is natural to expect that $h$ has linear growth, and therefore we will work with weighted H\"older norms allowing such behavior.
For $0<\alpha<1$ and $\gamma\in\R$, we define norms for functions defined on $\Sigma_0$ or $\R^2$ as follows:
\begin{align*}
\nonumber
[f]_{\gamma,\alpha}
& =
\sup_{ x\not=y} \, \min(1+|x|,1+|y|)^{\gamma+\alpha} \frac{|f(x) - f (y)|} {|x-y|^\alpha},\\
\|f\|_{\gamma,\alpha} &= \| (1+|x|)^\gamma f\|_{L^\infty} + [f]_{\gamma,\alpha},
\end{align*}
and
\begin{align}
\label{norm st}
\| h \|_* = \| (1+|x|)^{-1} h \|_{L^\infty} + \| \nabla h \|_{L^\infty} + \| (1+|x|) D^2 h\|_{L^\infty} + [D^2 h]_{1,\alpha}.\end{align}
Then we look for a solution $h$ of \eqref{eq30} with  $\|h\|_*<\infty$ and measure $\ve \JJ_{\Sigma_0}^s(h)$ in the norm
\begin{align}
\label{norm RHS}
\| f \|_{1-\ve,\alpha+\ve}
&= \| (1+|x|)^{1-\ve} f\|_{L^\infty} + [f]_{1-\ve,\alpha+\ve}
\end{align}
More explicitly,
\begin{align*}
\| f \|_{1-\ve,\alpha+\ve}
&= \| (1+|x|)^{1-\ve} f\|_{L^\infty} + \sup_{x\not=y} \min(1+|x|,1+|y|)^{1+\alpha} \frac{|f(x)-f(y)|}{|x-y|^{\alpha+\ve}} .
\end{align*}

An outline of the proof of Theorem~\ref{teo1} is the following.
In Section~\ref{sect error}, using estimates for $f_\ve$ obtained in Section~\ref{sect ode},
we will prove:
\begin{prop}
\label{prop error}
For $\ve>0$ sufficiently small we have
\begin{align}
\label{estimate E}
\| \ve H^s_{\Sigma_0} \|_{1-\ve,\alpha+\ve} \leq
\frac{C\ve^{\frac12}}{|\log\ve|} .
\end{align}
\end{prop}

The next result is about invertibility of the operator $ \ve \JJ^s_{\Sigma_0} $ on a weighted H\"older space.

\begin{prop}
\label{main linear prop}
There is a linear operator that to a function $f$ on $\Sigma_0$ such that $f$ is radially symmetric and symmetric with respect to $x_3=0$ with $\|f\|_{1-\ve,\alpha+\ve}<\infty$, gives a solution $\phi$ of
\begin{align*}
\ve\mathcal J^s_{\Sigma_0}(\phi) = f \quad \text{in } \Sigma_0 .
\end{align*}
Moreover $\phi$ has the same symmetries as $f$ and
\begin{align*}
\|\phi\|_* \leq C \|f\|_{1-\ve,\alpha+\ve}.
\end{align*}

\end{prop}
The proof is given in Section~\ref{sect linear}, based on preliminaries in Sections~\ref{sect linear1} and \ref{sect linear2}.

In Section~\ref{sect q} we obtain the estimate
\begin{prop}
\label{prop Nh}
There is $C$ independent of $\ve>0$ small such that for $\|h_i\|_* \leq \sigma_0 \ve^{\frac12}$, $i=1,2$ we have
\begin{align}
\label{est N1}
\ve\| N(h_1) - N(h_2)\|_{1-\ve,\alpha+\ve}\leq C \ve^{-\frac12} (\|h_1\|_*+\|h_2\|_* ) \|h_1-h_2\|_*.
\end{align}
\end{prop}
Here $\sigma_0>0$ is small and fixed.

With these results we can give a

\begin{proof}[Proof of Theorem~\ref{teo1}]
We need a solution $h$ to \eqref{eq30}
which we look for in the Banach space
$$
X = \{ h \in C^{2,\alpha}_{loc}(\Sigma_0) , \ \|h\|_* <\infty\},
$$
with norm  $\|\  \|_* $. Consider also the Banach space
$$
Y = \{ f \in C^{\alpha+\ve}_{loc} , \ \|f\|_{1-\ve,\alpha+\ve} <\infty\},
$$
with norm $\| \ \|_{1-\ve,\alpha+\ve}$.
In both spaces we restrict functions to be axially symmetric and symmetric with respect to $x_3 = 0$.

Let $T$ be the linear operator constructed in  Proposition~\ref{main linear prop}. Then we reformulate  \eqref{eq30} as
$$
2 h = A(h) := T ( -\ve H^s_{\Sigma_0} - \ve N(h) ) .
$$
We claim that for $\ve >0$ small, $A$ is a contraction on the ball
$$
B = \{ h\in X : \|h\|_* \leq M \frac{\ve^{\frac12}}{|\log\ve|} \} ,
$$
if we choose $M$ large.
Indeed,  for $h\in B$, by \eqref{estimate E} and \eqref{est N1}
\begin{align*}
\|A(h)\|_* &\leq C \|\ve H^s_{\Sigma_0}\|_{1-\ve,\alpha+\ve} + C \| \ve N(h)\|_{1-\ve,\alpha+\ve}\\
&\leq  \frac{\ve^{\frac12}}{|\log\ve|} ( C+ \frac{M^2}{|\log\ve|} )
\leq M  \frac{\ve^{\frac12}}{|\log\ve|} ,
\end{align*}
if we take $M=2C$  then let $\ve>0$ be small.
Next, for $h_1$, $h_2\in B$,
\begin{align*}
\|A(h_1)-A(h_2)\|_* \leq C \ve^{-\frac12} ( \|h_1\|_* + \|h_2\|_*) \|h_1 - h_2\|_*.
\end{align*}
But  $\ve^{-\frac12} ( \|h_1\|_* + \|h_2\|_*)  \leq \frac{C}{|\log\ve|} $ and so $A$ is a contraction on $B$ for $\ve>0$ small.
\end{proof}


\section{The ODE of the initial approximation}
\label{sect ode}
The purpose of this section is to analyze the solution $f_\ve(r)$  of \eqref{f eps}, which is used in the construction of the initial approximation. Thanks to \eqref{catenoid} we have
\begin{align}
\label{initial cond2}
\left\{
\begin{aligned}
f(\ve^{-\frac12})
& = C(\ve^{-\frac12}) = \frac12 |\log\ve| + \log 2 + O(\ve)
\\
f'(\ve^{-\frac12})
& = C'(\ve^{-\frac12}) = \sqrt \ve ( 1 + O(\ve)) .
\end{aligned}
\right.
\end{align}
Note that $f_\ve'(r)\geq 0$ so in particular
\begin{align}
\label{lower bound}
f_\ve(r)\geq f_\ve(\ve^{-\frac12})
\quad \text{for all } r \geq r^{-\frac12}.
\end{align}
\begin{lemma}
\label{lemma 2.1}
We have
\begin{align}
\label{16}
C_1 |\log\ve| \leq |f_\ve(r)|\leq C_2 |\log\ve| , \quad
|f_\ve'(r)|\leq C \ve^{\frac12}
\end{align}
$$
|f_\ve''(r)|\leq  \frac{C}{r^2} + \frac{C \ve}{|\log\ve|^2}
$$
for $\ve^{-\frac12} \leq r \leq |\log\ve|\ve^{-\frac12}$.
\end{lemma}
\begin{proof}
We make the change of variables
$$
f_\ve(r) = |\log\ve| \tilde f(\ve^{\frac12} r),
$$
and then $\tilde f$ satisfies
\begin{align}
\label{11}
\Delta \tilde f = \frac{1}{|\log\ve|^2 \tilde f},
\end{align}
for $r\geq 1$, with initial conditions
\begin{align}
\label{initial tilde f}
\tilde f(1) & = \frac12 + O(\frac{1}{|\log\ve|}) ,
\qquad
\tilde f'(1) = \frac{1+O(\ve)}{|\log\ve|} .
\end{align}
Integrating once \eqref{11} we get
\begin{align}
\label{12}
r \tilde f'(r) - \tilde f'(1) = \frac{1}{|\log\ve|^2}
\int_1^r \frac{s}{\tilde f(s)} \, d s
\end{align}
for $r\geq 1$. By \eqref{lower bound}
\begin{align}
\label{13}
\tilde f(r)\geq \frac12 + O(\frac{1}{|\log\ve|})
\quad\text{for } r\geq 1.
\end{align}
Therefore from \eqref{12} and \eqref{13} we obtain
$$
\tilde f'(r) \leq \frac 1r
\left(
\frac C{|\log\ve|}
+
\frac {C r^2}{|\log\ve|^2}
\right)
\quad\text{for }r\geq 1.
$$
This implies
$$
\tilde f'(r) \leq \frac{C}{|\log\ve|},
\quad \text{for } 1 \leq r \leq |\log\ve|,
$$
and using  \eqref{initial tilde f} also
$$
\tilde f(r) \leq C ,
\quad\text{for } 1\leq r \leq |\log\ve|.
$$
To estimate $f_\ve''$ we note that
\begin{align*}
|\tilde f_\ve''(r)|
&\leq \frac1r |\tilde f'(r)| + \frac{1 }{|\log\ve|^2 \tilde f}
\\
&\leq
\frac{C}{r^2} + \frac{C }{|\log\ve|^2}
\quad\text{for } r\geq 1.
\end{align*}
\end{proof}

We study now the asymptotic behavior of $f_\ve(r)$ as $r\to \infty$.
For this let us write
\begin{align}
\label{eq f0}
f_\ve(r) = |\log\ve| f_0^{(\ve)}(\frac{ \ve^{\frac12} }{|\log\ve|} r) ,\quad
\quad \text{for } r \geq  \frac{1}{|\log\ve|} ,
\end{align}
for a new function $f_0^{(\ve)}$.
Then $f_0^{(\ve)}$ satisfies
$$
\Delta f_0^{(\ve)} = \frac{1}{f_0^{(\ve)}} \quad \text{for } r \geq \frac{1}{|\log\ve|}
$$
and from \eqref{initial cond2}
\begin{align*}
f_0^{(\ve)}( \frac{1}{|\log\ve|} )
& = \frac12 + \frac{\log 2}{|\log\ve|} + O(\frac{\ve}{|\log\ve|})
\\
[ f_0^{(\ve)} ] ' ( \frac{1}{|\log\ve|} )
& =  1 + O(\ve) ,
\end{align*}
as $\ve\to0$.
\begin{lemma}
\label{lemma est f0}
For any $r_0>0$ and $\ve>0$ small there is $C$ such that
\begin{align*}
& |f_0^{(\ve)}(r) -r| \leq C ,
\qquad
|[f_0^{(\ve)}]'(r)-1|  \leq \frac Cr ,
\\
& |[f_0^{(\ve)}]''(r)|  \leq \frac Cr
\end{align*}
for all $r\geq r_0$.
\end{lemma}
\begin{proof}
Let us introduce the
Emden-Fowler change of variables
\begin{align}
\label{def psi}
f_0^{(\ve)} (r) = r \psi_\ve(t)
,
\qquad \text{where }r = e^t
\end{align}
for $t\geq  - \log|\log\ve|$. Then $\psi_\ve(t)>0$ and
\begin{align}
\label{eq psi}
\psi_\ve''+2\psi_\ve'+\psi_\ve = \frac{1}{\psi_\ve}
\quad \text{for } t\geq - \log|\log\ve|.
\end{align}
Let
$$
G_\ve(t) =
\frac12 (\psi_\ve')^2 + \frac12 \psi_\ve^2  - \log\psi_\ve
-\frac12
$$
and note that
\begin{align}
\label{energy decreasing}
G_\ve'(t)
= -2(\psi_\ve')^2 \leq 0 .
\end{align}
Using \eqref{16}
we see that $\psi_\ve(0) = O(1)$ and $\psi_\ve'(0)=O(1)$ as $\ve\to0$ and this implies that $G_\ve(0) = O(1)$ as $\ve\to 0$. Then by \eqref{energy decreasing} $G_\ve(t) \leq C$ for all $t\geq 0$ and all $\ve>0$ small.
This implies
that
\begin{align}
\label{simple bounds}
0< a\leq \psi_\ve(t) \leq b<\infty,
\quad
|\psi_\ve'(t)|\leq C
\quad\text{for all } t\geq 0 ,
\end{align}
and all $\ve>0$ small, for some uniform constants $0<a<b$ and $C>0$.

From \eqref{energy decreasing}
$$
\int_{0}^t \psi_\ve'(s)^2 \, d s = 2 G_\ve(0) - 2 G_\ve(t)
\leq C
$$
with $C$ independent of $\ve$ and $t\geq 0$. From this we see that
\begin{align}
\label{finite integral}
\int_{0}^\infty \psi_\ve^2(s) \, d s \leq C
\end{align}
with $C$ independent of $\ve$.
Using interpolation estimates (or elliptic estimates) for the equation for $Z_\ve = \psi_\ve'$:
$$
Z_\ve'' + 2 Z_\ve' + Z_\ve\Big(1+\frac1{\psi_\ve^2}\Big) = 0
$$
we have
$$
|\psi_\ve'(t)|=
|Z_\ve(t)|\leq C \Big( \int_{t-1}^{t+1} Z_\ve(s)^2 \, d s\Big)^{1/2}
\to 0
$$
as $t\to \infty$, by \eqref{finite integral}. We claim the convergence is uniform and exponential. To see this, define
$$
E_{2,\ve} = \frac12 (\psi_\ve'')^2 + \frac12 (\psi_\ve')^2
\Big(1+\frac{1}{\psi_\ve^2}\Big).
$$
Then
$$
E_{2,\ve}' = - 2 (\psi_\ve'')^2 - \Big(\frac{\psi_\ve'}{\psi}\Big)^3 .
$$
For $\alpha>0$ to be fixed later on consider
$$
\tilde G_\ve = \alpha E_{\ve} + E_{2,\ve} .
$$
Then
$$
\tilde G_\ve' =  - 2 (\psi_\ve'')^2 - \Big(\frac{\psi_\ve'}{\psi}\Big)^3 - 2\alpha (\psi_\ve')^2 .
$$
But by \eqref{simple bounds}
$$
-\Big(\frac{\psi_\ve'}{\psi}\Big)^3
\leq C (\psi_\ve')^2
$$
so that
$$
\tilde G_\ve' \leq -  2 (\psi_\ve'')^2 - (2\alpha-C) (\psi_\ve')^2.
$$
At this point we choose $\alpha$ so that $2\alpha = C$. We then obtain
\begin{align}
\label{Etilde prime}
\tilde G_\ve' \leq - (\psi_\ve'')^2 - (\psi_\ve')^2.
\end{align}
Using \eqref{eq psi} we note that for some $A>0$
\begin{align}
\label{ineq psi}
\Big(\frac1\psi_\ve -\psi_\ve\Big)^2
\leq 2 ( (\psi_\ve'')^2 + A (\psi_\ve')^2 ) .
\end{align}
Using again \eqref{simple bounds}
\begin{align*}
\tilde G_\ve
& =
\alpha \left[
\frac12 (\psi_\ve')^2 + \frac12 \psi_\ve^2  - \log\psi_\ve
-\frac12
\right]
+
\frac12 (\psi_\ve'')^2 + \frac12 (\psi_\ve)'^2
\Big(1+\frac{1}{\psi_\ve^2}\Big).
\\
& \leq C \left( (\psi_\ve'')^2 + (\psi_\ve')^2 + \Big(\frac1\psi_\ve -\psi_\ve \Big)^2\right)
\end{align*}
Combining \eqref{Etilde prime}, \eqref{ineq psi} and the last estimate we see that
$$
\tilde G_\ve \leq - C \tilde G_\ve' .
$$
This implies that
$$
\tilde G_\ve(t) \leq C e^{-\delta t}
\quad\text{for all }t \geq 0,
$$
for some constants $C$, $\delta>0$ independent of $\ve>0$ small.
From this we obtain
$$
|\psi_\ve'(t)| + |\psi_\ve(t)-1|
\leq C e^{-\delta t/2},
\quad\text{for all }t \geq 0.
$$
Then, after a fixed  $t_1$ independent of $\ve$, the point $\psi_\ve(t_1),\psi_\ve'(t_1)$ is sufficiently close to $(1,0)$.
Let
$$
v_1 = \frac1{\psi}, \quad
v_2 = \psi+\psi' .
$$
Then \eqref{eq psi} is equivalent to
\begin{align}
\label{system-ode}
\begin{aligned}
v_1 ' &= v_1 - v_1^2 v_2
\\
v_2' &= v_1 + v_2 .
\end{aligned}
\end{align}
For $t_1$ sufficiently large the point $(v_1(t_1), v_2(t_1))$ is sufficiently close to $(1,1)$, which is a hyperbolic stationary point of \eqref{system-ode}. The eigenvalues of the linearization at $(1,1)$ are $-1\pm i$ so that by applying a $C^1$ conjugacy to the linearization at $(1,1)$ we obtain
$$
| (v_1(t), v_2(t)) -(1,1) | \leq  C e^{-t}
\quad\text{for all }t\geq t_1.
$$
This implies
\begin{align}
\label{bounds psi}
|\psi_\ve'(t)| + |\psi_\ve(t)-1|
\leq C e^{-t},
\quad\text{for all }t \geq 0,
\end{align}
For the function $f_0^{(\ve)}$ we find
$$
|f_0^{(\ve)}(r)-r|\leq C ,
\qquad
|[f_0^{(\ve)}]'(r)-1| \leq \frac Cr
$$
for all $r\geq r_0$, for any $r_0>0$ fixed.
\end{proof}

\begin{coro}
\label{coro prop F}
We have the following properties of $F_\ve$:
\begin{align*}
& F_\ve(r) =
f_C(r) =  \log ( r + \sqrt{r^2-1}) = \log(2r) + O(r^{-2})
, && 1 \leq r \leq \ve^{-\frac12} ,
\\
& C_1|\log\ve|  \leq F_\ve(r) \leq C_2 |\log\ve|
, &&  \ve^{-\frac12} \leq r \leq \delta |\log\ve|\ve^{-\frac12} ,
\\
& F_\ve(r) =  \ve^{\frac12} r + O(|\log\ve|)
, &&  r \geq \delta |\log\ve|\ve^{-\frac12} ,
\end{align*}
\begin{align*}
& F_\ve'(r) =
C'(r) =  \frac1r + O(r^{-3})
, && 1 \leq r \leq \ve^{-\frac12} ,
\\
& F_\ve'(r) = O(\ve^{\frac12} )
, &&  \ve^{-\frac12} \leq r \leq \delta |\log\ve|\ve^{-\frac12} ,
\\
& F_\ve'(r) =  \ve^{\frac12}( 1 + O(\frac{|\log\ve|}{ \ve^{1/2} r} ) )
,&&
r \geq \delta |\log\ve|\ve^{-\frac12} ,
\end{align*}
\begin{align*}
&F_\ve''(r) = C''(r) = -\frac1{r^2} + O(r^{-4})
, && 1 \leq r \leq \ve^{-\frac12} ,
\\
& F_\ve''(r) = O (\frac{1}{r^2} + \frac{\ve}{|\log\ve|})
, &&  \ve^{-\frac12} \leq r \leq \delta |\log\ve|\ve^{-\frac12} ,
\\
& F_\ve''(r) = O(\frac{\ve^\frac12}{r})
, &&
r \geq \delta |\log\ve|\ve^{-\frac12} .
\end{align*}

\begin{align*}
&F_\ve'''(r) = C'''(r) = 2\frac1{r^3} + O(r^{-5})
, && 1 \leq r \leq \ve^{-\frac12} ,
\\
& F_\ve'''(r) = O ( \frac{\ve^{\frac12}}{r^2} )
, &&  \ve^{-\frac12} \leq r \leq \delta |\log\ve|\ve^{-\frac12} ,
\\
& F_\ve'''(r) = O(\frac{\ve^{\frac12}}{ r^2})
, &&
r \geq \delta |\log\ve|\ve^{-\frac12} .
\end{align*}
\end{coro}
\begin{proof}
The estimates for $F_\ve$, and first and second derivatives follow from the Lemmas~\ref{lemma 2.1} and \ref{lemma est f0}.
To estimate the third derivate we can differentiate the equation and use the previous estimates.
\end{proof}

It will be useful for later purposes to have also estimates for the elements in the linearization of \eqref{eq f0}. Namely consider
\begin{align}
\label{linearization f0}
\Delta z + \frac{1}{(f_0^{(\ve)})^2(r)} z =0 ,
\quad\text{for }
r \geq \frac 1{|\log\ve|} .
\end{align}
The function
\begin{align}
\label{def tilde z1}
\tilde z_1(r) = f_0^{(\ve)} - r [ f_0^{(\ve)}]'(r)
\end{align}
satisfies \eqref{linearization f0}, since the equation  \eqref{eq f0} is invariant by the scaling $f_\lambda(r) = \frac 1\lambda f(\lambda r)$, $\lambda>0$. We may construct a second independent solution $\tilde z_2$ of \eqref{linearization f0} by solving this equation with initial conditions
\begin{align*}
\tilde  z_2(r_0) = - \tilde z_1'(r_0),
\qquad
\tilde z_2'(r_0) = \tilde z_1(r_0).
\end{align*}
Here $r_0>0$ is fixed.

\begin{lemma}
\label{lemma kerne f0}
Fix $r_0>0$.
We have
$$
|\tilde z_i(r)| \leq C,
\quad
|\tilde z_i'(r)| \leq \frac{C}{r}
$$
for all $r\geq r_0$, $i=1,2$.
\end{lemma}
\begin{proof}
In terms of $\psi$ defined in \eqref{def psi}, we may write
$$
\tilde z_1(r) = - r \psi'(\log(r))
$$
so that the boundedness of $\tilde z_1$ is consequence of \eqref{bounds psi}.
For $\tilde z_2$, we may consider the equation
$$
\phi'' +2\phi' + 2\phi = g ,\quad\text{for } t\geq \log(r_0)
$$
with kernel  given by $\zeta_1(t) = e^{-t} \cos(t)$, $\zeta_2(t) = e^{-t} \sin(t)$. Then we may express $\tilde z_2$ as a perturbation of the correct linear combination of $\zeta_1$, $\zeta_2$.

\end{proof}

\section{Approximate equation and error}
\label{sect error}

The main result in this section is the proof of Proposition~\ref{prop error}, namely the estimate
$$
\| \ve H^s_{\Sigma_0} \|_{1-\ve,\alpha+\ve} \leq
\frac{C\ve^{\frac12}}{|\log\ve|} .
$$
For for $x\in \Sigma_0$ we compute $H_{\Sigma_0}^s (x) $
by splitting
\begin{align}
\label{spliiting}
H_{\Sigma_0}^s (x)  =
\int_{\R^3} \frac{\chi_{E_0}(y) - \chi_{ E_0^c}(y)}{|x-y|^{4-\ve}}\, d y
=
I_i + I_o,
\end{align}
where
$$
I_i = \int_{C_R(x)}
\frac{\chi_{E_0}(y) - \chi_{E_0^c}(y)}{|x-y|^{4-\ve}}\, d y ,\qquad
I_o = \int_{C_R(x)^c}
\frac{\chi_{E_0}(y) - \chi_{E_0^c}(y)}{|x-y|^{4-\ve}}\, d y ,
$$
are inner and outer contributions respectively.
The inner part is the integral on a cylinder $C_R(x)$ of radius $R$ centered at $x$ and the outer contribution the rest.
We take  $R$ as a function of $x\in \Sigma_0$, $x = (x',F_\ve(x'))$, defined by
\begin{align}
\label{def R}
R  = (1-\eta(|x'| - R_0) ) R_1 + \eta(|x'|-R_0) F_\ve(|x'|)
\end{align}
where $R_0>0$ is fixed large, $R_1>0$ is a small constant and $\eta$ is as in \eqref{def eta1}.

To define the cylinder, let $\Pi_1$, $\Pi_2$ be tangent vectors to $\Sigma_0$ at $x$, orthogonal and of length 1, and $\nu_{\Sigma_0}$ be the unit normal vector  to $\Sigma_0$ oriented such that $\nu_{\Sigma_0}(x) x_3>0$. Introduce coordinates $(t_1,t_2,t_3)$ in $\R^3$ by
$$
(t_1,t_2,t_3) \mapsto t_1 \Pi_1 + t_2 \Pi_2 + t_3 \nu_{\Sigma_0} .
$$
Define the cylinder of center $x$, radius $R$ and base plane the plane generated by $\Pi_1$, $\Pi_2$ as
$$
C_R(x) = \{ x+ t_1 \Pi_1 + t_2 \Pi_2 + t_3 \nu_{\Sigma_0}(x) : t_1^2+ t_2^2<R^2, |t_3|<R\} .
$$

For the computation of the inner integral,
we represent the surface $\Sigma_0$ near $x$ as the graph over its tangent plane at $x $.
More precisely, if $R_1>0$ in \eqref{def R} is chosen small and $\|h\|_*$ is small, there is a function $g=g_x:B_R(0) \subset \R^2$ to $\R$ of class $C^{2,\alpha}$ such that
\begin{align}
\label{repr g1}
\Sigma_0 \cap C_R(x)
=  \{ x + \Pi t + \nu_{\Sigma_0}  g(t):  |t|<R \} ,
\end{align}
where $t=(t_1,t_2)$ and
$$
\Pi = [\Pi_1, \Pi_2] .
$$
Then
$$
g(0) = 0, \quad \nabla g(0) = 0, \quad \Delta g(0) = 2 H_{\Sigma_0}(x),
$$
where $H_{\Sigma_0}$ is the mean curvature of $\Sigma_0$ at $x$.

In the following statements we use the notation
$$
[v]_{\alpha,D} = \sup_{x,y\in D, \; x\not=y} \frac{|v(x)-v(y)|}{|x-y|^\alpha}.
$$
\begin{lemma}
\label{lemma Rest1}
For $x\in \Sigma_0$ and $R=R(x)$ given by \eqref{def R} we have
\begin{align}
\label{part1}
I_i
=  -2\pi\frac{H_{\Sigma_0}(x) R^{\ve}}{\ve}  + Rest_1 ,
\end{align}
where
\begin{align}
\label{Rest1}
|Rest_1|
\leq
C [D^2 g]_{\alpha,B_R(0)} R^{1+\alpha-s} +
C
\|D^2 g\|_{L^\infty(B_R(0))}^3
R^{3-s} .
\end{align}
Here $C$ remains bounded as $s\to 1$ (i.e. $\ve\to0$).
\end{lemma}

The main contribution from the outer integral is given in the next result.
\begin{lemma}
\label{lemma R2 simple}
For $x = (x',F_\ve(x')) \in \Sigma_0$ and $R=R(x)$ given by \eqref{def R} we have
\begin{align}
\label{R2 simple}
|I_o|
\leq \frac{C}{R^{1-\ve}} ,
\end{align}
and if $|x'|\geq \ve^{-\frac12}$,
\begin{align}
\label{R2 far}
I_o
= \frac{ \pi }{R^{1-\ve}}\left( 1 + O(\ve^{\frac 12})  \right) .
\end{align}
\end{lemma}

%
%

By \eqref{part1} and \eqref{R2 far} we see that the equation  $H_{\Sigma_0}^s(x)=0$ takes the form
$$
-2 H_{\Sigma_0}(x) +\frac{\ve}{R} \approx 0,
$$
which motivates \eqref{xx1}.

%

\begin{lemma}
\label{lemma g}
Let $x \in \Sigma_0 $, and write $x =(x' ,F_\ve(x'))$, $r =|x'|$.
There is $\delta_0>0$ and $g:B_{\rho}(0) \to \R$ of class $C^{2,\alpha}$ such that
$$
\Sigma_0 \cap C_\rho(x)
=  \{ x + \Pi t + \nu  g(t):  |t|<\rho \} .
$$
where $\rho = \delta_0 r$.
In particular $g$ is well defined in $B_R(0)$ where $R$ is defined in \eqref{def R}.
Moreover $g$ satisfies
\begin{align*}
\|g\|_{ L^\infty( B_R(0) ) } &
\leq
\begin{cases}
C \ve^{\frac32} r
& \text{if } r \geq \delta|\log\ve| \ve^{-\frac12}
\\
C \frac{\ve^{\frac12} |\log\ve|}{r}
& \text{if } \ve^{-\frac12} \leq r \leq \delta|\log\ve| \ve^{-\frac12}
\\
C\frac{\log(r)^2}{r^2}
& \text{if } r \leq \ve^{-\frac12}
\end{cases}
\end{align*}
\begin{align*}
\| D g \|_{L^\infty(B_R(0))} &
\leq \begin{cases}
C \ve^{\frac12} & \text{if } r \geq \ve^{-\frac12}
\\
\frac{C}r & \text{if } R_0 \leq r \leq \ve^{-\frac12}
\end{cases}
\end{align*}

\begin{align*}
\| D^2 g \|_{B_R(0)}
\leq
\begin{cases}
\frac{C \ve^{\frac12}}{r} & \text{if } r \geq  \ve^{-\frac12}
\\
\frac C{r^2} & \text{if }  r\leq \ve^{-\frac12}
\end{cases}
\end{align*}

\begin{align}
\label{est holder g}
[D^2 g]_{\alpha,B_R} &
\leq
\begin{cases}
\frac{C\ve^{\frac12}}{r^{1+\alpha}}
& \text{if } r \geq  \ve^{-\frac12}
\\
\frac{C}{r^{2 +\alpha}}
& \text{if } r \leq  \ve^{-\frac12} .
\end{cases}
\end{align}
\end{lemma}
(Proof in Appendix \ref{app}).

\begin{proof}[Proof of Lemma~\ref{lemma Rest1}]
We compute
$$
I_i = \int_{C_R(x)} \frac{\chi_{E_0}(y) - \chi_{E_0^c}(y)}{|x-y|^{4-\ve}}\, d y
=
-2
\int_{|t|<R}
\int_0^{g(t)}
\frac{1}{(|t|^2 + t_3^2)^{\frac{4-\ve}{2}}}
dt_3 \, d t ,
$$
expanding
$$
\int_0^{z}
\frac{1}{(|t|^2 + t_3^2)^{\frac{4-\ve}{2}}}
dt_3
= \frac{z}{|t|^{4-\ve}}
-(4-\ve) z^2 \int_0^1
(1-\tau) \frac{\tau z}{(|t|^2 + (\tau z)^2)^{\frac{6-\ve}{2}}} \, d\tau .
$$
Then
$$
I_{i}
=
I_{i,1}+I_{i,2}+I_{i,3}
$$
where
\begin{align*}
I_{i,1}
&=
-2
\int_{|t|<R}
\frac{\frac12 D^2 g(0)[{t}^2]}{|t|^{4-\ve}} \, d t
\\
I_{i,2}
&=
-2
\int_{|t|<R}
\frac{g(t)- \frac12  D^2 g(0)[{t}^2]}{|t|^{4-\ve}} \, d t
\\
I_{i,3}
&=2(4-\ve)
\int_{|t|<R}
g( t)^2
\int_0^1 (1-\tau)
\frac{\tau g(t)}{(|t|^2 + (\tau g(t))^2)^{\frac{6-\ve}{2}}} \, d\tau
\, d t ,
\end{align*}
and $D^2 g$ denotes the Hessian matrix of $g$.
Then
$$
I_{i,1} = - \pi \frac{ \Delta g(0) R^{\ve}}{\ve}
=
-2\pi\frac{ H(x) R^{\ve}}{\ve} .
$$
We estimate
\begin{align}
\nonumber
|I_{i,2}|
&\leq
2
\int_{|t|<R}
\frac{|g(t)- \frac12  D^2 g(0)[{t}^2]|}{|t|^{4-\ve}} \, d t
\\
\label{I i2}
&\leq
C [D^2 g]_{B_R(0),\alpha} \int_{|t|<R, t\in \R^2} |t|^{\alpha-2+\ve} d t
\leq  C [D^2 g]_{B_R(0),\alpha} R^{\alpha+\ve} .
\end{align}
Using $|g(t)|\leq \|D^2 g\|_{L^\infty(B_R(0))} |t|^2$, we can bound $I_{i,3}$
\begin{align}
\label{I3}
|I_{i,3}|
&\leq
C
\|D^2 g\|_{L^\infty}^3
\int_{|t|<R, t\in\R^2}
|t|^{\ve}dt
\leq
C
\|D^2 g\|_{L^\infty}^3
R^{2+\ve} .
\end{align}
This proves \eqref{Rest1}.
\end{proof}

\begin{proof}[Proof of Lemma~\ref{lemma R2 simple}]
Let $x \in \Sigma_0$, $x = (x',F_\ve(x'))$.
We change variables $y =  R z$ and write $\tilde x_R = x/R$
\begin{align*}
\int_{C_R(x)^c}  \frac{\chi_{E_0}(y) - \chi_{E_0^c}(y)}{|x-y|^{4-\ve}}\, d y
&=
\frac{1}{R^{1-\ve}}
\int_{C_1(\tilde x_R)^c}
\frac{\chi_{E_0/R}(z) - \chi_{E_0^c/R}(z)}{|\tilde x_R-z|^{4-\ve}}\, d z ,
\end{align*}
where $C_1(\tilde x_R)$ denotes the cylinder of radius 1 centered at $\tilde x_R$ and base plane given by the tangent plane to $\partial E_0/R$ at $\tilde x_R$.
Then \eqref{R2 simple} follows since
$$
\left|
\int_{C_1(\tilde x_R)^c}
\frac{\chi_{E_0/R}(z) - \chi_{E_0^c/R}(z)}{|\tilde x_R-z|^{4-\ve}}\, d z
\right|
\leq C.
$$

To obtain the second estimate we first note that for any $\delta_0 >0$ fixed,
$$
\left|
\int_{|\tilde x_R-z|\geq \delta_0 \ve^{-\frac12}}
\frac{\chi_{E_0/R}(z) - \chi_{E_0^c/R}(z)}{|\tilde x_R-z|^{4-\ve}}\, d z
\right|
\leq C \ve^{\frac 12},
$$
and therefore we need to prove
$$
\left|
\int_{C_1(\tilde x_R)^c, |\tilde x_R-z|\leq \delta_0 \ve^{-\frac12}}
\frac{\chi_{E_0/R}(z) - \chi_{E_0^c/R}(z)}{|\tilde x_R-z|^{4-\ve}}\, d z
- \pi  \right|\leq C \ve^{\frac12}.
$$
We note that
$$
\int_{C_1(\tilde x_R)^c,|z-\tilde x_R|\leq \delta_0 \ve^{-\frac12} }
\frac{\chi_{[|z_3|>1]}-\chi_{[|z_3|<1]}}{|z-\tilde x_R|^{4-\ve}} \, d z
= \pi+ O(\ve^{\frac12}) .
$$
(here $z=(z',z_3)$, $z'\in \R^2$, $e_3=(0,0,1)$).
Indeed,
$$
\int_{C_1(\tilde x_R)^c,|z-\tilde x_R|\leq \delta_0 \ve^{-\frac12} }
\frac{\chi_{[|z_3|>1]}-\chi_{[|z_3|<1]}}{|z-\tilde x_R|^{4-\ve}} \, d z
$$
$$
=
\int_{|z-\tilde x_R|>1,|z-\tilde x_R|\leq \delta_0 \ve^{-\frac12} }
\frac{\chi_{[|z_3|>1]}-\chi_{[|z_3|<1]}}{|z-\tilde x_R|^{4-\ve}} \, d z
$$
since by symmetry the difference of the two integrals is zero.
Since
$$
\int_{|z-\tilde x_R|\geq \delta_0 \ve^{-\frac12} }
\frac{\chi_{[|z_3|>1]}-\chi_{[|z_3|<1]}}{|z-\tilde x_R|^{4-\ve}} \, d z
= O(\ve^{\frac12})
$$
we get
\begin{align*}
& \int_{C_1(\tilde x_R)^c,|z-\tilde x_R|\leq \delta_0 \ve^{-\frac12} }
\frac{\chi_{[|z_3|>1]}-\chi_{[|z_3|<1]}}{|z-\tilde x_R|^{4-\ve}} \, d z
\\
& = \int_{|z-\tilde x_R|>1 }
\frac{\chi_{[|z_3|>1]}-\chi_{[|z_3|<1]}}{|z-\tilde x_R|^{4-\ve}} \, d z
+O(\ve^{\frac12}) \\
&=\pi +O(\ve^{\frac12}) .
\end{align*}
Therefore
\begin{align*}
& \left|
\int_{C_1(\tilde X_R)^c, |\tilde X_R-Z|\leq \delta_0 \ve^{-\frac12}}
\frac{\chi_{E_0/R}(Z) - \chi_{E_0^c/R}(Z)}{|\tilde X_R-Z|^{4-\ve}}\, d Z
- \pi \right|
\\
& \leq
\left|
\int_{C_1(\tilde X_R)^c, |\tilde X_R-Z|\leq \delta_0 \ve^{-\frac12}}
\frac{ \chi_{E_0/R}(Z) -\chi_{[|z_3|>1]} + \chi_{[|z_3|<1]}- \chi_{E_0^c/R}(Z) ) }{|\tilde X_R-Z|^{4-\ve}}\, d Z
\right|
+
C \ve^{\frac12}.
\end{align*}

Note that the point $\tilde x_R$ has the form  $\tilde x_R= (\frac {x'}R,1)$.
Inside the region $C_1(\tilde x_R)^c \cap\{z: |\tilde x_R-z|\leq \delta_0 \ve^{-\frac12}\}$,  $\partial E_0$ can be represented by
$$
|z_3| = \frac 1R F_\ve( R|z'|)
$$
By Corollary~\ref{coro prop F} we have
$$
|\frac{d}{dr} ( \frac 1R F_\ve( R r) ) | \leq C \ve^{\frac12}   ,
$$
in
$C_1(\tilde x_R)^c \cap\{z: |\tilde x_R-z|\leq \delta_0 \ve^{-\frac12}\}$.
Let us consider the upper part, namely
$C_1(\tilde x_R)^c \cap\{z: |\tilde x_R-z|\leq \delta_0 \ve^{-\frac12}\} \cap \{ z_3 >0\}$.
Inside this region,  the symmetric difference of the two sets $E_0/R$ and $|z_3|>1$ is contained in the cone
$$
\tilde x_R + \{ (z',z_3) \in \R^2\times \R : |z'|\leq \delta_0 \ve^{-\frac12}, |z_3|\leq C \ve^{\frac12}|z'| \} .
$$
Therefore we can estimate
\begin{align*}
& \left|
\int_{C_1(\tilde x_R)^c, |\tilde x_R-z|\leq \delta_0 \ve^{-\frac12} , z_3>0}
\frac{ \chi_{E_0/R}(z) -\chi_{[|z_3|>1]} + \chi_{[|z_3|<1]}- \chi_{E_0^c/R}(z) ) }{|\tilde x_R-z|^{4-\ve}}\, d z
\right|
\\
&\leq
\int_{\frac1{10}\leq|z'|\leq \delta_0 \ve^{-\frac12}, |z_3|\leq C \ve^{\frac12}|z| }
\frac{1}{|z|^{4-\ve}} d Z
\leq C  \ve^{\frac12}.
\end{align*}
The integral over $C_1(\tilde x_R)^c \cap\{z: |\tilde x_R-z|\leq \delta_0 \ve^{-\frac12}\} \cap \{ z_3 <0\}$ can be handled similarly.
\end{proof}

\begin{proof}[Proof of Proposition~\ref{prop error}]
Let $x \in \Sigma_0$,  $x=(x',F_\ve(x'))$ where $|x'|\geq 1$.
Let $R=R(x)$ be given by \eqref{def R}.

By \eqref{spliiting}, \eqref{part1} we can write
\begin{align*}
\ve H_{\Sigma_0}^s(x) = -2\pi H_{\Sigma_0} R^\ve + \ve Rest_1 + \ve I_o .
\end{align*}
Since $\Sigma_0$ is a minimal surface for $r=|x|\leq \ve^{-\frac12}$, we have
\begin{align*}
\ve H_{\Sigma_0}^s(x)
&=E_1 + E_2 + E_3 + E_4 + E_5 ,
\end{align*}
where
\begin{align*}
E_1 &=\pi R^\ve \eta_\ve ( -2 H_{\Sigma_0} +\frac{\ve}{R})
\\
E_2
&= -2 \ve
\int_{|t|<R}
\frac{g(t)- \frac12  D^2 g(0)[{t}^2]}{|t|^{4-\ve}} \, d t
\\
E_3
&=\ve 2(4-\ve)
\int_{|t|<R}
g( t)^2
\int_0^1 (1-\tau)
\frac{\tau g(t)}{(|t|^2 + (\tau g(t))^2)^{\frac{5+s}{2}}} \, d\tau
\, d t
\\
E_4 &=
\ve  I_o (1-\eta_\ve)
\\
E_5 &=
(\ve  I_o- \frac{\pi \ve}{R^s} )
\eta_\ve ,
\end{align*}
and $\eta_\ve(r)=\eta(r-\ve^{-\frac12})$ with  $\eta$ is the cut-off function \eqref{def eta1}.
Here $g$ is a function such that we have the representation of $\Sigma_0$ near $X$ as the graph of $g$ over the tangent plane of $\Sigma_0$ at $X$, as in \eqref{repr g1}.


We start with $E_1$.
For $r\geq \ve^{-\frac12}+1$, $F_\ve$ satisfies $\Delta F_\ve = \frac{\ve}{F_\ve}$, so
\begin{align*}
E_1
&=
\pi F_\ve^\ve
\left(  \Delta F_\ve(1- \frac{1}{\sqrt{1+(F_\ve')^2}}) + \frac{(F_\ve')^2 F_\ve''}{(1+(F_\ve')^2)^{3/2}} \right) .
\end{align*}
But for this range $F_\ve'(r) = O(\ve^{\frac12})$, $F_\ve''(r)= O(\frac{\ve^{\frac12} }{r})$,  $F_\ve(r) \leq C \ve^{\frac12}r$ if $r\geq \delta \ve^{-\frac12} |\log\ve|$ and $F_\ve(r) \leq C |\log\ve| $ if $\ve^{-\frac12} r\leq \delta \ve^{-\frac12} |\log\ve|$ , so
$$
\sup_{r\geq \ve^{-1/2}+1} r^{1-\ve}|E_1| = O(\ve^{\frac32}) ,
\quad \text{as } \ve\to0.
$$

For  $r \in [\ve^{-\frac12},\ve^{-\frac12}+1]$ we have $\Delta f_\ve= O(\frac{\ve}{|\log\ve|})$, $ \Delta f_C = O(\ve^2)$, and so $(f_\ve-f_C)'=O(\frac{\ve}{|\log\ve|})$, $f_\ve-f_C=O(\frac{\ve}{|\log\ve|})$ in this region.
Then for these $r$
\begin{align*}
-\Delta F_\ve + \frac{\ve}{F_\ve} &=
-\eta_\ve \frac{\ve}{f_\ve} + \frac{\ve}{\eta_\ve f_\ve + (1-\eta_\ve) f_C}
- (1-\eta_\ve) \Delta f_C - 2 \eta_\ve' (f_\ve-f_C)' - \Delta \eta_\ve (f_\ve-f_C)
\\
&= O(\frac{\ve}{|\log\ve|}).
\end{align*}
It follows that
$$
\sup_{r \in [\ve^{-\frac12},\ve^{-\frac12}+1]} r^{1-\ve} |E_1| = O(\frac{\ve^{\frac12}}{|\log\ve|}).
$$
To estimate the H\"older part of the norm, i.e. $[E_1]_{1-\ve,\alpha+\ve}$, it is enough to show that
$$
\sup_{r \geq \ve^{-1/2} } r^{2-\ve} |E_1'(r)| \leq C \frac{\ve^{\frac12}}{|\log\ve|} ,
$$
and the computation is analogous to the previous one.

We estimate  $E_2 $.
By \eqref{I i2}, we need to estimate
$$
\ve \|[D^2 g]_{B_R(0),\alpha} R^{\alpha+\ve} \|_{1-\ve,\alpha+\ve}
$$
where $R=R(x)$, $g=g_x$, $x=(x',F_\ve(x')) \in \Sigma_0$.
In the regime $r = |x'| \geq \delta |\log\ve|\ve^{-\frac12}$ we have
\begin{align}
\label{2.12}
R = F_\ve(r) \leq C \ve^{\frac12} r ,
\end{align}
and by \eqref{est holder g}
\begin{align}
\label{D2 g Calpha}
[D^2 g]_{\alpha,B_R(0)}
\leq \frac{C \ve^{\frac12}}{r^{1+\alpha}}
\end{align}
Therefore
$$
\sup_{r \geq \delta |\log\ve|\ve^{-1/2}}
r^{1-\ve} \left(
\ve   [D^2 g]_{\alpha,B_R(0)} R^{\alpha+\ve}
\right)
\leq C  \ve^{\frac{3+\alpha}{2}} .
$$
In the region  $ \ve^{-\frac12} \leq r \leq\delta |\log\ve|\ve^{-\frac12}$ we have
$$
R = F_\ve(r) = O(|\log\ve|)
$$
and \eqref{D2 g Calpha} still holds. Hence
$$
\sup_{\ve^{-1/2} \leq r \leq \delta |\log\ve|\ve^{- 1/2}}
r^{1-\ve} \left(
\ve   [D^2 g]_{\alpha,B_R(0)} R^{\alpha+\ve}
\right)
\leq C  \ve^{\frac{3+\alpha}{2}} |\log\ve|^\alpha.
$$
Finally for $r\leq \ve^{-\frac12}$, $R = \log(r)+O(1)$ and
$$
[D^2 g]_{\alpha,B_R(0)}
\leq \frac{C}{r^{2+\alpha}}
$$
so
$$
\sup_{r\leq \ve^{-\frac12}} r^{1-\ve} \left(
\ve   [D^2 g]_{B_R(0),\alpha} R^{\alpha+\ve}
\right)
\leq C \ve.
$$
It follows that
$$
\| |x|^{1-\ve} E_2 \|_{L^\infty} \leq C \ve.
$$

We estimate the $C^\alpha$ norm of $E_2$. For this let $x_{1}=(x_1',F_\ve(x_1'))$, $x_{2}=(x_2',F_\ve(x_2'))\in \Sigma_0$, $R_i = R(x_i)$, and $g_i : B_{R_i} \to \R$ be such that $\Sigma_0$ can be represented as a graph of $g_i$ over its tangent plane at $x_i$.

We can assume that $|x_{1}| \leq |x_{2}|$
and $R_1 \leq R_2$ (if $R_2 \leq R_1$ the argument is the same). We can also assume
$|x_{1} -x_{2}|\leq \frac{1}{10}|x_{1}|$.

Let us write
$$
E_{1}(x_{1}) - E_2(x_{2}) = E_{1,1} + E_{1,2}  ,
$$
where
\begin{align*}
E_{1,1}
& =
\ve  \int_{|t|<R_1}
\frac{g_1(t)-\frac12 D^2 g_1(0)[t^2] -(g_2(t)-\frac12 D^2 g_2(0)[t^2] ) }{|t|^{4-\ve}} \, d t
\\
E_{1,2} & = -
\ve  \int_{R_1<|t|<R_2}
\frac{g_2(t)-\frac12 D^2 g_2(0)[t^2]}{|t|^{4-\ve}} \, d t .
\end{align*}

Assume $|x_{1}-x_{2}| \leq R_1$.
Then note that by the same computation as in Lemma~\ref{lemma Rest1}
and writing $\bar R =|x_{1}-x_{2}|$,
\begin{align*}
& \left|
\int_{|t|\leq \bar R}
\frac{g_1(t)-\frac12 D^2 g_1(0)[t^2] -(g_2(t)-\frac12 D^2 g_2(0)[t^2] ) }{|t|^{4-\ve}} \, d t\right|
\\
& \leq
C( [D^2 g_1]_{\alpha,B_{\bar R}(0)} + [D^2 g_2]_{\alpha,B_{\bar R}(0)} ) \bar R^{1+\alpha-s}  \\
&
\leq  C \frac{\ve^{\frac32}}{|x_{1}|^{1+\alpha}}
|x_{1}-x_{2}|^{\alpha+\ve} ,
\end{align*}
where we have used \eqref{est holder g}.
Let us estimate the integral over $\bar R \leq |t| \leq R_1$.
For this note that from Appendix \ref{app}
$$
\|D^2(g_1-g_2)\|_{L^{\infty}(B_{R_1})} \leq \frac{C\ve^{\frac12}}{|x_1|^2}|x_1-x_2|
$$
if $|x_1|\geq \delta |\log\ve|\ve^{-\frac12}$. In this case we see that
\begin{align*}
& \left|
\int_{\bar R \leq |t|\leq R_1}
\frac{g_{1,0}(t)-\frac12 D^2 g_{1}(0)[t^2] -(g_{2}(t)-\frac12 D^2 g_{2}(0)[t^2] ) }{|t|^{4-\ve}} \, d t
\right|
\\
& \leq
\frac{C\ve^{\frac12}}{|x_{1}|^2}
|x_{1}-x_{2}|
\int_{\bar R \leq |t|\leq R_1} |t|^{-1-s} \, dt
\\
& \leq
\frac{C\ve^{\frac12}}{|x_{1}|^{2}}
|x_{1}-x_{2}|^{\alpha+\ve}
\int_{\bar R \leq |t|\leq R_1} |t|^{-s-\alpha-\ve} \, dt
\\
& \leq
\frac{C\ve^{\frac12}}{|x_{1}|^{2}}
|x_{1}-x_{2}|^{\alpha+\ve}
R_1^{2-s-\alpha-\ve}
=\frac{C\ve^{\frac12}}{|x_{1}|^{2}}
|x_{1}-x_{2}|^{\alpha+\ve}
R_1^{1-\alpha}
.
\end{align*}
Then recalling that $R_1 = O( \ve^{\frac12}|X_{1}|)$
\begin{align*}
& \ve  \left|
\int_{\bar R \leq |t|\leq R_1}
\frac{g_{1}(t)-\frac12 D^2 g_{1}(0)[t^2] -(g_{2}(t)-\frac12 D^2 g_{2}(0)[t^2] ) }{|t|^{4-\ve}} \, d t
\right|
\\
&\leq
C\frac{\ve^{\frac{4-\alpha}{2}}}{|x_{1}|^{1+\alpha} } |x_{1}-x_{2}|^{\alpha+\ve}.
\end{align*}
Therefore
$$
|E_{1,1}|\leq C\frac{\ve^{\frac32}}{|x_1|^{1+\alpha} }
|x_1-x_2|^{\alpha+\ve},
$$
if $|x_1|\geq \delta|\log\ve|\ve^{-\frac12}$. The other cases can be handled similarly.

For the second term we have
$$
\left|
\int_{R_1<|t|<R_2}
\frac{g_2(t)-\frac12 D^2 g_2(0)[t^2]}{|t|^{4-\ve}} \, d t
\right|
\leq C [D^2 g_2]_{\alpha,B_{R_2}(0)}
(R_2^{1+\alpha-s} - R_1^{1+\alpha-s}).
$$
But we can estimate
$$
|R_2 -R_1| \leq C \frac{\ve^{\frac12}}{|x_1|^{\alpha-1}}
|x_1-x_2|^\alpha
$$
if $|x_1|\geq \ve^{-\frac12}$, and
$$
|R_2 -R_1| \leq C \frac{|x_1-x_2|^\alpha}{|x_1|^{\alpha}}
$$
if $|x_1|\leq \ve^{-\frac12}$.
Summarizing
$$
|E_2(x_1)-E_2(x_2)|  \leq C\ve \frac{|x_1-x_2|^{\alpha+\ve}}{|x_1|^{1+\alpha}}.
$$

Let us consider  $E_3$.
By \eqref{I3}
$$
|E_3|\leq C \ve \|D^2 g \|_{L^\infty(B_R)}^3 R^{3-s}.
$$
In the region  $|x| \geq \delta |\log\ve|\ve^{-\frac12}$  we use \eqref{2.12} and
$$
\|D^2 g\|_{L^\infty(B_R(0))}
\leq \frac{C \ve^{\frac12}}{|x|}
$$
to obtain
$$
\sup_{|x|\geq \delta |\log\ve|\ve^{-\frac12}}
|x|^{1-\ve} \left(
 \ve \|D^2 g\|_{L^\infty(B_R(0))}^3  R^{2+\ve}
\right) \leq C \ve^{\frac 72}.
$$
When  $ \ve^{-\frac12} \leq |x| \leq\delta |\log\ve|\ve^{-\frac12}$ we get
$$
\sup_{\ve^{-\frac12}\leq |x|\leq \delta |\log\ve|\ve^{-\frac12}}
|x|^{1-\ve} \left(
 \ve \|D^2 g\|_{L^\infty(B_R(0))}^3  R^{2+\ve}
\right) \leq C \ve^{\frac 72} |\log\ve|^2.
$$
Also
$$
\sup_{|x| \leq \ve^{-\frac12}}
|x|^{1-\ve} \left(
 \ve \|D^2 g\|_{L^\infty(B_R(0))}^3  R^{2+\ve}
\right) \leq C \ve.
$$
Similar computations as before show that if  $|x_{1}| \leq |x_{2}|$  and $|x_{1} -x_{2}|\leq \frac{1}{10}|x_{1}|$ then
$$
|E_3(x_1)-E_3(x_2)|  \leq C\ve \frac{|x_1-x_2|^{\alpha+\ve}}{|x_1|^{1+\alpha}}.
$$

To estimate $E_4 =  \ve  I_o (1-\eta_\ve)   $ we use \eqref{R2 simple}:
\begin{align*}
|\ve  I_o (1-\eta_\ve)|
\leq \frac{C \ve}{R^{1-\ve}} ,
\end{align*}
and note that it is supported in $r\leq \ve^{-\frac12}+1.$
But in this range $ R = \log(|x'|) +O(1) $ and this implies
$$
\sup_{x} |x|^{1-\ve} |E_4|
\leq C \ve \sup_{|x|\leq \ve^{-\frac12}} \frac{|x|^{1-\ve}}{|\log(|x|)|}
\leq \frac{C \ve^{\frac12}}{|\log\ve|}.
$$

To estimate the H\"older norm of $E_4$,
we actually claim that
\begin{align}
\label{der1}
\left| D_{x} E_4(x) \right|
\leq \frac{C}{|x| \log(|x|)^2}.
\end{align}
To obtain \eqref{der1}, we write $x= (x',F_\ve(x'))$ and write $x'=(x_1,x_2)$. We compute
$$
D_{x_i}
\int_{C_R(x)^c}
\frac{\chi_{E_0}(y) - \chi_{E_0^c}(y)}{|x-y|^{4-\ve}}\, d y
= B_1 + B_2 + B_3
$$
with
\begin{align*}
B_1 & =
- (4-\ve) \int_{C_R(x)^c}
\frac{\chi_{E_0}(y) - \chi_{E_0^c}(y)}{|x-y|^{5+s}}
\langle x-y,D_{x_i} x \rangle
\, d y
\\
B_2 & =
-
\int_{\partial C_R(x)}
\frac{\chi_{E_0}(y) - \chi_{E_0^c}(y)}{|x-y|^{4-\ve}}
\langle \nu , \frac{y-x}{|y-x|}  \rangle \, d y D_{x_i} R
\\
B_3 &=
-\int_{\partial C_R(x)}
\frac{\chi_{E_0}(y) - \chi_{E_0^c}(y)}{|x-y|^{4-\ve}}
\nu_i \, d y ,
\end{align*}
where $\nu = (\nu_1,\nu_2,\nu_3) $ is the unit exterior normal to $C_R(x)$.
But $D_{x_i} x = e_i + (0,0,D_{x_i} F_\ve)$ and so $B_1$ can be combined with $B_3$. Indeed
$$
B_1 = B_{1,1}  + B_{1,2}
$$
where
\begin{align*}
B_{1,1} &=
- (4-\ve) \int_{C_R(x)^c}
\frac{\chi_{E_0}(y) - \chi_{E_0^c}(y)}{|x-y|^{5+s}}
\langle x-y,e_i \rangle
\, d y
\\
B_{1,2} &= - (4-\ve) \int_{C_R(x)^c}
\frac{\chi_{E_h}(y) - \chi_{E_h^c}(y)}{|x-y|^{5+s}}
\langle x-y, (0,0,D_{x_i} F_\ve ) \rangle
\, d Y .
\end{align*}
But
\begin{align*}
B_{1,1}
&=
- \int_{C_R(x)^c}
D_{y_i}\frac{\chi_{E_0}(y) - \chi_{E_0^c}(y)}{|x-y|^{4-\ve}}
\, d Y
\\
&=
2 \int_{\partial E_h \setminus C_R(x)}\frac{1}{|x-y|^{4-\ve}} \nu_i \, d y
+
\int_{\partial C_R(x)} \frac{\chi_{E_h}(y) - \chi_{E_h^c}(y)}{|x-y|^{4-\ve}} \nu_i \, d y ,
\end{align*}
where we have used the unit normal $\nu$ pointing up on $\partial E_0$, and the exterior unit normal to $C_R(x)$.
Therefore
$$
B_1 + B_2 + B_3 =
2 \int_{\partial E_0 \setminus C_R(x)}\frac{1}{|x-y|^{4-\ve}} \nu_i \, d y + B_{1,2} + B_2 .
$$
We now estimate
\begin{align*}
\int_{\partial E_0 \setminus C_R(x)}\frac{1}{|x-y|^{4-\ve}} \nu_i \, d y
= \int_{\partial E_0 \cap B_\rho(x) \setminus C_R(x)} \ldots +
\int_{\partial E_0 \setminus  B_\rho(x)} \ldots
\end{align*}
where we take $\rho = |x|/100$.
For the outside integral we have
\begin{align*}
\left|
\int_{\partial E_0 \setminus  B_\rho(x)}
\frac{1}{|x-y|^{4-\ve}} \nu_i \, d y
\right| \leq C \rho^{-1-s} \leq C |x|^{-1-s} .
\end{align*}
For the inner region, we observe that $|\nu_i(y)|\leq \frac{C}{|y|} \leq \frac C{|x|}$
and so
\begin{align*}
\left|
\int_{\partial E_0 \cap B_\rho(x) \setminus C_R(x)}
\frac{1}{|x-y|^{4-\ve}} \nu_i \, d y
\right|
\leq \frac{C}{|x| \log(|x|)^{2-\ve}} .
\end{align*}
For $B_2$ we also get
$$
|B_2|\leq  \frac{C}{|x| \log(|x|)^{2-\ve}}
$$
since $D_{x_i} R = O(1/|x_0|)$ while the integral is $O (\frac{C}{ \log(|x_0|)^2})$.
Hence
$$
|B|\leq \frac{C}{|x_0| \log(|x_0|)^2}.
$$
and combined with the estimate for $A$ yields \eqref{der1}.

Using \eqref{der1} we can estimate the H\"older norm of $E_4$.
For this let $x_1,x_2\in \Sigma_0$, $x_i = (x_i',F_\ve(x_i'))$, $R_i =R(x_i)$. We can assume that $|x_1| \leq |x_2|$, $R_1 \leq R_2$ and also
$|x_1 -x_2|\leq \frac{1}{10}|x_1|$.

Then
\begin{align*}
|E_4(x_1) - E_4(x_2)|
& \leq C \ve \frac{|x_{0,1}|^{1+\alpha}}{|x_{0,1}| \log(|x_{0,1}|)^2 } |x_{0,1}-x_{0,2}|
\\
& \leq C \ve \frac{|x_{0,1}|}{ \log(|x_{0,1}|)^2 } |x_{0,1}-x_{0,2}|^\alpha
\\
& \leq C \frac{\ve^{\frac12}}{|\log\ve|}|x_{0,1}-x_{0,2}|^\alpha .
\end{align*}
Therefore
$$
\|E_4\|_{1-\ve,\alpha+\ve} \leq \frac{C \ve^{\frac12}}{|\log\ve|}.
$$

To estimate $E_5 =\ve (  I_o- \frac{\pi }{R^s} ) \eta_\ve$ we use \eqref{R2 far} to obtain
\begin{align*}
|E_5 |
\leq C\frac{\ve^{\frac32}}{R^s}.
\end{align*}
Since
$R = F_\ve(x') = \ve^{\frac12} r +O(|\log\ve|) $ for  $r \geq \delta|\log\ve| \ve^{-\frac12} $, where $r=|x'|$,
then
$$
\sup_{ r \geq \delta|\log\ve| \ve^{-\frac12}}
r^{1-\ve}
|E_5(r) |  \leq  C \ve.
$$
Also
$$
\sup_{ \ve^{-\frac12} \leq r \leq \delta|\log\ve| \ve^{-\frac12}}
r^{1-\ve}
|E_5 |
\leq C
\sup_{ \ve^{-\frac12} \leq r \leq \delta|\log\ve| \ve^{-\frac12}}
\frac{\ve r^{1-\ve}}{|\log\ve|}
\leq \frac{C \ve^{\frac12}}{|\log\ve|}.
$$

We estimate the H\"older estimate for $E_5$.
Let us write $x=(x',F_\ve(x'))$, $x'=(x_1,x_2)$, $r=|x'|$.
We claim that
\begin{align}
\label{est E4}
|\frac{d}{d x_i} E_5|\leq  C\frac\ve{r^{2-\ve}}
\quad \text{for } r\geq \delta |\log\ve|\ve^{-\frac12}.
\end{align}
As in the proof of Lemma~\ref{lemma R2 simple} we can  rewrite $E_5$ as
$$
E_5=
\frac{\ve}{R^{1-\ve}}\eta_\ve J ,
$$
where
$$
J=\int_{C_1(\tilde x_R)^c}
\frac{\chi_{E_0/R}(z) - \chi_{E_0^c/R}(z) }{|\tilde x_R -z|^{4-\ve}} \, d z ,
$$
$\tilde x_R = x/R$ and  $C_1(\tilde x_R)$ is the cylinder of radius 1 centered at $\tilde x_R$ and base plane given by the tangent plane to $\partial E_0/R$ at $\tilde x_R$.
Then
$$
\frac{d}{d x_i} E_5 = D_1 + D_2 + D_3
$$
where
\begin{align*}
D_1&=
-\frac{\ve(1-\ve)}{R^{2-\ve}} \frac{d R}{d x_i} ç\eta_\ve J ,
\qquad D_2 =
\frac\ve{R^{1-\ve}} \eta_\ve' \frac{d r}{d x_i} J ,
\\
D_3&=
\frac\ve{R^{1-\ve}}  \eta_\ve \frac{d J}{d x_i} .
\end{align*}
Let us estimate $D_3$. By a translation
$$
J = \int_{C_x^c} \frac{\chi_{E_x}-\chi_{E_x^c}}{|z|^{4-\ve}} \, dz
$$
where $C_x$ is the cylinder centered at the origin, with base a unit disk on a plane parallel to the tangent plane to $\Sigma_0$ at $x$, and unit height, and $E_x=(E_0 - x)/R$.

We can write
\begin{align*}
\frac{d J}{d x_i}
&= -2 \int_{\partial E_x \setminus C_x} \frac{1}{|z|^{4-\ve}}
\nu(z) \cdot ( \frac1{R^2} \frac{d R}{d x_i} z +\frac1R \frac{d x}{d x_i}) \, dz
\\
& \qquad + \int_{\partial C_x \cap E_x} \frac{1}{|z|^{4-\ve}}
\nu(z) \cdot ( \frac1{R^2} \frac{d R}{d x_i} z +\frac1R \frac{d x}{d x_i}) \, dz
\\ &\qquad
-
\int_{\partial C_x \cap E_x^c} \frac{1}{|z|^{4-\ve}}
\nu(z) \cdot ( \frac1{R^2} \frac{d R}{d x_i} z +\frac1R \frac{d x}{d x_i}) \, dz ,
\end{align*}
where for points on $\partial E_x$ $\nu$ represents the unit normal vector pointing into $E_0$, and on $\partial C_x$, $\nu$ points oust side of $C_x$. This follows from the transport theorem in the form
$$
\frac{d}{dt} \int_{T_t(U)} f(y) \, d y =
\int_{\partial T_t(U)} f(y) \nu(y) \cdot  (D_t T_t)(T_t^{-1}(y)) \, d y ,
$$
where $\nu$ points to the exterior of $T_t(U)$.
In our case $E_x = T_{x'}(E_0)$ where $T_{x'}(y) = \frac1R(y-x)$, $x=(x',F_\ve(x'))$. Hence $D_{x_i} T_{x'}(T_{x'}^{-1}(z)) = -\frac{1}{R^2} z \frac{d R}{d x_i}
-\frac1R \frac{d x}{d x_i}$.

We claim that  for $|x'|\geq \ve^{-\frac12}$ we have:
\begin{align}
\label{dJ dxi}
\frac{d J}{d x_i} = O(\frac{1}{r^2} ),
\end{align}
$r=|x'|$.
Indeed, we compute with detail the case when $|x'|\geq \delta |\log\ve|\ve^{-\frac12}$. For these $x'$, $R = \ve^{\frac12} r +O(|\log\ve|)$, $\frac{d R}{d x_i} = O(\ve^{\frac12})$.
Then
\begin{align*}
&\left|\int_{\partial E_x \setminus C_x \cap [z_3>-1] \cap [|z| \geq \ve^{-\frac12}/100]}
\frac{\nu(z) \cdot z}{|z|^{4-\ve}}
\, d z \right|
\\
& \leq
\int_{\partial E_x \setminus C_x \cap [z_3>-1] \cap [|z| \geq \ve^{-\frac12}/100]}
\frac{1}{|z|^{4-\ve}}
\, d z
= O (\ve).
\end{align*}
Inside the ball $|z|\leq \ve^{-\frac12}/100$, we have $\nu(z)\cdot z = O(\ve^{\frac12}) |z|$. Then
\begin{align*}
\int_{\partial E_x \setminus C_x \cap [z_3>-1] \cap [|z|\leq \ve^{-\frac12}/100]}
\frac{\nu(z) \cdot z}{|z|^{4-\ve}} \, d z = O (\ve^{\frac12}) .
\end{align*}
The estimate in the lower half are similar,  and therefore
\begin{align*}
\frac1{R^2} \frac{d R}{d x_i}
\int_{\partial E_x \setminus C_x} \frac{1}{|z|^{4-\ve}}
\nu(z) \cdot z  dz = O(\frac{1}{r^2})
\end{align*}
where $r=|x'|$.

In the upper half we have
\begin{align*}
\left|
\int_{\partial E_x \setminus C_x \cap [z_3>-1] \cap [|z| \geq \ve^{-1/2} /100} \frac{ \frac{d x}{d x_i} \cdot\nu(z)  }{|z|^{4-\ve}}  \, dz  \right|
& \leq  C
\int_{\partial E_x \setminus C_x \cap [z_3>-1] \cap [|z| \geq \ve^{-1/2} /100} \frac{dz }{|z|^{4-\ve}}
\\
&= O ( \ve) .
\end{align*}
For the integral over $|z|\leq \ve^{-\frac12}/100$, notice that before the  change of variables $y \mapsto z = (y-x)/R$, we have
$$
\left|\nu_{\Sigma_0} (y)\cdot \frac{d x}{d x_i}(x')\right|\leq C\frac{\ve^{\frac12}}{|x|}|y-x|
$$
for $y\in \Sigma_0$, $|y-x|\leq |x|/100$, since $\nu_{\Sigma_0} (y)\cdot \frac{d x}{d x_i}(x')$ vanishes at $y=x$ and has derivative of order $\frac{\ve^{\frac12}}{|x|}$. After the change of variables this implies
$$
| \frac{d x}{d x_i} \cdot\nu(z) |\leq C \frac{\ve^{\frac12}}{r} |z|.
$$
Therefore
$$
\int_{\partial E_x \setminus C_x \cap [z_3>-1] \cap [|z| \leq \ve^{\frac12} /100} \frac{ \frac{d x}{d x_i} \cdot\nu(z)  }{|z|^{4-\ve}}  \, dz = O ( \frac{\ve^{\frac12}}{r}) .
$$
The estimate in the lower half are similar,  and therefore
\begin{align*}
\frac1R
\int_{\partial E_x \setminus C_x} \frac{1}{|z|^{4-\ve}}
\nu(z) \cdot \frac{d x}{d x_i} \, dz
= O(\frac{1}{r^2}) .
\end{align*}

For the last 2 terms in $\frac{dJ}{dx_i}$ we observe that
\begin{align*}
\int_{\partial C_x \cap E_x} \frac{1}{|z|^{4-\ve}}
\nu(z) \cdot  z   \, dz -
\int_{\partial C_x \cap E_x^c} \frac{1}{|z|^{4-\ve}}
\nu(z) \cdot  z \, dz = O(\ve^{\frac12})
\end{align*}
since most of the integral cancels by symmetry, except a region of area $O(\ve^{\frac12})$ and similarly
\begin{align*}
\int_{\partial C_x \cap E_x} \frac{1}{|z|^{4-\ve}}
\nu(z) \cdot  \frac{d x}{d x_i} \, dz
-
\int_{\partial C_x \cap E_x^c} \frac{1}{|z|^{4-\ve}}
\nu(z) \cdot  \frac{d x}{d x_i} \, dz = O(\frac{\ve^{\frac12}}{r}) .
\end{align*}
This implies and the previous estimates imply the claim \eqref{dJ dxi} (the range $\ve^{-\frac12} \leq r \leq \delta \ve^{-\frac12}|\log\ve|$ is analogous).

The estimates for $D_1$, and $D_2$ are analogous, and since $R \approx \ve^{\frac12} |x_0|$ we obtain \eqref{est E4}.

Using \eqref{est E4}, we can show, as was done before, that
\begin{align*}
|E_5(x_1) - E_5(x_2)|
& \leq C \frac{\ve^{\frac12}}{|\log\ve|}|x_{0,1}-x_{0,2}|^\alpha .
\end{align*}
$x_1,x_2\in \Sigma_0$, with $x_i = (x_i',F_\ve(x_i'))$,
$|x_i'|\geq \ve^{-\frac12}$, and  $|x_1| \leq |x_2|$, $|x_1 -x_2|\leq \frac{1}{10}|x_1|$.
It follows that
$$
\|E_5\|_{1-\ve,\alpha+\ve} \leq C\frac{\ve^{\frac12}}{|\log\ve|}.
$$
\end{proof}

\section{Limit problem in $\Sigma_0$}
\label{sect linear1}

We want to build  a right inverse for the operator
$$
L_{0} (h)  =  \Delta h +  \frac{\ve} {F_\ve(r)^2} \eta_\ve( r)  h  ,
$$
which arises as the linearization of the approximate problem \eqref{eqf}. Here $\eta_\ve$ is any family continuous cut-off functions with $\eta_\ve(r)=0$ for $r\leq \ve^{-\frac12}$ and $\eta_\ve(r)=1$ for $r\geq \delta |\log\ve|\ve^{-\frac12}$, where $\delta>0$ is a sufficiently small number.

We then consider the equation
\be
L_0(\phi) = g, \quad\text{in }\R^2,
\label{xx3}\ee
and work in the class of radial functions.

\begin{prop}
\label{prop4}
Let $1\leq \gamma<2$.
If $\ve>0$ is small there is $C>0$ such that for $g$ radially symmetric with $ \|(1+|x|)^{\gamma} g\|_{L^\infty} < + \infty$  there exists a radially symmetric solution of \equ{xx3} $\phi = T(g)$ with $\|(1+|x|)^{\gamma-2}\phi\|_{L^\infty} < +\infty$ that defines a linear operator of $g$ with
$$
\| |x|^{\gamma-2}\phi\|_{L^\infty}\ \le \ C\,  \|(1+|x|)^{\gamma} g\|_{L^\infty},
$$
and $\phi(0)=0$.
\end{prop}

\begin{proof}
Since all functions are radial, we have to solve
\begin{align}
\label{ode regions23}
\phi''+\frac{1}{r}\phi' + \frac{\ve } {F_\ve(r)^2} \eta_\ve( r)  \phi = g, \quad r>0 .
\end{align}
We solve this ODE with initial condition $\phi(0)=\phi'(0)=0$.
For $r\leq \ve^{-\frac12}$ we obtain directly
$$
|\phi(r)|\leq C r^{2-\gamma} \|(1+|x|)^{\gamma} g\|_{L^\infty},
\quad\forall r\geq 0.
$$
We also have
\begin{align}
\label{initial cond1}
|\phi(\ve^{-\frac12}+1)|
& \leq C \ve^{\frac{\gamma-2}2} \|(1+|x|)^{\gamma} g\|_{L^\infty},
\\
\label{initial cond2b}
|\phi'(\ve^{-\frac12}+1)| & \leq C \ve^{\frac{\gamma-1}2}\|(1+|x|)^{\gamma} g\|_{L^\infty} .
\end{align}
Let us estimate $\phi(r)$ for  $r\geq \ve^{-\frac12}+1$. First we deal with the region $ \ve ^{-\frac12}+1\leq r \leq \delta |\log\ve| \ve ^{-\frac12}$, where $\delta>0$ is to be fixed later on.
Let us rewrite \eqref{ode regions23} as
$$
\phi_{rr} + \frac1r \phi_r = \tilde g
,
\quad \text{for }  \ve^{-\frac12}+1 \leq r \leq \delta |\log\ve| \ve^{-\frac12}
$$
where
$$
\tilde g = g -  \frac{\ve } {F_\ve(r)^2} \eta_\ve( r) \phi,
$$
and let $r_0= \ve^{-\frac12}+1$. Integrating we find
\begin{align}
\label{integral formula}
\phi(r) =\phi(r_0)
+ r_0 \phi'(r_0) \log(\frac{r}{r_0})
+ \int_{r_0}^r \frac1t\int_{r_0}^t \tau \tilde g(\tau) \, d \tau d t .
\end{align}
Let us introduce  the norm
\begin{align*}
\|\phi\|
= \sup_{r\in I} r^{\gamma-2} |\phi(r)| ,
\end{align*}
where $I = [\ve^{-\frac12} + 1,\delta|\log\ve| \ve^{-\frac12}]$.
We have from \eqref{16}
\begin{align}
\label{main linear term region2}
|\frac{\ve  } {F_\ve(r)^2} |
\leq \frac{C \ve}{|\log \ve|^2} ,
\quad\text{for } \ve^{-\frac12}+1\leq r \leq \delta |\log\ve|\ve^{-\frac12} .
\end{align}
Using formulas \eqref{initial cond1}, \eqref{initial cond2b},\eqref{integral formula}, \eqref{main linear term region2}  we find
$$
\|\phi\| \leq C \delta^2  \|\phi\| + C \|(1+|x|)^{\gamma} g\|_{L^\infty}  .
$$
Then we can choose $\delta>0$ small so that for all $\ve>0$ small we find
$$
\|\phi\| \leq C \|(1+|x|)^{\gamma} g\|_{L^\infty} .
$$
This is the desired estimate in the region $\ve^{-\frac12} \leq r \leq \delta \ve^{-\frac12} |\log\ve|$.

Let us consider the range  $ r \geq r_1$ where $r_1 = \delta |\log\ve| \ve ^{-\frac12}$. Then by the previous step we have
$$
|\phi(r_1)|\leq C r_1^{2-\gamma} \|(1+|x|)^{\gamma} g\|_{L^\infty} ,
\quad |\phi'(r_1)|\leq C r_1^{1-\gamma }\|(1+|x|)^{\gamma} g\|_{L^\infty}  .
$$
We write the solution $\phi$ in terms of elements in the kernel of the linear operator $\Delta + \frac{ \ve}{f_\ve^2}$, where  $f_\ve $ is defined in \eqref{f eps}.
Let $\tilde z_i$ be the functions introduced in \eqref{def tilde z1} and
$$
z_i(r) = \tilde z_i\Big( \frac{\ve^{\frac12} r}{|\log\ve|}\Big) ,
\quad r\geq \frac{\delta |\log\ve|}{\ve^{\frac12}}.
$$
By Lemma~\ref{lemma kerne f0} we have
\begin{align}
\label{behav zi}
|z_i(r)| \leq C ,
\qquad
|z_i'(r) | \leq \frac Cr ,
\quad r\geq r_1.
\end{align}
We write now
\begin{align}
\label{formula phi}
\phi(r) = c_1 z_1(r) + c_2 z_2(r) + \phi_0(r),
\qquad r \geq r_1 ,
\end{align}
where $c_1$, $c_2$ are determined so that
$$
\phi(r_1) = c_1 z_1(r_1) + c_2 z_2(r_1) ,
\quad
\phi'(r_1) = c_1 z_1'(r_1) + c_2 z_2'(r_1)
$$
and
$$
\phi_0(r) = -z_1(r) \int_{r_1}^r \frac{z_2(s) h(s)}{W(s)} \, d s
+
z_2(r) \int_{r_1}^r \frac{z_1(s) h(s)}{W(s)} \, d s .
$$
Here $W = z_1' z_2 - z_1 z_2'$ is the Wronskian. Then $W = \frac{c}{r}$ for some $c$ and using \eqref{behav zi} gives $c = O(1)$.
Also by \eqref{behav zi} we see that
$$
|c_1|+|c_2| \leq C r_1^{2-\gamma}  \|(1+|x|)^{\gamma} g\|_{L^\infty}  .
$$
Using the estimates \eqref{behav zi} we obtain
$$
\sup_{r \geq  r_1} r^{\gamma-2} | \phi_0(r) |  \| \leq C \|(1+|x|)^{\gamma} g\|_{L^\infty}  .
$$
Therefore \eqref{formula phi} yields
$$
\sup_{r \geq  r_1} r^{\gamma-2} |\phi(r)|
\leq C \|(1+|x|)^{\gamma} g\|_{L^\infty} ,
$$
which is the desired estimate
\end{proof}

\section{Fractional exterior problem}
\label{sect linear2}

In this section we will construct a linear bounded operator that maps $f$ defined on $\Sigma_0$ to  $\phi$ defined also on $\Sigma_0$
with the property
\begin{align}
\label{eq50}
\ve \JJ_{\Sigma_0}^s(\phi)(x) = f(x) \quad \text{for } x\in \Sigma_0 , \ |x|\geq R,
\end{align}
where $R>0$ will be a large fixed constant.

\begin{prop}
\label{prop exterior slow}
If $R$ is fixed large,  there is a linear operator $f\mapsto \phi$ defined for radial, symmetric functions $f$ on $\Sigma_0$ with $ \|f\|_{1-\ve,\alpha+\ve}<\infty$, such that $\phi$ is radial, symmetric, satisfies \eqref{eq50} and
$$
\|\phi\|_{*}\leq C \|f\|_{1-\ve,\alpha+\ve}.$$
\end{prop}
Here the norms $\| \ \|_*$ and $\|\ \|_{1-\ve,\alpha+\ve}$ are the ones defined in \eqref{norm st}, \eqref{norm RHS}.

We will also need a version of this result for right hand sides with fast decay. Let $0<\tau<1$.

\begin{prop}
\label{prop fast decay}
If $R$ is fixed large,  there is a linear operator $f\mapsto \phi$ defined for $f$ radial, symmetric and  $\| |x|^{2+\tau-\ve} f\|_{L^\infty(\Sigma_0)}<\infty$, such that $\phi$ is symmetric, satisfies \eqref{eq50} and
$$
\| |x|^{\tau} \phi\|_{L^\infty(\Sigma_0)} \leq C\| |x|^{2+\tau-\ve} f\|_{L^\infty(\Sigma_0)}.
$$

\end{prop}

In order to prove Propositions~\ref{prop exterior slow} and \ref{prop fast decay}  we study first
\begin{align}
\label{exterior problem}
L_\ve(\phi) + W_\ve(r) \phi= f \quad \text{in } \R^2 ,
\end{align}
where
\begin{align}
\label{Leps}
L_\ve(\phi)(x) = \ve\, \frac2\pi \text{p.v.} \int_{\R^2} \frac{\phi(y)-\phi(x)}{|x-y|^{4-\ve}} \,  d y ,
\end{align}
and
$$
W_\ve(r) = \frac{\ve } {F_\ve(r)^{2-\ve}} \eta_\ve(r),
\quad r = |x|
$$
where
\begin{align}
\label{def eta eps}
\eta_\ve(r) = \eta(\ve^{-\frac12} r -1)
\end{align}
and $\eta$ is a smooth cut-off function with $\eta(t)=1$ for $t\geq 1$ and $\eta(t)=0$ for $t\leq 0$.

\medskip

We start  with a version of Proposition~\ref{prop exterior slow} for problem \eqref{exterior problem}.

\begin{lemma}
\label{prop slow r2}
There is a linear  operator that given a radial function $f $ in $\R^2$ such that $\|  f\|_{1-\ve,\alpha+\ve}<\infty$ produces a radial solution $\phi$ of \eqref{exterior problem} with the property
\begin{align}
\label{eq92}
\| \phi\|_{*} \leq C \|  f\|_{1-\ve,\alpha+\ve}.
\end{align}
\end{lemma}
Then norms are the ones defined in \eqref{norm st}, \eqref{norm RHS} in the context of functions defined on $\R^2$.

For smooth bounded functions $h$, $L_\ve(h)$ has the expansion
$$
L_\ve(h) =  \Delta h(x) + O(\ve) \quad\text{as }\ve\to0,
$$
so  equation \eqref{exterior problem} can be considered a perturbation of
\begin{align*}
\Delta h  +W(x) h= g \quad\text{in }\R^2.
\end{align*}
where
$$
W(x) = \frac{\ve } {F_\ve(x)^2} \eta_\ve(x)
$$

The next lemma is a standard estimate for convolutions.
\begin{lemma}
\label{conv0}
Assume $\gamma,\beta<2$, $\gamma+\beta>2$. Let $\| (1+|x|)^\gamma f\|_{L^\infty} < \infty$. Then
$$
\left|\int_{\R^2} \frac{1}{|x-y|^\beta} f(y)\, d y \right|\leq C \| (1+|x|)^\gamma f\|_{L^\infty} (1+|x|)^{2-\beta-\gamma} .
$$
\end{lemma}

\begin{lemma}
\label{solv medium decay}
Let $g$ be radial with $\| (1+|x|)^{\gamma-\ve} g \|_{L^\infty}<\infty$ where $\gamma \in (1,2)$.
Then for $\ve>0$ small \eqref{exterior problem}
has a radial solution $h$ depending linearly on $g$ with $h(0)=0$. Moreover
$$
\| (1+|x|)^{\gamma-2} h\|_{L^\infty} \leq C \| (1+|x|)^{\gamma-\ve} g \|_{L^\infty} .
$$
\end{lemma}
\begin{proof}
Instead of looking directly for a solution of  \eqref{exterior problem} we will solve
\begin{align}
\label{reformulation2}
D_{r} h(x) =  c_{2,\ve} \,  \text{p.v.} \int_{\R^2} \frac{|x|-\langle y, \frac{x}{|x|} \rangle}{|x-y|^{2+\ve}}(W_\ve h-g) \, d y ,
\end{align}
for a radial function $h$ with $h(0)=0$. Here $D_r$ is the radial derivative.
In \eqref{reformulation2} the integral converges if $ \| (1+|x|)^{\gamma-\ve} (W_\ve h-g) \|_{L^\infty}<\infty$ by Lemma~\ref{conv0}.

Equation \eqref{reformulation2} is equivalent to
\begin{align}
\label{ref2b}
D_{r} h - A_\ve(h) = B_\ve(g)
\end{align}
where
\begin{align*}
A_\ve(h)(x) & = c_{2,\ve} \,  \text{p.v.} \int_{\R^2} \frac{|x|-\langle y, \frac{x}{|x|} \rangle}{|x-y|^{2+\ve}}W_\ve(y) h(y)\, d y
\\
\nonumber
B_\ve(g)(x)&= - c_{2,\ve} \,  \text{p.v.} \int_{\R^2} \frac{|x|-\langle y, \frac{x}{|x|} \rangle}{|x-y|^{2+\ve}}g(y) \, d y .
\end{align*}
Let $A_0$ be the operator
$$
A_0(h)(x) =  c_{2} \,  \text{p.v.} \int_{\R^2} \frac{|x|-\langle y, \frac{x}{|x|} \rangle}{|x-y|^{2}}W(y) h(y)\, d y .
$$
Then \eqref{ref2b} is equivalent to
\begin{align}
\label{eq81}
D_r h  - A_0(h) = A_\ve(h) - A_0(h) + B_\ve(g).
\end{align}

We claim that given $\psi$ radial in $\R^2$ with $\|(1+r)^{\gamma-1}\psi\|_{L^\infty}<\infty$ we can find a radial solution $h$ to
\begin{align}
\label{eqh1}
D_{r} h - A_0(h) = \psi
\end{align}
satisfying $h(0)=0$ and
\begin{align}
\label{eq80}
\| (1+r)^{\gamma-1} h'\|_{L^\infty} + \| r^{\gamma-2}  h\|_{L^\infty}  \leq C  \|(1+r)^{\gamma-1}\psi\|_{L^\infty}.
\end{align}

Indeed, we need to solve
\begin{align*}
h'(r) + \frac1r\int_0^r W(s) h(s) s\, d s = \psi(r)
\quad \text{for all } r>0.
\end{align*}
Let
$$
\tilde \psi(r) = \int_0^r \psi(s)\, d s, \quad \tilde h(r) = h(r)-\tilde \psi(r).
$$
Then we look for $\tilde h$ satisfying
$$
\tilde h'(r)  + \frac1r\int_0^r W(s) \tilde h(s)s \, d s = -
 \frac1r\int_0^r W(s) \tilde \psi(s) s\, d s
$$
which we write as
$$
\Delta \tilde h + W(r) \tilde h(r) = W(r) \tilde \psi(r) , \quad 0<r<\infty.
$$
We solve this equation using Proposition~\ref{prop4} and obtain
$$
\| ( 1+r^{\gamma-1}\tilde h'\|_{L^\infty} + \| r^{\gamma-2} \tilde h\|_{L^\infty} \leq C  \| (1+r)^{\gamma-2} \tilde \psi\|_{L^\infty} .
$$
Then $h=\tilde h+\tilde \psi$ satisfies \eqref{eqh1}, $h(0)=0$ and estimate \eqref{eq80}.

Let $T$ denote the operator that to a radial function $\psi\in L^\infty(\R^2)$ gives the radial solution $h$ to \eqref{eqh1} just constructed, and note that by \eqref{eq80}
\begin{align}
\label{eq80b}
\|T(\psi)\|_a \leq C  \|(1+r)^{\gamma-1}\psi\|_{L^\infty}.
\end{align}
where
\begin{align*}
\|\varphi\|_a=
\| |x|^{\gamma-2} \varphi\|_{L^\infty} + \|(1+|x|)^{\gamma-1} \nabla \varphi\|_{L^\infty} .
\end{align*}
We rewrite \eqref{eq81} as
\begin{align}
\label{ref2c}
h = T(  A_\ve(h)-A_0(h) +B_\ve(g) )
\end{align}
in the space $X = \{ h \in W^{1,\infty}_{loc}(\R^2): h \text{ is radial}, \|h\|_a<\infty\}$ with norm $\| \ \|_a$.

We solve \eqref{ref2c} by the contraction mapping principle.
Consider  the difference
$$
\int_{\R^2} \left( \frac{x_i-y_i}{|x-y|^2} W(y)  - \frac{x_i-y_i}{|x-y|^{2+\ve}} W_\ve(y) \right)  \varphi(y) d y
$$
where we assume that $\varphi $ is radial and $\|\varphi\|_a<\infty$.
Let
$$
D = \left\{ y : \max\Big( \frac{|x-y|}{|y|} , \frac{|y|}{|x-y|}\Big) \leq \ve^{-m} \right\}
$$
where $0<m<1$ is fixed.
Let us estimate the integral outside $D$.
Then we can estimate separately
$$
\int_{D^c} \frac{1}{|x-y|} \frac{1}{|y|^2}
\eta_\ve(y) \varphi(y) d y ,
\quad
\int_{ D^c} \frac{1}{|x-y|^{1+\ve}} \frac{1}{|y|^{2-\ve}}
\eta_\ve(y) \varphi(y) d y  ,
$$
respectively, since $\ve^{\frac12} r \leq C F_\ve(r)$ for all $r\geq \ve^{-\frac12}$.
First, for the integral over the region $|y|\geq |x-y|\ve^{-m}$
note that this condition is equivalent to $|y-(1+\delta)x|^2 \leq (\delta+\delta^2)|x|^2$ where $\delta = O(\ve^{2m})$ as $\ve\to 0$. So, for $y$ in this region $|y|\sim |x|$ and  hence
\begin{align*}
\Big|\int_{ \{|y|\geq  \ve^{-m} |x-y|\}} &\frac{1}{|x-y|^{1+\ve}} \frac{1}{|y|^{2-\ve}}
\eta_\ve(y) \varphi(y) d y  \Big|
\\
& \leq
\|\varphi\|_a
\frac{C}{(1+|x|)^{1-\ve}} \int_{|y-x| \leq C\delta^{1/2}|x|} \frac{1}{|x-y|^{1+\ve}}\, d y \\
&\leq \|\varphi\|_a \frac{C}{(1+|x|)^{1-\ve}} (\delta^{1/2}|x|)^{1-\ve} \leq  o(1) \|\varphi\|_a,
\end{align*}
as $\ve\to 0$.
Similarly
$$
\Big|\int_{\{|y|\geq \ve^{-m} |x-y|\}} \frac{1}{|x-y|} \frac{1}{|y|^{2}}
\eta_\ve(y) \varphi(y) d y  \Big|
\leq o(1) \|\varphi\|_a,
$$
as $\ve\to 0$.

Next, for the integral over the region $|y|\leq \ve^{m}|x-y|$, note that this condition is equivalent to $|y+\delta x|^2 \leq (\delta + \delta^2)|x|^2$ with $\delta = O(\ve^{2m})$ as $\ve\to 0$.
In this region $|x-y|\sim |x|$ and then we estimate
\begin{align*}
\int_{ \{|y|\leq \ve^m |x-y|\}} \frac{1}{|x-y|^{1+\ve}} \frac{1}{|y|^{2-\ve}}
\eta_\ve(y) \varphi(y) d y
&\leq \frac{C \|\varphi\|_a }{(1+|x|)^{1+\ve}} \int_{|y|\leq C \delta^{1/2}|x|} \frac{1}{|y|^{1-\ve}}dy\\
&\leq C \|\varphi\|_a \delta^{\frac{n-1+\ve}{2}}\leq o(1) \|\varphi\|_a .
\end{align*}
Similarly
$$
\int_{ \{|y|\leq \ve^{m}|x-y|\}} \frac{1}{|x-y|^{n-1}} \frac{1}{|y|^{2}}
\eta_\ve(y) \varphi(y) d y
\leq o(1) \|\varphi\|_a.
$$

Next consider the integrals inside $D$.
For this we let
\begin{align*}
A_1 &= \{ y \in \R^2: |y|\geq 2|x|\} \cap D
\\
A_2 &= \{ y \in \R^2 : |y|\leq 2|x|, |x-y|\geq |x|/2 \}\cap D
\\
A_3 &= \{ y \in \R^2: |x-y|\leq |x|/2\}\cap D.
\end{align*}
We have now to estimate
\begin{align*}
I&= \int_{D} \left( \frac{x_i-y_i}{|x-y|^{2}} W(y)  - \frac{x_i-y_i}{|x-y|^{2+\ve}} W_\ve(y)\right)  \varphi(y) d y \\
&= \int_{D}  \frac{x_i -y_i}{|x-y|^2} \left( 1 - \frac{F_\ve(y)^\ve}{|x-y|^{\ve}}   \right)g(y) d y ,
\end{align*}
where
$$
g(y) = \frac{\ve \eta_\ve(y)}{F_\ve(y)^2} \varphi(y) .
$$
Note that
$$
\|(1+|x|)^{\gamma}g\|_{L^\infty} \leq C \|\varphi\|_a.
$$
Inside $D$ we have $1 - \frac{F_\ve(y)^\ve}{|x-y|^{\ve}} =O (\ve |\log\ve|)$. We assume $|x|\geq 10$. In $A_1$, $|x-y|\sim|y|$ so
\begin{align*}
\left| \int_{A_1} \frac{x_i -y_i}{|x-y|^2} \left( 1 - \frac{F_\ve(y)^\ve}{|x-y|^{\ve}}   \right) g(y) d y\right|
& \leq
C\ve |\log\ve| \, \|\varphi\|_a  \int_{|y|\geq 2 |x|} \frac{1}{|y|^{1+\gamma}} dy
\\
&\leq C\ve |\log\ve| \, \|\varphi\|_a |x|^{1-\gamma}.
\end{align*}
For the integral in $A_2$ note that $|x-y|\sim|x|$ and hence
\begin{align*}
\left|\int_{A_2} \frac{x_i -y_i}{|x-y|^{2}}  \left( 1 - \frac{F_\ve(y)^\ve}{|x-y|^{\ve}}   \right)  g(y)  \right|
&\leq
C\ve |\log\ve| \,  \|\varphi\|_a |x|^{-1}
\int_{A_2} \frac{1}{|y|^{\gamma}}\, d y
\\
&\leq C\ve |\log\ve| \,  \|\varphi\|_a |x|^{1-\gamma}.
\end{align*}
For $y\in A_3$ we have $|y|\sim|x|$ and therefore
\begin{align*}
\left|\int_{A_3}  \frac{x_i-y_i}{|x-y|^{2}} \left( 1 - \frac{F_\ve(y)^\ve}{|x-y|^{\ve}}   \right)  g(y)\, d y \right|
&\leq C\ve |\log\ve| \,  \|\varphi\|_a  |x|^{-\gamma} \int_{|y-x|\leq |x|/2} \frac{1}{|x-y|}\, d y\\
&\leq C\ve |\log\ve| \,  \|\varphi\|_a  |x|^{1-\gamma} .
\end{align*}

Using the previous calculation we see that the map from $X$ to itself given by $T(  A_\ve(h)-A_0(h) +B_\ve(g) )$ is a contraction for $\ve>0$ small, and hence has a unique fixed point.
This fixed point satisfies
$$
\|h\|_a \leq C \| T( B_\ve(g) )\|_a \leq C  \|(1+r)^{\gamma-1}  B_\ve(g)\|_{L^\infty}
$$
by \eqref{eq80b}. Using then Lemma~\ref{conv0}
$$
\|h\|_a   \leq C \|(1+|x|)^{\gamma-\ve}g\|_{L^\infty}.
$$

We need to verify that $h$ solves also \eqref{exterior problem}. We define
$$
\tilde w_k (x)= \tilde c_{n,\ve} \,  \int_{\R^2} \frac{1}{|x-y|^{\ve}}(\frac{\eta_\ve(y)}{|y|^{2-\ve}}h-g) \eta_k(y)\, d y .
$$
and
$$
w_k(x) = \tilde w_k(x)- \tilde w_k(0)
$$
where $\eta_k$ is a sequence of smooth cut-off functions
with support in $B_{k^2}$, $\eta_k = 1$ in $B_k$ and $|D \eta_k|\leq 1/k^2$.
Hence $w_k$ are well defined and satisfy
$$
\ve\, \text{p.v.} \int_{\R^2}\frac{w_k(y)-w_k(x)}{|x-y|^{4-\ve}} \, d y = (g-\frac{\eta_\ve(x)}{|x|^{2-\ve}}h) \eta_k
\quad\text{in }\R^n.
$$
By Lemma~\ref{conv0}
\begin{align*}
|x|^{\gamma-1} |D_{x_i} w_k (x)|
&\leq C \| (1+|x|)^{\gamma-\ve} (g-\frac{\eta_\ve(x)}{|x|^{2-\ve}}h)  \eta_k \|_{L^\infty}
\\
&\leq  C  \|(1+|x|)^{\gamma-\ve}g\|_{L^\infty}.
\end{align*}
Then up to subsequence $w_k\to w$ uniformly on compact sets of $\R^2$, $\|(1+|x|)^{\gamma-1} D w\|_\infty \leq C  \|(1+|x|)^{\gamma-\ve}g\|_{L^\infty}$, and $w$ satisfies
\begin{align}
\label{eq61}
\ve\, \text{p.v.} \int_{\R^2}\frac{w(y)-w(x)}{|x-y|^{n+2-\ve}} \, d y = g-\frac{\eta_\ve(x)}{|x|^{2-\ve}}h \quad\text{in }\R^n.
\end{align}
From this equation
\begin{align*}
D_{x_i} w(x)&=
c_{n,\ve} \,  \int_{\R^n} \frac{x_i-y_i}{|x-y|^{n+\ve}}(\frac{\eta_\ve(y)}{|y|^{2-\ve}}h-g) \, d y =D_{x_i}h(x).
\end{align*}
Hence $w$ and $h$ differ by a constant and from \eqref{eq61} we see that $h$ solves \eqref{exterior problem}.

\end{proof}

\begin{proof}[Proof of Lemma~\ref{prop slow r2}]
The proof is based on the following apriori estimate for radial solutions $h$ of \eqref{exterior problem} such that $\||x|^{-1} h\|_{L^\infty}<\infty$:
\begin{align}
\label{eq91b}
\| |x|^{-1} h\|_{L^\infty}\leq C \|(1+|x|)^{1-\ve} g\|_{L^\infty} ,
\end{align}
and we claim it holds if $\ve>0$ is sufficiently small.

We argue by contradiction, assuming that there are sequences $\ve_i\to0$, radial functions $g_i$, $h_i$ solving \eqref{exterior problem} and satisfying
\begin{align}
\label{eq90b}
\| |x|^{-1} h_i\|_{L^\infty}=1, \quad \|(1+|x|)^{1-\ve_i} g_i\|_{L^\infty} \to0
\end{align}
as $i\to\infty$. Let $x_i \in \R^2$ be such that
$$
(1+|x_i|)^{-1} |h_i(x_i)|\geq\frac12.
$$
Assume first that $x_i$ remains bounded and, up to a subsequence $x_i\to x$ as $i\to\infty$. The bounds \eqref{eq90b} and standard estimates for $L_\ve$, uniform as $\ve\to 0$, show that $h_i$ is bounded in $C^{1,\alpha}_{loc}$. Therefore passing to a subsequence  we find $h_i\to h$ locally uniformly in $\R^2$. Let $\varphi \in C_0^\infty(\R^2)$. Multiplying \eqref{exterior problem} by $\varphi$ and integrating we find
$$
\int_{\R^2} h_i L_{\ve_i}(\varphi) + W_{\ve_i} h_i \varphi_i = \int_{\R^2} g_i \varphi.
$$
Taking the limit we find that $h$ is harmonic in $\R^2$. But also $|h(x)|\geq \frac12$, $h$ is radial  and $|h(r)|\leq r$ for all $r\geq 0$, which is impossible.

Suppose that $x_i$ is unbounded so that up to subsequence $r_i=|x_i| \to \infty$  as $i\to\infty$. Let
$$
\tilde h_i(x) = \frac{1}{r_i} h( r_i x ) , \quad \tilde g_i(x) = r_i^{1-\ve_i} g(r_i x)
$$
so that
$$
L_{\ve_i} (\tilde h_i)+ W_i(x)\tilde h_i = \tilde g_i \quad \text{in } \R^2 ,
$$
where
$$
W_i(x) = \frac{\ve_i \eta_{\ve_i}(r_i x) r_i^{2-\ve_i}}{F_{\ve_i}(r_i x) ^{2-\ve_i}}
$$
Also
$$
\| |x|^{-1} \tilde h_i\|_{L^\infty}=1, \quad \||x|^{1-\ve_i} \tilde g_i\|_{L^\infty}\to0
$$
as $i\to\infty$. Up to subsequence $\tilde h_i\to h$ locally uniformly in $\R^2$ and $x_i/r_i \to \hat x$. Moreover $|h(\hat x)|\geq\frac12$.


If $\ve_i^{-\frac12} |\log\ve_i| r_i^{-1} \to \infty$ as $i\to\infty$ then $W_i(x) \to 0$ uniformly on compact sets and we reach a contradiction as before.

If $\ve_i^{-\frac12} |\log\ve_i| r_i^{-1} \to R_0$, then $W_i(x) \to W(x)$ uniformly on compact sets where $W(x)$ is bounded for $|x|\leq R_0$ and $W(x) = \frac{1}{|x|^2}$ for $|x|\geq R_0$. Then $h$ solves
$$
\Delta h + W h = 0 \quad\text{in } \R^2
$$
with $|h(r) |\leq r$ for all $r\geq 0$. This implies $h\equiv 0$, a contradiction.

Finally, if $\ve_i^{-\frac12} |\log\ve_i| r_i^{-1} \to 0$, then $h$ satisfies
$$
\Delta h + \frac{1}{|x|^2} h = 0 \quad\text{in } \R^2 \setminus\{0\}
$$
with $|h(r) |\leq r$ for all $r>0$. Again this implies that $h$ is trivial.

Existence of a solution to \eqref{exterior problem} can be deduced from the solvability obtained in Lemma~\ref{solv medium decay} and the apriori estimate \eqref{eq91b}, with an approximation argument. Namely, let $g$ be radial with $\|(1+|x|)^{1-\ve}g \|_{L^\infty}<\infty$ and $\eta$ be a smooth cut-off function with $\eta(x)=1$ for $|x|\leq 1$, $\eta(x)=0$ for $|x|\geq2$. Thanks to  Lemma~\ref{solv medium decay} there is a radial solution $h_n$ of \eqref{exterior problem} with right hand side $g \eta(x/n)$. By  \eqref{eq91b} we have $\|(1+|x|)^{-1} h_n\|_{L^\infty}\leq C$ and by standard estimates $h_n$ is bounded is $C^{1,\alpha}_{loc}$. Up to subsequence $h_n$ converges to a solution $h$ satisfying
$$
\|(1+|x|)^{-1} h\|_{L^\infty}\leq C\|(1+|x|)^{1-\ve} g\|_{L^\infty}.
$$
Finally estimate \eqref{eq92} follows from a standard scaling argument and Schauder estimates for $L_\ve$, which is $(-\Delta)^{\frac{1+s}{2}}$ up to constant, and which are uniform as $\ve\to1$.
\end{proof}

Next we give a result analogous to Lemma~\ref{prop slow r2} but for functions with fast decay.

\begin{lemma}
There is a linear  operator that given a radial function $f $ in $\R^2$ such that $\| (1+|x|)^{2+\tau-\ve} f\|_{L^\infty}<\infty$  produces a solution $\phi$ of \eqref{exterior problem} with the property
\begin{align}
\label{est2b}
\| |x|^{\tau} \phi\|_{L^\infty} \leq C \| (1+|x|)^{2+\tau-\ve} f\|_{L^\infty}.
\end{align}
\end{lemma}
\begin{proof}
Let $Y$ denote the space of radial functions in $\R^2$ satisfying $\| |x|^{\tau} \phi\|_{L^\infty} <\infty$.
We claim there exists $\phi\in Y$ that depends linearly on $f$ satisfying
\begin{align}
\label{int2a}
\nabla \phi(x) = c_{2,\ve} \int_{\R^2} \left( \frac{x-y}{|x-y|^{2+\ve}} -\frac{x}{|x|^{2+\ve}}\right) \left( f(y) - \frac{\eta_\ve(|y|)}{|y|^{2-\ve}} \phi(y) \right) \, d y
\end{align}
and the estimate \eqref{est2b}. This function is the desired solution.
Here $c_{2,\ve}\to\frac1{2\pi}$ as $\ve\to0$.

Similar to Lemma~\ref{conv0} we have the following estimate.
Assume $0<\beta<2$, $2<\gamma<3$ and $\gamma+\beta>2$. Let $\| (1+|x|)^\gamma f\|_{L^\infty} < \infty$. Then
$$
\left|\int_{\R^2} \left( \frac{x-y}{|x-y|^{\beta+1}}-\frac{x}{|x|^{\beta+1}}\right) f(y)\, d y \right|\leq C \| (1+|x|)^\gamma f\|_{L^\infty} |x|^{2-\beta-\gamma} .
$$
Using this estimate with  $\beta=1+\ve$ we see that the integral  \eqref{int2a} is well defined if $\|(1+|x|)^{2+\tau-\ve} f\|_{\infty}<\infty$ and $\phi\in Y$.

We treat \eqref{int2a} as a perturbation of the case $\ve=0$.
So first we consider the equation
$$
\Delta \phi + \frac{\eta_\ve}{r^2} \phi = f\quad \text{in } \R^2
$$
with $\eta_\ve$ as in \eqref{def eta eps},
for which we want to construct a solution such that
\begin{align}
\label{est phi sigma}
\| |x|^{\tau} \phi\|_{L^\infty(\R^2)}  \leq \| (1+|x|)^{2+\tau} f\|_{L^\infty(\R^2)}.
\end{align}

For $r\geq \ve^{-\frac12}+1$ the equation is given by
$$
\frac{1}{r}(r\phi')' + \frac{1}{r^2} \phi = f , \quad r \geq \ve^{-\frac12},
$$
hence we take $\phi$ of the form
$$
\phi(r) =  \cos(\log(r)) \int_r^\infty \sin(\log(t)) t f(t) d t - \sin(\log(r)) \int_r^\infty \cos(\log(t)) t f(t) d t
$$
for $r\geq \ve^{-\frac12}+1$.
From this formula we get directly
$$
\sup_{r\geq \ve^{-\frac12}} r^\tau |\phi(r)| \leq \| r^{2+\tau} f\|_{L^\infty}.
$$
For $0<r\leq\ve^{-\frac12}+1$ we define $\phi$ as the unique solution of the equation
$$
\frac{1}{r}(r\phi')' + \frac{\eta_\ve(r)}{r^2} = f , \quad r \leq \ve^{-\frac12}+1,
$$
with initial conditions at $\ve^{-\frac12}+1$ to make $\phi$ a global solution for $r\in (0,\infty)$.
Note that
$$
\phi(\ve^{-\frac12}) = O(\ve^{\frac\tau2}), \quad \phi'(\ve^{-\frac12})=O(\ve^{\frac{1+\tau}2}) .
$$
Let $r_0=\ve^{-\frac12}$.
Then for $r\leq r_0$ we can represent
$$
\phi(r) = c_1 + c_2 \log(\frac{r}{r_0}) + \int_r^{r_0} \frac 1s \int_s^{r_0} t f(t) dt ds ,
$$
where $c_1,c_2$ have to satisfy
$$
c_1 = \phi(r_0) = O(\ve^{\frac\tau2}), \quad c_2 = r_0 \phi'(r_0) =O(\ve^{\frac\tau2}).
$$
With this formula we can verify \eqref{est phi sigma}.
The previous solution satisfies
\begin{align*}
\phi(x) = \frac{1}{2\pi} \int_{\R^2} \log\frac{1}{|x-y|} \left (f(y) - \frac{\eta_\ve(|y|)}{|y|^2} \phi(y) \right) \, dy + A \log|x| + B
\end{align*}
where $A$, $B$ depend on $f$ and are such that $\phi(x)\to0$ as $|x|\to\infty$. Therefore for the gradient we have
\begin{align}
\nonumber
\nabla \phi(x)
& = \frac{1}{2\pi} \int_{\R^2} \frac{x-y}{|x-y|^2} \left (f(y) - \frac{\eta_\ve(|y|)}{|y|^2} \phi(y) \right) \, dy + A \frac{x}{|x|^2}
\\
\label{int1a}
& = \frac{1}{2\pi} \int_{\R^2} \left( \frac{x-y}{|x-y|^2} -\frac{x}{|x|^2}\right) \left (f(y) - \frac{\eta_\ve(|y|)}{|y|^2} \phi(y) \right) \, d y .
\end{align}
Let $\phi=T(f)$ denote the operator that associates the function $\nabla\phi$ constructed above, so that in particular \eqref{est phi sigma} and \eqref{int1a} hold. To find a solution of \eqref{int2a} it then suffices to find $\phi\in Y$ such that
$$
\nabla\phi = T ( B_\ve(f) + A_0(\phi) - A_\ve(\phi))
$$
where the operators $B_\ve$, $ A_0$, $A_\ve$ are defined as
\begin{align*}
B_\ve(f)(x) & =  c_{2,\ve} \int_{\R^2} \left( \frac{x-y}{|x-y|^{2+\ve}} -\frac{x}{|x|^{2+\ve}}\right) f(y) \, d y
\\
A_\ve(\phi)(x) &= c_{2,\ve} \int_{\R^2} \left( \frac{x-y}{|x-y|^{2+\ve}} -\frac{x}{|x|^{2+\ve}}\right) \frac{\eta_\ve(|y|)}{|y|^{2-\ve}} \phi(y) \, d y
\\
A_0(\phi)(x) &= c_{2,\ve} \int_{\R^2} \left( \frac{x-y}{|x-y|^{2}} -\frac{x}{|x|^{2}}\right)  \frac{\eta_\ve(|y|)}{|y|^{2}} \phi(y)  \, d y  ,
\end{align*}
and $\phi$ is defined from $\nabla \phi$ by integration such that $\lim_{|x|\to\infty}\phi(x)=0$ (here all functions are radial).
Similarly as in Lemma~\ref{solv medium decay} we can show that for $\ve>0$ small the map from $Y$ to $Y$ given by $\phi\mapsto T(B_\ve(f)+A_0(\phi)-A_\ve(\phi))$ is a contraction.

\end{proof}

For the proof of Proposition~\ref{prop exterior slow} we need an estimate of
$$
a_\ve(x) = \ve
\int_{\Sigma_0}
\frac{1-\langle \nu_{\Sigma_0}(y), \nu_{\Sigma_0} (y) \rangle }{|x-y|^{4-\ve}} dy .
$$

\begin{lemma}
\label{lemma a epsilon}
Let $x=(x',F_\ve(x'))\in \Sigma_0$. Then
\begin{align*}
a_\ve(x) & = \pi|A_{\Sigma_0}|^2 |x'|^{\ve} + O(\frac{\ve}{(1+|x|)^{2-\ve}}) + O(\frac{\ve}{\log(|x|)^{2-\ve}}) \chi_{|x|\leq \ve^{-\frac12}}
\\&\qquad+  \pi \frac{\ve}{F_\ve(x')^{2-\ve}}(1+o(1))\chi_{|x|\geq \ve^{-\frac12}} ,
\end{align*}
where $|A_{\Sigma_0}|$ is the norm of the second fundamental form of $\Sigma_0$ and $O()$, $o()$ are uniform  $x$ as $\ve\to0$.
\end{lemma}

\begin{proof}
Let $R_1>0$ and  write
$$
a_\ve = a_\ve^+ + a_\ve^-,
$$
where
\begin{align*}
a_{\ve}^\pm (x)&= \ve\int_{\Sigma_0^\pm } \frac{ 1- \langle  \nu_{\Sigma_0}(y), \nu_{\Sigma_0} (x) \rangle }{|x-y|^{4-\ve}} dy,
\end{align*}
and   $\Sigma_0^\pm$ is $\Sigma_0$ intersected with $x_3\geq 0$ or $x_3\leq 0$ respectively.
Let us split
$$
a_\ve^+= a_{1,\ve}^+ + a_{2,\ve}^+,
$$
where
\begin{align*}
a_{1,\ve}^+ &= \ve\int_{\Sigma_0^+ \cap C_{R_1}(x) } \frac{ 1- \langle  \nu_{\Sigma_0}(y), \nu_{\Sigma_0} (x) \rangle }{|x-y|^{4-\ve}} dy,
\\
a_{2,\ve}^+ &= \ve\int_{\Sigma_0^+ \setminus C_{R_1}(x) } \frac{ 1- \langle  \nu_{\Sigma_0}(y), \nu_{\Sigma_0} (x) \rangle }{|x-y|^{4-\ve}} dy ,
\end{align*}
and $C_{R_1}(x)$ is the cylinder with base the disk of radius $R_1$ on the tangent plane to $\Sigma_0$ at $x$, and height $R_1$, which  will be chosen later depending on $x$. Let $g:B_{R_1}(0)\subset \R^2\to\R$ be such that $\Sigma_0$ can be described as the graph of $g$ over the tangent plane at $X$. Then
$$
a_{1,\ve}^+ =\ve \int_{|t|\leq R_1} \frac{\sqrt{1+|\nabla g|^2}-1}{(|t|^2+g(t)^2)^{\frac{4-\ve}{2}}} \, d t .
$$
A calculation gives
$$
a_{1,\ve}^+(x) = \pi |A_{\Sigma_0}(X)|^2 R_1^\ve + O(\ve [D^2g]_{\alpha,B_{R_1}}R_1^{\alpha+\ve}).
$$
We choose now $R_1$ as follows. Recall that $x= (x',F_\ve(x'))$ and $|x|\sim|x'|$. If $|x'|\leq 100$ we take $R_1>0$ a fixed small constant so that the representation of $\Sigma_0 \cap C_{R_1}(x)$ by a graph is possible. If $|x'|>100$, we take $R = \delta|x'|$ with $\delta>0$ a small positive constant.
By estimates \eqref{est holder g}
$$
[D^2 g]_{\alpha,B_{R_1}}
\leq
\begin{cases}
\frac{C\ve^{\frac12}}{|x'|^{1+\alpha}}
& \text{if } |x'| \geq  \ve^{-\frac12}
\\
\frac{C}{|x'|^{2 +\alpha}}
& \text{if } |x'| \leq  \ve^{-\frac12} .
\end{cases}
$$
and therefore

------  X -- x

$$
a_{1,\ve}^+ = \pi|A_{\Sigma_0}|^2 |x'|^{\ve} + O(\frac{\ve}{(1+|x|)^{2-\ve}})
$$
as $\ve \to 0$. On the other hand a direct estimate gives
$$
a_{2,\ve}^+ = O(\frac{\ve}{(1+|x|)^{2-\ve}}) .
$$
Therefore
$$
a_\ve^+ = \pi |A_{\Sigma_0}|^2 |x|^\ve +  O(\frac{\ve}{(1+|x|)^{2-\ve}}).
$$

We can write explicitly
$$
a_{\ve}^-(x) = \ve \int_{|y|\ge 1}
\frac{ \sqrt{1+|\nabla F_\ve(y)|^2} + \frac{1-F_\ve'(x)F_\ve'(y)}{ \sqrt{1+|\nabla F_\ve(x)|^2}}}{(|x-y|^2+(F_\ve(x)+F_\ve(y))^2)^\frac{4-\ve}{2}} \, dy
$$
For $10\leq |x|\leq \ve^{-\frac12}$ we estimate
\begin{align*}
|a_{\ve}^-(x)| &\leq \ve C\int_{|y|\ge 1}
\frac{ 1}{(|x-y|^2+ F_\ve(x)^2 )^\frac{4-\ve}{2}} \, dy\\
&\leq\frac{C\ve}{F_\ve(x)^{2-\ve}}
\leq\frac{C\ve}{\log(x)^{2-\ve}} .
\end{align*}
For $|x|\geq\ve^{-\frac12}$ we split $a_\ve^- = a_{1,\ve}^- + a_{2,\ve}^-$ where
\begin{align*}
a_{1,\ve}^-(x) &= \ve\int_{\Sigma_0^- \cap \tilde C_{R_2}(x,0) } \frac{ 1- \langle  \nu_{\Sigma_0}(Y), \nu_{\Sigma_0} (X) \rangle }{|X-Y|^{4-\ve}} dY,
\\
a_{2,\ve}^-(x) &= \ve\int_{\Sigma_0^- \setminus \tilde C_{R_2}(x,0) } \frac{ 1- \langle  \nu_{\Sigma_0}(Y), \nu_{\Sigma_0} (X) \rangle }{|X-Y|^{4-\ve}} dY ,
\end{align*}
where $\tilde C_{R_2}(X)$ is the cylinder with base a disk of radius $R_2$ on $\R^2$ centered at $(x,0)$. We choose $R_2 = |\log\ve|^{-\frac12} |x|$. Then
\begin{align*}
|a_{2,\ve}^-(x)| &= \ve C\int_{|y-x|\geq R} \frac{ 1}{(|x-y|^2+ F_\ve(x)^2 )^\frac{4-\ve}{2}} \, dy\\
&\leq \frac{C \ve |\log\ve|}{|x|^{2-\ve}}.
\end{align*}
For $|x|\geq\ve^{-\frac12}$ and $|y-x|\leq R_2$, $F_\ve'(x)=O(\ve^{\frac12}) $ and $F_\ve(y) = F_\ve(x) + O(\ve^{\frac12} R) $ so
\begin{align*}
a_{\ve}^-(x)& =  \ve \int_{|x-y|\leq R}
\frac{2+O(\ve)}{(|x-y|^2+(F_\ve(x)+F_\ve(y))^2)^\frac{4-\ve}{2}} \, dy\\
&= \pi \frac{\ve}{F_\ve(x)^{2-\ve}}(1+o(1))
\end{align*}
where $o(1)\to0$ uniformly as $\ve\to0$.

\end{proof}

\begin{proof}[Proof of Propositions~\ref{prop exterior slow} and \ref{prop fast decay}]
The idea is to reduce problem \eqref{eq50} to one in $\R^2$.
Suppose that $\phi$ is a radial function on $\Sigma_0$, symmetric with respect to $x_3=0$ vanishing in $B_{2R}(0)$. Here $R>0$ is large and fixed, to be chosen later.
Since $\phi$ is symmetric with respect to $x_3=0$, we can  define $\tilde \phi$ globally in $\R^2$ by
$$
\tilde \phi(x) = \phi(x,\pm F_\ve(x)) , \quad |x| \geq R,
$$
and $\tilde \phi=0$ in $B_R(0)$.
Let $C_R$ be the cylinder
$$
C_R = \{ (x_1,x_2,x_3)\in\R^3: x_1^2+x_2^2<R^2\}.
$$
Then, for $X\in \Sigma_0$ of the form $X=(x,F_\ve(x))$ with $|x|\geq R$, we have
\begin{align*}
& \text{p.v.} \int_{\Sigma_0 \setminus C_R} \frac{\phi(Y)-\phi(X)}{|Y-X|^{4-\ve}} \, dY
\\
& = \text{p.v.}\int_{\R^2\setminus B_R} \frac{\tilde\phi(y)-\tilde\phi(x)}{(|x-y|^2 +(F_\ve(x)-F_\ve(y))^2)^{\frac{4-\ve}{2}}} \sqrt{1+|\nabla F_\ve(y)|^2}\, dy
\\
&\quad+\int_{\R^2\setminus B_R} \frac{\tilde\phi(y)-\tilde\phi(x)}{(|x-y|^2 +(F_\ve(x)+F_\ve(y))^2)^{\frac{4-\ve}{2}}} \sqrt{1+|\nabla F_\ve(y)|^2} \, dy
\end{align*}
Then we find for $|X|\geq R$, $X=(x,F_\ve(x))$,
$$
\text{p.v.}
\int_{\Sigma_0 } \frac{\phi(Y)-\phi(X)}{|Y-X|^{4-\ve}} \, dY
=
\text{p.v.}\int_{\R^2} \frac{\tilde \phi(y)-\tilde\phi(x)}{|y-x|^{4-\ve}}\, d y
+ b(x) \tilde \phi(x)
+ B_1(\tilde\phi)(x)
$$
where
\begin{align*}
b(x) & = \int_{B_R} \frac{1}{|x-y|^{4-\ve}}\,dy
-\int_{\Sigma_0 \cap C_R}\frac{1}{|(x,F_\ve(x))-Y|^{4-\ve}}\, d Y
\\
B_1(\tilde\phi)(x)&= \int_{\R^2\setminus B_R} \left( \tilde\phi(y)-\tilde\phi(x)\right) \left( \frac{\sqrt{1+|\nabla F_\ve(y)|^2}}{(|x-y|^2 +(F_\ve(x)-F_\ve(y))^2)^{\frac{4-\ve}{2}}}  - \frac{1}{|x-y|^{4-\ve}}\right) \, dy
\\
&\quad
+\int_{\R^2\setminus B_R} \frac{\tilde\phi(y)-\tilde\phi(x)}{(|x-y|^2 +(F_\ve(x)+F_\ve(y))^2)^{\frac{4-\ve}{2}}} \sqrt{1+|\nabla F_\ve(y)|^2} \, dy .
\end{align*}
Let
$$
a_\ve(X) = \ve
\int_{\Sigma_0}
\frac{1-\langle \nu_{\Sigma_0}(Y), \nu_{\Sigma_0} (X) \rangle }{|X-Y|^{3+s}} dY .
$$
Then \eqref{eq50} reads as
\begin{align}
\label{eq91}
L_\ve(\tilde\phi)
+ \frac{\eta_\ve}{|x|^{2-\ve}}\tilde \phi(x)
+ \ve B_1(\tilde\phi)(x)
+ ( \ve b(x) + a_\ve - \frac{\eta_\ve}{|x|^{2-\ve}} )\tilde \phi(x) = \tilde f(x)
\end{align}
where $\tilde f(x) = f(x,F_\ve(x))$ and $L_\ve$ is the operator \eqref{Leps}. We look for $\tilde \phi$ of the form $\tilde \phi = \eta \varphi$, where $\eta$ is a smooth radial cut-off function such that $\eta(x)=1$ for $|x|\geq 3R$ and $\eta(x)=0$ for $|x|\leq 2R$. Then we ask that $\varphi$ solves
\begin{align}
\label{eq90}
L_\ve(\varphi)
+ \frac{\eta_\ve}{|x|^{1-s}}\varphi
+ \ve B_2(\varphi)
+ \eta ( \ve b(x) + a_\ve - \frac{\eta_\ve}{|x|^{1-s}} )\varphi = \tilde f(x) \quad\text{in }\R^2 ,
\end{align}
where
$$
B_2(\varphi)(x) =
\ve \tilde \eta(x)\int_{\R^2} \varphi(y) \frac{\eta(y)-\eta(x)}{|x-y|^{4-\ve}}\, d y
+ \ve \tilde \eta(x) B_1[\eta\varphi](x),
$$
and where $\tilde \eta$ is another radial smooth cut-off function such that $\tilde \eta(x)=1$ for $|x|\geq 5R$, $\tilde \eta(x)=0$ for $|x|\leq 4R$.
If $\varphi$ solves \eqref{eq90}, then $\tilde \phi = \eta \varphi$ will satisfy \eqref{eq91} for $|x|\geq 5 R$. Let $T$ denote the operator constructed in Lemma~\ref{prop slow r2}, so that $\phi=T(f)$ is a radial solution to \eqref{exterior problem} satisfying the estimate \eqref{eq92}. Then we rewrite \eqref{eq90} as the fixed point problem
$$
\varphi = T(-\ve B_2(\varphi) - \eta ( \ve b(x) + a_\ve - \frac{\eta_\ve}{|x|^{1-s}} )\varphi +\tilde f).
$$
We can apply the contraction mapping principle by the following estimates
$$
\| \ve B_2(\varphi) \|_{1-\ve,\alpha} \leq o(1) \|\varphi\|_*
$$
$$
\|\eta ( \ve b(x) + a_\ve - \frac{\eta_\ve}{|x|^{2-\ve}} ) \varphi\|_{1-\ve,\alpha}\leq o(1) \|\varphi\|_*
$$
where $o(1)\to 0$ as $\ve\to0$ and $R\to\infty$, which can be proved using Lemma~\ref{lemma a epsilon}.

The proof of Proposition~\ref{prop fast decay} follows the same lines as the one of Proposition~\ref{prop exterior slow}.
\end{proof}

\section{Linear theory}
\label{sect linear}

The purpose here is to construct a linear operator $f \mapsto \phi$ which gives a solution to the problem
\begin{align}
\label{mainLeq}
\ve\mathcal J^s_{\Sigma_0}(\phi) = f \quad \text{in } \Sigma_0 ,
\end{align}
where $\mathcal J_{\Sigma_0}^s$ is the nonlocal Jacobi operator
$$
\mathcal J^s_{\Sigma_0}(\phi) (x)=
\text{p.v.} \int_{\Sigma_0} \frac{\phi(y)-\phi(x)}{|x-y|^{4-\ve}} \, d y+\phi(x) \int_{\Sigma_0} \frac{(\nu(x)-\nu(y))\cdot \nu(x)}{|x-y|^{4-\ve}} dy ,
$$
and $\Sigma_0$ is the surface defined in \eqref{def sigma0}.

The main result is stated in Proposition~\ref{main linear prop}, which we recall:
there is a linear operator that to a function $f$ on $\Sigma_0$ such that $f$ is radially symmetric and symmetric with respect to $x_3=0$ with $\|f\|_{1-\ve,\alpha+\ve}<\infty$, gives a solution $\phi$ of \eqref{mainLeq}. Moreover
$$
\|\phi\|_* \leq C \|f\|_{1-\ve,\alpha+\ve}.
$$
The norms $ \|\ \|_{1-\ve,\alpha+\ve}$ and $\| \ \|_*$ are defined in \eqref{norm RHS}, \eqref{norm st}.

As $\ve \to 0$, $\Sigma_0$ approaches the standard catenoid
$\cat$ on compact sets, which can be  described by  the parametrization
$$
y\in \R \mapsto \big( \sqrt{1+y^2} \cos(\theta) , \sqrt{1+y^2} \sin(\theta) , \log( y +  \sqrt{1+y^2} ) \big)
$$
with $y \in\R$, $\theta \in [0,2\pi]$.
Hence for smooth bounded $\phi$ we have
$$
\ve \mathcal J^s_{\Sigma_0}(\phi) \to \frac\pi2( \Delta_{\cat} \phi + |A|^2\phi )
$$
uniformly over compact sets
as $\ve \to0$, where $\Delta_\cat$ is the Laplace-Beltrami operator and $|A|$ the norm of the second fundamental form of $\cat$ (see Lemmas~\ref{conv lapl} and \ref{lemma conv A}).

Let us recall
the standard nondegeneracy property of the Jacobi operator $ \Delta_{\cat}  + |A|^2$ on the catenoid.
Linearly  independent elements in its kernel are the functions
\begin{align}
\label{Z1 Z2}
Z_1 (y) = \frac{y} { \sqrt{y^2+1}}, \quad  Z_2(y) = - 1 + \frac{y} { \sqrt{y^2+1}}\log( y+ \sqrt{y^2+1}) .
\end{align}
The knowledge of these elements in the kernel of $ \Delta_{\cat}  + |A|^2$ immediately yields


\begin{lemma}
\label{lemma jacobi catenoid}
If $\phi$ is a bounded axially symmetric solution of $\Delta_{\cat} \phi  + |A|^2\phi =0$ in $\cat$ then $\phi = c Z_1$ for some $c\in\R$.
\end{lemma}	



%
%
%

Let
$$
a_\ve(x) = \ve
\int_{\Sigma_0}
\frac{1-\langle \nu_{\Sigma_0}(y), \nu_{\Sigma_0} (x) \rangle }{|x-y|^{3+s}} dy
$$
and
$$
b_\ve(x) = a_\ve(x) \eta_\ve(x)
$$
where $\eta_\ve$ is smooth, radial, $\eta(x)=0$ for $|x|\geq \ve^{-\frac12}+1$, and $\eta(x)=1$ for $|x|\leq \ve^{-\frac12}$.

Let us write
$$
L_\ve(\phi)(x) = \ve \, \text{p.v.}\int_{\Sigma_0} \frac{\phi(y)-\phi(x)}{|x-y|^{4-\ve}}dy
$$
and consider the equation
\begin{align}
\label{frac cat}
L_\ve(\phi) +  b_\ve(x) \phi =  f \quad \text{in } \Sigma_0 .
\end{align}
%

We will consider from now only right hand sides $f:\Sigma_0 \to \R$ which are symmetric with respect to the plane $x_3=0$, and symmetric solutions  $\phi$.

Let $0<\tau<1$.

\begin{prop}
\label{prop exist frac cat}
For $\ve>0$ small there is a linear operator that takes   $f$ symmetric with respect to $x_3$ with  $\| y^{2+\tau -\ve} f \|_{L^\infty}<\infty$ to a  a symmetric  bounded  solution $\phi$ of \eqref{frac cat}. Moreover
$$
\|\phi\|_{L^\infty} \leq C \| y^{2+\tau-\ve} f\|_{L^\infty},
$$
\begin{align}
\label{grad}
\| (1+|y|)^{1+\tau} \nabla \phi\|_{L^\infty} \leq C \| y^{2+\tau-\ve} f\|_{L^\infty},
\end{align}
and  $ \lim_{|x|\to\infty} \phi(x)  $ exists.
\end{prop}

The counterpart of this result for the Jacobi operator $\Delta_\cat+|A|^2$, without assuming any symmetry on $f$ or $\phi$ is: if $\||y|^{2+\tau}f\|_{L^\infty}<\infty$ and $\int_\cat f Z_1=0$, there is a bounded solution $\phi$ of
$$
\Delta_\cat \phi + |A|^2 \phi=f \quad \text{in }\cat,
$$
and this solution is unique except a constant times $Z_1$. Moreover $\phi$ has limits at both ends, which have to coincide. In the nonlocal setting, to simplify we work with functions that are symmetric with respect to $x_3$, so in some sense the condition $\int_\cat f Z_1=0$ is automatic.

For the existence part in Proposition~\ref{prop exist frac cat}
 we study the truncated problem
\begin{align}
\label{truncated frac cat}
\left\{
\begin{aligned}
& L_{\ve}(\phi ) + b_\ve \phi  = f \quad \text{in } \Sigma_0 \cap B_R(0)
\\
& \phi = 0\quad\text{on } \Sigma_0 \setminus B_R(0)
\end{aligned}
\right.
\end{align}

Let
$$
\sigma = \frac{1+s}2 =1-\frac\ve2 .
$$
Given in $f\in L^2(\Sigma_0 \cap B_R(0))$ there is a weak solution $\phi \in H^{\sigma}(\Sigma_0)$  of
\begin{align*}
\left\{
\begin{aligned}
& - L_{\ve}(\phi )  = f \quad \text{in } \Sigma_0 \cap B_R(0)
\\
& \phi = 0\quad\text{on } \Sigma_0 \setminus B_R(0)
\end{aligned}
\right.
\end{align*}
By weak solution we mean $\phi \in H^\sigma(\Sigma_0)$, $\phi=0$ on $\Sigma_0 \setminus B_R(0)$ and
$$
\int_{\Sigma_0}\int_{\Sigma_0}\frac{(\phi(y)-\phi(x))(\varphi(y)-\varphi(x))}{|x-y|^{2+2 \sigma}} \, d y dx = \int_{\Sigma_0} f(x)\varphi(x)\, dx
$$
for all $\varphi \in H^{\sigma}(\Sigma_0)$ with $\varphi = 0 $ in $ \Sigma_0 \setminus B_R(0)$. This solution can be found by minimizing the functional
$$
\frac14\int_{\Sigma_0}\int_{\Sigma_0} \frac{(\phi(y)-\phi(x))^2}{|x-y|^{2+2\sigma}} \, d y d x - \int_{\Sigma_0} f(x) \phi(x)\, d x
$$
over the space $\{\phi\in H^{\sigma}(\Sigma_0): \phi=0\quad \text{on }\Sigma_0\setminus B_R(0)\}$. For $f$ locally bounded and $\ve>0$ small ($\sigma$ is close to 1), the solution belongs to $C^{1,\alpha}_{loc}$.

First we establish an apriori estimate for solutions of \eqref{truncated frac cat}.
\begin{lemma}
\label{lemma apriori trcat}
Suppose $f$ is symmetric and  $\| |y|^{2+\tau-\ve} f \|_{L^\infty}<\infty$.
There are $\ve_0,R_0,C>0$ such that for $0<\ve\leq \ve_0$, $R\geq R_0$, and any symmetric solution $\phi$ of \eqref{truncated frac cat} we have
$$
\|\phi\|_{L^\infty} \leq C \| |y|^{2+\tau-\ve} f\|_{L^\infty}.
$$
\end{lemma}
\begin{proof}
If the conclusion fails, there are sequences $\ve_n\to 0$, $R_n\to\infty$, $\phi_n$ solving \eqref{truncated frac cat} for some $f_n$ such that
$$
\|\phi_n\|_{L^\infty}=1 , \quad  \| |y|^{2+\tau-\ve_n} f_n \|_{L^\infty} \to 0
$$
as $n\to\infty$.
We show that for any $\rho>0$ fixed
$$
\sup_{\Sigma_0 \cap B_\rho(0)} |\phi_n|\to 0\quad \text{as }n\to\infty.
$$
If not, then passing to a subsequence, for some $x_n \in \Sigma_0 \cap B_\rho(0)$,
$$
|\phi_n(x_n)|\geq \delta>0.
$$
By standard estimates, $\phi_n$ is bounded in $C^{\alpha}_{loc}$. Hence by passing to a new subsequence, $\phi_n \to \phi$ locally uniformly as $n\to \infty$.  We pass to the limit in the weak formulation  and obtain a bounded symmetric  solution $\phi \not\equiv 0$ of
$$
\Delta_{\cat} \phi + |A|^2 \phi  = 0 \quad\text{in } \cat.
$$
But by Lemma~\ref{lemma jacobi catenoid} the only bounded solution is $c Z_1$, which is odd. Hence $\phi\equiv0$ and this is a contradiction.

We claim that
$$
\|\phi_n\|_{L^\infty(\Sigma_0\cap B_{R_n}(0))} \to 0
$$
as $n\to\infty$, which is a contradiction.

Indeed, let $w= 1  -  \delta |y|^{-\tau}$.
One can check that
$$
L_{\ve_n}(w) \leq  - c_{\ve_n}  \delta |y|^{-\tau-2+\ve_n}
$$
for $|y|\geq \bar R$ where $\bar R$ is large and fixed and $c_{\ve_n}$ converges to a positive constant as $\ve_n\to0$.
Next we choose $\delta >0$ such that $\inf_{\Sigma_0 \cap B_{\bar R}(0))} w>0$.
We claim that
\begin{align}
\label{qq1}
\phi_n \leq C (\|\phi\|_{L^\infty(\Sigma_0 \cap B_{\bar R}(0))} +
\||y|^{\tau+2-\ve_n} f_n \|_{L^\infty}) w
\end{align}
in $ \Sigma_0 \cap ( B_{R_n}(0) \setminus B_{\bar R}(0)) $.
Note that  \eqref{qq1} holds for $C$ large depending on $\phi_n$ because $\phi_n$ is bounded. The claim is that this holds for $C=C_0$ with
$$
C_0 = \max \left(  2 (\inf_{\Sigma_0 \cap B_{\bar R}(0))} w)^{-1} ,
\sup \frac{|f_n|}{c_{\ve_n} \delta |y|^{-\tau-2-+\ve_n}}\right)
$$
The comparison can be done by sliding.
\end{proof}

%
%
%
%

Using the Fredholm alternative, we deduce the following result.

\begin{lemma}
Suppose $f$ is symmetric and  $\| |y|^{2+\tau-\ve } f \|_{L^\infty}<\infty$.  For $0<\ve\leq \ve_0$ and $R\geq R_0$ there is a unique symmetric solution $\phi$ of \eqref{truncated frac cat}.
\end{lemma}

\begin{proof}[Proof of Proposition~\ref{prop exist frac cat}]
We fix $0<\ve\leq \ve_0$ for $R\geq R_0$ and let $\phi_R$ be the solution of \eqref{truncated frac cat}. Then for a sequence $R_j\to\infty$, $\phi = \lim_{j\to\infty} \phi_{R_j}$ exists and is a solution of \eqref{frac cat}.
Estimate \eqref{grad} is obtained by scaling and the gradient estimates of Caffarelli and Silvestre \cite{caffarelli-silvestre-arma}.
Finally $ \lim_{|x|\to\infty} \phi(x)  $ exists because of  \eqref{grad}.
\end{proof}

We need a solvability theory with a constraint on the right hand side so that the solution decays. For this we consider the equation
\begin{align}
\label{eq70}
L_\ve(\phi) + b_\ve \phi = f -  c Z_2 \eta_1 \quad\text{in } \Sigma_0,
\end{align}
where $\eta_1 $ is a smooth radial symmetric cut-off function on $\Sigma_0$, such that $\eta_1 (x)=1$ for $|x|\leq A_1$, $\eta_1(x)=0$ for $|x|\geq A_1+1$ and $A_1$ is a fixed large constant.
The function $Z_2 \eta_1$ in the right hand side can be replaced by any $f_0$ with $f_0(x) = O (|x|^{-2-\tau+\ve})$, $\int_{\Sigma_0} f_0 Z_2 \not=0$.
\begin{prop}
\label{prop exist frac cat decay}
There is $\ve_0>0$ such that for all $0<\ve\leq\ve_0$ and any $f$  symmetric with respect to $x_3$ with  $\| |y|^{2+\tau-\ve } f \|_{L^\infty}<\infty$ there is a unique solution $\phi$, $c$ of \eqref{eq70} such that $\phi$ is symmetric and $\||y|^{\tau} \phi\|_{L^\infty}<\infty$.
Moreover
$$
\||y|^{\tau} \phi\|_{L^\infty} +|c|\leq C \| |y|^{2+\tau-\ve} f\|_{L^\infty}.
$$
\end{prop}
\begin{proof}
First we prove existence.
For this we let $\phi_0$ be the solution of \eqref{frac cat} constructed in Proposition~\ref{prop exist frac cat} with right hand side $Z_2 \eta_1$. Then $\lim_{|x|\to\infty}\phi_0(x) = \Lambda_\ve$ exists. We claim that $\Lambda_\ve>0$ stays bounded and bounded away from 0 as $\ve\to 0$. To prove this,
let $Z_2$ be given as in \eqref{Z1 Z2}. Multiply \eqref{frac cat} by $Z_0 \eta_R$ where $\eta_R(x) = \eta(x/R)$ and $\eta$ is a radial, symmetric, smooth cut-off function such that $\eta(x)=1$ for $|x|\leq 1$, $\eta(x)=0$ for $|x|\geq 2$.
Then for $R\geq A_1 +1$ we find
\begin{align*}
& \ve \int_{\Sigma_0} \phi_0 (x) \int_{\Sigma_0} Z_2(y)\frac{\eta_R(y)-\eta_R(x)}{|x-y|^{4-\ve}} \, d y \, d x
+ \int_{\Sigma_0} \phi_0 \eta_R g_0
= \int_{\Sigma_0} Z_2^2 \eta_1
\end{align*}
where
$$
g_0=L_{\ve}(Z_2)+ b_{\ve} Z_2.
$$

Let us consider the first term
\begin{align*}
& \ve \int_{\Sigma_0(\ve)} \phi_0 (x) \int_{\Sigma_0(\ve)} Z_2(y)\frac{\eta_R(y)-\eta_R(x)}{|x-y|^{4-\ve}} \, d y \, d x
=I_1-I_2
\end{align*}
with
\begin{align*}
I_1 &  =
\frac12\ve\int_{\Sigma_0} \int_{\Sigma_0}
\phi_0(x)
\frac{(Z_2(y)-Z_2(x))(\eta_R(y)-\eta_R(x))}{|x-y|^{4-\ve}} \, d y d x
\\
I_2 &=
\frac12\ve\int_{\Sigma_0}\int_{\Sigma_0} Z_2(y) \frac{(\phi_0(y)-\phi_0(x))(\eta_R(y)-\eta_R(x))}{|x-y|^{4-\ve}} \, d y d x .
\end{align*}
Since $\phi_0(x) = \Lambda_\ve + O(R^{-1-\tau})$ for $|x|\geq R/4$ it is possible to show that
$$
I_1 = a_\ve \Lambda_\ve + o(1)
$$
where $a_\ve>0$ remains bounded and bounded away from 0  and $o(1)\to0$ as as $\ve\to0$ and $R\to\infty$ with $R^{\ve} \to 1$.
Indeed,  consider the regions $R_1=\{x \in \Sigma_0 : |x|\leq R/2\}$, $R_2 = \{R/2\leq |x| \leq 4R\}$, $R_3 = \{|x|\geq 4R\}$.
Then
$$
\ve\int_{x\in R_j}\int_{y\in R_j}
\phi_0(x)
\frac{(Z_2(y)-Z_2(x))(\eta_R(y)-\eta_R(x))}{|x-y|^{4-\ve}} \, d y d x
 = 0
$$
for $j=1,3$.
We have
\begin{align*}
& \left|
\ve\int_{x\in R_1}\int_{y\in R_2}
\phi_0(x)
\frac{(Z_2(y)-Z_2(x))(\eta_R(y)-\eta_R(x))}{|x-y|^{4-\ve}} \, d y d x
\right|
\\
& \leq 2 \ve \|\phi_0\|_{L^\infty} \log(R)
\int_{|x|\leq  R/4}\int_{|y|\geq R/2}
\frac{1}{|x-y|^{4-\ve}} \, d y d x
\\
& \leq C \ve \log(R) R^{-\ve}.
\end{align*}
and
\begin{align*}
&
\ve\int_{x\in R_2}\int_{y\in R_2}
\phi_0(x)
\frac{(Z_2(y)-Z_2(x))(\eta_R(y)-\eta_R(x))}{|x-y|^{4-\ve}} \, d y d x
\\
&= ( \Lambda_\ve + O(R^{-1-\tau}) )
\ve\int_{x\in R_2}\int_{y\in R_2}
\frac{(Z_2(y)-Z_2(x))(\eta_R(y)-\eta_R(x))}{|x-y|^{4-\ve}} \, d y d x
\\
&=
( \Lambda_\ve + O(R^{-1-\tau}) )
R^{\ve} \pi (1 + O(\ve) + O(R^{-1}) ).
\end{align*}
In the last integral we have rescaled by $R$ and used the expansion for $Z_2$.
Other terms in $I_1$ can be handled similarly.
Also, similar calculations show that $I_2\to 0$.

Now let  $\hat\phi$ be the solution of \eqref{frac cat} constructed in Proposition~\ref{prop exist frac cat} with right hand side $f$.
Let $\ell=\lim_{|x|\to\infty} \hat\phi(x)$, which exists by Proposition~\ref{prop exist frac cat}. Then $\phi = \hat\phi - \frac{\ell}{\Lambda_\ve} \phi_0$ satisfies
$$
\Lambda_\ve(\phi) + b_\ve \phi = f - \frac{\ell}{\Lambda_\ve} Z_2 \eta_1  .
$$
Moreover we have the estimates $|\ell|\leq C  \| |y|^{-2-\tau} f\|_{L^\infty}$ and
$$
\| |y|^\tau \phi\|_{L^\infty} \leq C \| |y|^{-2-\tau} f\|_{L^\infty}
$$
by \eqref{grad}.

Let us prove uniqueness. Suppose that for a sequence $\ve_n\to 0$ there is a nontrivial solution $\phi_n$, $c_n$ of \eqref{eq70} with $f=0$. We can assume
\begin{align}
\label{norma1}
\||y|^{\tau} \phi\|_{L^\infty} = 1.
\end{align}
To estimate $c_n$,
we test equation \eqref{eq70} with $Z_2 \eta_n$ where $\eta_n$ is a smooth cut-off function such that $\eta_n(r)=1$ for $r\leq R_n$ and $\eta_n(r)=0$ for $r\geq 2 R_n$, with
\begin{align}
\label{cond Rn}
R_n\to\infty \quad\text{and}\quad R_n \ve_n^{\frac12} \to0.
\end{align}
We get
\begin{align*}
& \ve_n \int_{\Sigma_0} \phi_n (x) \int_{\Sigma_0} Z_2(y)\frac{\eta_n(y)-\eta_n(x)}{|x-y|^{4-\ve_n}} \, d y \, d x
+ \int_{\Sigma_0} \phi_n \eta_n g_n
\\
& = 
- c_n \int_{\Sigma_0} b_{\ve_n} Z_2 \eta_n ,
\end{align*}
where
$$
g_n = L_{\ve_n}(Z_2)+ b_{\ve_n} Z_2.
$$
We claim that
$$
\ve_n \int_{\Sigma_0(\ve_n)} \phi_n (x) \int_{\Sigma_0(\ve_n)} Z_2(y)\frac{\eta_n(y)-\eta_n(x)}{|x-y|^{4-\ve_n}} \, d y \, d x
\to 0
$$
as $n\to \infty$.
Indeed
\begin{align*}
& \ve_n \int_{\Sigma_0(\ve_n)} \phi_n (x) \int_{\Sigma_0(\ve_n)} Z_2(y)\frac{\eta_n(y)-\eta_n(x)}{|x-y|^{4-\ve_n}} \, d y \, d x
\\
& \qquad =
\int_{\Sigma_0} \phi_n Z_2  L_{\ve_n}(\eta_n)
+ \ve_n \int_{\Sigma_0} \phi_n(x) \int_{\Sigma_0}
\frac{(Z_2(y)-Z_2(x))(\eta_n(y)-\eta_n(x))}{|x-y|^{4-\ve_n}}\, dy \, d  x .
\end{align*}
By calculation
\begin{align*}
L_{\ve_n}(\eta_n) (x) =
\begin{cases}
O(R_n^{\ve_n-2}) & \text{if } |x|\leq 10 R_n \\
O(\frac{\ve_n R_n^{4-\ve_n}}{|x|^{4-\ve_n}}) & \text{if } |x|\geq 10R_n
\end{cases}
\end{align*}
Then by \eqref{norma1}
\begin{align*}
\left|
\int_{\Sigma_0} \phi_n Z_2  L_{\ve_n}(\eta_n)
\right|
& \leq C R_n^{\ve_n-2} \int_{B_{R_n} \subset \R^2} |x|^{-\tau} \log(2+|x|)
\\
& \qquad + C \ve_n R_n^{4-\ve_n} \int_{B_{R_n}^c \subset \R^2} |x|^{\ve_n-4-\tau}\log(2+|x|)
\\
& \leq C R_n^{-\tau} \log(R_n) + C  \ve_n R_n^{2-\ve_n-\tau} \log(R_n)\to 0
\end{align*}
as $n\to\infty$ by \eqref{cond Rn}.
Similarly
\begin{align*}
\ve_n\int_{\Sigma_0}
\frac{(Z_2(y)-Z_2(x))(\eta_n(y)-\eta_n(x))}{|x-y|^{4-\ve_n}}\, dy
=
\begin{cases}
O(R_n^{\ve_n-2}) & \text{if } |x|\leq 10 R_n\\
O(\frac{\ve_n R_n^{4-\ve_n}}{|x_n|^{4-\ve_n}} \log(\frac{|x|}{R_n} )
& \text{if } |x|\geq 10R_n.
\end{cases}
\end{align*}
This implies
\begin{align*}
\left|\ve_n \int_{\Sigma_0} \phi_n(x) \int_{\Sigma_0}
\frac{(Z_2(y)-Z_2(x))(\eta_n(y)-\eta_n(x))}{|x-y|^{4-\ve_n}}\, dy \, d  x \right| \to 0
\end{align*}
as $n\to\infty$ as before.

We also have
\begin{align}
\label{claim 111}
\int_{\Sigma_0} \phi_n(y) \eta_n(y) g_n(y) \, d y \to 0
\end{align}
as $n\to \infty$. Indeed, $g_n = L_{\ve_n}(Z_2) + b_{\ve_n}Z_2 \to \frac\pi2 ( \Delta_\cat + |A|^2) Z_2 = 0 $ uniformly on compact sets (Lemmas~\ref{conv lapl} and \ref{lemma conv A}), so for any fixed $\rho>0$
\begin{align}
\label{int B rho}
\int_{\Sigma_0 \cap B_\rho } \phi_n(y) \eta_n(y) g_n(y) \, d y \to 0
\end{align}
as $n\to\infty$. For the integral in $\Sigma_0 \setminus B_\rho $,
we note that
$$
|L_{\ve_n}(Z_2)(x) = O(\log(|x|) |x|^{\ve_n-4})
$$
and
by Lemma~\ref{lemma a epsilon}
\begin{align*}
b_\ve(x) =
\begin{cases}
\pi|A_{\Sigma_0}|^2 |x|^{\ve} + O(\frac{\ve}{\log(|x|)^{2-\ve}}) & \text{for } |x|\leq \ve^{-\frac12}+1,\\
0 & \text{for }|x|\geq \ve^{-\frac12}+1.
\end{cases}
\end{align*}
In the region $|x|\leq \ve^{-\frac12}$, $\Sigma_0$ is the catenoid and hence
$$
|A_{\Sigma_0}|^2=O(|x|^{-4}).
$$
This implies that for $|x|\leq \ve^{-\frac12}$
$$
|b_{\ve_n}(x)|\leq C |x|^{\ve-4} + C \frac{\ve_n}{\log(|x|)^{2-\ve_n}}
$$
It follows that
$$
|g_n(x)| \leq C \log(|x|) |x|^{\ve-4} + C \frac{\ve_n}{\log(|x|)^{1-\ve_n}}
$$
Hence
$$
\left| \int_{\Sigma_0\setminus B_\rho} \phi_n(y) \eta_n(y) g_n(y) \, d y \right| \leq C \rho^{-2-\tau+\ve_n} \log(\rho) + C \ve_n R_n^{2-\tau}.
$$
Using this and \eqref{int B rho} we deduce the claim \eqref{claim 111}.

It follows that
$$
c_n\to0\quad\text{as } n\to\infty.
$$
As in Lemma~\ref{lemma apriori trcat}, $\phi_n\to 0$ uniformly on compact sets. Then by \eqref{norma1}
there is a point $x_n \in \Sigma_0$ such that
$$
(1+|x_n|)^\tau |\phi_n(x_n)| \geq \frac12.
$$
and $|x_n|\to \infty$.
By scaling and translating we obtain a non-trivial $\phi$ satisfying
$$
\Delta  \phi=0\quad \text{in }\R^2 \setminus\{0\}
$$
with
$$
|\phi(x)|\leq C |x|^{-\tau},
$$
which is impossible.
\end{proof}

\bigskip

Next we establish an a priori estimate for decaying solutions of \eqref{mainLeq}. We do not expect solutions of this problem to decay, but that this will be the case if $f$ satisfies a constraint. For this reason, instead of \eqref{mainLeq} we consider a projected equation
\begin{align}
\label{eq36}
\ve
\mathcal J_{\Sigma_0}^s(\phi)=f - c f_0\quad \text{in } \Sigma_0 .
\end{align}
where $f_0$ is an appropriate function.
For $f_0$ we can take almost any smooth function with compact support, but it will be important that
$$
\int_{\Sigma_0} f_0 Z_2 \not =0,
$$
and that we have a solution $\phi_0$ with $\|\phi_0\|_*<\infty$ of
$$
\ve \mathcal J_{\Sigma_0}^s(\phi_0)=f_0\quad\text{in } \Sigma_0.
$$
One possibility to achieve this is the following.
Let $R>0$ the number given in Proposition~\ref{prop exterior slow}.
For $\rho>R$ let $\eta_\rho(x) = \eta(x/\rho)$ where $\eta$ is a smooth radial cut-off function in $\R^3$, such that $\eta(x)=1$ for $|x|\leq 1$ and $\eta(x)=0$ for $|x|\geq 2$. Let $f_\rho = Z_2\eta_\rho$ and $\phi_\rho$ be the function constructed in Proposition~\ref{prop exterior slow}. We recall that it satisfies
$$
\ve \mathcal J_{\Sigma_0}^s(\phi_\rho)(X) = f_\rho(X)
\quad\text{for }X\in \Sigma_0, \ |X|\geq R,
$$
and the estimate
$$
\|\phi_\rho\|_* \leq C \|f_\rho\|_{1-\ve,\alpha+\ve}.
$$
Note that
$$
\|f_\rho\|_{1-\ve,\alpha+\ve} \leq C \rho\log(\rho).
$$
and that since $f_\rho$ is smooth, $\phi_\rho$ is also smooth. Using elliptic estimates we deduce that $\|\phi_\rho\|_{C^{2,\alpha}(B_R)} \leq C \rho \log(\rho)$. Let
$$
\tilde f_\rho = \ve \mathcal J_{\Sigma_0}^s(\phi_\rho).
$$
Then
$$
\int_{\Sigma_0} \tilde f_\rho Z_2 = \int_{\Sigma_0\cap B_R}  \ve \mathcal J_{\Sigma_0}^s(\phi_\rho)Z_2 +
\int_{\Sigma_0\setminus B_R} Z_2^2 \eta_\rho .
$$
Since
$$
\int_{\Sigma_0\cap B_R}  \ve \mathcal J_{\Sigma_0}^s(\phi_\rho)Z_2 = O (\rho \log(\rho)) , \quad  \int_{\Sigma_0\setminus B_R} Z_2^2 \eta_\rho = c \rho^2 \log(\rho)^2 (1+o(1))
$$
as $\rho\to \infty$, where $c>0$, we find that for $\rho>0$ large
$$
\int_{\Sigma_0} \tilde f_\rho Z_2  \not = 0.
$$
We fix $\rho$ large and take
\begin{align}
\label{def phi0}
\phi_0 = \phi_\rho, \qquad  f_0 = \tilde f_\rho.
\end{align}


\begin{lemma}
\label{l14}
Assume  $\| |x|^{2+\tau-\ve} f\|_{L^\infty(\Sigma_0)}<\infty$ and $\phi$, $c$ is a solution of \eqref{eq36} such that $\| |x|^{\tau} \phi \|_{L^\infty(\Sigma_0)}<\infty$. If $\ve$ is small enough, then there is $C$ independent of $f$, $\phi$, $c$ such that
$$
\| |x|^{\tau} \phi \|_{L^\infty(\Sigma_0)} +|c|\leq C \| |x|^{2+\tau-\ve} f\|_{L^\infty(\Sigma_0)}.
$$
\end{lemma}
\begin{proof}

Assume by contradiction that there are sequences $\ve_n\to0$, $\phi_n$, $c_n$ solving \eqref{eq36} with right hand side $f_n$ such that
$$
\| (1+|x|)^{\tau} \phi_n \|_{L^\infty(\Sigma_0)}=1, \quad\| (1+|x|)^{2+\tau-\ve_n} f_n \|_{L^\infty(\Sigma_0)} \to 0
$$
as $n\to\infty$.
Recall that  $\Sigma_0 = \Sigma_0(\ve_n)$.

To estimate $c_n$, let $Z_2$ be given as in \eqref{Z1 Z2}.
We test equation \eqref{eq36} with $Z_2 \eta_n$ where $\eta_n$ is a smooth cut-off function such that $\eta_n(r)=1$ for $r\leq R_n$ and $\eta_n(r)=0$ for $r\geq 2 R_n$, with $R_n\to\infty$ and
$$
R_n<<\ve_n^{-\frac12}.
$$
We get
\begin{align*}
& \ve_n \int_{\Sigma_0(\ve_n)} \phi_n (x) \int_{\Sigma_0(\ve_n)} Z_2(y)\frac{\eta_n(y)-\eta_n(x)}{|x-y|^{4-\ve_n}} \, d y \, d x
+ \int_{\Sigma_0(\ve_n)} \phi_n(y) \eta_n(y) \JJ_{\Sigma_0}(Z_2)(y) \, d y
\\
& = \int_{\Sigma_0(\ve_n)} f_n Z_2 \eta_n - c_n \int_{\Sigma_0(\ve_n)} f_0 Z_2 \eta_n .
\end{align*}
By a calculation
$$
\ve_n \int_{\Sigma_0(\ve_n)} \phi_n (x) \int_{\Sigma_0(\ve_n)} Z_2(y)\frac{\eta_n(y)-\eta_n(x)}{|x-y|^{4-\ve_n}} \, d y \, d x
\to 0
$$
as $n\to \infty$, and
$$
\int_{\Sigma_0(\ve_n)} \phi_n(y) \eta_n(y) \JJ_{\Sigma_0}[Z_2](y) \, d y  \to 0
$$
as $n\to \infty$. It follows that
$$
c_n\to0\quad\text{as } n\to\infty.
$$

There is a point $x_n \in \Sigma_0(\ve_n)$ such that
$$
(1+|x_n|)^\tau |\phi_n(x_n)| \geq \frac12.
$$
If $x_n$ remains bounded, then up to subsequence $\phi_n\to\phi$ uniformly on compact sets of the catenoid $\cat$ and $\phi$ is a nontrivial solution of
$$
\Delta_{\cat}\phi + |A|^2 \phi = 0\quad\text{on }\cat
$$
with $|\phi(x)|\leq  (1+|x|)^{-\tau}$. By Lemma~\ref{lemma jacobi catenoid} $\phi$ must be zero, a contradiction.

Hence $x_n$ is unbounded.
By scaling and translating we obtain a non-trivial $\phi$ satisfying
$$
\Delta \phi + \frac{\tilde\eta}{r^2} \phi=0\quad \text{in }\R^2
$$
with
$$
|\phi(x)|\leq C |x|^{-\tau},
$$
where $0\leq \tilde \eta\leq 1$ is a radial, non-decreasing function  such that $\tilde \eta=1$ for all $|x|\geq m$, where $m\geq 0$.
For $r\geq m$ we get
$$
\phi(r) = a \cos(\log(r)) + b\sin(\log(r))
$$
but then $a=b=0$, so $\phi\equiv0$, a contradiction.
\end{proof}

\begin{proof}[Proof of Proposition~\ref{main linear prop}]
We want to solve \eqref{mainLeq} where $f$ is radial and symmetric such that $\|f\|_{1-\ve,\alpha+\ve}<\infty$.
First we reduce the problem to one where the right hand side has fast decay. Let $\bar\phi = \bar\phi(f)$ be the function constructed in Proposition~\ref{prop exterior slow} with right hand side $f$, namely $\bar\phi$ satisfies
$$
\ve\mathcal J_{\Sigma_0}^s(\bar\phi)(X)=f \quad X\in \Sigma_0, |X|\geq R
$$
where $R>0$ is fixed in this proposition.
Then we look for $\phi$ of the form $\phi = \phi_1 + \eta \bar\phi$ where $\eta \in C^\infty(\R^2)$ is a cut-off function such $\eta(x)=1$ for $|x|\geq 2R$, $\eta(x)=0$ for $|x|\leq R$. The function $\phi_1$ then needs to satisfy
$$
\ve\mathcal J_{\Sigma_0}^s(\phi_1)=f_1 \quad \text{in } \Sigma_0
$$
where
$$
f_1(x) = (1-\eta(x)) f(x) - \ve \int_{\Sigma_0} \bar\phi(y) \frac{\eta(y)-\eta(x)}{|y-x|^{4-\ve}} \, dy .
$$
Since the second term decays like $|x|^{-4+\ve}$ as $|x|\to\infty$,  $f_1$ has fast decay, meaning $\| (1+|x|)^{2+\tau-\ve} f\|_{L^\infty(\Sigma_0)}<\infty$.

In the sequel, we assume that $f$ is symmetric, radial with $\| (1+|x|)^{2+\tau-\ve} f\|_{L^\infty(\Sigma_0)}<\infty$.
First, we claim that it is possible to find a solution $\phi$, $c$ to \eqref{eq36}, which depends linearly on $f$ and such that
$$
\|(1+|x|)^\tau\phi\|_{L^\infty}+|c|\leq  C \|(1+|x|)^{2+\tau-\ve}f\|_{L^\infty}.
$$
We construct this solution by looking for it in the form
$$
\phi = \varphi + \eta_0 \psi
$$
and we ask that
\begin{align}
\label{problemcat2}
& L_\ve(\varphi)+b_\ve \varphi = - [L_\ve,\eta_0](\psi) + (1-\eta_0)f + c f_0
\quad\text{in } \Sigma_0
\\
\label{problpsi}
& L_\ve(\psi) + a_\ve \psi  = - a_\ve (1-\eta_\ve) \varphi +  f \quad\text{in }\Sigma_0 \setminus B_R(0)
\end{align}
Here
$$
[L_\ve,\eta](\psi) = L_\ve(\eta_0\psi)- \eta_0 L_\ve(\psi)
=\ve\, \text{p.v.} \int_{\Sigma_0} \psi(y)\frac{\eta_0(y)-\eta_0(x)}{|x-y|^{4-\ve}}\,dy ,
$$
and $R$ is the same as in Proposition~\ref{prop fast decay}.
The smooth cut-off functions,  $\eta_0$ and $\eta_\ve$ are
radial in $\R^3$ and such that
\begin{align*}
& \eta_0(x) = 0 \text{ for } |x|\leq R,
&&
\eta_0(x) = 1 \text{ for } |x|\geq 2R,
\\
& \eta_\ve(x) = 1 \text{ for } |x|\leq \ve^{-\frac12},
&&
\eta_\ve(x) = 0 \text{ for } |x|\geq \ve^{-\frac12}+1 .
\end{align*}
We rewrite this system as a fixed point problem as follows. Let $Y$ be the space $Y = \{ \varphi \in L^\infty(\Sigma_0): \| ( 1+|x|)^\tau \varphi\|_{L^\infty}<\infty\}$ with the norm $\|\varphi\|_Y =  \| ( 1+|x|)^\tau \varphi\|_{L^\infty}$.  Given $\varphi \in Y$ we solve \eqref{problpsi} using Proposition~\ref{prop fast decay} and obtain a solution $\psi = \psi(\varphi)$. With this $\psi$ we solve now problem \eqref{problemcat2} using  Proposition~\ref{prop exist frac cat decay} and obtain a solution $\tilde \varphi = \tilde \varphi(\varphi) \in Y$.  Let $T(\varphi) = \tilde \varphi(\varphi)$ denote the operator defined in this way, so that $T:Y\to Y $ is an affine linear operator.

We claim that $T$ is compact. Assume that $\varphi_n$ is a bounded sequence in $Y$, and let $\psi_n$ be the corresponding solution of \eqref{problpsi}. By Proposition~\ref{prop fast decay}  $\|\psi_n\|_Y \leq C $. Let $\tilde\varphi_n$, $c_n$ be the solution of \eqref{problemcat2} with $\psi$ replaced by $\psi_n$ and $c$ by $c_n$.
We claim that up to subsequence $\tilde \varphi_n$ converges in $Y$.
By standard regularity
$\tilde \varphi_n$ is bounded in $C^{1,\alpha}_{loc}(\Sigma_0)$ (any $0<\alpha<1$). Then for a subsequence (denoted the same), $\tilde \varphi_n\to\tilde\varphi$ uniformly on compact sets of $\Sigma_0$ as $n\to\infty$.
Let $\tau' \in (\tau,1)$. Then note that $[L_\ve,\eta][\psi_n]$ and  $(1-\eta_0) f + c_n f_0$  have fast decay uniform in $\ve$, more precisely
$$
\| (1+|x|)^{2+\tau'-\ve} (  - [L_\ve,\eta_0](\psi_n) + (1-\eta_0) f + c_n f_0 ) \|_{L^\infty} \leq C.
$$
By Proposition~\ref{prop exist frac cat decay}
$$
\| (1+|x|)^{\tau'} \tilde \varphi_n\|_{L^\infty} \leq C
$$
and hence also $ \| (1+|x|)^{\tau'} \tilde \varphi\|_{L^\infty} <\infty$.
It follows that for any $r>0$
\begin{align*}
& \limsup_{n\to\infty} \sup_{\Sigma_0\cap B_r(0)} (1+|x|)^\tau |\tilde \varphi_n -\varphi| = 0
\\
& \limsup_{n\to\infty} \sup_{\Sigma_0\setminus B_r(0)} (1+|x|)^\tau |\tilde \varphi_n -\varphi| \leq C r^{\tau-\tau'},
\end{align*}
so that $\limsup_{n\to\infty}\|\tilde\varphi_n-\varphi\|_Y\leq Cr^{\tau-\tau'}$.
Since $r$ is arbitrary, $\|\tilde\varphi_n-\tilde \varphi\|_Y\to0$ as $n\to\infty$. This proves that $T$ is compact.
By Lemma~\ref{l14} and the Fredholm alternative there is a unique solution of the system \eqref{problemcat2}, \eqref{problpsi} and hence we find a unique solution $\phi$ to \eqref{eq36}. Moreover
$$
\|(1+|x|)^\tau\phi\|_{L^\infty}+|c|\leq  C \|(1+|x|)^{2+\tau-\ve}f\|_{L^\infty} ,
$$
by Lemma~\ref{l14}.

Finally, we solve  \eqref{mainLeq} when $\|(1+|x|)^{2+\tau-\ve}f\|_{L^\infty}<\infty$. For this let $\phi_0$ be be defined by \eqref{def phi0}.  We look now for a solution $\phi$ of \eqref{mainLeq} of the form $\phi = \phi_1 + \alpha \phi_0$, where we want $\phi_1$ to have fast decay. Then  \eqref{mainLeq} is equivalent to
$$
\ve\JJ_{\Sigma_0}^s(\phi_1) = f -  \alpha f_0.
$$
Given $\alpha\in\R$, by the previous results we know that there exists $c_1=c_1(\alpha) $ and $\phi_1=\phi_1(\alpha) $ of fast decay solving
$$
\ve\JJ_{\Sigma_0}^s(\phi_1) = f -  (\alpha + c_1(\alpha) ) f_0.
$$
We claim that it is possible to choose $\alpha$ such that $c_1(\alpha)=0$.
For this, consider the function $Z_2$ of \eqref{Z1 Z2} and $\eta$ a smooth cut-off function on $\Sigma_0$ such that $\eta(x)=1$ for $|x|\leq \tilde R$ and $\eta(x)=0$ for $|x|\geq 2 \tilde R$ with $\tilde R$ such that $\tilde R\to\infty$ and $\ve\tilde R^2 \log(\tilde R)\to0$. By the same calculation as in Proposition~\ref{prop exist frac cat decay} we get
\begin{align}
\nonumber
& \ve \int_{\Sigma_0} \phi_1 (x) \int_{\Sigma_0} Z_2(y)\frac{\eta(y)-\eta(x)}{|x-y|^{4-\ve}} \, d y \, d x
+ \int_{\Sigma_0} \phi_1(y) \eta(y) \JJ_{\Sigma_0}(Z_2)(y) \, d y
\\
\label{eqalpha}
& = \int_{\Sigma_0} f Z_2 \eta -(\alpha+ c_1(\alpha)  ) \int_{\Sigma_0} f_0 Z_2 \eta .
\end{align}
For the first 2 terms, we have
\begin{align*}
\left|\ve \int_{\Sigma_0} \phi_1 (x) \int_{\Sigma_0} Z_2(y)\frac{\eta(y)-\eta(x)}{|x-y|^{4-\ve}} \, d y \, d x\right| & =o(1)\|(1+|x|)^\tau \phi_1\|_{L^\infty}
\\
&\leq o(1) ( \|(1+|x|)^{2+\tau-\ve}f\|_{L^\infty} + |\alpha| )
\end{align*}
and
\begin{align*}
\left|\int_{\Sigma_0} \phi_1(y) \eta(y) \JJ_{\Sigma_0}(Z_2)(y) \, d y\right| & =o(1)\|(1+|x|)^\tau \phi_1\|_{L^\infty}
\\
&\leq o(1) ( \|(1+|x|)^{2+\tau-\ve}f\|_{L^\infty} + |\alpha| )
\end{align*}
where $o(1)\to 0$ as $\tilde R\to\infty$ and $\ve\to0$.
Then the equation \eqref{eqalpha} for $\alpha$ is uniquely solvable if $\ve$ is small.
\end{proof}

\section{The nonlinear term}
\label{sect q}

Consider $h_1. h_2$ defined on $\Sigma_0$ with  $\|h_i\|_* \leq \sigma_0 \ve^{\frac12}$, where $\sigma_0>0$ is a small constant. The main result in this section is the following estimate stated in Proposition~\ref{prop Nh}:
$$
\ve\| N(h_1) - N(h_2)\|_{1-\ve,\alpha+\ve}\leq C \ve^{-\frac12} (\|h_1\|_*+\|h_2\|_* ) \|h_1-h_2\|_*.
$$
Note the ``extra'' $\ve^{-\frac12}$ in the left hand side.

We rewrite the fractional mean curvature in the following way.
For a point $x=(x',F_\ve(x')) \in \Sigma_0$  let $x_h = x + \nu_{\Sigma_0}(x) h(x)$ and  let $L_h(x)$ denote the half space defined by
$$
L_h(x)=\{ y \in\R^3: \langle y-x_h ,\nu_{\Sigma_h}(x_h)\rangle \geq 0 \},
$$
where $\nu_{\Sigma_h}$ is the unit normal vector to $\partial E_h$ pointing into $E_h$.
Then
$$
H_{E_h}^s(x_h) = 2 \int_{\R^3} \frac{\chi_{E_h}(y)- \chi_{L_h(x)}(y)}{|x_h-y|^{3+s}} \, dy
$$
which has the advantage that the integral is convergent.

To compute the previous integral restricted to a ball around $x$, let us represent $\Sigma_h$ near this point as a graph over the tangent plane to $\Sigma_0$ at $X$.
We start with
$r,\theta$  polar coordinates for $x\in\R^2$, i.e.
$
x=(r\cos\theta,r\sin\theta) $ and let
$\hat r = \frac{x'}{r} = (\cos\theta,\sin\theta)^T$,
$\hat \theta = (-\sin\theta,\cos\theta)^T$.
Given a point $x \in \Sigma_0$, $x = (x',F_\ve(x'))$ we let
\begin{align}
\label{pi1 pi2}
\Pi_1(x ) =
\frac{1}{\sqrt{1+ F_\ve'(x')^2}}
\left[
\begin{matrix}
\hat r\\
F_\ve'( x')
\end{matrix}
\right] ,
\quad
\Pi_2(x) = \left[
\begin{matrix}
\hat \theta
\\
0
\end{matrix}
\right] \in \R^3 ,
\end{align}
$$
\Pi = [\Pi_1,\Pi_2].
$$
The unit normal vector to $\Sigma_0$ at $X$ pointing up is then given by
\begin{align}
\label{nu sigma0}
\nu_{\Sigma_0}(X) =
\frac{1}{\sqrt{1+ F_\ve'( x')^2}}
\left[
\begin{matrix}
-  F_\ve'(x') \hat r\\
1
\end{matrix}
\right] .
\end{align}
Then we consider coordinates $t=(t_1,t_2)$ and $t_3$ defined by
$$
(t_1,t_2,t_3) \mapsto
\Pi_1(x) t_1 + \Pi_2(x)t_2 + \nu_{\Sigma_0}(x) t_3.
$$
Let
$$
R_x = \delta |x|
$$
where $\delta>0$ is a small fixed constant, and let us define
$t_0=t_0(x)$ such that $\Pi(x) t_0$ is the orthogonal projection of $x$ onto the plane generated by $\Pi_1(x)$, $\Pi_2(x)$.

Using the implicit function theorem (see Appendix~\ref{app}),
given $h$ on $\Sigma_0$ with $\|h\|_* \leq \sigma_0 \ve^{\frac12}$, we can represent $\partial E_h$ near $x_h = x + \nu_{\Sigma_0}(x) h(x)$ as
$$
\Pi(x)t+\nu_{\Sigma_0}(x) g_h(t) , \quad |t -t_0(x)|\leq 2 R_x
$$
where $g_h$ is of class $C^{2,\alpha}$ in the ball $B_{4 R_x}(t_0(x))$.
We call $G_x$ the operator defined by
\begin{align}
\label{def G}
g_h = G_x(h) .
\end{align}

Let
\begin{align}
\label{cut N}
\eta_x(t,t_3) = \eta(\frac{|t-t_0(x)|}{R_x}) \eta(\frac{100|t_3|}{\ve^{\frac12}|x|})
\end{align}
where $\eta \in C^\infty(\R)$ is such that $\eta(s)=1$ for $s\leq 1$ and $\eta(s)=0$ for $s\geq 2$. We also require $\eta'\leq 0$.

Let us write
$$
H_{\partial E_h}^s(x_h) = H_i(h)(x) +  H_o(h)(x)
$$
where
\begin{align*}
H_i(h)(x_h) &=  2 \int_{\R^3} \eta_x(y-x_h) \frac{\chi_{E_h}(y)- \chi_{L_h(x)}(y)}{|x_h-y|^{3+s}} \, d y\\
H_o(h)(x_h) &=  2 \int_{\R^3} (1-\eta_x(y-x_h)) \frac{\chi_{E_h}(y)- \chi_{L_h(x)}(y)}{|x_h-y|^{3+s}} \, dy .
\end{align*}

Let us explain the choice of cut-off function \eqref{cut N}.
For this, let us write
$$
D_{R_x}(x) = \{ \Pi(x) t + x : t\in\R^2,\ |t-t_0(x)|<R_x\} ,
$$
which is a 2-dimensional disk on the tangent plane to $\Sigma_0$ at $x$, centered at $x$, and of radius $R_x=\delta |x|$.
Let us call
$$
C(x)= \{ \Pi(x) t + t_3 \nu_{\Sigma_0}(x) + x : t\in\R^2,\ |t-t_0(x)|<R_x , \ |t_3|< \frac{\ve^{\frac12}|x|}{100}\} ,
$$
the cylinder with base the disk $D_{R_x}$ and height $\ve^{\frac12}|x|/100$, and
$$
\tilde C(x)= \{ \Pi(x) t + t_3 \nu_{\Sigma_0}(x) + x : t\in\R^2,\ |t-t_0(x)|<2 R_x , \ |t_3|< \frac{\ve^{\frac12}|x|}{50}\} ,
$$
which is a similar cylinder with twice the radius and height.
The cut-off function \eqref{cut N} is zero outside the $\tilde C(x)$, while it is one on $C(x)$.
Since we assume  $\|h\|_*\leq \sigma_0\ve^{\frac12}$, we have $\|D g_h\|_{L^\infty}= O(\ve^{\frac12})$ and then the set $\Sigma_h$ separates from $\Sigma_0$ in the $\nu_{\Sigma_0}(x)$ direction  an amount bounded by $O(\ve^{\frac12} 2 R_x) = O (\delta \ve^{\frac12}|x|)$ over the disk $D_{2 R_x}(x)$. By choosing $\delta<<100$ we achieve that the parts of $\Sigma_h$ and the plane $\partial L_h$ inside $\tilde C(x)$ are in fact contained in a cylinder with base $D_{2R_x}(x)$ but height $O(\delta \ve^{\frac12}|x|)$, which is much small than the height of $C(x)$.

We expand $H_i$, $H_0$
\begin{align*}
H_i(h)(x_h) = H_i(0)(x) + H_i'(0)(h)(x) + N_i(h)(x)\\
H_o(h)(x_h) = H_o(0)(x) + H_o'(0)(h)(x) + N_o(h)(x) .
\end{align*}

Estimate \eqref{est N1} will follow from similar estimates of $N_o(h)$ and $N_i(h)$, which we state in the next lemmas.
\begin{lemma}
\label{lemma Ni}
There is $C$ independent of $\ve>0$ small such that for $\|h_i\|_* \leq \sigma_0 \ve^{\frac12}$, $i=1,2$ we have
\begin{align*}
\| N_i(h_1) - N_i(h_2)\|_{1-\ve,\alpha+\ve}\leq \frac C\ve (\|h_1\|_*+\|h_2\|_* ) \|h_1-h_2\|_*.
\end{align*}
\end{lemma}

\begin{lemma}
\label{lemma No}
There is $C$ independent of $\ve>0$ small such that for $\|h_i\|_* \leq \sigma_0 \ve^{\frac12}$, $i=1,2$ we have
\begin{align*}
\| N_o(h_1) - N_o(h_2)\|_{1-\ve,\alpha+\ve}\leq \frac C{\ve^{\frac32}} (\|h_1\|_*+\|h_2\|_* ) \|h_1-h_2\|_*.
\end{align*}
\end{lemma}



For the integral involved in $H_i$  we can write
\begin{align*}
H_i(h)(x_h) &=2\int_{B_{2 R_x}(0)}  \frac{\eta(\frac{|t|}{R_x})}{|t|^{3-\ve}}\left(\psi(\frac{\nabla g_h(t_0(x))t}{|t|}) - \psi(\frac{g_h(t+t_0(x))-g_h(t_0(x))}{|t|})\right)dt
\end{align*}
where
$$
\psi(s) = \int_0^s\frac{d\tau}{(1+\tau^2)^{\frac{4-\ve}2}}.
$$
For a given $C^{2,\alpha}$ function $g$ defined on $B_{2 R_x}(t_0(x))$ let
$$
\tilde H_x(g) = 2\int_{B_{2 R_x}(0)}  \frac{\eta(\frac{|t|}{R_x})}{|t|^{3-\ve}}\left(\psi(\frac{\nabla g(t_0(x))t}{|t|}) - \psi(\frac{g(t+t_0(x))-g(t_0(x))}{|t|})\right)dt
$$
so that
$$
H_i(h) = \tilde H_x(G_x(h)) ,
$$
where $G_x$ is the operator defined in \eqref{def G}.


For the expansion of $\tilde H_X$ it will be convenient to rewrite it as
\begin{align*}
\tilde H_X(g) = 2 \int_0^1  \int_{\R^2} \frac{\eta(\frac{|z|}{R_X})}{|z|^{3-\ve}}  \psi' \left( A_t(g) \right)B(g) \, d z d t ,
\end{align*}
where
\begin{align*}
A_t(g)(X,z) &= t\frac{g(z+t_0(X))-g(t_0(X))}{|z|}+(1-t)\frac{\nabla g(t_0(X))z}{|z|} ,\\
B(g)(X,z) &= \frac{g(z+t_0(X))-g(t_0(X))-\nabla g(t_0(X))z}{|z|}.
\end{align*}
%

Note that
\begin{align}
\label{der H}
D H_i(h)[h_1] &= D \tilde H_X( G_X(h) )[ D G_X(h)[h_1]],
\\
\label{second der H}
D^2 H_i (h)[h_1,h_2] &= D^2 \tilde H_X(G_X(h))[D G_X(h)[h_1],D G_X(h)[h_2] ] \\
\nonumber
&\quad + D \tilde H_X(G_X(h))[D^2 G_X(h)[h_1,h_2]].
\end{align}
and
\begin{align*}
& D \tilde H_X(g)[g_1] =
\int_{\R^2} \frac{\eta(\frac{|z|}{R_X})}{|z|^{3-\ve}}  \big[ \psi'' ( A_t(g)  (X,z) ) A_t(g_1)  (X,z) B(g) (X,z)   \\
&\qquad\qquad\qquad\qquad + \psi'(A_t(g) (X,z) ) B(g_1) (X,z)  \big]\, d z ,
\\
& D^2 \tilde H_X(g)[g_1,g_2] \\
&= \int_0^1 \int_{\R^2} \frac{\eta(\frac{|z|}{R_X})}{|z|^{3-\ve}}   \big[ \psi''' ( A_t(g)(X,z) ) B(g)(X,z) A_t(g_1)(X,z) A_t(g_2)(X,z)
\\
&\qquad \qquad\qquad+  \psi''(A_t(g)(X,z))  A_t(g_1)(X,z) B(g_2)(X,z) \\
& \qquad \qquad\qquad+ \psi''(A_t(g)(X,z)) A_t(g_2) (X,z)B(g_1)(X,z) \big] \, d z d t .
\end{align*}

For later computations we will need the following properties of
$D G_X$, $D^2 G_X$.

\begin{lemma}
\label{lemma est G}
Let $\|h\|_*,\|h_1\|_*,\|h_2\|_*\leq \sigma_0 \ve^{\frac12}$, $X\in \Sigma_0$ and
\begin{align*}
g& = G_X(h) , \quad
g_{i} = D G_X(h)[h_i]\quad i=1,2, \quad
\hat g = D^2 G_X(h)[h_1,h_2] .
\end{align*}
Then
$$
\|G_X(h)\|_b \leq C
$$
where
\begin{align}
\label{norm b}
\|g\|_b & = |X|^{-1} \|g\|_{L^\infty(B_X)} +\|\nabla g\|_{L^\infty(B_X)}    + |X| \|D^2 g\|_{L^\infty(B_X)}  + |X|^{1+\alpha} [D^2 g]_{\alpha,B_X}  .
\end{align}
and $B_X = B_{2 R_X}(t_0(X))$.
Also, for $z \in B_X$:
\begin{align}
\label{at g}
|A_t(g)(X,z)| & \leq C\|h\|_*\\
\label{B g}
|B(g)(X,z)| & \leq C \frac{\|h\|_*}{|X|} |z| ,
\\
\label{at gi}
|A_t(g_i)(X,z) | &  \leq C \|h_i\|_*\\
\label{202}
|B(g_i)(X,z) | & \leq C \frac{\|h_i\|_*}{|X|} |z| .
\end{align}
\end{lemma}

We leave the proof of these estimate for the appendix.

\begin{lemma}
\label{lemma 7.4}
Let $h,h_1,h_2$ be defined on $\Sigma_0$ with $\|h\|_*,\|h_i\|_*\leq \sigma_0 \ve^{\frac12}$. Let $X\in \Sigma_0$ and
\begin{align*}
g& = G_X(h) , \quad
g_{i} = D G_X(h)[h_i]\quad i=1,2, \quad
\hat g = D^2 G_X(h)[h_1,h_2] .
\end{align*}
Then
\begin{align*}
\ve|D \tilde H_X(g)[ \hat g](X)| &\leq  \frac{C }{|X|^{1-\ve} } \|h_1\|_* \|h_2\|_* \\
\ve\left|D^2 \tilde H(g)[g_1,g_2] (X)\right|&\leq  \frac{C }{|X|^{1-\ve} } \|h_1\|_* \|h_2\|_* .
\end{align*}
\end{lemma}
\begin{proof}

Let us start with the first term  in $D \tilde H_X(g)[g_1]$. Using \eqref{at g}, \eqref{at gi}
\begin{align*}
\left|\int_{\R^2} \frac{\eta(\frac{|z|}{R_X})}{|z|^{3-\ve}}  \psi'' ( A_t(g) )  A_t(g_1)   B(g)   \, d z \right|
& \leq \|\psi''\|_{L^\infty} \|A_t(g_1)\|_{L^\infty}  \int_{B_{2R_X}(0)} \frac{1}{|z|^{3-\ve}}   |B(g)| \,dz
\\
&\leq C \|h_1\|_* \int_{B_{2R_X}(0)} \frac{1}{|z|^{3-\ve}}   |B(g)| \,dz
.
\end{align*}
Then by \eqref{B g}
\begin{align*}
\int_{B_{2R_X}(0)} \frac{1}{|z|^{3-\ve}}   |B(g)| \, d z
&\leq  \frac{\|h\|_*}{|X|} \int_{B_{2 R_X}(0)} \frac{1}{|z|^{2-\ve}} \, d z \\
&\leq \frac{C}{|X|} \|h\|_* \frac{R_X^\ve}{\ve} \leq \frac{C}{\ve |X|^{1-\ve}} \|h\|_*.
\end{align*}
Therefore
\begin{align*}
\left|\int_{\R^2} \frac{\eta(\frac{|z|}{R_X})}{|z|^{3-\ve}}  \psi'' ( A_t(g) )  A_t(g_1)   B(g)   \, d z \right|
\leq\frac{C}{\ve^|X|^{1-\ve}} \|h_1\|_*.
\end{align*}
For the second term observe that
\begin{align*}
\left|\int_{\R^2} \frac{\eta(\frac{|z|}{R_X})}{|z|^{3-\ve}}   \psi'(A_t(g) ) B(g_1) \, d z \right|&\leq
C \int_{B_{2R_X}(0)} |A_t(g) B(g_1) | \, d z\\
& \leq \frac{C}{\ve|X|^{1-\ve}} \|g_1\|_b ,
\end{align*}
which is obtained using \eqref{at g} and \eqref{202}.

For the  first term in $D^2 \tilde H_X(g)[g_1,g_2]$, we have, using \eqref{B g}  and \eqref{at gi},
\begin{align*}
& \left| \int_{\R^2} \frac{\eta(\frac{|z|}{R_X})}{|z|^{3-\ve}}   \psi''' ( A_t(g) ) A_t(g_1) A_t(g_2) B(g) \, dz \right| \\
&\leq
\|\psi'''\|_{L^\infty} \|A_t(g_1)\|_{L^\infty} \|A_t(g_2)\|_{L^\infty}  \int_{B_{2R_X}(0)} \frac{1}{|z|^{3-\ve}}   |B(g)| \,dz \\
&\leq \frac{C}{\ve|X|^{1-\ve}} \|h_1\|_* \|h_2\|_* .
\end{align*}

Similarly, for the second and third terms
\begin{align*}
\left|  \int_{\R^2}  \frac{\eta(\frac{|z|}{R_X})}{|z|^{3-\ve}}  \psi''(A_t(g)) A_t(g_1)B(g_2) \, d z \right|
&\leq\|\psi''\|_{L^\infty} \|A_t(g_1)\|_{L^\infty}\int_{B_{2R_X}(0)} \frac{1}{|z|^{3-\ve}}   |B(g_2)| \, d z \\
&\leq \frac{C}{\ve |X|^{1-\ve} } \|h_1\|_* \|h_2\|_*.
\end{align*}
\end{proof}

Now we deal with the H\"older part of the norm $\| \ \|_{1-\ve,\alpha+\ve} $.

\begin{lemma}
\label{lemma n2}
Let $X_1=(x_1,F_\ve(x_1))$, $X_2=(x_2,F_\ve(x_2))\in \Sigma_0$, be such that $|X_1|\leq |X_2|$ and $|X_1-X_2|\leq \frac1{10}|X_1|$.
Let
\begin{align*}
g_{X_j} & = G_{X_j}(h) \quad j=1,2\\
g_{i,X_j} &= D G_{X_j}(h_0)[h_i]\quad i,j=1,2.
\end{align*}
Then
\begin{align}
\nonumber
& | D^2 \tilde H_{X_1}(g_{X_1})[g_{1,X_1},g_{2,X_1}] -
D^2 \tilde H_{X_2}(g_{X_2})[g_{1,X_2},g_{2,X_2}] | \\
\label{est nonl holder}
& \leq \frac{C}{\ve} ( \|h_1\|_* + \|h_2\|_*)\|h_1-h_2\|_* \, \frac{|X_1 - X_2|^{\alpha+\ve} }{ |X_1|^{1+\alpha}} .
\end{align}
\end{lemma}
\begin{proof}

Thanks to \eqref{der H} and \eqref{second der H}, to prove \eqref{est nonl holder} it is enough to show
\begin{align}
\nonumber
|D^2 \tilde H_{X_1}(g_{X_1})[g_{1,X_1},g_{2,X_1}] &-D^2 \tilde H_{X_2}(g_{X_2})[g_{1,X_2},g_{2,X_2}] |
\\
\label{est nonl holder 01}
&\leq \frac{C}{\ve} \|h_1\|_* \|h_2\|_* \, \frac{|X_1 - X_2|^{\alpha+\ve} }{ |X_1|^{1+\alpha}} ,
\end{align}
and
\begin{align}
\nonumber
|D \tilde H_{X_1}(g_{X_1})[D^2 G_{X_1}(h)[h_1,h_2]]
&-D \tilde H_{X_2}(g_{X_2})[D^2 G_{X_2}(h)[h_1,h_2]] | \\
\label{est nonl holder 02}
& \leq \frac{C}{\ve} \|h_1\|_* \|h_2\|_* \, \frac{|X_1 - X_2|^{\alpha+\ve} }{ |X_1|^{1+\alpha}} .
\end{align}

Let us show \eqref{est nonl holder 01}.
For this, write
\begin{align*}
D^2 \tilde H_{X_j}(g_{X_j})[g_{1,X_j},g_{2,X_j}] = A_{1}(X_j) +  A_{2}(X_j)  +  A_{3}(X_j)
\end{align*}
where
\begin{align*}
A_{1}(X_j)  &= \int_0^1 \int_{\R^2} \frac{\eta(\frac{|z|}{R_{X_j}})}{|z|^{3-\ve}}    \psi''' ( A_t(g_{X_j})(X_j,z) ) B(g_{X_j})(X_j,z) A_t(g_{1,X_j})(X_j,z) A_t(g_{2,X_j})(X_j,z) \, d z d t ,\\
A_{2}(X_j) &= \int_0^1 \int_{\R^2} \frac{\eta(\frac{|z|}{R_{X_j}})}{|z|^{3-\ve}}  \psi''(A_t(g_{X_j})(X_j,z))  A_t(g_{1,X_j})(X_j,z) B(g_{2,X_j})(X_j,z) \, d z d t ,\\
A_{3}(X_j) &= \int_0^1 \int_{\R^2} \frac{\eta(\frac{|z|}{R_{X_j}})}{|z|^{3-\ve}}  \psi''(A_t(g_{X_j})(X_j,z)) A_t(g_{2,X_j}) (X_j,z)B(g_{1,X_j})(X_j,z)  \, d z d t .
\end{align*}

Let us estimate the difference
\begin{align*}
A_{1}(X_j) -A_{1}(X_j) = \int_0^1( B_1+B_2+B_3+B_4+B_5)\,dt
\end{align*}
where
\begin{align*}
B_1&=\int_{\R^2} \frac{\eta(\frac{|z|}{R_{X_1}}) - \eta(\frac{|z|}{R_{X_2}}) }{|z|^{3-\ve}}  \psi''' ( A_t(g_{X_1})(X_1,z) ) B(g_{X_1})(X_1,z)
\\&\qquad\qquad\cdot A_t(g_{1,X_1})(X_1,z) A_t(g_{2,X_1})(X_1,z) \, d z  ,
\end{align*}
\begin{align*}
B_2&=\int_{\R^2}  \frac{\eta(\frac{|z|}{R_{X_2}})}{|z|^{3-\ve}}   \big( \psi''' ( A_t(g_{X_1})(X_1,z) )- \psi''' ( A_t(g_{X_2})(X_2,z) )\big) B(g_{X_1})(X_1,z)\\
&\qquad\qquad\cdot A_t(g_{1,X_1})(X_1,z) A_t(g_{2,X_1})(X_1,z) \, d z  ,
\end{align*}
\begin{align*}
B_3&=\int_{\R^2}  \frac{\eta(\frac{|z|}{R_{X_2}})}{|z|^{3-\ve}}    \Big)\psi''' ( A_t(g_{X_2})(X_2,z) )\big( B(g_{X_1})(X_1,z) -B(g_{X_2})(X_2,z) \big)  \\&\qquad\qquad\cdot A_t(g_{1,X_1})(X_1,z) A_t(g_{2,X_1})(X_1,z) \, d z  ,
\end{align*}
\begin{align*}
B_4&=\int_{\R^2}  \frac{\eta(\frac{|z|}{R_{X_2}})}{|z|^{3-\ve}}    \Big)\psi''' ( A_t(g_{X_2})(X_2,z) )B(g_{X_2})(X_2,z) \big(  A_t(g_{1,X_1})(X_1,z)-A_t(g_{1,X_2})(X_2,z)\big) \\&\qquad\qquad\cdot A_t(g_{2,X_1})(X_1,z) \, d z  ,
\end{align*}
\begin{align*}
B_5&=\int_{\R^2}  \frac{\eta(\frac{|z|}{R_{X_2}})}{|z|^{3-\ve}}    \Big)\psi''' ( A_t(g_{X_2})(X_2,z) )B(g_{X_2})(X_2,z)  A_t(g_{1,X_2})(X_2,z)\\&\qquad\qquad\cdot \big( A_t(g_{2,X_1})(X_1,z) - A_t(g_{2,X_2})(X_2,z) \Big)\, d z  .
\end{align*}

We  estimate $B_1$:
\begin{align*}
|B_1|\leq \|\psi'''\|_{L^\infty} \|A_t(g_1)\|_{L^\infty}\|A_t(g_2)\|_{L^\infty}
\int_{B_{2R_{X_2}}(0)} \frac{\eta(\frac{|z|}{R_{X_2}})-\eta(\frac{|z|}{R_{X_1}})}{|z|^{3-\ve}}  |B(g)|\, d z
\end{align*}
where we have used $\eta'\leq 0$ and $R_2\geq R_1$.
Thanks to \eqref{at gi}, \eqref{B g} we find
\begin{align*}
|B_1|
&\leq C \frac{\|h_1\|_* \|h_2\|_* }{|X_1}
\int_{B_{2R_{X_2}}(0)} \frac{\eta(\frac{|z|}{R_{X_2}})-\eta(\frac{|z|}{R_{X_1}})}{|z|^{2-\ve}} \, d z \\
& \leq \frac{C}{\ve |X_1|} \|h_1\|_* \|h_2\|_* (R_{X_2}^\ve-R_{X_1}^\ve)\\
& \leq \frac{C |X_1-X_2|^{\alpha+\ve}}{\ve |X_1|^{1+\alpha}} \|h_1\|_b \|h_2\|_b .
\end{align*}

Let us consider $B_2$. Using \eqref{at gi} we get
\begin{align*}
|B_2| 
& \leq C \|h_1\|_* \|h_2\|_* \int_{B_{2 R_{X_2}}(0)} |A_t(g_{X_1})(X_1,z) - A_t(g_{X_2})(X_2,z) | |B(g)(X_1,z)|\, d z .
\end{align*}
For $z\in B_{2 R_{X_2}}(0)$ we have
\begin{align*}
&|A_t(g_{X_1})(X_1,z) - A_t(g_{X_2})(X_2,z) | \\
&\leq
|A_t(g_{X_1})(X_1,z) - A_t(g_{X_1})(X_2,z) |
+|A_t(g_{X_1})(X_2,z) - A_t(g_{X_2})(X_2,z) | .
\end{align*}
For the first term
\begin{align*}
|A_t(g_{X_1})(X_1,z) - A_t(g_{X_1})(X_2,z)  | & \leq \|g_{X_1}''\|_{L^\infty(B_{2R_{X_2}}(t_0(X_1)))} |t_0(X_1)-t_0(X_2)|\\
&\leq C \frac{\|g_{X_1}\|_b}{|X_1|} |X_1-X_2|.
\end{align*}
For the second term we have
\begin{align*}
|A_t(g_{X_1})(X_2,z) - A_t(g_{X_2})(X_2,z) |\leq
\|g_{X_1}-g_{X_2}\|_b \leq C \frac{|X_1 - X_2|}{|X_1|},
\end{align*}
where the last inequality follow from ... \comment{make ref}
Therefore for $z\in B_{2 R_{X_2}}(0)$
\begin{align*}
|A_t(g_{X_1})(X_1,z) - A_t(g_{X_2})(X_2,z) |
\leq  \frac{C}{|X_1|} |X_1-X_2|.
\end{align*}
This combined with \eqref{B g}  gives
\begin{align*}
& \int_{B_{2 R_{X_2}}(0)} |A_t(g_{X_1})(X_1,z) - A_t(g_{X_2})(X_2,z) | |B(g)(X_1,z)|\, d z\\
& \leq C \frac{|X_1-X_2|}{|X_1|^2}
\int_{B_{2 R_{X_2}}(0)} |z|^{\ve -2}\,dz \\
&\leq  C \frac{|X_1-X_2|}{\ve |X_1|^2} R_{X_2}^\ve
\leq \frac{C |X_1-X_2|^{\alpha+\ve}}{|X_1|^{1+\alpha}} ,
\end{align*}
and therefore
$$
|B_2|\leq  C \|h_1\|_* \|h_2\|_* \frac{ |X_1-X_2|^{\alpha+\ve}}{|X_1|^{1+\alpha}} .
$$

For $B_3$ we proceed as follows:
\begin{align}
\label{prev B}
|B_3| & \leq \|\psi'''\|_{L^\infty} \|g_1\|_b \|g_2\|_b
\int_{B_{2 R_{X_2}}(0)} \frac{|B(g_{X_1})(X_1,z) - B(g_{X_2})(X_2,z)|}{|z|^{3-\ve}} \, dz
\end{align}
Let  $R = 10  |X_1-X_2|$ and assume that $R \leq \frac12\min(|X_1|,|X_2|)$.
We split
\begin{align*}
& |B(g_{X_1})(X_1,z) - B(g_{X_2})(X_2,z)|\\&\leq|B(g_{X_1})(X_1,z) - B(g_{X_1})(X_2,z)| + |B(g_{X_1})(X_2,z) - B(g_{X_2})(X_2,z)|.
\end{align*}
For the first term we have\begin{align*}
& \int_{B_R(0)} \frac{|B(g_{X_1})(X_1,z) -B(g_{X_1})(X_2,z)|}{|z|^{3-\ve}} \, d z\\
&\leq \int_0^1 (1-\tau) \int_{B_R(0)}\frac{|g_{X_1}''(t_0(X_1)+\tau z)[z^2]-g_{X_1}''(t_0(X_2)+\tau z)[z^2]|}{|z|^{4-\ve}} \, d z \, d \tau\\
&\leq C|t_0(X_1)-t_0(X_2)|^\alpha [g_{X_1}'']_{\alpha,B_R(0)}\frac{R^\ve}\ve\\
&\leq \frac{C}{\ve } \frac{|X_1-X_2|^{\alpha+\ve}}{|X_1|^{1+\alpha}}  \|g_{X_1}\|_b.
\end{align*}
We next estimate the integral in $B_{2 R_{X_2}}(0) \setminus B_R(0)$ and for this we compute for $z\in B_{2 R_{X_2}}(0) \setminus B_R(0)$,
\begin{align*}
& |z|( B(g_{X_1})(X_1,z) -B(g_{X_1})(X_2,z) )\\
&= g_{X_1}(t_0(X_1)+z)-g_{X_1}((t_0(X_1))-\nabla g_{X_1}((t_0(X_1))z\\
& \qquad - [g_{X_1}((t_0(X_2)+z)-g_{X_1}(t_0(X_2))-\nabla g_{X_1}(t_0(X_2))z]\\
&=\int_0^1 ( \nabla g_{X_1}(x_\tau+z)-\nabla g_{X_1}(x_\tau) - g_{X_1}''(x_\tau) z)(t_0(X_1)-t_0(X_2)) \, d\tau\\
&=\int_0^1 \int_0^1 (g_{X_1}''(x_\tau+\rho z)z-g_{X_1}''(x_\tau)z)	( t_0(X_1) - t_0(X_2) ) \, d \rho d \tau,
\end{align*}
where $x_\tau = \tau t_0(X_1) + (1-\tau)t_0(X_2)$.
Then
\begin{align*}
|z| | B(g_{X_1})(X_1,z) -B(g_{X_1})(X_2,z) | &\leq [g_{X_1}'']_{\alpha,B_{\bar R}(\bar x)} |z|^{1+\alpha}|t_0(X_1) - t_0(X_2)| \\
&\leq C \|g_{X_1}\|_b |X_1|^{-1-\alpha} |z|^{1+\alpha}|X_1 - X_2| .
\end{align*}
Integrating
\begin{align*}
\int_{B_{2 R_{X_2}}(0) \setminus B_R(0)} &\frac{|B(g_{X_1})(X_1,z) -B(g_{X_1})(X_2,z)|}{|z|^{3-\ve}} \, d z \\
&\leq C \|g_{X_1}\|_* |X_1|^{-1-\alpha} |X_1-X_2| R^{\alpha+\ve-1}\\
&\leq C\|g_{X_1}\|_b |X_1|^{-1-\alpha} |X_1-X_2|^{\alpha+\ve} .
\end{align*}
To estimate
\begin{align*}
\int_{B_{2 R_{X_2}}(0)} \frac{|B(g_{X_1})(X_2,z) - B(g_{X_2})(X_2,z)|}{|z|^{3-\ve}} \, dz
\end{align*}
we observe that
\begin{align*}
|B(g_{X_1})(X_2,z) - &B(g_{X_2})(X_2,z)| \\
&\leq  |z| \int_0^1 (1-\tau) |g_{X_1}''(t_0(X_2)+\tau z ) -g_{X_2}''(t_0(X_2)+\tau z ) | \, dz
\\
&\leq |z|^{1+\alpha} \frac{\|g_{X_1}-g_{X_2}\|_b}{|X_1|^{1+\alpha}}
\leq \frac{C  |z|^{1+\alpha}}{|X_1|^{2+\alpha}} |X_1-X_2| .
\end{align*}
Integrating we find
$$
\int_{B_{R_{X_2}}(0)} \frac{|B(g_{X_1})(X_2,z) - B(g_{X_2})(X_2,z)|}{|z|^{3-\ve}} \, dz \leq \frac{C}{|X_1|^{1+\alpha}} |X_1-X_2|^{\alpha+\ve}.
$$

This shows that
$$
|B_3 |\leq \frac{C}{\ve} \frac{ |X_1-X_2|^{\alpha+\ve}}{ |X_1|^{1+\alpha}} \|g_1\|_b \|g_2\|_b ,
$$

The estimates of $B_4$ and $B_5$ are similar and we omit the details.
This proves the estimate
\begin{align*}
|A_1(X_1)-A_1(X_2)|\leq \frac{C}{\ve} \|h_1\|_* \|h_2\|_* \, \frac{|X_1 - X_2|^{\alpha+\ve} }{ |X_1|^{1+\alpha}} .
\end{align*}

\bigskip

Let us estimate the difference
\begin{align*}
A_{2}(X_1) -A_{2}(X_2) = \int_0^1( B_1+B_2+B_3+B_4)\,dt
\end{align*}
with
\begin{align*}
B_1&=
\int_{\R^2} \frac{\eta(\frac{|z|}{R_{X_1}}) - \eta(\frac{|z|}{R_{X_2}}) }{|z|^{3-\ve}}   \psi''(A_t(g_{X_1})(X_1,z))  A_t(g_{1,X_1})(X_1,z) B(g_{2,X_1})(X_1,z) \, d z
\end{align*}
\begin{align*}
B_2&=
\int_{\R^2} \frac{\ \eta(\frac{|z|}{R_{X_2}}) }{|z|^{3-\ve}}   \big( \psi''(A_t(g_{X_1})(X_1,z)) - \psi''(A_t(g_{X_2})(X_2,z)) \big) A_t(g_{1,X_1})(X_1,z) \\ & \qquad\qquad \cdot B(g_{2,X_1})(X_1,z) \, d z
\end{align*}
\begin{align*}
B_3&=
\int_{\R^2} \frac{\ \eta(\frac{|z|}{R_{X_2}}) }{|z|^{3-\ve}}    \psi''(A_t(g_{X_2})(X_2,z)) \big( A_t(g_{1,X_1})(X_1,z) - A_t(g_{1,X_2})(X_2,z)\big) \\ & \qquad\qquad \cdot B(g_{2,X_1})(X_1,z) \, d z
\end{align*}
\begin{align*}
B_4&=
\int_{\R^2} \frac{\ \eta(\frac{|z|}{R_{X_2}}) }{|z|^{3-\ve}}    \psi''(A_t(g_{X_2})(X_2,z)) A_t(g_{1,X_2})(X_2,z) \\ & \qquad\qquad\cdot \big(  B(g_{2,X_1})(X_1,z)  - B(g_{2,X_2})(X_2,z)  \big) \, d z
\end{align*}
The terms $B_1$, $B_2$, $B_3$ are similar as before and we have
\begin{align*}
|B_1|+|B_2|+|B_3|& \leq\frac{C}{\ve} \frac{ |X_1-X_2|^{\alpha+\ve}}{ |X_1|^{1+\alpha}} \|g_1\|_b \|g_2\|_b
\\
&\leq\frac{C}{\ve} \frac{ |X_1-X_2|^{\alpha+\ve}}{ |X_1|^{1+\alpha}} \|h_1\|_* \|h_2\|_*.
\end{align*}
Let us focus on
\begin{align}
\label{301}
|B_4|\leq C \|h_1\|_* \int_{B_{2 R_{X_2}}(0)} \frac{|B(g_{2,X_1})(X_1,z)  - B(g_{2,X_2})(X_2,z)|}{|z|^{3-\ve}}\,dz .
\end{align}
The difference with the estimate for \eqref{prev B} is that now we cannot control $\|g_{2,X_1}\|_b$ only assuming $\|h\|_*$ bounded, since  $g_{2,X_1}$ involves derivatives of $h$.

We proceed by the following observation.
We will see that $g_{2,X_j}$ can be decomposed in 2 parts, one of them being regular enough to perform the previous calculations, and the second part having a special form. A model for the second part is
$$
D_t \tilde h(t+t_0(X_j)) b(X_j,t+t_0(X_j))
$$
where $\tilde h$ is $h$ composed with an appropriate change of variables (this in reality also depends on $X_j$, but for to explain the idea here we will omit this dependence), and
$$
b(X_j,t_0(X_j)) = 0.
$$
Let us see what we get if we assume for the moment that
$$
g_{2,X_j} = D_t \tilde h(t+t_0(X_j)) b(X_j,t+t_0(X_j)) .
$$
Then we have
\begin{align*}
B(g_{2,X_j})(X_j,z) & = \frac{1}{|z|} \big[ D_t \tilde h(z+t_0(X_j))  b(X_j,z+t_0(X_j)) -  D_t \tilde h(t_0(X_j))  D_t b(X_j,t_0(X_j))\big]
\end{align*}
and so
\begin{align*}
|B(g_{2,X_1})(X_1,z)-B(g_{2,X_2})(X_2,z)|
\leq A+B
\end{align*}
where
\begin{align*}
A &= \frac{1}{|z|}\Big|  ( D_t \tilde h(z+t_0(X_1)) -  D_t \tilde h(t_0(X_1))) b(X_1,z+t_0(X_1)) \\
& \qquad - ( D_t \tilde h(z+t_0(X_2)) -  D_t \tilde h(t_0(X_2))) b(X_2,z+t_0(X_2)) \Big|
\\
B&= \frac{1}{|z|}\Big| D_t \tilde h(t_0(X_1))) ( D_t b(X_1,t_0(X_1)) z-b(X_1,t_0(X_1))) \\
&\qquad -D_t \tilde h(t_0(X_2))) ( D_t b(X_2,t_0(X_2)) z-b(X_2,t_0(X_2)))  \Big| .
\end{align*}
For $A$ we write
\begin{align*}
A \leq A_1+A_2
\end{align*}
where
\begin{align*}
A_1 &=\frac{1}{|z|}\big| D_t \tilde h(z+t_0(X_1)) - D_t \tilde h(t_0(X_1))
\\ & \qquad - (D_t \tilde h(z+t_0(X_2)) - D_t \tilde h(t_0(X_2)) ) \big| |b(X_1,z+t_0(X_1))|
\\
A_2 &= \frac{1}{|z|}\big| (D_t \tilde h(z+t_0(X_2)) - D_t \tilde h(t_0(X_2)) ) (b(X_1,z+t_0(X_1)) -b(X_2,z+t_0(X_2))  )\big| \end{align*}

For the first term we split the integral in $B_R(0)$ and outside, where $R = 10 |X_1-X_2|$ and we assume $R \leq \frac{1}{10}|X_1|$.
For $z\in B_R(0)$ we estimate
\begin{align*}
A_1
&  \leq \frac{1}{|z|} \int_0^1 \big| D_{tt} \tilde h(\tau z+t_0(X_1))
 - D_{tt} \tilde h(\tau z+t_0(X_2))  \big| \, d \tau |z| |b(X_1,z+t_0(X_1))|
\\
&\leq C \|\frac{b(X_1,z+t_0(X_1))}{|z|}\|_{L^\infty} \frac{|X_1-X_2|^\alpha |z|}{|X_1|^{1+\alpha}},
\end{align*}
where we have used
\begin{align}
\label{est b1}
\|\frac{b(X_1,z+t_0(X_1))}{|z|}\|_{L^\infty}<\infty
\end{align}
and the norm is computed in a ball $B_{4R_{X_1}}(0)$.
Therefore
\begin{align*}
\int_{B_R(0)}  \frac{1}{|z|^{3-\ve}}A_1\,dz
&\leq C \frac{|X_1-X_2|^\alpha}{|X_1|^{1+\alpha}} \int_{B_R(0)} |z|^{\ve-2}\, d z
\\
&= C \frac{|X_1-X_2|^\alpha}{|X_1|^{1+\alpha}}
\frac{R^\ve}{\ve}
\\
& = \frac{C}{\ve}  \frac{|X_1-X_2|^{\alpha+\ve}}{|X_1|^{1+\alpha}} .
\end{align*}
For the integral outside $B_R(0)$ we estimate
\begin{align*}
A_1&\leq C\int_0^1 \big| D_{tt} \tilde h(z+t_0(X_\tau)) -  D_{tt} \tilde h(t_0(X_\tau) ) \big| \, d \tau |X_1-X_2|  \|\frac{b(X_1,z+t_0(X_1))}{|z|}\|_{L^\infty} |z|,
\end{align*}
where $\{X_\tau:\tau \in [0,1]\}$ denotes a path joining $X_1$ to $X_2$, with $|\frac{d}{d\tau} X_\tau| \leq C |X_1-X_2|$.
Hence
\begin{align*}
A_1&\leq C  \|\frac{b(X_1,z+t_0(X_1))}{|z|}\|_{L^\infty}  \frac{|X_1-X_2|}{|X_1|^{1+\alpha}} |z|^\alpha.
\end{align*}
Integrating,
\begin{align*}
\int_{B_{2R_{X_2}}(0) \setminus B_R(0)} \frac{1}{|z|^{3-\ve}} A_1 \, d z & \leq C \frac{|X_1-X_2|}{|X_1|^{1+\alpha}}\int_{B_{2R_{X_2}}(0) \setminus B_R(0)} |z|^{\ve-3+\alpha}\,dz
\\
&  \leq C \frac{|X_1-X_2|}{|X_1|^{1+\alpha}} R^{\ve-1+\alpha}
\\
& = C \frac{|X_1-X_2|^{\alpha+\ve}}{|X_1|^{1+\alpha}} .
\end{align*}
The estimate for $A_2$ works directly by using
$$
|b(X_1,z+t_0(X_1)) - b(X_2,z+t_0(X_2))|\leq C \frac{|z| |X_1-X_2|}{|X_1|^2}
$$
and there is no need to split the integral.

For $B$ we estimate as
\begin{align*}
B\leq B_1+B_2
\end{align*}
where
\begin{align*}
B_1 &=\frac{1}{|z|}\big| ( D_t \tilde h(t_0(X_1)) - D_t \tilde h(t_0(X_2)) ) ( D_t b(X_1,t_0(X_1)) z - b(X_1,z+t_0(X_1))) \big|
\\
B_2 &= \frac{1}{|z|}\big|  D_t \tilde h(t_0(X_2))  ( D_t b(X_1,t_0(X_1)) z - b(X_1,z+t_0(X_1))
\\
&\qquad - ( D_t b(X_2,t_0(X_2)) z - b(X_2,z+t_0(X_2)))  ) \big|.
\end{align*}
Using
\begin{align}
\label{est b2}
\big| D_t b(X_1,t_0(X_1)) z - b(X_1,z+t_0(X_1)) \big|
\leq C \frac{|z|^2}{|X_1|}
\end{align}
we get
\begin{align*}
B_1 & \leq C\frac{1}{|z|} \frac{|X_1-X_2|}{|X_1|} \big| D_t b(X_1,t_0(X_1)) z - b(X_1,z+t_0(X_1)) \big|
\\
&\leq  C\frac{1}{|z|} \frac{|X_1-X_2|}{|X_1|} \frac{|z|^2}{|X_1|}
\end{align*}
and then
$$
\int_{B_{2R_{X_2}}(0)} \frac{1}{|z|^{3-\ve}} B_1 \, dz \leq C \frac{|X_1-X_2|^{\alpha+\ve}}{|X_1|^{1+\alpha}} .
$$
For $B_2$, let  $R = 10 |X_1-X_2|$ and we assume $R \leq \frac{1}{10}|X_1|$.
Then,  using
\begin{align}
\nonumber
& \big|   D_t b(X_1,t_0(X_1)) z - b(X_1,z+t_0(X_1))
\\
\label{est b3}
& - ( D_t b(X_2,t_0(X_2)) z - b(X_2,z+t_0(X_2)))   \big| \leq C \frac{|X_1-X_2|^\alpha |z|^2}{|X_1|^{1+\alpha} },
\end{align}
we have
\begin{align*}
B_2 \leq \frac{C}{|z|} \frac{|X_1-X_2|^\alpha |z|^2}{|X_1|^{1+\alpha}}
\end{align*}
and we obtain
$$
\int_{B_{R}(0)} \frac{1}{|z|^{3-\ve}} B_2 \, dz \leq C \frac{|X_1-X_2|^{\alpha+\ve}}{|X_1|^{1+\alpha}} .
$$
To estimate the integral outside $B_R(0)$ we use
\begin{align}
\nonumber
& \big|   D_t b(X_1,t_0(X_1)) z - b(X_1,z+t_0(X_1))
\\
\label{est b4}
& - ( D_t b(X_2,t_0(X_2)) z - b(X_2,z+t_0(X_2)))   \big| \leq C \frac{|X_1-X_2| |z|^{1+\alpha}}{|X_1|^{1+\alpha} } ,
\end{align}
and we get
\begin{align*}
\int_{B_{2R_{X_2}}(0) \setminus B_{R}(0)} \frac{1}{|z|^{3-\ve}} B_2 \, dz
&\leq C\frac{|X_1-X_2|}{|X_1| ^{1+\alpha}}
\int_{B_{2R_{X_2}}(0) \setminus B_{R}(0)} \frac{1}{|z|^{3-\ve}} |z|^\alpha \, dz\\
&\leq \frac C\ve \frac{|X_1-X_2|^{\alpha+\ve} }{|X_1| ^{1+\alpha}} \end{align*}

Let us verify the assertions on $b$ made in \eqref{est b1}--\eqref{est b4}.
The function $b(X,z+t_0(X_0))$ is given at main order by
\begin{align}
\nonumber
& b(X,t+t_0(X)) \\
\label{form b}
& = \nu^{(i)} (x) \nu^{(j)}(x+ r_0  y(x, \frac{t-t_0(X)}{r_0})) -\nu^{(j)} (x) \nu^{(i)}(x+ r_0  y(x,\frac{t-t_0(X)}{r_0}))  ,
\end{align}
where $\nu^{(i)}$ are the components of the unit normal vector to $\Sigma_0$, which we consider as functions of $x$ with $X = (x,F_\ve(x))$, $r_0 = \delta |X|$, and $y$ is a  change of variables from variables $(t_1,t_2)$ parametrizing the tangent plane to $\Sigma_0$ at $X$ to $\R^2$. It has a bounded $C^{2,\alpha}$ norm.
Then
\begin{align*}
|b(X,t+t_0(X))|\leq C  \|\nabla \nu\|_{L^\infty} |t|
\end{align*}
but $r_0 = \delta |X| $ and $\|\nabla \nu\|_{L^\infty} = O(\frac{\ve^{1/2}}{ |X|})$, and this implies \eqref{est b1}.

To prove \eqref{est b2} note that
\begin{align*}
\big| D_t b(X,t_0(X)) z - b(X,z+t_0(X)) \big|
\leq \|D_{tt} b(X,\cdot)\|_{L^\infty} |z|^2 ,
\end{align*}
where the $L^\infty$ norm is in ball of center $x$ and radius $O(\delta |X|)$.
By using formula \eqref{form b} we get $ \|D_{tt} b(X_1,\cdot)\|_{L^\infty} = \frac{\ve^{\frac12}}{|X|^2}$ and we obtain \eqref{est b2}.

Estimates \eqref{est b3} and \eqref{est b4} can be prove similarly.



\bigskip

The complete argument is given next.
Thanks to \eqref{decomp DG} we can decompose
$$
g_{2,X_j} = \bar g_{2,X_j} + \tilde g_{2,X_j}
$$
where  $\bar g_{2,X_j} $ can be chosen so that it does not involve derivatives of $h$, and satisfies the estimate
$$
\| \bar g_{2,X_j} \|_b \leq C \|h_2\|_*.
$$

Then one can prove as was done for \eqref{prev B}:
\begin{align*}
\int_{B_{2R_{X_2}}(0)} \frac{|B(\bar g_{2,X_1})(X_1,z)  - B(\bar g_{2,X_2})(X_2,z)|}{|z|^{3-\ve}}\,dz&\leq \frac{C}{\ve } \frac{|X_1-X_2|^{\alpha+\ve}}{|X_1|^{1+\alpha}}  \|\bar g_{2,X_1}\|_b
\\&\leq \frac{C}{\ve } \frac{|X_1-X_2|^{\alpha+\ve}}{|X_1|^{1+\alpha}}  \|h_2\|_* .
\end{align*}
For $\tilde g_{2,X_j} $ we claim that the same estimate holds, that is, we claim that
\begin{align}
\label{tilde g2}	
\int_{B_{2R_{X_2}}(0)} \frac{|B(\tilde g_{2,X_1})(X_1,z)  - B(\tilde g_{2,X_2})(X_2,z)|}{|z|^{3-\ve}}\,dz
&\leq \frac{C}{\ve } \frac{|X_1-X_2|^{\alpha+\ve}}{|X_1|^{1+\alpha}}  \|h_2\|_* .
\end{align}
To prove this, we use representation
\begin{align}
\label{tilde g2 Xj}
\tilde g_{2,X_j} (t) = \tilde h_{2}(X_j,t) Q_0(X_j,t,\tilde h(X_j,t),D_t \tilde h(X_j,t) ) ,
\end{align}
where the functions $\tilde h_{2}(X_j,t)$, $\tilde h(X_j,t)$ are obtained from $h_2$ and $h$ through a change of variables:
$$
\tilde h(X_j,t) = h(y_{X_j}(t)) ,\quad  \tilde h_2(X_j,t)= h_2(y_{X_j}(t)) ,
$$
as in the Appendix~\ref{app}.
Here, for a given $X \in \Sigma_0$, $Q_0 = Q_0(X,t,h,\xi)$,
is defined for $t\in \R^2$, $ |t-t_0(X)|\leq 4 R_X$, $h\in \R$, $\xi\in\R^2$ and is explicitly written in \eqref{def Q}. It has the properties
$$
Q_0(X_j,t_0(X_j),h,\xi)=0
$$
$$
D_h Q_0(X_j,t_0(X_j),h,\xi)=0
$$
$$
D_{\xi_k} Q_0(X_j,t_0(X_j),h,\xi)=0 .
$$
Let us define
$$
Q(X,t+t_0(X),\xi ) = Q_0(X,t+t_0(X),\tilde h(X,z),\xi )
$$
so that
$$
\tilde g_{2,X_j}(t) =  \tilde h_2(X_j,t)  Q(X_j,t+t_0(X_j),D_t \tilde h(X_j,t)) ) .
$$

Now we write
\begin{align*}
& B(\tilde g_{2,X_j})(X_j,z) =\frac{\tilde g_{2,X_j}(z+t_0(X_j)) -\tilde g_{2,X_j}(t_0(X_j)) - \nabla \tilde g_{2,X_j}(t_0(X_j))z}{|z|}\\
&=\frac{\tilde h_2(X_j,z+t_0(X_j)) Q(X_j,z+t_0(X_j),\ldots) - \tilde h_2(X_j,t_0(X_j)) D_t Q(X_j,t_0(X_j),\ldots) z}{|z|}
\\
&=\frac{ ( \tilde h_2(X_j,z+t_0(X_j)) -\tilde h_2(X_j,t_0(X_j)) ) Q(X_j,z+t_0(X_j),\ldots)  }{|z|}
\\
&\qquad+\frac{\tilde h_2(X_j,t_0(X_j)) ( Q(X_j,z+t_0(X_j),\ldots) - D_t Q(X_j,t_0(X_j),\ldots) z ) }{|z|} ,
\end{align*}
where we are using the notation
\begin{align*}
Q(X_j,z+t_0(X_j),\ldots) &= Q(X_j,z+t_0(X_j),  D_t \tilde h(X_j,z+t_0(X_j) ) )
\\
Q(X_j,t_0(X_j),\ldots) &=Q(X_j,t_0(X_j), D_t \tilde h(X_j,t_0(X_j) ) ) .
\end{align*}

We have to estimate
\begin{align*}
B(\tilde g_{2,X_1})(X_1,z)  - B(\tilde g_{2,X_1})(X_1,z)
= D_1+D_2+D_3
\end{align*}
where
\begin{align*}
D_1 &=
\frac{ ( \tilde h_{2}(X_1,z+t_0(X_1)) -\tilde h_{2}(X_1,t_0(X_1)) ) Q(X_1,z+t_0(X_1),\ldots)  }{|z|}
\\
&\qquad-\frac{ ( \tilde h_{2}(X_2,z+t_0(X_2)) -\tilde h_{2}(X_2,t_0(X_2)) ) Q(X_2,z+t_0(X_2),\ldots)  }{|z|}
\\
D_2&= (\tilde h_{2}(X_1,t_0(X_1))-\tilde h_{2}(X_2,t_0(X_2)))\frac{ Q(X_1,z+t_0(X_1),\ldots) - D_t Q(X_1,t_0(X_1),\ldots) z  }{|z|}
\\
D_3 &=\tilde h_{2,X_2}(t_0(X_2)) \Big[ \frac{  Q(X_1,z+t_0(X_1),\ldots) - D_t Q(X_1,t_0(X_1),\ldots) z )
}{|z|}
\\
&\qquad \qquad  - \frac{ Q_{X_2}(z+t_0(X_2),\ldots) - D_t Q_{X_2}(t_0(X_2),\ldots) z }{|z|} \Big].
\end{align*}
The estimate
\begin{align*}
\int_{B_{2 R_{X_2}}(0)} \frac{1}{|z|^{3-\ve}} (|D_1|+|D_2|)\,z
\leq C\frac{|X_1-X_2|^{\alpha+\ve}}{|X_1|^{1+\alpha}}
\end{align*}
can be proved in the same way as before, since for $D_1$ difference quotients of $Q$ involve only difference quotients of $D_t \tilde h$ which can be controlled by $\|h\|_*$, and for $D_2$ we need only to consider difference quotients of $\tilde h$.

The estimate of $D_3$ is more delicate, and we proceed with detail.
We further split
\begin{align*}
D_3 = \tilde h_{2,X_2}(t_0(X_2)) ( D_{3,a}+D_{3,b})
\end{align*}
where
\begin{align*}
D_{3,a}&=\frac{1}{|z|}\big[ Q(X_1,z+t_0(X_1),  D_t \tilde h(X_1,z+t_0(X_1) ) )
- Q(X_1,z+t_0(X_1), D_t \tilde h(X_1,t_0(X_1) ) )
\\
&\qquad - \big( Q(X_2,z+t_0(X_2),  D_t \tilde h(X_2,z+t_0(X_2) ) )
-  Q(X_2,z+t_0(X_2),  D_t \tilde h(X_2,t_0(X_2) ) )  \big) \big]
\end{align*}
\begin{align*}
D_{3,b}&=\frac{1}{|z|} \big[ Q(X_1,z+t_0(X_1),D_t\tilde h(X_1,t_0(X_1)))
- D_t Q(X_1t_0(X_1),D_t\tilde h(X_1,t_0(X_1)))
\\
& \qquad
-
\big( Q(X_2,z+t_0(X_2),D_t\tilde h(X_2,t_0(X_2)))
- D_t Q(X_2,t_0(X_2),D_t\tilde h(X_2,t_0(X_2)))\big)
\big] -
\end{align*}
To estimate the integral of $\frac{1}{|z|^{3-\ve}} |D_{3,a}|$ over $B_{R_{X_2}}(0)$ we divide the region of integration in $B_R(0)$ and $B_{R_{X_2}}(0) \setminus B_{R}(0)$,
where $R=10|X_1-X_2|$.
To estimate the integral inside $B_R(0)$ we compute
\begin{align*}
D_{3,a} &=
\frac{1}{|z|}
\int_0^1 \frac{d}{d\tau}
\big[
Q(X_1, z+t_0(X_1), D_t \tilde h(X_1,\tau  z+t_0(X_1) ) )
\\
&\qquad\qquad
- Q(X_2, z+t_0(X_2), D_t \tilde h(X_2,\tau  z+t_0(X_2) ) )
\big]\, d\tau
\\
&= \frac{1}{|z|}\int_0^1 \big[
D_\xi Q(X_1, z+t_0(X_1), D_t \tilde h(X_1,\tau  z+t_0(X_1) ) ) D_{tt} \tilde h(X_1,\tau z + t_0(X_1))
\\
&\qquad -
D_\xi Q(X_2, z+t_0(X_2), D_t \tilde h(X_2,\tau  z+t_0(X_2) ) ) D_{tt} \tilde h(X_2,\tau z + t_0(X_2))
\big] \, d\tau .
\end{align*}
From this
\begin{align*}
|D_{3,a}|\leq D_{3,a,1} + D_{3,a,2}
\end{align*}
where
\begin{align*}
D_{3,a,1}
&= \frac{1}{|z|} \int_0^1  \big|
\big[ D_\xi Q(X_1, z+t_0(X_1), D_t \tilde h(X_1,\tau  z+t_0(X_1) ) )
\\
&\qquad \qquad - D_\xi Q(X_2, z+t_0(X_2), D_t \tilde h(X_2,\tau  z+t_0(X_2))) \big] D_{tt} \tilde h(X_1,\tau z + t_0(X_1))
 \big| \,d\tau
\\
D_{3,a,2}&=\int_0^1  \big|
 D_\xi Q(X_2, z+t_0(X_2), D_t \tilde h(X_2,\tau  z+t_0(X_2) ) )  )
\\
&\qquad\qquad\cdot
\big[ D_{tt} \tilde h(X_1,\tau z + t_0(X_1))-D_{tt} \tilde h(X_2,\tau z + t_0(X_2))\big]
 \big| \,d\tau
\end{align*}
Using regularity of $Q(X,t,h,\xi)$ with respect to $X$ and that we we have control of $D^2 \tilde h$ we
have
\begin{align*}
&
 \big| D_\xi Q(X_1, z+t_0(X_1) , D_t \tilde h(X_1,\tau  z+t_0(X_1) ) ) \\
&
\qquad \qquad -
D_\xi  Q(X_2, z+t_0(X_2),  D_t \tilde h(X_2,\tau  z+t_0(X_2) ) ) \big|
\leq C \frac{|z| |X_1-X_2|}{|X_1|}  .
\end{align*}
Using this and that  $|D_{tt}^2 \tilde h|\leq \|h\|_*/|X_1|$ we find
\begin{align*}
& \int_{B_{2 R_{X_2}}(0)}
\frac{1}{|z|^{4-\ve}}| ( D_\xi Q(X_1,z + t_0(X_1), \ldots)  - D_\xi Q(X_2, z + t_0(X_2), \ldots) ) D_{tt}  \tilde h(X_1,\tau z + t_0(X_1))z| \,dz
\\
& \leq C \frac{|X_1-X_2|}{|X_1|^2} \int_{B_{R_{X_2}}(0)} \frac{1}{|z|^{2-\ve}}\,dz \leq \frac C\ve\frac{|X_1-X_2|^{\alpha+\ve}}{|X_1|^{1+\alpha}} .
\end{align*}
For the second term we use
$$
|D_\xi Q(X_2, z + t_0(X_2), \ldots)|\leq C |z|
$$
and
\begin{align*}
|D_{tt} \tilde h(X_1,\tau z + t_0(X_1)) -   D_{tt} \tilde h(X_2,\tau z + t_0(X_1))|\leq C \frac{|X_1-X_2|^\alpha}{|X_1|^{1+\alpha}} .
\end{align*}
Then we obtain
\begin{align*}
&  \int_{B_{R}(0)}
\frac{1}{|z|^{4-\ve}}| D_\xi Q(X_2, z + t_0(X_2), \ldots) ( D_{tt} \tilde h(X_1,\tau z + t_0(X_1))z -   D_{tt} \tilde h(X_2,\tau z + t_0(X_1))z ) |\,dz
\\
&\leq \frac C\ve \frac{|X_1-X_2|^{\alpha+\ve}}{|X_1|^{1+\alpha}} .
\end{align*}
Therefore we get
\begin{align*}
\int_{B_{R}(0)}\frac{1}{|z|^{3-\ve}}|D_{3,a}|\,dz \leq \frac C\ve \frac{|X_1-X_2|^{\alpha+\ve}}{|X_1|^{1+\alpha}} .
\end{align*}

Let us proceed with the estimate of the integral of $\frac{1}{|z|^{3-\ve}}|D_{3,a}|$ over the region $B_{R_{X_2}}(0)\setminus B_R(0)$.

Recall that the points $X_j$ have the form $X_j = (x_j , F_\ve(x_j))$, $j=1,2$.
For for $\tau \in [0,1]$ we let $X_\tau = (x_\tau,F_\ve(x_\tau))$ where  $x_\tau = \tau x_1 + (1-\tau)x_2$. We compute
\begin{align*}
D_{3,a} &=
\frac{1}{|z|}
\int_0^1 \frac{d}{d\tau}
\big[
Q(X_\tau,z+t_0(X_\tau),  D_t \tilde h(X_\tau,  z+t_0(X_\tau) ) ) \\
&\qquad\qquad
 - Q(X_\tau, z+t_0(X_\tau), D_t \tilde h(X_\tau,  t_0(X_\tau) ) )
\big]\, d\tau
\end{align*}
and so
\begin{align*}
|D_{3,a}|\leq\int_0^1 (\bar D_{3,a,1}+\bar D_{3,a,2}+\bar D_{3,a,3}+\bar D_{3,a,4}) \, d\tau
\end{align*}
where
\begin{align*}
\bar D_{3,a,1} &= \frac{ |X_1-X_2|}{|z|} \big| D_X Q(X_\tau,z+t_0(X_\tau),D_t \tilde h(X_\tau,z+t_0(X_\tau)))
\\
&\qquad \qquad- D_X Q(X_\tau,z+t_0(X_\tau),D_t \tilde h(X_\tau,t_0(X_\tau))) \big|
\end{align*}
\begin{align*}
\bar D_{3,a,2} &= \frac{ |X_1-X_2|}{|z|}| ( D_t Q(X_\tau,z+t_0(X_\tau),D_t \tilde h(X_\tau,z+t_0(X_\tau)))
\\
&\qquad\qquad
- D_t Q(X_\tau,z+t_0(X_\tau),D_t \tilde h(X_\tau,t_0(X_\tau))) ) D_X t_0(X_\tau)|
\end{align*}
\begin{align*}
\bar D_{3,a,3} &= \frac{ |X_1-X_2|}{|z|}\big| \big[ D_\xi Q(X_\tau,z+t_0(X_\tau),D_t \tilde h(X_\tau,z+t_0(X_\tau)))  \\
&\qquad\qquad
- D_\xi Q(X_\tau,z+t_0(X_\tau),D_t \tilde h(X_\tau,t_0(X_\tau))) \big] D^2_{X,t} \tilde h(X_\tau,z+t_0(X_\tau))
\big|
\end{align*}
\begin{align*}
\bar D_{3,a,4} &= \frac{ |X_1-X_2|}{|z|}\big| D_\xi Q(X_\tau,z+t_0(X_\tau),D_t \tilde h(X_\tau,t_0(X_\tau)))
\\
&\qquad\qquad \cdot \big [D^2_{X,t} \tilde h(X_\tau,z+t_0(X_\tau)) -D^2_{X,t} \tilde h(X_\tau,t_0(X_\tau)) \big]\big|
\end{align*}
\begin{align*}
\bar D_{3,a,5}&= \frac{ |X_1-X_2|}{|z|} \big| \big[ D_\xi Q(X_\tau,z+t_0(X_\tau),D_t \tilde h(X_\tau,z+t_0(X_\tau)))
\\
&\qquad\qquad
- D_\xi Q(X_\tau,z+t_0(X_\tau),D_t \tilde h(X_\tau,t_0(X_\tau))) \big] D^2_{tt} \tilde h(X_\tau,z+t_0(X_\tau))   D_X t_0(X_\tau)|
\end{align*}
\begin{align*}
\bar D_{3,a,6}&= \frac{ |X_1-X_2|}{|z|} \big|
D_\xi Q(X_\tau,z+t_0(X_\tau),D_t \tilde h(X_\tau,t_0(X_\tau)))
\\
&\qquad \qquad\big[D^2_{tt} \tilde h(X_\tau,z+t_0(X_\tau)) -D^2_{tt} \tilde h(X_\tau,t_0(X_\tau)) \big] D_X t_0(X_\tau) .
\end{align*}
The most delicate terms are the ones involving differences of second derivatives of $\tilde h$. For example, for $\bar D_{3,a,6}$ we use
$$
| D_\xi Q(X_\tau,z+t_0(X_\tau),D_t \tilde h(X_\tau,t_0(X_\tau)))|\leq C |z|
$$
and
$$
|D^2_{tt} \tilde h(X_\tau,z+t_0(X_\tau)) -D^2_{tt} \tilde h(X_\tau,t_0(X_\tau)) |\leq C\frac{|z|^\alpha}{|X_1|^{1+\alpha}}
$$
and obtain
\begin{align*}
\int_{B_{2 R_{X_2}}(0)}\frac{1}{|z|^{3-\ve}} \bar D_{3,a,6}\,dz
&\leq C\frac{|X_1-X_2|}{|X_1|^{1+\alpha}}
\int_{B_{2R_{X_2}}(0)} |z|^{\alpha-3+\ve}\,dz\\
&\leq  C\frac{|X_1-X_2|^{\alpha+\ve}}{|X_1|^{1+\alpha}}.
\end{align*}
Other terms in $D_{3,a}$ are estimated similarly and we find
\begin{align*}
\int_{B_{2 R_{X_2}}(0) \setminus B_{R}(0)}\frac{1}{|z|^{3-\ve}}|D_{3,a}|\,dz \leq \frac C\ve \frac{|X_1-X_2|^{\alpha+\ve}}{|X_1|^{1+\alpha}} .
\end{align*}
All other terms can be handled with analogous computations, and this establishes \eqref{est nonl holder 01}.

Let us prove now \eqref{est nonl holder 02}. Let
$$
\hat g_{X_j} = D^2 G_{X_j}(h)[h_1,h_2] .
$$
We claim that
$$
|D \tilde H_{X_1}(g_{X_1})[\hat g_{X_1}]-D \tilde H_{X_2}(g_{X_2})[\hat g_{X_2}]|\leq \frac{C}{\ve} \|h_1\|_* \|h_2\|_* \frac{|X_1-X_2|^{\alpha+\ve}}{|X_1|^{1+\alpha}}.
$$
For this we write
$$
 D \tilde H_{X_j}(g_{X_j})[\hat g_{X_j} ] = A_4(X_j) + A_5(X_j)
$$
where
\begin{align*}
A_4(X_j) &=\int_{\R^2} \frac{\eta(\frac{|z|}{R_{X_j}})}{|z|^{3-\ve}}
 \psi'' ( A_t(g_{X_j})  (X_j,z) ) A_t(\hat g_{X_j})  (X_j,z) B(g_{X_j}) (X_j,z) \, d z
\\
A_5(X_j) &=
\int_0^1  \frac{\eta(\frac{|z|}{R_{X_j}})}{|z|^{3-\ve}}
\psi'(A_t(g_{X_j}) (X_j,z) ) B(\hat g_{X_j}) (X_j,z)  \, d z .
\end{align*}
The most delicate difference is
$$
A_5(X_1)  - A_5(X_2) = \int_0^1 (B_1+B_2+B_3)\,dt
$$
where
\begin{align*}
B_1 &=\int_{\R^2} \frac{\eta(\frac{|z|}{R_{X_1}}) - \eta(\frac{|z|}{R_{X_2}}) }{|z|^{3-\ve}}  \psi'(A_t(g_{X_1}) (X_1,z) ) B(\hat g_{X_1}) (X_1,z)  \, dz ,
\\
B_2 &=\int_{B_{R_{X_2}}(0)} \frac{1}{|z|^{3-\ve}} \big( \psi'(A_t(g_{X_1}) (X_1,z) ) -  \psi'(A_t(g_{X_2}) (X_2,z) ) \big) B(\hat g_{X_1}) (X_1,z)  \, dz ,
\\
B_3 &=\int_{B_{R_{X_2}}(0)} \frac{1}{|z|^{3-\ve}}  \psi'(A_t(g_{X_2}) (X_2,z) ) \big( B(\hat g_{X_1}) (X_1,z) - B(\hat g_{X_2}) (X_2,z) \big) \, dz .
\end{align*}
Let us focus on the most delicate term, $B_3$. It can be estimated as follows:
\begin{align*}
|B_3|\leq \int_{B_{R_{X_2}}(0)} \frac{1}{|z|^{3-\ve}} |B(\hat g_{X_1})(X_1,z) - B(\hat g_{X_2})(X_2,z)|\,dz
\end{align*}
and we claim that
\begin{align*}
 \int_{B_{R_{X_2}}(0)} \frac{1}{|z|^{3-\ve}} |B(\hat g_{X_1})(X_1,z) - B(\hat g_{X_2})(X_2,z)|\,dz \leq \frac{C}{\ve} \|h_1\|_* \|h_2\|_* \frac{|X_1-X_2|^{\alpha+\ve}}{|X_1|^{1+\alpha}}.
\end{align*}
The computation involves a similar difficulty as in the estimate of \eqref{301}, except that now the functions $\hat g_{X_j}$ involve up to second derivatives of $h$. They can also be decomposed as follows
\begin{align*}
\hat g_{X_j} = \hat g_{0,X_j} + \hat g_{1,X_j} +\hat g_{2,X_j}  .
\end{align*}
Since $B$ is linear we have to estimate
\begin{align}
\label{hat gk X}
\int_{B_{R_{X_2}}(0)} \frac{1}{|z|^{3-\ve}} |B(\hat g_{k,X_1})(X_1,z) - B(\hat g_{k,X_2})(X_2,z)|\,d z
\end{align}
for $k=0,1,2$.
For  $ \hat g_{0,X_j}$ we have
\begin{align}
\label{400}
\| \hat g_{0,X_j}\|_b \leq C \|h_1\|_* \|h_2\|_* .
\end{align}
and so we can estimate the integral as we did for \eqref{prev B}, using \eqref{400}.
For $ \hat g_{1,X_j}$ we have the same properties as for $\tilde g_{2,X_j}$ in \eqref{tilde g2 Xj} and so the estimate of the integral can be done in the same way as in the proof of \eqref{tilde g2}.

Let us show that \eqref{hat gk X} holds for $k=2$. This function has the form:
\begin{align*}
\hat g_{2,X_j} = D_{tt} \tilde h(X_j,t+t_0(X_j)) b(X_j,t+t_0(X_j)) ,
\end{align*}
with $b$ now satisfying
$$
b(X_j,t+t_0(X_j))  = O(\frac{|t|^2}{|X_j|}).
$$

Let us sketch the computations, assuming no dependence on the first variables in $h$ and $b$, and $t_0(X)=X$  i.e.
\begin{align*}
B(\hat g_{2,X_j})(X_j,z) = \frac{D_{tt} \tilde h(z+X_j) b(z)}{|z|}
\end{align*}
and so
\begin{align*}
&
B(\hat g_{2,X_1})(X_1,z) -B(\hat g_{2,X_2})(X_2,z) \\
&= \frac{1}{|z|}\big[
D_{tt} \tilde h(z+X_1) b(z) - D_{tt} \tilde h(z+X_2) b(z)
\big]
\end{align*}

Note that the functions
$$
\left| \frac{1}{|z|^{4-\ve}}
D_{tt} \tilde h(z+X_j) b(z) \right| \leq C \frac{1}{|z|^{2-\ve}}
$$
are integrable.

We have to estimate
\begin{align*}
&
\int_{B_{R_{X_2}}(0) } \frac{1}{|z|^{4-\ve}}\big[
D_{tt} \tilde h(z+X_1) b(z) - D_{tt} \tilde h(z+X_2) b(z)\big] \,dz
\\
&=
\int_{B_{R_{X_2}}(0) } \frac{1}{|z|^{4-\ve}}
D_{tt} \tilde h(z+X_1) b(z) \,dz-
\int_{B_{R_{X_2}}(0) } \frac{1}{|z|^{4-\ve}} D_{tt} \tilde h(z+X_2) b(z)\,dz
\end{align*}

Let $\bar R = 10 |X_1 - X_2|$ and $\bar X$ be the middle point between $X_1$ and $X_2$: $\bar X = \frac12(X_1+X_2)$ in this simplified calculation.

Then
\begin{align*}
& \left|\int_{B_R(0)}\frac{1}{|z|^{4-\ve}}
D_{tt} \tilde h(z+X_1) b(z) \,dz-
\int_{B_R(0)} \frac{1}{|z|^{4-\ve}} D_{tt} \tilde h(z+X_2) b(z)\,dz\right|
\\
&\leq \int_{B_R(0)} \frac{1}{|z|^{4-\ve}}  | (D_{tt} \tilde h(z+X_1)-D_{tt} \tilde h(z+X_2))b(z) |\,dz
\\
& \leq C [D^2 \tilde h]_\alpha |X_1-X_2|^\alpha \sup\frac{|b(z)|}{|z|^2} \int_{B_R(0)} \frac{1}{|z|^{2-\ve}}  \,dz
\\
&\leq C [D^2 \tilde h]_\alpha |X_1-X_2|^\alpha \sup\frac{|b(z)|}{|z|^2} \frac{R^\ve}\ve
\\
& \leq \frac {C|X_1-X_2|^{\alpha+\ve}  }{|X_1|^{1+\alpha}}
\end{align*}
and we need here
$$
[D^2 \tilde h]_\alpha\leq \frac{C}{|X_1|^{1+\alpha}}
,\\quad
\sup\frac{|b(z)|}{|z|^2} \leq C.
$$
The remaining part is
\begin{align*}
(a)=& \int_{B_{R_{X_2}}(0)\setminus B_R(0) } \frac{1}{|z|^{4-\ve}}
D_{tt} \tilde h(z+X_1) b(z) \,dz-
\int_{B_{R_{X_2}}(0) \setminus B_R(0)} \frac{1}{|z|^{4-\ve}} D_{tt} \tilde h(z+X_2) b(z)\,dz \\
& = \int_{B_{R_{X_2}}(0)\setminus B_R(0) } \frac{1}{|z|^{4-\ve}}
( D_{tt} \tilde h(z+X_1) -D_{tt} \tilde h(\bar X)) b(z) \,dz\\
& \qquad -
\int_{B_{R_{X_2}}(0) \setminus B_R(0)} \frac{1}{|z|^{4-\ve}} ( D_{tt} \tilde h(z+X_2)-D_{tt} \tilde h(\bar X)) b(z)\,dz \\
&=
\int_{B_{R_{X_2}}(X_1)\setminus B_R(X_1) } \frac{1}{|z-X_1|^{4-\ve}}
( D_{tt} \tilde h(z) - D_{tt} \tilde h(\bar X) )b(z-X_1) \,dz
\\
&\qquad -
\int_{B_{R_{X_2}}(X_2) \setminus B_R(X_2)} \frac{1}{|z-X_2|^{4-\ve}} ( D_{tt} \tilde h(z) - D_{tt} \tilde h(\bar X)) b(z-X_2)\,dz ,
\end{align*}
where we have added and subtracted
$$
D_{tt} \tilde h(\bar X) \int_{B_{R_{X_2}}(0)\setminus B_R(0) } \frac{1}{|z|^{4-\ve}} b(z) \,dz .
$$
Let us decompose
\begin{align*}
(a) &= \int_A \frac{1}{|z|^{2-\ve}}| D_{tt} \tilde h(z) - D_{tt} \tilde h(\bar X)|\left|\frac{b(z-X_1)}{|z-X_1|^{4-\ve}} - \frac{b(z-X_2)}{|z-X_2|^{4-\ve}}\right|\,dz
\\
&\qquad+
\int_{R_1} \frac{1}{|z-X_1|^{4-\ve}}
( D_{tt} \tilde h(z) - D_{tt} \tilde h(\bar X) )b(z-X_1) \,dz\\
&\qquad-
\int_{R_2}\frac{1}{|z-X_2|^{4-\ve}} ( D_{tt} \tilde h(z) - D_{tt} \tilde h(\bar X)) b(z-X_2)\,dz ,
\\
&\qquad+\int_{R_3}  \frac{1}{|z-X_1|^{4-\ve}}
( D_{tt} \tilde h(z) - D_{tt} \tilde h(\bar X) )b(z-X_1) \,dz\\
&\qquad-\int_{R_4} \frac{1}{|z-X_2|^{4-\ve}} ( D_{tt} \tilde h(z) - D_{tt} \tilde h(\bar X)) b(z-X_2)\,dz .
\end{align*}
where
\begin{align*}
A & = [ B_{R_{X_2}}(X_1)\setminus B_R(X_1) ] \cap [B_{R_{X_2}}(X_2) \setminus B_R(X_2)]
\\
R_1 &= B_{R_{X_2}}(X_1)\setminus B_{R_{X_2}}(X_2)  \\
R_2 &= B_{R_{X_2}}(X_2)\setminus B_{R_{X_2}}(X_1)  \\
R_3 &= B_R(X_2)\setminus B_R(X_1)\\
R_4 &= B_R(X_1)\setminus B_R(X_2).
\end{align*}
Note that
$$
B_{R/2}(\bar X) \subset A \subset B_{2 R_{X_2} }(\bar X) .
$$
We estimate
\begin{align*}
& \int_A \frac{1}{|z|^{2-\ve}}| D_{tt} \tilde h(z) - D_{tt} \tilde h(\bar X)|\left|\frac{b(z-X_1)}{|z-X_1|^{4-\ve}} - \frac{b(z-X_2)}{|z-X_2|^{4-\ve}}\right|\,dz
\\
& \leq C  [D^2 \tilde h]_\alpha
\int_A |z-\bar X|^\alpha\left|\frac{b(z-X_1)}{|z-X_1|^{4-\ve}} - \frac{b(z-X_2)}{|z-X_2|^{4-\ve}}\right|\,dz \\
& \leq  C [D^2\tilde h]_\alpha
\int_0^1 \int_A |z-\bar X|^\alpha\left|
\frac{d}{dt} \frac{b(z-X_t)}{|z-X_t|^{4-\ve}}
\right|\,dz dt
\end{align*}
where $X_t = tX_2+(1-t)X_1$.
Assuming
\begin{align*}
\sup\frac{b(z)}{|z|^2}& \leq C\\
\sup\frac{|\nabla b(z)|}{|z|}& \leq C
\end{align*}
we get
\begin{align*}
\left|\frac{b(z-X_1)}{|z-X_1|^{4-\ve}} - \frac{b(z-X_2)}{|z-X_2|^{4-\ve}}\right| \leq  C \frac{|X_1-X_2|}{|z-X_t|^{3-\ve}}.
\end{align*}
Then
\begin{align*}
& \int_A \frac{1}{|z|^{2-\ve}}| D_{tt} \tilde h(z) - D_{tt} \tilde h(\bar X)|\left|\frac{b(z-X_1)}{|z-X_1|^{4-\ve}} - \frac{b(z-X_2)}{|z-X_2|^{4-\ve}}\right|\,dz
\\
&\leq C [D^2 \tilde h]_\alpha |X_1-X_2|
\int_0^1\int_A \frac{|z-\bar X|^\alpha }{|z-X_t|^{3-\ve}}\,dz dt
\\
&\leq C [D^2 \tilde h]_\alpha |X_1-X_2|
\int_{B_{R/2}(\bar X)^c} |z-X_t|^{\alpha+\ve-3}\,dz
\\
&\leq C [D^2 \tilde h]_\alpha |X_1-X_2|
R^{\alpha+\ve-1}
\\
&\leq C [D^2 \tilde h]_\alpha |X_1-X_2|^{\alpha+\ve} .
\end{align*}
Now we estimate
\begin{align*}
& \left| \int_{R_1} \frac{1}{|z-X_1|^{4-\ve}}
( D_{tt} \tilde h(z) - D_{tt} \tilde h(\bar X) )b(z-X_1) \,dz\right|\\
& \leq C [D^2 \tilde h]_\alpha \int_{B_{R_{X_2}+|X_1-X_2|}(X_1) \setminus B_{R_{X_2}-|X_1-X_2|}(X_1)} |z-X_1|^{\alpha+\ve-2} \,dz\\
& \leq  C [D^2 \tilde h]_\alpha( (R_{X_2} + |X_1-X_2|)^{\alpha+\ve} -  (R_{X_2} - |X_1-X_2|)^{\alpha+\ve} )
\\
& \leq  C [D^2 \tilde h]_\alpha R_{X_2}^{\alpha+\ve-1} |X_1-X_2|\\
& \leq  C [D^2 \tilde h]_\alpha  |X_1-X_2|^{\alpha+\ve} .
\end{align*}
The estimate of $R_2$, $R_3$ and $R_4$ are analogous.

\end{proof}

\begin{proof}[Proof of Lemma~\ref{lemma Ni}]
Write
\begin{align*}
N_i(h_1) - N_i(h_2)
& = H_i(h_1) -H_i(h_2) - D H_i(0)[h_1-h_2]
\\
&=\int_0^1  ( D H_i( t h_1 +(1-t) h_2)[h_1-h_2]  - D H_i(0)[h_1-h_2] )\, d t \\
&=\int_0^1 \int_0^1  D^2 H_i( s( t h_1 +(1-t) h_2))[h_1-h_2, t h_1 +(1-t) h_2]\, d s d t
\end{align*}
Using  Lemma \ref{lemma 7.4}
we get
\begin{align*}
|N_i(h_1)(X) - N_i(h_2)(X)|\leq  \frac{C }{\ve |X|^{1-\ve} } \|h_1-h_2\|_* (\|h_1\|_*+\|h_2\|_*).
\end{align*}
By Lemma~\ref{lemma n2}, if $|X_1-X_2|\leq \frac{1}{10}\min(|X_1|,|X_2)$,
\begin{align*}
& |N_i(h_1)(X_1)- N_i(h_2)(X_1)- (N_i(h_1)(X_2)- N_i(h_2)(X_2))|
\\
& \leq\frac{C}{\ve} \frac{|X_1-X_2|^{\alpha+\ve}}{\min(|X_1|,|X_2)^{1+\alpha}} \|h_1-h_2\|_* (\|h_1\|_*+\|h_2\|_*).
\end{align*}
\end{proof}

\begin{proof}[Proof of Lemma~\ref{lemma No}]
By a direct and long computation we obtain
$$
\ve |D^2 H_o(h)[h_1,h_2](x)
|\leq \frac{C}{\ve^{\frac12}|x|^{1-\ve}}\|h_1\|_*\|h_2\|_*
$$
for $x\in \Sigma_0$,
and if $x_1,x_2\in\Sigma_0$, $|x_1-x_2|\leq \frac{1}{10}|x_1|$, then
$$
\ve|D^2 H_o(h)[h_1,h_2](x_1)-D^2 H_o(h)[h_1,h_2](x_2)|\leq C\frac{|x_1-x_2|^{\alpha+\ve}}{\ve^{\frac12}|x_1|^{1+\alpha}} \|h_1\|_*\|h_2\|_*.
$$
Then the lemma follows as in the proof of Lemma~\ref{lemma Ni}.


\end{proof}

\section{Layered fractional minimal surfaces}

\begin{proof}[Proof of Theorem~\ref{teo2}]
The proof is essentially the same as for Theorem~\ref{teo1}.
This time we look for a set $E \subseteq \R^3$ of the form
$$
E = \{ (x',x_3):\in\R^3: |x_3|>f(x')\},
$$
where $f:\R^2\to\R$ is a positive radially symmetric function.
We take as a first approximation
$$
E_0 = \{ (x',x_3):\in\R^3: |x_3|>f_{\ve}(x')\},
$$
where $f_{\ve}$ is the unique radial solution to
\begin{align}
\label{eqf2}
\Delta f_{\ve} = \frac{\ve}{f_{\ve} }  ,  \quad f_\ve>0,  \quad \text{in }\R^2,
\end{align}
with $f_{\ve} (0)=1$. Then $f_{\ve}(x) = f_1(\ve^{\frac12}x)$ where $f_1$ is the radial solution of $\Delta f = \frac 1f$ with $f_1(0)=1$. The same analysis of Section~\ref{sect ode} applies to show that $f_1(r) = r + O(1)$ as $r\to\infty$ and one obtains the same estimates for $f_\ve$ as for $F_\ve$. This leads to the estimate
$$
\|\ve H_{\Sigma_0}^s\|_{1-\ve,\alpha+\ve} \leq C \ve.
$$

As before, we construct the surface $\Sigma$ and the corresponding set $E$ by perturbing the surface $\Sigma_0$ in the normal direction $\nu_{\Sigma_0}$ (it could also be done using vertical perturbations). That is, for a function $h$ defined on $\Sigma_0$ (small with a suitable norm)
we let
$$
\Sigma_h = \{ x +  h(x)\nu_{\Sigma_0}(x)\ /\ x\in \Sigma_0\} .
$$
As before, we are led  to find $h$ such that
$$
H^s_{\Sigma_0} + 2 \JJ^s_{\Sigma_0}(h) + N(h) = 0 .
$$
We solve for $h$ in this equation using the contraction mapping principle, employing the same norms as in \eqref{norm st}, \eqref{norm RHS}.
The solvability of the linearized problem
$$
\ve \JJ_{\Sigma_0}^s(h)=f \quad \text{in }\Sigma_0
$$
in weighted H\"older space and the estimates for $N(h)$  are very similar to the ones in Theorem~\ref{teo1}.
\end{proof}

We can also construct axially symmetric solutions with multiple layers.
Suppose that
$$
f_1 > f_2 > \ldots > f_k,
$$
are radially symmetric functions on $\R^n$ and consider the surface $\Sigma$ defined by
$$
\Sigma = \{(x,x_{n+1}) \in \R^n \times \R^1 :  x_{n+1} = f_i(x), \text{ for some } i\}.
$$
We claim it is possible to choose $f_i$ such that this surface is $s$-minimal for $s$ close to 1.

We will not give a detailed proof  of this statement, but only derive formally the form of the elliptic system that plays the role of the equation \eqref{eqf2} for the case of two layers and mention a few of its properties.

For the derivation of the system, we assume that the functions $f_i$ have small gradient, a condition that a posteriori is verified. Note that  the surface $\Sigma$ is the boundary of the region $E$ given by
$$
E=\{(x,x_{n+1}) \in \R^{n+1} :   f_{j}(x)  >  x_{n+1} > f_{j+1}(x) \text{ for $j$ even, } j\in \{0,\ldots,k\} \},
$$
with the convention $f_0 = \infty$, $f_{k+1}=-\infty$.

Consider a point $X = (x,x_{n+1})$ with $x_{n+1}= f_i(x)$.
We split the integral
$$
\int_{\R^{n+1} \setminus B_R(X)}\frac{\chi_E(Y) - \chi_{E^c}(Y)}{|X-Y|^{n+1+s}} \, dY
= \int_{y\in \R^n} \int_{y_{n+1} > f_{i-1}(y)}\frac{\chi_E(Y) - \chi_{E^c}(Y)}{|X-Y|^{n+1+s}} \, dY
$$
$$
+ \int_{y\in \R^n} \int_{y_{n+1} < f_{i+1}(y)}\frac{\chi_E(Y) - \chi_{E^c}(Y)}{|X-Y|^{n+1+s}} \, dY
+\int_{y\in \R^n} \int_{ f_{i-1}(y)> y_{n+1} >f_{i+1}(y) }\frac{\chi_E(Y) - \chi_{E^c}(Y)}{|X-Y|^{n+1+s}} \, dY .
$$
Remark that for $b\geq a > f_i(x)$
\begin{align*}
& \int_{\R^n}
\int_{a}^{b}
\frac{1}{(|y-x|^2+(y_{n+1}-f_i(x))^2)^{\frac{n+1+s}2}}
\, d y_{n+1}
\, d y
\\
&=
const
\int_{a-f_i(x)}^{b-f_i(x)} \frac{1}{|t|^{1+s}} \, d t
= const ( \frac{1}{(a-f_i(x))^s} - \frac{1}{ (b-f_i(x))^s}  )
\\
& \approx const ( \frac{1}{a-f_i(x)} - \frac{1}{ b-f_i(x)} )
\end{align*}
and for $f_i(x)> b\geq a $
\begin{align*}
& \int_{\R^n}
\int_{a}^{b}
\frac{1}{(|y-x|^2+(y_{n+1}-f_i(x))^2)^{\frac{n+1+s}2}}
\, d y_{n+1}
\, d y
\\
& = const ( \frac{1}{ f_i(x)-b} - \frac{1}{f_i(x)-a} ),
\end{align*}
where $const>0$.

By decomposing into a ball and its complement and assuming for instance $f_{i-1} - f_i \geq f_i - f_{i+1}$, we have
\begin{align}
\nonumber
& \int_{y\in \R^n} \int_{ f_{i-1}(y)> y_{n+1} >f_{i+1}(y) }\frac{\chi_E(Y) - \chi_{E^c}(Y)}{|X-Y|^{n+1+s}} \, dY
\\
\nonumber
& \approx (-1)^{i} \frac{\Delta f_i }{1-s}
+
(-1)^{i-1}
\int_{\R^n}
\int_{2 f_{i}(y) - f_{i+1}(y)}^{f_{i-1}(y)}
\frac{1}{(|y-x|^2+(y_{n+1}-f_i(x))^2)^{\frac{n+1+s}2}}
\, d y_{n+1}
\, d y
\\
\nonumber
&\approx
(-1)^{i} \frac{\Delta f_i }{1-s}
+(-1)^{i-1}
\int_{\R^n}
\int_{2 f_{i}(x) - f_{i+1}(x)}^{f_{i-1}(x)}
\frac{1}{(|y-x|^2+(y_{n+1}-f_i(x))^2)^{\frac{n+1+s}2}}
\, d y_{n+1}
\, d y
\\
\label{la i-1 i+1}
&\approx
(-1)^{i} \frac{\Delta f_i }{1-s}
+(-1)^{i-1}
const\left( -\frac{1}{f_{i-1}(x)-f_{i}(x)} + \frac{1}{f_{i}(x)-f_{i+1}(x)}\right)
\end{align}
The case $f_{i-1} - f_i \leq f_i - f_{i+1}$ leads to the same formula.

We compute
\begin{align}
\nonumber
&\int_{y\in \R^n} \int_{y_{n+1} > f_{i-1}(y)}\frac{\chi_E(Y) - \chi_{E^c}(Y)}{|X-Y|^{n+1+s}} \, dY
\\
\nonumber
&=
\sum_{j=0}^{i-2}
(-1)^j
\int_{\R^n}
\int_{f_{j+1}(y)}^{f_j(y)}
\frac{1}{(|y-x|^2+(y_{n+1}-f_i(x))^2)^{\frac{n+1+s}2}}
\, d y_{n+1}
\, d y
\\
\nonumber
&\approx
\sum_{j=0}^{i-2}
(-1)^j
\int_{\R^n}
\int_{f_{j+1}(x)}^{f_j(x)}
\frac{1}{(|y-x|^2+(y_{n+1}-f_i(x))^2)^{\frac{n+1+s}2}}
\, d y_{n+1}
\, d y
\\
\label{la 0 i-2}
&\approx
const
\sum_{j=0}^{i-2} (-1)^j \left(\frac{1}{f_{j+1}(x)-f_i(x)}- \frac{1}{f_j(x)-f_i(x)} \right)
\end{align}
Similarly
\begin{align}
\nonumber
& \int_{y\in \R^n} \int_{y_{n+1} < f_{i+1}(x)}\frac{\chi_E(Y) - \chi_{E^c}(Y)}{|X-Y|^{n+1+s}} \, dY
\\
\nonumber
&=\sum_{j=i+1}^{k}
(-1)^j
\int_{\R^n}
\int_{f_{j+1}(x)}^{f_j(x)}
\frac{1}{(|y-x|^2+(y_{n+1}-f_i(x))^2)^{\frac{n+1+s}2}}
\, d y_{n+1}
\, d y
\\
\label{la i+1 k}
&\approx
const
\sum_{j=i+1}^{k} (-1)^j \left( -\frac{1}{f_j(x)-f_i(x)} +  \frac{1}{f_{j+1}(x)-f_i(x)}\right)
\end{align}

Adding \eqref{la i-1 i+1}, \eqref{la 0 i-2} and \eqref{la i+1 k} we find
$$
0\approx
(-1)^{i}\frac{\Delta f_i}{1-s}
+2 const (-1)^{i-1} \sum_{j\not=i} \frac{(-1)^{i+j}}{f_i-f_j}
$$
and hence we are lead to
\begin{align}
\label{system}
\Delta f_i = 2 const \ve \sum_{j\not= i} \frac{ (-1)^{i+j+1} }{f_i-f_j} .
\end{align}

\bigskip

To be able to carry out the construction of a solution with multiple layers we need a solution of \eqref{system}, and we show next how to find a certain family. For this we shall work with $\ve=1$, that is, we consider now
\begin{align}
\label{system ve=1}
\Delta f_i = 2 \sum_{j\not= i} \frac{ (-1)^{i+j+1} }{f_i-f_j} .
\end{align}

We look for a solution of the form
\begin{align}
\label{approx}
f_i = a_i f_0,\quad
\Delta f_0 = \frac{1}{f_0} .
\end{align}
Then the $a_i$ have to satisfy
\begin{align}
\label{system ai2}
a_i = 2 \sum_{j\not=i} \frac{(-1)^{i+j+1}}{a_i-a_j}
\end{align}
Note that  $\sum_{i=1}^k f_i$ is harmonic and radially symmetric, so it is constant.
Since $\sum f_i = f_0 \sum a_i$ is a constant we must have $\sum a_i=0$.

A solution of the system \eqref{system ai2} can be obtained by minimization of
$$
E(a_1,\ldots,a_k) = \frac12 \sum_{i=1}^k a_i^2 + \sum_{i,j :i\not= j} (-1)^{i+j}\log(|a_i-a_j|)
$$
subject to
$$
\sum_{i=1}^k a_i=0.
$$
Let
$$
\Lambda = \{ ( a_1,\ldots,a_k) \in \R^k: \; a_1>a_2>\ldots>a_k, \; a_j = - a_{k-j+1} \; \forall j \in \{1,\ldots,k\} \}.
$$
\begin{prop}
The minimum of $E$ over $\Lambda$ is achieved.
\end{prop}
\begin{proof}
Consider a sequence $a^{(i)} \in \Lambda$ such that $E(a^{(i)}) \to \inf_\Lambda E$ as $i\to\infty$. We claim that $a^{(i)}$ remains bounded and
\begin{align}
\label{separation}
\liminf_{i\to\infty} \min_{j=1,\ldots,k-1} a^{(i)}_j - a^{(i)}_{j+1} >0.
\end{align}
To prove these claims,  for $j \in \{1,\ldots,k-1\}$ let $x^{(i)}_j = a^{(i)}_j-a^{(i)}_{j+1}>0$. Then
$$
E(a^{(i)}) = \frac12\sum_{j=1}^k (a^{(i)}_j)^2
- 2\sum_{j=1}^{k-1} \log(x^{(i)}_j) + 2\sum_{j=1}^{k-2}\log(x^{(i)}_j+x^{(i)}_{j+1})+
$$
$$
+\ldots + (-1)^k 2\log( x^{(i)}_1 + \ldots + x^{(i)}_k) .
$$
Consider $l $ an odd integer in $\{ 1,\ldots,k-1\} $ and
$j_l \in \{ 1\ldots k-l \}$.
Then
\begin{align}
\label{ineq E}
\sum_{j=1}^{k-(l+1)} \log(x^{(i)}_j + \ldots + x^{(i)}_{j+l})
\geq
\sum_{j=1,j\not=j_l}^{k-l} \log(x^{(i)}_j + \ldots + x^{(i)}_{j+l-1}) ,
\end{align}
since this is equivalent to
$$
\prod_{j=1}^{k-(l+1)} (x^{(i)}_j + \ldots + x^{(i)}_{j+l})
\geq
\prod_{j=1,j\not=j_l}^{k-l} (x^{(i)}_j + \ldots + x^{(i)}_{j+l-1}) .
$$
The product in the right hand side is always present as a term on the left hand side, while the other terms in the left hand side are positive.

From \eqref{ineq E} we have
$$
- \sum_{j=1}^{k-l} \log(x^{(i)}_j + \ldots + x^{(i)}_{j+l-1})
+ \sum_{j=1}^{k-(l+1)} \log(x^{(i)}_j + \ldots + x^{(i)}_{j+l})
\geq -\log(x^{(i)}_{j_l} + \ldots x^{(i)}_{j_l + l -1}) .
$$
Suppose that $k$ is odd and let $m=k-2$ be the largest odd integer that is less or equal than $k-1$. Then $j_m$ is either 1 or 2 and
\begin{align}
\nonumber
E(a^{(i)})
& \geq \frac12\sum_{j=1}^k (a^{(i)}_j)^2
- 2\log(x^{(i)}_{j_1}) -2 \log(x^{(i)}_{j_3} + x^{(i)}_{j_3+1})
-\ldots
\\
\label{k odd}
& \qquad -2\log(x^{(i)}_{j_m} +  \ldots+ x^{(i)}_{j_m+m-1}) .
\end{align}
The last term is $\log(x^{(i)}_1 + \ldots + x^{(i)}_{k-2} )=\log(a^{(i)}_1 - a^{(i)}_{k-1}) $ or  $\log(x^{(i)}_2 + \ldots + x^{(i)}_{k-1} )=\log(a^{(i)}_2 - a^{(i)}_{k}) $ depending on whether $j_m=1$ or $j_m=2$. In any case both terms are equal by the symmetry.
Then we obtain
$$
C\geq
E(a^{(i)})
\geq \frac12\sum_{j=1}^k (a^{(i)}_j)^2 - (m+1) \log(a^{(i)}_1 - a^{(i)}_{k-1})
$$
and we deduce  that $a^{(i)}$ remains bounded as $i\to\infty$.

In the case that $k$ is even, let $m=k-1$. Then $j_m=1$
\begin{align}
\nonumber
E(a^{(i)})
& \geq \frac12\sum_{j=1}^k (a^{(i)}_j)^2
- 2\log(x^{(i)}_{j_1}) -2 \log(x^{(i)}_{j_3} + x^{(i)}_{j_3+1})
-\ldots
\\
\label{k even}
& \qquad -2\log(x^{(i)}_{1} +  \ldots+ x^{(i)}_{k-1})
\end{align}
and the last term is $\log(a^{(i)}_1-a^{(i)}_k)$. Again from  $E(a^{(i)})\leq C$ we see that $a^{(i)}$ remains bounded as $i\to\infty$.

Using now that   $a^{(i)}$ remains bounded as $i\to\infty$, and \eqref{k odd} or \eqref{k even} we obtain \eqref{separation}.
Once we have established that  $a^{(i)}$ is bounded and \eqref{separation} it is direct that up to subsequence $a^{(i)}$ converges as $i\to\infty$ to a minimizer of $E$ over $\Lambda$.
\end{proof}

\bigskip

There is however a further restriction on a solution $a_i$ to \eqref{system ai2} that we need to impose for our method to work, and it is related to the linearization of the system \eqref{system ve=1} around a solution of the form \eqref{approx} .
Indeed,
the linearized operator around the approximate solution \eqref{approx} is given by
$$
\Delta \phi_i - 2 \sum_{j\not=i}(-1)^{i+j}\frac{\phi_i-\phi_j}{(f_i-f_j)^2} .
$$
Let us write this operator acting on the vector $\Phi=(\phi_1,\ldots,\phi_k)$ as
$$
\Delta \Phi + \frac{1}{f_0^2} A \Phi
$$
where $A=(a_{ij})$ has entries
$$
a_{ij} =
\begin{cases}
2 \frac{(-1)^{i+j}}{(a_i-a_j)^2}   & \text{if } i\not=j
\\
-2 \sum_{k\not=i} \frac{(-1)^{i+k}}{(a_i-a_k)^2}& \text{if } i=j
\end{cases}
$$
Note that $f_0 \sim r$ as $r\to\infty$, so the linearized operator is asymptotic to
$$
\Delta \Phi + \frac{1}{r^2} A \Phi  ,
$$
as $r\to\infty$.

As done before, a natural space to find the solution $\Phi$ should involve norms allowing linear growth. We see that it is possible to find such solutions for a given right hand side of the form $\sim 1/r$ if  the matrix $A$ has no eigenvalue equal to $-1$, since otherwise, $\Phi(r) = r v$ with $v$ an eigenvector of $A$ associated to eigenvalue 1 would be in the kernel of the operator.

We note that
$$
D_{a_i,a_k}^2 E
=
\begin{cases}
2(-1)^{i+k}\frac{1}{(a_i-a_k)^2} & \text{if } i\not=k
\\
1-2\sum_{j\not=i} (-1)^{i+j}\frac{1}{(a_i-a_j)^2} & \text{if } i=k,
\end{cases}
$$
so that
$$
D^2 E = I + A  .
$$
At a local minimum of $E$, $D^2 E\geq 0$ which means that eigenvalues of $A$ are greater or equal than $-1$. If $(a_i,\ldots,a_k)$ is a non degenerate local minimum of $E$ then $D^2 E>0$ and the eigenvalues of $A$ are greater than $-1$.

\section{Existence of $s$-Lawson cones}
\label{sect exist unique}

\begin{proof}[Proof of Theorem~\ref{teo3}]

\

Let us write
\begin{align}
\label{def solid cone}
E_\alpha  = \{ x = (y,z) \, : \ y \in \R^m, \, z \in \R^n, \, |z| > \alpha |y| \ \} ,
\end{align}
so that $C_\alpha = \partial E_\alpha$.

\medskip

\noindent{\bf Existence.}
We fix $N$, $m$, $n$ with $N=m+n$, $n\leq m$ and also fix $0<s<1$.
If $m=n$ then $C_1$ is a minimal cone, since \eqref{1} is satisfied by symmetry. So we concentrate next on the case $n<m$.

Before proceeding we remark that for a cone $C_\alpha$ the quantity appearing in \eqref{1} has a fixed sign for all $p\in C_\alpha$, $p\not=0$, since by rotation we can always assume that $p = r p_\alpha$ for some $r>0$ where
\begin{align*}
p_{\alpha} = \frac{1}{\sqrt{1+\alpha^2}}(e_1^{(m)}, \alpha e_1^{(n)} )
\end{align*}
with
\begin{align}
\label{notation e}
e_1^{(m)} = (1,0,\ldots,0) \in \R^m
\end{align}
and similarly for $ e_1^{(n)}$. Then we observe that
$$
\text{p.v.}
\int_{\R^N} \frac{\chi_{E_\alpha} (x)-\chi_{E_\alpha^c}(x)} {|x-r p_{\alpha}|^{N+s}} \, d x =
\frac{1}{r^s}
\text{p.v.}
\int_{\R^N} \frac{\chi_{E_\alpha}  (x)-\chi_{E_\alpha^c} (x)} {|x-p_{\alpha}|^{N+s}} \, d x .
$$
Let us define
\begin{align}
\label{def Halpha}
H(\alpha )
=
\text{p.v.}
\int_{\R^N} \frac{\chi_{E_\alpha} (x)-\chi_{E_\alpha^c}(x)} {|x-p_{\alpha}|^{N+s}} \, d x
\end{align}
and note that it is a continuous function of $\alpha \in (0,\infty)$.

\medskip
\noindent{\bf Claim 1.} We have
\begin{align}
\label{h1}
H(1) \leq 0.
\end{align}

Indeed, write $y \in \R^m$ as $y = (y_1,y_2)$ with $y_1 \in \R^n$ and $y_2\in\R^{m-n}$.
Abbreviating $e_1 = e_1^{(n)} = (1,0,\ldots,0) \in \R^n$ we rewrite
\begin{align*}
H(1)
&= \lim_{\delta\to0}
\int_{\R^N \setminus B(p_{1},\delta)} \frac{\chi_{E_1} (x)-\chi_{E_1^c}(x)} {|x-p_{1}|^{N+s}} \, d x
\\
&
=
\lim_{\delta\to0}
\int_{A_\delta} \frac{1}{( |y_1 - \frac1{\sqrt2}e_1|^2 + |y_2|^2 + |z-\frac1{\sqrt2}e_1|^2)^{\frac{N+s}{2}} }
\\
&\qquad -
 \lim_{\delta\to0}
\int_{B_\delta} \frac{1}{( |y_1 - \frac1{\sqrt2}e_1|^2 + |y_2|^2 + |z-\frac1{\sqrt2}e_1|^2)^{\frac{N+s}{2}} } ,
\end{align*}
where
\begin{align*}
A_\delta
& = \{ |z|^2>|y_1 |^2 + |y_2|^2, \
|y_1 - \frac1{\sqrt2}e_1|^2 + |y_2|^2 + |z-\frac1{\sqrt2}e_1|^2>\delta^2\}
\\
B_\delta
&= \{ |z|^2<|y_1 |^2 + |y_2|^2,  \
|y_1 - \frac1{\sqrt2}e_1|^2 + |y_2|^2 + |z-\frac1{\sqrt2}e_1|^2>\delta^2 \} .
\end{align*}
But the first integral can be rewritten as
\begin{align*}
& \int_{A_\delta} \frac{1}{( |y_1 - \frac1{\sqrt2}e_1|^2 + |y_2|^2 + |z-\frac1{\sqrt2}e_1|^2)^{\frac{N+s}{2}} }
\\
& =
\int_{\tilde A_\delta} \frac{1}{( |y_1 - \frac1{\sqrt2}e_1|^2 + |y_2|^2 + |z-\frac1{\sqrt2}e_1|^2)^{\frac{N+s}{2}} }
\end{align*}
where
$$
\tilde A_\delta = \{
|y_1|^2>|z |^2 + |y_2|^2, \
|y_1 - \frac1{\sqrt2}e_1|^2 + |y_2|^2 + |z-\frac1{\sqrt2}e_1|^2>\delta^2\}
$$
(we just have exchanged $y_1$ by $z$ and noted that the integrand is symmetric in these variables).
But $\tilde A_\delta \subset B_\delta$ and so
\begin{align*}
& \int_{\R^N \setminus B(p_{1},\delta)} \frac{\chi_{E_1} (x)-\chi_{E_1^c}(x)} {|x-p_{1}|^{N+s}} \, d x
\\
& =
- \int_{B_\delta \setminus \tilde A_\delta}
\frac{1}{( |y_1 - \frac1{\sqrt2}e_1|^2 + |y_2|^2 + |z-\frac1{\sqrt2}e_1|^2)^{\frac{N+s}{2}} }
\leq 0.
\end{align*}
This shows the validity of \eqref{h1}.

\medskip
\noindent{\bf Claim 2.} We have
\begin{align}
\label{limit H 0}
H(\alpha) \to+\infty \quad\text{as }\alpha\to 0.
\end{align}

Let $0<\delta<1/2$ be fixed and write
$$
H(\alpha) = I_\alpha + J_\alpha
$$
where
\begin{align*}
I_\alpha
&=\int_{\R^N \setminus B(p_{\alpha},\delta)} \frac{\chi_{E_\alpha} (x)-\chi_{E_\alpha^c}(x)} {|x-p_{\alpha}|^{N+s}} \, d x
\\
J_\alpha &=
\text{p.v.} \int_{B(p_{\alpha},\delta)} \frac{\chi_{E_\alpha} (x)-\chi_{E_\alpha^c}(x)} {|x-p_{\alpha}|^{N+s}} \, d x .
\end{align*}
With $\delta$ fixed
\begin{align}
\label{limit I}
\lim_{\alpha\to 0} I_\alpha
= \int_{\R^N \setminus B(p_{\alpha},\delta)} \frac{1} {|x-p_{0}|^{N+s}} \, d x  >0.
\end{align}
For $J_\alpha$ we make a change of variables $x = \alpha \tilde x + p_\alpha$ and obtain
\begin{align}
\label{scaling J}
J_\alpha
=
\text{p.v.} \int_{B(p_{\alpha},\delta)} \frac{\chi_{E_\alpha} (x)-\chi_{E_\alpha^c}(x)} {|x-p_{\alpha}|^{N+s}} \, d x
=
\frac{1}{\alpha^s}
\text{p.v.}
\int_{B(0,\delta/\alpha)}
\frac{\chi_{F_\alpha}(\tilde x) -\chi_{F_\alpha^c} (\tilde x)}{|\tilde x|^{N+s}} d \tilde x
\end{align}
where $F_\alpha = \frac1\alpha ( E_\alpha-p_\alpha)$. But
$$
\text{p.v.}
\int_{B(0,\delta/\alpha)}
\frac{\chi_{F_\alpha}(\tilde x) -\chi_{F_\alpha^c} (\tilde x)}{|\tilde x|^{N+s}} d \tilde x
\to
\text{p.v}
\int_{\R^N}
\frac{\chi_{F_0}(x) -\chi_{F_0^c} (x)}{|x|^{N+s}} dx
$$
as $\alpha\to 0$ where
$F_0 = \{ x=(y,z) : y\in \R^m, z\in \R^n, \ |z+e_1^{(n)}|>1 \} $.
But writing $z=(z_1,\ldots,z_n)$ we see that
\begin{align*}
\text{p.v}
\int_{\R^N}
\frac{\chi_{F_0}(x) -\chi_{F_0^c} (x)}{|x|^{N+s}} d x
&\geq
\text{p.v}
\int_{\R^N}
\frac{\chi_{[z_1>0 \text{ or } z_1<-2]} -\chi_{[-2<z_1<0]}}{|x|^{N+s}} dx
\\
&\geq
\int_{\R^N}
\frac{\chi_{ [ \ |z_1| >2 \ ]} }{|x|^{N+s}}  d x
\end{align*}
and this number is positive. This and \eqref{scaling J} show that $J_\alpha\to+\infty$ as $\alpha\to0$ and combined with \eqref{limit I} we obtain the desired conclusion.


By  \eqref{h1}, \eqref{limit H 0} and continuity we obtain the existence of $\alpha \in (0,1]$ such that $H(\alpha)=0$.

\bigskip
\noindent
{\bf Uniqueness.}
Consider 2 cones $C_{\alpha_1}$, $C_{\alpha_2}$ with $\alpha_1>\alpha_2>0$, associated to solid cones $E_{\alpha_1}$ and $E_{\alpha_2}$.
We claim that there is a rotation $R$ so that $R(E_{\alpha_1}) \subset E_{\alpha_2}$ (strictly) and that
$$
H(\alpha_1) =
\text{p.v.}
\int_{\R^N}
\int_{\R^N} \frac{\chi_{R(E_{\alpha_1})} (x)-\chi_{R(E_{\alpha_1})^c}(x)} {|x-p_{\alpha_2}|^{N+s}} \, d x .
$$
Note that the denominator in the integrand is the same that appears in \eqref{def Halpha} for $\alpha_2$ and then
\begin{align}
\nonumber
H(\alpha_1)
&=
\text{p.v.}
\int_{\R^N}
\int_{\R^N} \frac{\chi_{R(E_{\alpha_1})} (x)-\chi_{R(E_{\alpha_1})^c}(x)} {|x-p_{\alpha_2}|^{N+s}} \, d x
\\
\label{strict ineq}
&<
\text{p.v.}
\int_{\R^N}
\int_{\R^N} \frac{\chi_{E_{\alpha_2}} (x)-\chi_{E_{\alpha_2}^c}(x)} {|x-p_{\alpha_2}|^{N+s}} \, d x = H(\alpha_2) .
\end{align}
This shows that $H(\alpha)$ is decreasing in $\alpha$ and hence the uniqueness.
To construct the rotation let us write as before $x = (y,z) \in \R^N$, with $y\in\R^m$, $z\in \R^n$, and $y = (y_1,y_2)$ with $y_1\in\R^n$, $y_2 \in \R^{m-n}$ (we assume alway $n\leq m$). Let us write the vector $(y_1,z)$ in spherical coordinates of $\R^{2n}$ as follows
$$
y_1
= \rho
\left[
\begin{matrix}
\cos(\varphi_1)\\
\sin(\varphi_1) \cos(\varphi_2)\\
\sin(\varphi_1) \sin(\varphi_2) \cos(\varphi_3)\\
\vdots\\
\sin(\varphi_1) \sin(\varphi_2) \sin(\varphi_3)\ldots \sin(\varphi_{n-1}) \cos(\varphi_n)\\
\end{matrix}
\right]
$$
$$
z
= \rho
\left[
\begin{matrix}
\sin(\varphi_1) \sin(\varphi_2) \sin(\varphi_3)\ldots \sin(\varphi_{n}) \cos(\varphi_{n+1})\\
\vdots\\
\sin(\varphi_1) \sin(\varphi_2) \sin(\varphi_3)\ldots \sin(\varphi_{2n-2}) \cos(\varphi_{2n-1})\\
\sin(\varphi_1) \sin(\varphi_2) \sin(\varphi_3)\ldots \sin(\varphi_{2n-2}) \sin(\varphi_{2n-1})\\
\end{matrix}
\right]
$$
where $\rho>0$, $\varphi_{2n-1}\in[0,2\pi)$, $\varphi_j\in[0,\pi]$ for $j=1,\ldots,2n-2$. Then
$$
|z|^2 = \rho^2 \sin(\varphi_1)^2 \sin(\varphi_2)^2 \ldots \sin(\varphi_n)^2
,
\quad |y_1|^2+|z|^2 = \rho^2.
$$
The equation for the solid cone $E_{\alpha_i}$, namely $|z|>\alpha_i |y|$, can be rewritten as
$$
\rho^2 \sin(\varphi_1)^2 \sin(\varphi_2)^2 \ldots \sin(\varphi_n)^2
> \alpha_i^2 (|y_1|^2 + |y_2|^2).
$$
Adding $\alpha_i^2 |z|^2$ to both sides this is equivalent to
$$
\sin(\varphi_1)^2 \sin(\varphi_2)^2 \ldots \sin(\varphi_n)^2
> \sin(\beta_i)^2
( 1+\frac{ |y_2|^2}{\rho^2})
$$
where $\beta_i = \arctan(\alpha_i)$.
We let $\theta = \beta_1 - \beta_2 \in (0,\pi/2)$,
and define the rotated cone $R_\theta(E_{\alpha_1})$ by the equation
$$
\sin(\varphi_1+\theta)^2 \sin(\varphi_2)^2 \ldots \sin(\varphi_n)^2
> \sin(\beta_1)^2
( 1+\frac{ |y_2|^2}{\rho^2}) .
$$
We want to show that $R_\theta(E_{\alpha_1}) \subset E_{\alpha_2}$. To do so, it suffices to prove that for any given $t\geq 1$, if $\varphi$ satisfies the inequality $ |\sin(\varphi+\theta)|> \sin(\beta_1) t $
then it also satisfies $ |\sin(\varphi)|> \sin(\beta_2) t $.  This in turn can be proved from the inequality
\begin{align*}
\arccos(\sin(\beta_1) t) + \theta < \arccos(\sin(\beta_2) t)
\end{align*}
for $ 1< t \leq \frac{1}{\sin(\beta_1)}$. For $t=1$ we have equality by definition of $\theta$. The inequality for $ 1< t \leq \frac{1}{\sin(\beta_1)}$  can be checked by computing a derivative with respect to $t$. The strict inequality in \eqref{strict ineq} is because $R(E_{\alpha_1}) \subset E_{\alpha_2}$ strictly.
\end{proof}

\section{Stability and instability}
\label{sect stability}

We consider the nonlocal minimal cone $C_m^n(s) =  \partial E_\alpha$
where $E_\alpha$ is defined in \eqref{def solid cone} and $\alpha$ is the one of Theorem~\ref{teo3}.
For $0 \leq s <1$ we obtain a characterization of
their stability in terms of constants that depend on   $m$, $n$ and $s$. For the case $s=0$ we consider the limiting cone with parameter $\alpha_0$ given in Proposition~\ref{prop alpha s=0} below.
Note that in the case $s=0$ the limiting Jacobi operator $\mathcal J^0_{C_{\alpha_0}}$ is well defined for smooth functions with compact support.

For brevity, in this section we write $\Sigma = C_m^n(s)$.

\subsection{Characterization of stability}

Recall that
$$
\mathcal J^s_{\Sigma}[\phi](x)
=
\text{p.v.}
\int_{\Sigma}
\frac{\phi(y) - \phi(x)}{|y-x|^{N+s}} dy
+ \phi(x) \int_{\Sigma}
\frac{1-\langle \nu(x),\nu(y)\rangle}{|x-y|^{N+s}} dy
$$
for $\phi \in C_0^\infty(\Sigma\setminus\{0\})$.
Let us rewrite this operator in the form
$$
\mathcal J^s_{\Sigma}[\phi](x)
=
\text{p.v.}
\int_{\Sigma}
\frac{\phi(y) - \phi(x)}{|x-y|^{N+s}} dy
+ \frac{A_0(m,n,s)^2}{|x|^{1+s}}
\phi(x)
$$
where
\begin{align*}
A_0(m,n,s)^2 = \int_{\Sigma}
\frac{\langle \nu (\hat p)- \nu(x), \nu (\hat p) \rangle }{|\hat p-x|^{N+s}} dx \geq0
\end{align*}
and this integral is evaluated at any $\hat p\in {\Sigma}$ with $|\hat p|=1$.
We can think of $\mathcal J^s_{\Sigma}$ as analogous to the fractional Hardy operator
$$
-(-\Delta)^{\frac{1+s}2} \phi
+ \frac c{|x|^{1+s}} \phi
\quad \text{in } \R^{N-1} ,
$$
for which positivity is related to a fractional Hardy inequality
with best constant, see Herbst \cite{herbst}.
This suggests that the positivity of $\mathcal J_\Sigma$ is related to the existence of $\beta$ in an appropriate range such that
$\mathcal J^s_{\Sigma}[|x|^{-\beta}]\leq 0$, and it turns out that the best choice of $\beta$ is $\beta = \frac{N-2-s}{2}$.
This motivates the definition
$$
H(m,n,s) =
\text{p.v.}
\int_{\Sigma}
\frac{1 - |y|^{-\frac{N-2-s}{2}}}{|\hat p-y|^{N+s}} dy
$$
where $\hat p\in \Sigma$ is any point with $|\hat p|=1$.



We have then the following Hardy inequality with best constant:
\begin{prop}
\label{prop hardy ineq}
For any $\phi \in C_0^\infty(\Sigma\setminus\{0\})$ we have
\begin{align}
\label{fract hardy}
H(m,n,s)
\int_{\Sigma}
\frac{\phi(x)^2}{|x|^{1+s}} dx
\leq
\frac{1}{2}
\int_{\Sigma}
\int_{\Sigma}
\frac{(\phi(x)-\phi(y) ) ^2}{|x-y|^{N+s}} d x dy
\end{align}
and $H(m,n,s)$ is the best possible constant in this inequality.
\end{prop}

As a result we have:
\begin{corollary} The cone $C_m^n(s)$ is stable if and only if
$H (m,n,s) \geq A_0(m,n,s)^2$.
\end{corollary}

Other related fractional Hardy inequalities have appeared in the literature, see for instance \cite{bogdan-dyda,dyda-frank}.

\medskip
\noindent
\begin{proof}[Proof of Proposition~\ref{prop hardy ineq}]
Let us write  $H = H(m,n,s)$ for simplicity.
To prove the validity of \eqref{fract hardy}
let $w(x) = |x|^{-\beta}$ with $\beta = \frac{N-2-s}{2}$ so that from the definition of $H$ and homogeneity we have
\begin{align*}
\text{p.v.}
\int_{\Sigma} \frac{w(y) - w(x)}{|y-x|^{N+s} } d y
+ \frac{H}{|x|^{1+s}} w(x) = 0
\quad \text{for all } x \in\Sigma\setminus \{0\}.
\end{align*}
Now the same argument as in the proof of corollary~\ref{coro entire stable} shows that
\begin{align}
\label{equiv form}
\frac12
\int_{\Sigma}\int_{\Sigma}
\frac{(\phi(x) - \phi(y))^2}{|x-y|^{N+s} } d x d y
& =
\int_{\Sigma}\frac{H}{|x|^{1+s}}\phi(x)^2 dx
\\
\nonumber
& \quad +
\frac12
\int_{\Sigma}\int_{\Sigma}
\frac{(\psi(x) - \psi(y) )^2 w(x) w(y) }{|x-y|^{N+s} } d x d y .
\end{align}
for all $\phi \in C_0^\infty(\Sigma\setminus\{0\})$  with $\psi = \frac\phi w \in C_0^\infty(\Sigma\setminus\{0\})$

Now let us show that $H$ is the best possible constant in \eqref{fract hardy}. Assume that
$$
\tilde H
\int_{\Sigma}
\frac{\phi(x)^2}{|x|^{1+s}} dx
\leq
\frac{1}{2}
\int_{\Sigma}
\int_{\Sigma}
\frac{(\phi(x)-\phi(y) ) ^2}{|x-y|^{N+s}} d x dy
$$ for all
$\phi \in C_0^\infty(\Sigma \setminus\{0\})$.
Using \eqref{equiv form} and letting $\phi = w \psi$ with $\psi\in \in C_0^\infty(\Sigma \setminus\{0\})$ we then have
\begin{align*}
\tilde H
\int_{\Sigma}
\frac{w(x)^2 \psi(x)^2}{|x|^{1+s}} dx
& \leq
H \int_{\Sigma}
\frac{w(x)^2 \psi(x)^2}{|x|^{1+s}} dx
\\
& \qquad +\frac12
\int_{\Sigma}\int_{\Sigma}
\frac{(\psi(x) - \psi(y) )^2 w(x) w(y) }{|x-y|^{N+s} } d x d y .
\end{align*}
For $R>3$ let $\psi_R:\Sigma\to [0,1]$ be a radial function such that $\psi_R(x)=0$ for $|x|\leq 1$, $\psi_R(x)=1$ for $2\leq |x|\leq 2 R$, $\psi_R(x)=0$ for $|x|\geq 3 R$. We also require $|\nabla\psi_R(x)|\leq C$ for $|x|\leq 3$, $|\nabla\psi_R(x)|\leq C/R$ for $2 R\leq |x|\leq 3 R$. We claim that
\begin{align}
\label{ineq 1}
a_0 \log(R) - C
\leq
\int_{\Sigma}
\frac{w(x)^2 \psi_R(x)^2}{|x|^{1+s}} dx \leq a_0 \log(R) +C
\end{align}
where $a_0>0$, $C>0$ are independent of $R$, while
\begin{align}
\label{ineq 2}
\left|
\int_{\Sigma}\int_{\Sigma}
\frac{(\psi_R(x) - \psi_R(y) )^2 w(x) w(y) }{|x-y|^{N+s} } d x d y
\right|\leq C .
\end{align}
Letting then $R\to \infty$ we deduce that $\tilde H \leq H$.

To prove the upper bound in \eqref{ineq 1} let us write points in $\Sigma$ as $x = (y,z) $, with $y \in \R^m$, $z\in\R^n$.
Let us write
$ y = r \omega_1 $, $z = r\omega_2$, with $r>0$,  $\omega_1 \in S^{m-1}$, $\omega_2 \in S^{n-1}$ and use spherical coordinates $(\theta_1,\ldots,\theta_{m-1})$ and $(\varphi_1,\ldots,\varphi_{n-1})$ for $\omega_1$ and $\omega_2$ as in \eqref{omega1} and \eqref{omega2} . We assume here that $m\geq n \geq 2 $. In the remaining cases the computations are similar. Then we have
$$
\int_{\Sigma}
\frac{w(x)^2 \psi_R(x)^2}{|x|^{1+s}} dx \leq
a_0
\int_1^{4 R} \frac{1}{r^{N-2-s}} \frac1{r^{1+s}} r^{N-2} dr
\leq a_0 \log(R) + C
$$
where
$$
a_0 = \sqrt{1+\alpha^2} A_{m-1} A_{n-1}
$$
and $A_k$ denotes the area of the sphere $S^k \subseteq \R^{k+1}$ and is given by
\begin{align}
\label{Am}
A_{k} = \frac{2\pi^{\frac{k+1}2}}{\Gamma(\frac{k+1}2)}.
\end{align}
The lower bound in \eqref{ineq 1} is similar.

To obtain \eqref{ineq 2} we split $\Sigma$ into the regions $R_1 = \{x: |x|\leq 3\}$, $R_2 = \{x : 3\leq x \leq R\}$, $R_3=\{x: R\leq |x|\leq4R\}$ and $R_4 = \{x:|x|\geq 4R\}$ and let
$$
I_{i,j} = \int_{x\in R_i} \int_{y\in R_j}
\frac{(\psi_R(x) - \psi_R(y) )^2 w(x) w(y) }{|x-y|^{N+s} } d x d y .
$$
Then $I_{i,j} = I_{j,i}$ and $I_{j,j}=0$ for $j=2,4$. Moreoover $I_{1,1}=O(1)$ since the region of integration is bounded and $\psi_R$ is uniformly Lipschitz.

Estimate of $I_{1,2}$: We bound $w(x)\leq C$ for $|x|\geq 1$ and then
\begin{align*}
|I_{1,2}|\leq C  \int_{y\in R_2} \frac{w(y)}{|p-y|^{N+s}} dy
\leq C \int_2^R \frac{1}{r^{\frac{N-2-s}{2}}} \frac1{r^{N+s}} r^{N-2} dr \leq C,
\end{align*}
where  $p\in\Sigma$ is fixed with $|p|=2$.

By the same argument  $I_{1,3}=O(1) $ and $I_{1,4}=O(1)$ as $R\to\infty$.

Estimate of $I_{2,3}$: for $y \in R_3$, $w(y)\leq C R^{-\frac{N-2-s}2}$, so
\begin{align*}
|I_{2,3}|
&\leq
C R^{-\frac{N-2-s}2}
\int_{x\in R_2} \frac{1}{|x|^{\frac{N-2-s}2}}
\int_{y\in R_3}
\frac{(\psi_R(x)-\psi_R(y))^2}{|x-y|^{N+s}} dy dx
\\
&
\leq
C R^{-\frac{N-2-s}2}
\frac{Vol(R_3)}{R^{N+s}} \int_{x\in R_2}
\frac{1}{|x|^{\frac{N-2-s}2}}  dx \leq C.
\end{align*}

Estimate of $I_{2,4}$:
\begin{align*}
|I_{2,4}|
\leq C
\int_{x\in R_2}
\frac{1}{|x|^{\frac{N-2-s}2}}
\int_{y\in R_4}
\frac{1}{|x-y|^{N+s}}
\frac{1}{|y|^{\frac{N-2-s}{2}}}
dy
dx .
\end{align*}
By scaling
\begin{align*}
\int_{y\in R_4}
\frac{1}{|x-y|^{N+s}}
\frac{1}{|y|^{\frac{N-2-s}{2}}}
dy \leq C R^{-\frac N2  - \frac s2}
\quad \text{for } x \in R_2 ,
\end{align*}
so that
\begin{align*}
|I_{2,4}|
\leq C R^{-\frac N2  - \frac s2}
\int_{x\in R_2}
\frac{1}{|x|^{\frac{N-2-s}2}}
dx \leq C  .
\end{align*}

To estimate $I_{3,3}$ we use $|\psi_R(x) -\psi_R(y)|\leq \frac CR |x-y$ for $x,y \in R_3$, which yields
\begin{align*}
|I_{3,3}|
&\leq \frac{C}{R^2} \frac{1}{R^{N-2-s}} \int_{x,y\in R_3} \frac{1}{|x-y|^{N+s-2}}
dy dx .
\end{align*}
The integral is finite and by scaling we see that is bounded by $C R^{N-s}$, so that
$$
|I_{3,3}|\leq C.
$$

Estimate of $I_{3,4}$:
\begin{align*}
|I_{3,4}|
\leq C R^{-\frac{N-2-s}2}
\int_{x\in R_3}
\int_{y\in R_4}
\frac{1}{|x-y|^{N+s}} \frac{1}{|y|^{\frac{N-2-s}{2}}} d y d x .
\end{align*}
By scaling
$$
\int_{y\in R_4}
\frac{1}{|x-y|^{N+s}} \frac{1}{|y|^{\frac{N-2-s}{2}}} d y
\leq \frac{C}{|x|^{\frac{N+s}2}}
$$
for $x\in R_3$. Therefore
$$
|I_{3,4}|
\leq C R^{-\frac{N-2-s}2}
\int_{x\in R_3} \frac{1}{|x|^{\frac{N+s}2}} dx
\leq C.
$$
This concludes the proof of \eqref{ineq 2}.
\end{proof}

\subsection{Minimal cones for $s=0$}

Here we derive the limiting value $\alpha_0 = \lim_{s\to0} \alpha_s$ where $\alpha_s$ is such that $C_{\alpha_s}$ is an $s$-minimal cone.

\begin{prop}
\label{prop alpha s=0}
Assume that $n\leq m$ in \eqref{def solid cone}, $N=m+n$.
The number $\alpha_0$ is the unique solution to
$$
\int_{\alpha}^\infty
\frac{t^{n-1}}
{(1 + t^2  )^\frac{N}2}
d t - \int_0^{\alpha}
\frac{t^{n-1}}
{(1 + t^2  )^\frac{N}2}
d t
=0 .
$$
\end{prop}
\begin{proof}
We write $x = (y,z)\in\R^N$ with $y\in \R^m$, $z\in \R^n$.
Let us assume in the rest of the proof that $n\geq 2$. The case $n=1$ is similar.
We evaluate the integral in \eqref{1} for the point $p=(e_1^{(m)},\alpha e_1^{(n)})$
using spherical coordinates for $y = r \omega_1$ and $z = \rho \omega_2$
where $r,\rho>0$ and
\begin{align}
\label{omega1}
\omega_1 = \left[
\begin{matrix}
\cos(\theta_1)\\
\sin(\theta_1) \cos(\theta_2)\\
\vdots\\
\sin(\theta_1) \sin(\theta_2) \ldots \sin(\theta_{m-2}) \cos(\theta_{m-1})\\
\sin(\theta_1) \sin(\theta_2) \ldots
\sin(\theta_{m-2}) \sin(\theta_{m-1})
\end{matrix}
\right]
\end{align}
\begin{align}
\label{omega2}
\omega_2 = \left[
\begin{matrix}
\cos(\varphi_1)\\
\sin(\varphi_1) \cos(\varphi_2)\\
\vdots\\
\sin(\varphi_1) \sin(\varphi_2) \ldots \sin(\varphi_{n-2}) \cos(\varphi_{n-1})\\
\sin(\varphi_1) \sin(\varphi_2) \ldots
\sin(\varphi_{n-2}) \sin(\varphi_{n-1})
\end{matrix}
\right] ,
\end{align}
where $\theta_j\in [0,\pi]$ for $j=1,\ldots,m-2$, $\theta_{m-1} \in [0,2\pi]$, $\varphi_j\in [0,\pi]$ for $j=1,\ldots,n-2$, $\varphi_{n-1} \in [0,2\pi]$.
Then
$$
|(y,z)-(e_1^{(m)},\alpha e_1^{(n)})|^2
=r^2 + 1 - 2 r \cos(\theta_1) + \rho^2 + \alpha^2 -2\rho\alpha\cos(\varphi_1) .
$$
Assuming that $\alpha = \alpha_s>0$ is such that $C_{\alpha_s}$ is an $s$-minimal cone, \eqref{1} yields the following equation for $\alpha$
\begin{align}
\label{eq alpha}
\text{p.v.}
\int_0^\infty
r^{m-1}
( A_{\alpha,s}(r) - B_{\alpha,s}(r) ) d r = 0
\end{align}
where
\begin{align*}
A_{\alpha,s}(r)
&=
\int_{r\alpha}^\infty
\int_0^\pi
\int_0^\pi
\frac{\rho^{n-1}
\sin(\theta_1)^{m-2}
\sin(\varphi_1)^{n-2}
}{(r^2 + 1 - 2 r \cos(\theta_1) + \rho^2 + \alpha^2 -2\rho\alpha\cos(\varphi_1) )^\frac{N+s}2}
d\theta_1
d\varphi_1
d\rho
\\
 B_{\alpha,s}(r)
 &=
\int_0^{r\alpha}
\int_0^\pi
\int_0^\pi
\frac{\rho^{n-1}
\sin(\theta_1)^{m-2}
\sin(\varphi_1)^{n-2}
}{(r^2 + 1 - 2 r \cos(\theta_1) + \rho^2 + \alpha^2 -2\rho\alpha\cos(\varphi_1) )^\frac{N+s}2}
d\theta_1
d\varphi_1
d\rho ,
\end{align*}
which are well defined for $r\not=1$.
Setting $\rho = r t$ we get
\begin{align*}
& A_{\alpha,s}(r)
\\
& =
r^{-m-s}
\int_{\alpha}^\infty
\int_0^\pi
\int_0^\pi
\frac{t^{n-1}\sin(\varphi_1)^{m-2}
\sin(\theta_1)^{n-2}}
{(1 + \frac1{r^2} - \frac2r  \cos(\theta_1) + t^2 + \frac{\alpha^2}{r^2} - \frac2r t \alpha\cos(\varphi_1) )^\frac{N+s}2}
d\theta_1
d\varphi_1
d t
\\
& =
c_{m,n}
r^{-m-s}
\int_{\alpha}^\infty
\frac{t^{n-1}}
{(1 + t^2  )^\frac{N+s}2}
d t + O(r^{-m-s-1})
\end{align*}
as $r\to\infty$ and this is uniform in $s$ for $s>0$ small.
Here $c_{m,n}>0$ is some constant.
Similarly
\begin{align*}
B_{\alpha,s}(r) &=
c_{m,n}
r^{-m-s}
\int_0^{\alpha}
\frac{t^{n-1}}
{(1 + t^2  )^\frac{N+s}2}
d t + O(r^{-m-s-1})
\end{align*}
Then \eqref{eq alpha} takes the form
\begin{align*}
0&=\int_0^2 \ldots dr + \int_2^\infty \ldots dr
= O(1) + C_s(\alpha) \int_2^\infty r^{-1-s} dr
=
O(1) + \frac{2^{-s}}s C_s(\alpha)
\end{align*}
where
$$
C_s(\alpha) = \int_{\alpha}^\infty
\frac{t^{n-1}}
{(1 + t^2  )^\frac{N+s}2}
d t - \int_0^{\alpha}
\frac{t^{n-1}}
{(1 + t^2  )^\frac{N+s}2}
d t
$$
and $O(1)$ is uniform as $s\to0$, because $0<\alpha_s \leq 1$ by Theorem~\ref{teo3},
and the only singularity in \eqref{eq alpha} occurs at $r = 1$.
This implies that $\alpha_0 = \lim_{s\to0} \alpha_s$ has to satisfy $C_0(\alpha_0)=0$.
\end{proof}

\subsection{Proof of Theorem~\ref{thm stability}}

\noindent
In what follows we will obtain expressions for $H(m,n,s)$ and $A_0(m,n,s)^2$ for $m\geq 2$, $n\geq 1$, $0\leq s <1$. We always assume $m\geq n$.
For the sake of generality, we will compute
$$
C(m,n,s,\beta) =
\text{p.v.}
\int_{\Sigma}
\frac{1 - |x|^{-\beta}}{|\hat p-x|^{N+s}} dx
$$
where $\hat p\in \Sigma$, $|\hat p|=1$, and $\beta \in (0,N-2-s)$, so that $H(m,n,s) = C(m,n,s,\frac{N-2-s}{2})$.

Let $x = (y,z) \in \Sigma$, with $y \in \R^m$, $z\in\R^n$.
For simplicity in the next formulas we take $p = (e_1^{(m)},\alpha e_2^{(n)})$ (see the notation in \eqref{notation e}), and $h(y,z) = |y|^{-\beta}$,
so that
$$
C(m,n,s,\beta)
=
(1+\alpha^2)^{\frac{1+s}2}
\text{p.v.}
\int_{\Sigma}
\frac{h(p) - h(x)}{|p-x|^{N+s}} \, d x .
$$

\medskip

\noindent
{\bf Computation of $C(m,1,s,\beta)$.}
Write
$ y = r \omega_1 $, $z = \pm \alpha r$, with $r>0$,  $\omega_1 \in S^{m-1}$.
Let us use the notation $\Sigma_{\alpha}^+ = \Sigma \cap [ z>0]$, $\Sigma_{\alpha}^- = \Sigma \cap [ z<0]$.
Using polar coordinates $(\theta_1,\ldots,\theta_{m-1})$
for $\omega_1$  as in \eqref{omega1}  we have
$$
|x-p|^2 = |r\theta_1-e_1^{(m)}|^2 + \alpha^2 |r\theta_1-e_1^{(m)}|^2
=  r^2 + 1 - 2 r \cos(\theta_1)  + \alpha^2 (r-1)^2,
$$
for $x\in \Sigma_{\alpha}^+$ and
$$
|x-p|^2 = |r\theta_1-e_1^{(m)}|^2 + \alpha^2 |r\theta_1-e_1^{(m)}|^2
=  r^2 + 1 - 2 r \cos(\theta_1)  + \alpha^2 (r+1)^2,
$$
for $x\in \Sigma_{\alpha}^-$.
Hence, with $h(y,z) = |y|^{-\beta}$
\begin{align}
\label{int n1 h}
\text{p.v.}
\int_{\Sigma}
\frac{h(p)- h(x)}{|x-p|^{N+s}} dx
=
\sqrt{1+\alpha^2}
A_{m-2}
\text{p.v.}
\int_0^\infty (1-r^{-\beta})
( I_+(r) + I_-(r) )
r^{N-2}
dr
\end{align}
where
\begin{align*}
I_+(r) & =
\int_0^\pi
\frac{\sin(\theta_1)^{m-2} }
{(r^2 + 1 - 2 r \cos(\theta_1)  + \alpha^2 (r-1)^2)^{\frac{N+s}2}}
d\theta_1
\\
I_-(r) & =
\frac{\sin(\theta_1)^{m-2} }
{(r^2 + 1 - 2 r \cos(\theta_1)  + \alpha^2 (r+1)^2)^{\frac{N+s}2}}
d\theta_1 ,
\end{align*}
and $A_{m-2}$ is defined in \eqref{Am} for $m\geq 2$.
From \eqref{int n1 h}
we obtain
\begin{align}
\label{Cm1sbeta}
C(m,1,s,\beta)
=
(1+\alpha^2)^{\frac{3+s}2}
A_{m-2}
\int_0^1 (r^{N-2}-r^{N-2-\beta}+r^s- r^{\beta+s})
( I_+(r) + I_-(r) )
d r .
\end{align}

\medskip
\noindent
{\bf Computation of $A_0(m,1,s)^2$.}
Let $x = (r \theta_1,\pm \alpha r) $, $p = (e_1^{(n)},\alpha)$
so that
\begin{align*}
\nu(x ) = \frac{ ( -\alpha  \omega_1, \pm 1)}{\sqrt{1+\alpha^2}} ,
\quad
\nu(p) = \frac{ ( -\alpha e_1^{(n)}, 1)}{\sqrt{1+\alpha^2}} ,
\end{align*}
and hence
\begin{align*}
\int_{\Sigma}
\frac{1-\langle \nu(x), \nu (p) \rangle }{|p-x|^{N+s}} dx
&=
\sqrt{1+\alpha^2}
A_{m-2}
\int_0^\infty
( J_+(r) + J_-(r)) r^{N-2} dr
\\
&=
\sqrt{1+\alpha^2}
A_{m-2}
\int_0^1
(r^{N-2} + r^s)
( J_+(r) + J_-(r))  dr ,
\end{align*}
where
\begin{align*}
J_+(r)
&=
\frac{\alpha^2}{1+\alpha^2}
\int_0^\pi
\frac{(1-\cos(\theta_1)) \sin(\theta_1)^{m-2} }
{(r^2 + 1 - 2 \cos(\theta_1)  + \alpha^2 (r-1)^2)^{\frac{N+s}2}}
d\theta_1
\\
J_-(r)
&=
\frac{1}{1+\alpha^2}
\int_0^\pi
\frac{[ 2 +\alpha^2-\alpha^2 \cos(\theta_1))\sin(\theta_1)^{m-2} }
{(r^2 + 1 - 2 r \cos(\theta_1)  + \alpha^2 (r+1)^2)^{\frac{N+s}2}}
d\theta_1
\end{align*}
Therefore we find
\begin{align*}
A_0(m,1,s)^2
& =
(1+\alpha^2)^{\frac{3+s}2}
A_{m-2}
\int_0^1
(r^{N-2} + r^s)
( J_+(r) + J_-(r))  dr .
\end{align*}

\noindent
{\bf Computation of $C(m,n,s,\beta)$ for $n\geq 2$.}
Write
$ y = r \omega_1 $, $z = r\omega_2$, with $r>0$,  $\omega_1 \in S^{m-1}$, $\omega_2 \in S^{n-1}$ and let us use spherical coordinates $(\theta_1,\ldots,\theta_{m-1})$ and $(\varphi_1,\ldots,\varphi_{n-1})$ for $\omega_1$ and $\omega_2$ as in \eqref{omega1}  and \eqref{omega2}. Recalling that
$p = (e_1^{(m)},\alpha e_2^{(n)})$, we have
$$
|x-p|^2 = |r\theta_1-e_1^{(m)}|^2 +|r\theta_1-e_1^{(m)}|^2
=  r^2 + 1 - 2 r \cos(\theta_1)  + \alpha^2 ( r^2 + 1 - 2 r \cos(\varphi_1)).
$$
Hence, with $h(y,z) = |y|^{-\beta}$
\begin{align*}
\text{p.v.}
&\int_{\Sigma}
\frac{h(p)- h(x)}{|x-p|^{N}} dx
=
\sqrt{1+\alpha^2}
A_{m-2}A_{n-2}
\text{p.v.}
\int_0^\infty (1-r^{-\beta})
I(r) r^{N-2}  dr
\\
&=
\sqrt{1+\alpha^2}
A_{m-2}A_{n-2}
\int_0^1 (r^{N-2}-r^{N-2-\beta}+r^s- r^{\beta+s})
I(r)
d r
\end{align*}
where
$$
I(r)
=
\int_0^\pi
\int_0^\pi \frac{\sin(\theta_1)^{m-2} \sin(\varphi_1)^{n-2}}{(r^2 + 1 - 2 r \cos(\theta_1)  + \alpha^2 ( r^2 + 1 - 2 r \cos(\varphi_1)))^{\frac{N+s}{2}}}
d\theta_1 d\varphi_1.
$$
We find then that
\begin{align}
\label{Cmnsbeta}
C(m,n,s,\beta) =
(1+\alpha)^{\frac{3+s}{2}}
A_{m-2}A_{n-2}
\int_0^1 (r^{N-2}-r^{N-2-\beta}+r^s- r^{\beta+s})
I(r)
d r .
\end{align}
\medskip
\noindent
{\bf Computation of $A_0(m,n,s)^2$ for $n\geq 2$.}
Similarly as before we have, for $x=(r \omega_1,\alpha r \omega_2)\in \Sigma$, and $p=(e_1^{(m)},\alpha e_2^{(n)} )$:
\begin{align*}
\nu(x ) = \frac{ ( -\alpha \omega_1, \omega_2)}{\sqrt{1+\alpha^2}} ,
\quad
\nu(p) = \frac{ ( -\alpha e_1^{(n)}, 1)}{\sqrt{1+\alpha^2}} .
\end{align*}
Hence
\begin{align*}
\int_{\Sigma}
\frac{1-\langle \nu(x), \nu (p) \rangle }{|p-x|^{N+s}} dx
& =
\sqrt{1+\alpha^2}
A_{m-2}
A_{n-2}
\int_0^\infty r^{N-2} J(r) dr
\\
&=
\sqrt{1+\alpha^2}
A_{m-2}
A_{n-2}
\int_0^1(r^{N-2} + r^s) J(r) dr
\end{align*}
where
\begin{align*}
J(r) = \frac{1}{1+\alpha^2}
\int_0^\pi
\int_0^\pi
\frac{(1+\alpha^2 -\alpha^2 \cos(\theta_1) - \cos(\varphi_1) ) \sin(\theta_1)^{m-2} \sin(\varphi_1)^{n-2}}{ ( r^2 + 1 - 2r\cos(\theta_1) + \alpha^2 ( r^2 + 1 - 2 r \cos(\varphi_1)
)^{\frac{N+s}{2}}} d\theta_1 d\varphi_1 .
\end{align*}
We finally obtain
\begin{align*}
A_0(m,n,s)^2
= (1+\alpha^2)^{\frac{3+s}{2}}
A_{m-2}
A_{n-2}
\int_0^1(r^{N-2} + r^s) J(r) dr.
\end{align*}

In table~1 we show the  values obtained for $H(m,n,0)$ and $A_0(m,n,0)^2$, divided by $(1+\alpha^2)^{\frac{3+s}{2}}
A_{m-2}
A_{n-2}$, from numerical approximation of the integrals.
From these results we can say that for $s=0$, $\Sigma$ is stable if $n+m=7$ and unstable if $n+m\leq 6$.
The same holds for $s>0$ close to zero by continuity of the values with respect to $s$.

\begin{table}

\begin{tabular}{|c|l|l|l|l|l|l|l|l|}
\hline
&&\multicolumn{6}{c}{$n$}&
\\
\hline
 &  & 1 & 2 & 3 & 4 & 5 & 6 & 7
\\
$m$ & & & & & & & &
\\
\hline
2 & $H $  & 0.8140 & 1.0679 & & & & &
\\
 & $A_0^2$  & 3.2669 & 2.3015 & & & & &
\\
\hline
3 & $H $ & 1.1978 & 1.2346 & 0.3926 & & & &
\\
 & $A_0^2$ & 2.5984 & 1.7918 & 0.4463 & & & &
\\
\hline
4 & $H $ & 1.3968 & 1.3649 & 0.4477 & 0.1613 & & &
\\
 & $A_0^2$ & 2.0413 & 1.5534 & 0.4288 & 0.1356 & & &
\\
\hline
5 & $H $ & 1.5117 & 1.4570 & 0.4895 & 0.1845 & 0.06978 & &
\\
 & $A_0^2$ & 1.7332 & 1.3981 &0.4118 & 0.1398 & 0.04849 & &
\\
\hline
6 & $H $ & 1.5833 & 1.5231 & 0.5215 & 0.2031 & 0.08013 & 0.03113 &
\\
 & $A_0^2$ & 1.5318 & 1.2841 & 0.3955 &  0.1412 & 0.05173 &0.01885 &
\\
\hline
7 & $H $ & 1.6303 & 1.5719 & 0.5465 & 0.2182 & 0.08885 & 0.03583 & 0.01416
\\
& $A_0^2$ & 1.3872 & 1.1951 & 0.3802 & 0.1409 & 0.05381 & 0.02051 &0.007704
\\
\hline

\end{tabular}

\bigskip

\caption{ Values of $H(m,n,0)$ and $A_0(m,n,0)^2$ divided by $(1+\alpha^2)^{\frac{3+s}{2}}
A_{m-2}
A_{n-2}$}
\end{table}

\begin{remark}
We see from formulas \eqref{Cm1sbeta} and \eqref{Cmnsbeta} that $C(m,n,s,\beta)$ is symmetric with respect to $\frac{N-2-s}{2}$ and is maximized for $\beta = \frac{N-2-s}{2}$.
\end{remark}

\begin{remark}
In table 2 we give some numerical values of $\alpha$, $H(m,n,s)$ and $A_0(m,n,s)^2$ divided by $(1+\alpha^2)^{\frac{3+s}{2}}
A_{m-2}
A_{n-2}$ for $m=4$, $n=3$, which show how in this dimension stability depends on $s$. One may conjecture that there is $s_0$ such that the cone is stable for $0 \leq s \leq s_0$ and unstable for $s_0<s<1$.
\end{remark}

%
%
%
%
%

\appendix

\section{Asymptotics}
\label{sect asymptotics}

We prove convergence of geometric fractional quantities as $s\to1$ ($\ve=1-s\to0$).
Let $\Sigma\subset\R^{n+1}$ be a smooth embedded hyper surface.

\begin{lemma} Assume $\Sigma = \partial E$. Then for any $X\in \Sigma$
$$
(1-s) \int_{\R^{n+1}} \frac{\chi_E(Y)-\chi_{E^c}(Y)}{|X-Y|^{n+1+s}} \, d Y = - H_\Sigma(X) n \omega_n + O(1-s) ,
$$
as $s\to1$, where $H_\Sigma(X) = \frac{\kappa_1+\ldots+\kappa_n}{n}$ is the mean curvature of $\Sigma$ at $X$  and $\omega_n$ is the volume of the unit ball in $\R^n$.
\end{lemma}
\begin{proof}
Let us fix $R>0$ and $X\in \Sigma$ and assume $X=0$ for simplicity.
Let $\Sigma_R $ be $\Sigma$ intersected with the cylinder $B_R(0)\times (-R,R)$, $B_R(0)\subset \R^n$.
After rotation, we describe $\Sigma_R$  as the graph of $g:B_R(0)\to \R$ with
$$
g(0) =0,\quad D g(0) = 0,
$$
and assume $E$ lies above $\Sigma_R$.

Note that
$$
\int_{( B_R(0)\times(-R,R))^c} \frac{\chi_E(Y)-\chi_{E^c}(Y)}{|X-Y|^{n+1+s}} \, d Y=O(1)
$$
as $s\to1$.
We compute
$$
I=\int_{B_R(0)\times(-R,R)} \frac{\chi_E(Y)-\chi_{E^c}(Y)}{|X-Y|^{n+1+s}} \, d Y =
-2
\int_{B_R\subset\R^n}
\int_0^{g(t)}
\frac{1}{(|t|^2 + t_3^2)^{\frac{n+1+s}{2}}}
dt_3 \, d t ,
$$
expanding
$$
\int_0^{z}
\frac{1}{(|t|^2 + t_3^2)^{\frac{n+1+s}{2}}}
dt_3
= \frac{z}{|t|^{3+s}}
-(n+1+s) z^2 \int_0^1
(1-\tau) \frac{\tau z}{(|t|^2 + (\tau z)^2)^{\frac{n+3+s}{2}}} \, d\tau .
$$
Then
$$
I
=
I_1+I_2+I_3
$$
where
\begin{align*}
I_1
&=
-2
\int_{|t|<R}
\frac{\frac12 D^2 g(0)[{t}^2]}{|t|^{n+1+s}} \, d t ,
\qquad\qquad
I_2
=
-2
\int_{|t|<R}
\frac{g(t)- \frac12  D^2 g(0)[{t}^2]}{|t|^{n+1+s}} \, d t ,
\\
I_3
&=2(3+s)
\int_{|t|<R}
g( t)^2
\int_0^1 (1-\tau)
\frac{\tau g(t)}{(|t|^2 + (\tau g(t))^2)^{\frac{n+3+s}{2}}} \, d\tau
\, d t ,
\end{align*}
where $D^2 g$ denotes the Hessian matrix of $g$.
Then
$$
I_1 = - \frac{ \omega_n \Delta g(0) R^{1-s}}{1-s}
=
- n\omega_n \frac{ H_\Sigma(X) R^{1-s}}{(1-s)} .
$$
For the other terms we have $I_2=O(1)$ and $I_3=O(1)$ as $s\to1$.

\end{proof}

For the next results we assume that there is $C$ such that for all $0<s<1$ and $X\in\Sigma$
$$
\int_{Y\in\Sigma,|Y-X|\geq 1} \frac{1}{|X-Y|^{n+1+s}}\,  dY \leq C.
$$

\begin{lemma}
\label{conv lapl}
If $h$ is $ C^{2,\alpha}(\Sigma)$ and bounded,
$$
(1-s)\text{p.v.}
\int_{\Sigma}
\frac{h(Y)-h(X)}{|X-Y|^{n+1+s}} dY
=  \frac{\omega_n}2 \Delta_{\Sigma}h(X) + O(1-s) ,
$$
as $s\to1$, where $\Delta_{\Sigma}$ is the Laplace-Beltrami operator on $\Sigma$ and $\omega_n=\frac{area(S^{n-1})}{n}$ is the volume of the unit ball in $\R^n$.
\end{lemma}


For the proof we use the following computation.
\begin{lemma}
If $\phi \in C^{2,\alpha}(\overline B_R(0))$,
\begin{align}
\label{b12}
(1-s)\int_{B_R\subset\R^n} \frac{\phi(t)-\phi(0)}{|t|^{n+1+s}} \, d t = \frac{\omega_n}{2} \Delta \phi(0) + O(1-s),
\end{align}
as $s\to1$.
\end{lemma}
\begin{proof}
We expand
$$
\phi(t) = \phi(0) + D\phi(0)t+\frac12 D^2\phi(0)[t^2] + O(|t|^{2+\alpha})
$$
as $t\to0$ and compute
\begin{align*}
\int_{B_R} \frac{\phi(t)-\phi(0)}{|t|^{n+1+s}} \, d t
&=
\frac12
\int_{B_R} \frac{D^2\phi(0)[t^2]}{|t|^{n+1+s}} \, d t + O(1)\\
&= \frac12 \frac{area(S^{n-1})}{n} \frac{R^{1-s}}{1-s} \Delta \phi(0) + O(1)
\end{align*}
as $s\to1$.
\end{proof}

\begin{proof}[Proof of Lemma~\ref{conv lapl}]

Let us fix $R>0$ and $X\in \Sigma$ and assume $X=0$ for simplicity.
Let $\Sigma_R $ be $\Sigma$ intersected with the cylinder $B_R(0)\times (-R,R)$, $B_R(0)\subset \R^n$.
After rotation, we describe $\Sigma_R$  as the graph of $g:B_R(0)\to \R$ with
$$
g(0) =0,\quad D g(0) = 0.
$$
Then
\begin{align*}
\int_{\Sigma_R^c}\frac{h(Y)-h(X)}{|X-Y|^{n+1+s}} dY
=O(1)
\end{align*}
as $s\to1$.
We have
\begin{align*}
\int_{\Sigma_R}
\frac{h(Y)-h(X)}{|X-Y|^{n+1+s}} dY
=\int_{B_R(0)} \frac{h(g(t)) - h(g(0))}{(g(t)^2+|t|^2)^{\frac{n+1+s}{2}}}	 \sqrt{1+|D g(t)|^2}\,dt
\end{align*}

The previous lemma also holds if $\phi$ depends on $s$ and $\phi_s\to\phi$ in $C^{2,\alpha}$ as $s\to1$.
We apply \eqref{b12} to
$$
\phi_s(t) =  \frac{h(g(t)) - h(g(0))}{(\frac{g(t)^2}{|t|^2}+1)^{\frac{n+1+s}{2}}}	 \sqrt{1+|D g(t)|^2}
$$
and note that $\phi_s\to\phi$ as $s\to 1$, where
$$
\phi(t) =  \frac{h(g(t)) - h(g(0))}{(\frac{g(t)^2}{|t|^2}+1)^{n+2}}	 \sqrt{1+|D g(t)|^2}
$$
and
$$
\Delta\phi(0) = \sum_{i=1}^n D_i (h\circ g)(0)= \Delta_\Sigma h(0).
$$
\end{proof}

\begin{lemma}
\label{lemma conv A}
Let $\nu$ be smooth choice of normal vector $\nu$ on $\Sigma$. Then
$$
(1-s) \int_{\Sigma} \frac{(\nu(x)-\nu(y))\cdot \nu(x)}{|x-y|^{n+1+s}} dy =  \frac{\omega_n}2 |A(x)|^2 + O(1)
$$
as $s\to0$, where $|A(x)|^2 $ is the norm squared of the second fundamental form at $x$, i.e. $ \sum_{i=1}^n\kappa_i^2$, where $\kappa_1$, \ldots , $\kappa_n$ are the principal curvatures at $x$.
\end{lemma}
\begin{proof}
We apply Lemma~\ref{conv lapl} with $h(y) = \nu(y)\cdot \nu(x)-1$ and use that
$$
\Delta_\Sigma h (x)= - |A(x)|^2.
$$
\end{proof}

\section{The Jacobi operator}
\label{sect jacobi}
In this section we prove formula \eqref{j1} and derive the formula for the nonlocal Jacobi operator \eqref{nonlocal jac}.

Let $E \subset \R^N$ be an open set with smooth boundary and $\Omega$ be a bounded open set.
Let $\nu$ be the unit normal vector field  of $\Sigma = \pp E$ pointing to the exterior of $E$.
Given $h\in C_0^\infty (\Omega \cap \Sigma)$ and $t$  small, let $E_{th}$ be the set whose boundary $\pp E_{th} $ is parametrized as
$$
\pp E_{th} =  \{ x+ th(x) \nu(x) \ /\ x\in \pp E \},
$$
with exterior normal vector close to $\nu$.

\begin{prop}
\label{prop second var}
For $h\in C_0^\infty (\Omega \cap \Sigma)$
\begin{align}
\label{second var}
\frac {d^2}{dt^2} Per_{s} (E_{th},\Omega) \Big|_{t=0} \ =\ -2 \int_\Sigma  \JJ^s_\Sigma [h]\, h
-\int_{\Sigma} h^2 H H_\Sigma^s ,
\end{align}
where $ \JJ^s_\Sigma $ is the nonlocal Jacobi operator defined in \eqref{nonlocal jac}, $H$ is the classical mean curvature of $\Sigma$ and $H_\Sigma^s$ is the nonlocal mean curvature defined in \eqref{1}.
\end{prop}

In case that $\Sigma$ is a nonlocal minimal surface in $\Omega$ we obtain formula \eqref{j1}.
Another related formula is the following.

\begin{prop}
\label{prop derivative H}
Let $\Sigma_{th} = \partial E_{th}$.
For $p\in \Sigma$ fixed let
$p_t = p + t h(p) \nu(p) \in \Sigma_{th}$.
Then for $h\in C^\infty ( \Sigma) \cap L^\infty(\Sigma)$
\begin{align}
\label{der mean C}
\frac{d}{d t}
H_{\Sigma_{th}}^s(p_t)
\Big|_{t=0}
= 2 \JJ_\Sigma^s[h] (p) .
\end{align}
\end{prop}

A consequence of proposition~\ref{prop derivative H} is that entire nonlocal minimal graphs are stable.

\begin{corollary}
\label{coro entire stable}
Suppose that $\Sigma = \partial E$ with
$$
E = \{ (x',F(x')) \in \R^N: x'\in \R^{N-1} \}
$$
is a nonlocal minimal surface. Then
\begin{align}
\label{linearly stable}
-\int_\Sigma  \JJ^s_\Sigma [h]\, h\, \ge \, 0\foral h\in C_0^\infty(\Sigma).
\end{align}
\end{corollary}

\medskip
\noindent
\begin{proof}[Proof of Proposition~\ref{prop second var}]
Let
$$
K_\delta (z) = \frac{1}{|z|^{N+s}} \eta_\delta(z)
$$
where
$\eta_\delta(x) = \eta(x/\delta)$
($\delta >0$) and $\eta \in C^\infty(\R^N)$ is a radially symmetric cut-off function with $\eta(x)=1 $ for $|x|\geq 2$, $\eta(x) = 0$ for $|x|\leq 1$.

Consider
\begin{align}
\label{def per delta}
Per_{s,\delta}(E_{th},\Omega)
=
\int_{E_{th}\cap\Omega} \int_{\R^N \setminus E_{th}}
K_\delta(x-y) \, dy \, d x
+  \int_{E_{th}\setminus\Omega} \int_{\Omega \setminus E_{th}}
K_\delta(x-y) d y d x .
\end{align}
We will show that $\frac{d^2}{dt^2}
Per_{s,\delta}(E_{th},\Omega)  $ approaches a certain limit $D_2(t)$ as $\delta \to 0$, uniformly for $t$ in a neighborhood of $0$ and that
\begin{align*}
D_2(0)&=
-2 \int_\Sigma  \JJ^s_\Sigma [h]\, h
-\int_{\Sigma} h^2 H H_\Sigma^s .
\end{align*}

First we need some extensions of $\nu$ and $h$ to $\R^N$.
To define them, let $K \subset \Sigma$ be the support of $h$ and
$U_0$ be an open bounded neighborhood of $K$ such that  for any $x\in U_0$, the closest point $\hat x \in \Sigma$ to $x$ is unique and defines a smooth function of $x$.
We also take $U_0$ smaller if necessary as to have $\overline U_0 \subset \Omega$.
Let $\tilde \nu:\R^N \to \R^N$ be a globally defined smooth unit vector field such that
$\tilde \nu(x)=\nu(\hat x)$ for $x\in U_0$.
We also extend $h$ to $\tilde h: \R^N \to \R$ such that it is smooth with compact support contained in $\Omega$ and $\tilde h(x) = h(\hat x)$ for $x\in U_0$.
From now one we omit the tildes ($\tilde\ $) in the definitions of the extensions of $\nu$ and $h$.
For $t$ small  $\bar x \mapsto \bar x + t h(\bar x) \nu(\bar x) $ is a global diffeomorphism in $\R^N$. Let us write
$$
u(\bar x) =  h(\bar x) \nu(\bar x)  \quad \text{for } \bar x\in \R^N ,
$$
$$
\nu = (\nu^1,\ldots,\nu^N), \quad u=  (u^1,\ldots,u^N)
$$
and let
$$
J_t(\bar x) = J_{id + t u}(\bar x)
$$
be the Jacobian determinant of  $id + t u$.

We change variables
$$
x = \bar x + t u (\bar x) ,
\quad
y = \bar y + t u (\bar y) ,
$$
in \eqref{def per delta}
\begin{align*}
Per_{s,\delta}(E_{th},\Omega)
& =
\int_{E\cap \phi_t(\Omega) } \int_{\R^N \setminus E}
K_\delta(x-y)
J_t(\bar x) J_t(\bar y)
d\bar yd \bar x,
\\
& \qquad
+
\int_{E\setminus \phi_t(\Omega)} \int_{\phi_t(\Omega) \setminus E}
K_\delta(x-y)
J_t(\bar y) d\bar yd \bar x ,
\end{align*}
where $\phi_t$ is the inverse of the map $\bar x \mapsto \bar x + t u(\bar x)$.

Differentiating with respect to $t$:
\begin{align*}
\frac{d}{dt}
Per_{s,\delta}(E_{th},\Omega)
 & =
 \int_{E\cap \phi_t(\Omega) } \int_{\R^N \setminus E}
 \Big[
\nabla K_\delta(x-y) (u(\bar x) - u(\bar y)) J_t(\bar x) J_t(\bar y)
\\
& \qquad
+
K_\delta(x-y) ( J_t'(\bar x) J_t(\bar y) + J_t(\bar x) J_t'(\bar y) )
\Big]
d\bar y d \bar x
\\
& \qquad
+
\int_{E\setminus \phi_t(\Omega)} \int_{\phi_t(\Omega) \setminus E}
\Big[
\nabla K_\delta(x-y) (u(\bar x) - u(\bar y)) J_t(\bar x) J_t(\bar y)
\\
&\qquad
+
K_\delta(x-y) ( J_t'(\bar x) J_t(\bar y) + J_t(\bar x) J_t'(\bar y) )
\Big]
d\bar y d \bar x ,
\end{align*}
where
$$
J_t'(\bar x ) = \frac{d}{dt} J_t(\bar x) .
$$
Note that there are no integrals on $\partial \phi_t(\Omega)$ for $t$ small because $u$ vanishes in a neighborhood of $\partial \Omega$.

Since the integrands in $\frac{d}{dt}
Per_{s,\delta}(E_{th},\Omega) $ have compact support contained in $\phi_t(\Omega)$ ($t$ small), we can write
\begin{align*}
\frac{d}{dt}
Per_{s,\delta}(E_{th},\Omega)
&=
\int_{E} \int_{\R^N\setminus E}
\Big[
\nabla K_\delta(x-y) (u(\bar x) - u(\bar y)) J_t(\bar x) J_t(\bar y)
\\
& \qquad +
 K_\delta(x-y) ( J_t'(\bar x) J_t(\bar y) + J_t(\bar x) J_t'(\bar y) )
\Big]
d\bar y d \bar x .
\end{align*}
Differentiating once more
$$
\frac{d^2}{dt^2}
Per_{s,\delta}(E_{th},\Omega)
= A(\delta,t) + B(\delta,t) + C(\delta,t)
$$
where
\begin{align*}
A(\delta,t) & =
\int_{E} \int_{\R^N\setminus E}
D^2 K_\delta(x-y) (u(\bar x) - u(\bar y)) (u(\bar x) - u(\bar y)) J_t(\bar x) J_t(\bar y) d\bar y d \bar x
\\
B(\delta,t) & =
2 \int_{E} \int_{\R^N\setminus E}
\nabla K_\delta(x-y) (u(\bar x) - u(\bar y))
( J_t'(\bar x) J_t(\bar y) + J_t(\bar x) J_t'(\bar y) )
d\bar y d \bar x
\\
C(\delta,t) & =
\int_{E} \int_{\R^N\setminus E}
K_\delta(x-y)
(
J_t''(\bar x) J_t(\bar y)
+2 J_t'(\bar x) J_t'(\bar y)
+ J_t(\bar x) J_t''(\bar y) )
d\bar y d \bar x .
\end{align*}

We claim that $A(\delta,t)$, $B(\delta,t)$ and $C(\delta,t)$ converge as $\delta \to 0$ for uniformly for $t$ near 0, to limit expressions $A(0,t)$, $B(0,t)$ and $C(0,t)$, which are the same as above replacing $\delta$ by 0, and that the integrals appearing in $A(0,t)$, $B(0,t)$ and $C(0,t)$ are well defined.
Indeed, we can estimate
$$
|A(\delta,t)-A(0,t)|
\leq C
\int_{x\in E \cap K_0}
\int_{ y \in E^c, |x-y|\leq 2 \delta}
\frac{1}{|x-y|^{N+s}} \, d y \, d x,
$$
where $K_0$ is a fixed bounded set.
For $x\in E \cap K_0$ we see that
$$
\int_{ y \in E^c, |x-y|\leq 2 \delta}
\frac{1}{|x-y|^{N+s}} \, d y
\leq \frac{C}{dist(x,E^c)^s},
$$
and therefore
$$
|A(\delta,t)-A(0,t)|
\leq C
\leq C
\int_{x\in E \cap K_0, \ dist(x,E^c) \leq 2 \delta}
\frac{1}{dist(x,E^c)^s} \, d x
\leq C \delta^{1-s}.
$$
The differences  $B(\delta,t)-B(0,t)$, $C(\delta,t)-C(0,t)$ can be estimated similarly.
This shows that
$$
\frac{d^2}{dt^2}
Per_{s}(E_{th},\Omega)
\Big|_{t=0}
=
\lim_{\delta\to 0}
\frac{d^2}{dt^2}
Per_{s,\delta}(E_{th},\Omega)
\Big|_{t=0}
= \lim_{\delta \to 0}
A(\delta,0)+B(\delta,0)+C(\delta,0) .
$$
In what follows we will evaluate $A(\delta,0)+B(\delta,0)+C(\delta,0)$.
At $t=0$  we have
\begin{align*}
A(\delta,0)
&=
\int_{E}\int_{\R^N\setminus E}
D_{x_i x_j}
K_\delta(x-y)(u^i(x)-u^i(y))(u^j(x)-u^j(y))
\, d y \, d x
\\
&=A_{11} + A_{12} + A_{21} + A_{22}
\end{align*}
where
\begin{align*}
A_{11} &=
\int_{E} \int_{\R^N\setminus E}
D_{x_i x_j} K_\delta(x-y) u^i(x) u^j(x)
\, d y \, d x
\\
A_{12}
&=
-\int_{E} \int_{\R^N\setminus E}
D_{x_i x_j} K_\delta(x-y) u^i(x) u^j(y)
\, d y \, d x
\\
A_{21}
&=
-\int_{E} \int_{\R^N\setminus E}
D_{x_i x_j} K_\delta(x-y) u^i(y) u^j(x)
\, d y \, d x
\\
A_{22}
&=
\int_{E} \int_{\R^N\setminus E}
D_{x_i x_j} K_\delta(x-y) u^i(y) u^j(y)
\, d y \, d x .
\end{align*}
Let us also write
\begin{align*}
B(\delta,0)
&=
2 \int_{E} \int_{\R^N\setminus E}
D_{x_j} K_\delta(x-y) (u^j(x) - u^j(y)) ( {\rm div}\, (u)(x) + {\rm div}\, (u)(y))
\, d y \, d x
\\
&= B_{11} + B_{12} + B_{21} + B_{22} ,
\end{align*}
where
\begin{align*}
B_{11} &=
2 \int_{E} \int_{\R^N\setminus E}
D_{x_j} K_\delta(x-y) u^j(x) {\rm div}\, (u)(x)
\, d y \, d x
\\
B_{12}
&=
2 \int_{E} \int_{\R^N\setminus E}
D_{x_j} K_\delta(x-y)u^j(x){\rm div}\, (u)(y)
\, d y \, d x
\\
B_{21}
&=
-
2 \int_{E} \int_{\R^N\setminus E}
D_{x_j} K_\delta(x-y)u^j(y){\rm div}\, (u)(x)
\, d y \, d x\\
B_{22} &=
2 \int_{E} \int_{\R^N\setminus E}
D_{y_j} K_\delta(x-y) u^j(y) {\rm div}\, (u)(y)
\, d y \, d x ,
\end{align*}
and
\begin{align*}
C(\delta,0)=C_1 + C_2 + C_3 ,
\end{align*}
where
\begin{align*}
C_1 & =
\int_{E} \int_{\R^N\setminus E}
K_\delta(x-y)
\Big[ {\rm div}\, (u)(x)^2 - tr(D u(x)^2) \Big]
\, d y \, d x
\\
C_2
&=
\int_{E} \int_{\R^N\setminus E}
K_\delta(x-y)
\Big[
{\rm div}\, (u)(y)^2 - tr(D u(y)^2)
\Big]
\, d y \, d x
\\
C_3 &=
2 \int_{E} \int_{\R^N\setminus E}
K_\delta(x-y)
{\rm div}\, (u)(x) {\rm div}\, (u)(y)
\, d y \, d x .
\end{align*}

We compute
\begin{align*}
A_{11}
&=
\int_{E} \int_{\R^N\setminus E}
D_{x_i} \Big[ D_{x_j} K_\delta(x-y) u^i(x) u^j(x) \Big]
\, d y \, d x
\\
& \qquad
-
\int_{E} \int_{\R^N\setminus E}
D_{x_j} K_\delta(x-y) D_{x_i} \Big[  u^i(x) u^j(x) \Big]
\, d y \, d x
\\
&=
\int_{\partial E} \int_{\R^N\setminus E}
D_{x_j} K_\delta(x-y) u^i(x) u^j(x) \nu^i(x)
\, d y \, d x
\\
& \qquad
-
\int_{E} \int_{\R^N\setminus E}
D_{x_j} K_\delta(x-y)  \Big[  D_{x_i} u^i(x) u^j(x) +  u^i(x) D_{x_i}  u^j(x)\Big]
\, d y \, d x .
\end{align*}
Therefore
\begin{align*}
A_{11} + B_{11}
&=
\int_{\partial E} \int_{\R^N\setminus E}
D_{x_j} K_\delta(x-y) u^i(x) u^j(x) \nu^i(x)
\, d y \, d x
\\
&\qquad
+
\int_{E} \int_{\R^N\setminus E}
D_{x_j} K_\delta(x-y)  \Big[  D_{x_i} u^i(x) u^j(x) -  u^i(x) D_{x_i}  u^j(x)\Big]
\, d y \, d x .
\end{align*}
We express the first term as
\begin{align*}
\int_{\partial E} \int_{\R^N\setminus E}
&
D_{x_j} K_\delta(x-y) u^i(x) u^j(x) \nu^i(x)
\, d y \, d x
\\
&
=
- \int_{\partial E} \int_{\R^N\setminus E}
D_{y_j} K_\delta(x-y) u^i(x) u^j(x) \nu^i(x)
\, d y \, d x
\\
&=
\int_{\partial E} \int_{\partial E}
K_\delta(x-y) u^i(x) u^j(x) \nu^i(x) \nu^j(y)
\, d y \, d x
\\
&=
\int_{\partial E} \int_{\partial E}
K_\delta(x-y) h(x)^2 \nu(x) \nu(y)
\, d y \, d x .
\end{align*}
For the second term of $A_{11}+B_{11}$ let us write
\begin{align*}
\int_{E} \int_{\R^N\setminus E}
& D_{x_j} K_\delta(x-y)   D_{x_i} u^i(x) u^j(x)
\, d y \, d x
\\
& =
\int_{E} \int_{\R^N\setminus E}
D_{x_j} \Big[ K_\delta(x-y)   D_{x_i} u^i(x) u^j(x) \Big]
\, d y \, d x
\\
& \qquad -
\int_{E} \int_{\R^N\setminus E}
K_\delta(x-y) D_{x_j} \Big[   D_{x_i} u^i(x) u^j(x) \Big]
\, d y \, d x
\\
& =
\int_{\partial E} \int_{\R^N\setminus E}
K_\delta(x-y)   D_{x_i} u^i(x) u^j(x) \nu^j(x)
\, d y \, d x
\\
& \qquad -
\int_{E} \int_{\R^N\setminus E}
K_\delta(x-y) \Big[   D_{x_j x_i} u^i(x) u^j(x) + {\rm div}\, (u)(x)^2 \Big]
\, d y \, d x .
\end{align*}
The third term of $A_{11}+B_{11}$ is
\begin{align*}
- \int_{E} \int_{\R^N\setminus E}
&
D_{x_j} K_\delta(x-y)    u^i(x) D_{x_i}  u^j(x)
\, d y \, d x
\\
&=
- \int_{E} \int_{\R^N\setminus E}
D_{x_j} \Big[ K_\delta(x-y)    u^i(x) D_{x_i}  u^j(x) \Big]
\, d y \, d x
\\
&\qquad
+
\int_{E} \int_{\R^N\setminus E}
K_\delta(x-y)    D_{x_j} \Big[ u^i(x) D_{x_i}  u^j(x) \Big]
\, d y \, d x
\\
&=
- \int_{\partial E} \int_{\R^N\setminus E}
K_\delta(x-y)    u^i(x) D_{x_i}  u^j(x) \nu^j(x)
\, d y \, d x
\\
&\qquad
+
\int_{E} \int_{\R^N\setminus E}
K_\delta(x-y)     \Big[ D_{x_j} u^i(x) D_{x_i}  u^j(x) + u^i(x) D_{x_j x_i}  u^j(x)\Big]
\, d y \, d x  .
\end{align*}
Therefore
\begin{align*}
A_{11}+B_{11}
&=
\int_{\partial E} \int_{\partial E}
K_\delta(x-y) h(x)^2 \nu(x) \nu(y)
\, d y \, d x
\\
& \qquad
+
\int_{\partial E} \int_{\R^N\setminus E}
K_\delta(x-y)
\Big[ D_{x_i} u^i(x) u^j(x) \nu^j(x)
- u^i(x) D_{x_i}  u^j(x) \nu^j(x) \Big]
\, d y \, d x
\\
& \qquad
+
\int_{E} \int_{\R^N\setminus E}
K_\delta(x-y) \Big[  D_{x_j} u^i(x) D_{x_i}  u^j(x)- {\rm div}\, (u)(x)^2 \Big]
\, d y \, d x ,
\end{align*}
so that
\begin{align*}
A_{11} + B_{11} + C_1
&=
\int_{\partial E} \int_{\partial E}
K_\delta(x-y) h(x)^2 \nu(x) \nu(y)
\, d y \, d x
\\
& \qquad
+
\int_{\partial E} \int_{\R^N\setminus E}
K_\delta(x-y)
\Big[ D_{x_i} u^i(x) u^j(x) \nu^j(x)
- u^i(x) D_{x_i}  u^j(x) \nu^j(x) \Big]
\, d y \, d x .
\end{align*}
But using $u = \nu h$ and ${\rm div}\, (\nu) = H$ where $H$ is the mean curvature of $\partial E$ we have
\begin{align*}
D_{x_i} u^i(x) u^j(x) \nu^j(x)
- u^i(x) D_{x_i}  u^j(x) \nu^j(x)
&=
h(x)^2 H(x)
\end{align*}
and therefore
\begin{align*}
A_{11} + B_{11} + C_1
&=
\int_{\partial E} \int_{\partial E}
K_\delta(x-y) h(x)^2 \nu(x) \nu(y)
\, d y \, d x
+
\int_{\partial E} \int_{\R^N\setminus E}
K_\delta(x-y)
h(x)^2 H(x) .
\end{align*}

In a similar way, we have
\begin{align*}
A_{22} + B_{22} + C_2
&=
\int_{\partial E} \int_{\partial E}
K_\delta(x-y) h(y)^2 \nu(x) \nu(y)
\, d y \, d x
\\
& \qquad
-
\int_{ E} \int_{\partial E}
K_\delta(x-y)
\Big[ D_{y_i} u^i(y) u^j(y) \nu^j(y)
- u^i(y) D_{y_i}  u^j(y) \nu^j(y) \Big]
\, d y \, d x
\\
&=
\int_{\partial E} \int_{\partial E}
K_\delta(x-y) h(y)^2 \nu(x) \nu(y)
\, d y \, d x
-
\int_{ E} \int_{\partial E}
K_\delta(x-y)
h(y)^2 H(y) \, d y \, d x .
\end{align*}
Further calculations show that
\begin{align*}
A_{12}
&=
-\int_{\partial E} \int_{\partial E}
K_\delta(x-y) h(x) h(y) \, d y d x
\\
&
\qquad
- \int_{\partial E} \int_{\R^N\setminus E}
K_\delta(x-y) {\rm div}\, (u)(y) u^i(x) \nu^i(x)
\, d y \, d x
\\
& \qquad
+
\int_{ E} \int_{\partial E}
K_\delta(x-y) {\rm div}\, (u)(x) u^i(y) \nu^i(y)
\, d y \, d x
\\
& \qquad
+
\int_{E} \int_{\R^N\setminus E}
K_\delta(x-y) {\rm div}\, (u)(x) {\rm div}\, (u)(y)
\, d y \, d x ,
\end{align*}
\begin{align*}
A_{21}&=
-\int_{\partial E} \int_{\partial E}
K_\delta(x-y) h(x) h(y) \, d y d x
\\
& \qquad -
\int_{\partial E}\int_{\R^N\setminus E}
K_\delta(x-y) {\rm div}\, (u)(y) u^j(x) \nu^j(x)
\, d y \, d x
\\
&\qquad
+
\int_{E} \int_{\partial E}
K_\delta(x-y) {\rm div}\, (u)(x) u^i(y) \nu^i(y)
\, d y \, d x
\\
& \qquad
+\int_{E}\int_{\R^N\setminus E}
K_\delta(x-y) {\rm div}\, (u)(x) div (u)(y)
\, d y \, d x ,
\end{align*}
and
\begin{align*}
B_{12} + B_{21} &=
2\int_{\partial E} \int_{\R^N\setminus E}
K_\delta(x-y) {\rm div}\, (u)(y) u^j(x) \nu^j(x)
\, d y \, d x
\\
& \qquad
-
2\int_{E} \int_{\partial E}
K_\delta(x-y) {\rm div}\, (u)(x) u^j(y) \nu^j(y)
\, d y \, d x
\\
& \qquad
-4 \int_{E}\int_{\R^N\setminus E}
K_\delta(x-y) {\rm div}\, (u)(x) div (u)(y)
\, d y \, d x ,
\end{align*}
so that
\begin{align*}
A_{12}+A_{21} + B_{12}+B_{21} + C_3
=
-2
\int_{\partial E} \int_{\partial E}
K_\delta(x-y) h(x) h(y) \, d y d x .
\end{align*}
Therefore
\begin{align*}
\frac{d^2}{dt^2}
Per_{s,\delta}(E_{th},\Omega)
\Big|_{t=0}
&=
2\int_{\partial E} \int_{\partial E}
K_\delta(x-y)  h(x)^2  ( \nu(x) \nu(y)-1)
\, d y \, d x
\\
& \qquad
-2
\int_{\partial E}
h(x)
\int_{\partial E}
K_\delta(x-y)  ( h(y) - h(x)) \, d y d x
\\
& \qquad
-
\int_{\partial E}
h(x)^2 H(x)
\int_{\R^N}
(\chi_E(y)- \chi_{E^c}(y))
K_\delta(x-y)
\, d y \, d x .
\end{align*}
Taking the limit as $\delta\to 0$ we find \eqref{second var}.
\end{proof}

\medskip
\noindent
\begin{proof}[Proof of Proposition~\ref{prop derivative H}]
Let $\nu_t(x)$ denote the unit normal vector to $\partial E_t$ at $x \in \partial E_t$ pointing out of $E_t$.
Note that $\nu(x) = \nu_0(x)$.
Let $L_t$ be the half space defined by $L_t = \{ x: \langle x - p_t , \nu_t(p_t)\rangle>0\}$.
Then
\begin{align}
\label{int ht}
H_{\Sigma_{th}}^s(p_t) =
\int_{\R^N} \frac{\chi_{E_t} (x)-\chi_{L_t}(x)-\chi_{E^c}(x) + \chi_{L_t^c}(x)} {|x-p_t|^{N+s}} \, d x
\end{align}
since the function $1-2 \chi_{L_t}$ has zero principal value. Note that the integral in \eqref{int ht} is well defined and
\begin{align*}
H_{\Sigma_{th}}^s(p_t)
& =
2
\int_{\R^N} \frac{\chi_{E_t} (x)-\chi_{L_t}(x)} {|x-p_t|^{N+s}} \, d x .
\end{align*}

For $\delta >0$ let $\eta \in C^\infty(\R^N)$ be a radially symmetric cut-off function with $\eta(x)=1 $ for $|x|\geq 2$, $\eta(x) = 0$ for $|x|\leq 1$. Define $\eta_\delta(x) = \eta(x/\delta)$ and
write
$$
\int_{\R^N }
\frac{\chi_{E_t} (x)-\chi_{L_t}(x)} {|x-p_t|^{N+s}} \, d x =
f_\delta (t) + g_\delta(t)
$$
where
\begin{align*}
f_\delta(t)
&=
\int_{\R^N }
\frac{\chi_{E_t} (x)-\chi_{L_t}(x)} {|x-p_t|^{N+s}} \eta_\delta(x-p_t) \, d x
\end{align*}
and $g_\delta(t)$ is the rest.
Then it is direct that $f_\delta$ is differentiable and
\begin{align*}
f_\delta'(0)
&=
\int_{\partial E}
\frac{h(x)}{|x-p|^{N+s}} \eta_\delta(x-p)
\\
& \qquad
-
\int_{\partial L_0}
\frac{h(p) \langle \nu(p) , \nu(p) \rangle
-\langle x-p,\frac{\partial \nu_t(p_t)}{\partial t}|_{t=0}\rangle
}{|x-p|^{N+s}}\eta_\delta(x-p)
\\
&
\qquad
+ (N+s)
h(p)
\int_{\R^N}
\frac{\chi_{E}(x) - \chi_{L_0}(x)}{|x-p|^{N+s+2} }
\langle x-p,\nu(p)\rangle \eta_\delta(x-p) dx
\\
&
\qquad
-
h(p)
\int_{\R^N}
\frac{\chi_{E}(x) - \chi_{L_0}(x)}{|x-p|^{N+s} }
\langle \nabla \eta_\delta(x-p),\nu(p)\rangle  dx .
\end{align*}
We integrate the third term by parts
\begin{align*}
& (N+s)
\int_{\R^N}
\frac{\chi_{E}(x) - \chi_{L_0}(x)}{|x-p|^{N+s+2} }
\langle x-p,\nu(p)\rangle \eta_\delta(x-p) dx
\\
&=
-
\int_{\R^N}
( \chi_{E}(x) - \chi_{L_0}(x) )
\langle \nabla\frac{1}{|x-p|^{N+s} }
,\nu(p)\rangle \eta_\delta(x-p) dx
\\
& =
-
\int_{\partial E} \frac{\langle \nu(x) ,\nu(p) \rangle }{|x-p|^{N+s}}
\eta_\delta(x-p)
+
\int_{\partial L_0}
\frac{\langle \nu(p), \nu(p) \rangle }{|x-p|^{N+s}}
\eta_\delta(x-p)
\\
& \qquad +\int_{\R^N}\frac{\chi_{E(x)} - \chi_{L_0}(x)}{|x-p|^{N+s}}
\langle \nabla \eta_\delta(x-p) , \nu(p) \rangle dx .
\end{align*}
Since $\eta_\delta $ is radially symmetric,
$$
\int_{\partial L_0}
\frac{
\langle x-p, \frac{\partial \nu_t(p_t)}{\partial t}|_{t=0} \rangle
}{|x-p|^{N+s}}\eta_\delta(x-p) \, d x=0
$$
and then
\begin{align*}
f_\delta'(0)
&=
\int_{\partial E}
\frac{h(x)}{|x-p|^{N+s}} \eta_\delta(x-p)
dx
-
h(p)
\int_{\partial E} \frac{\langle \nu(x) ,\nu(p) \rangle }{|x-p|^{N+s}}
\eta_\delta(x-p) dx,
\end{align*}
which we write as
\begin{align*}
f_\delta'(0)
&=
\int_{\partial E}
\frac{h(x)-h(p)}{|x - p|^{N+s}} \eta_\delta(x-p)
\, d x
+
h(p)
\int_{\partial E} \frac{1-\langle \nu(x) ,\nu(p) \rangle }{|x-p|^{N+s}}
\eta_\delta(x-p) \, d x.
\end{align*}
We claim that  $g_\delta'(t)\to 0$ as $\delta \to 0$, uniformly for $t$ in a neighborhood of 0.
Indeed, in a neighborhood of $p_t$ we can represent $\partial E_t$ as a graph of a function $G_t$ over $L_t \cap B(p_t,2\delta)$, with $G_t$ defined in a neighborhood of $0$ in $\R^{N-1}$, $G_t(0)=0$, $\nabla_{y'} G_t(0)=0$ and smooth in all its variables (we write $y' \in \R^{N-1}$).
Then $g_\delta(t)$ becomes
$$
g_\delta(t)
=
\int_{|y'|<2\delta}
\int_0^{G_t(y')}
\frac{1}{(|y'|^2 + y_N^2)^{\frac{N+s}{2}}}
(1-\eta_\delta(y', y_N) )
d y_N
d y'
$$
so that
$$
g_\delta'(t)
=
\int_{|y'|<2\delta}
\frac{1}{(|y'|^2 + G_t(y')^2)^{\frac{N+s}{2}}}
\frac{\partial G_t}{\partial t}(y')
(1-\eta_\delta(y', y_N) )
d y' .
$$
But $|G_t(y')|\leq K |y'|^2$ and $| \frac{\partial G_t}{\partial t} (y') |\leq K |y'|^2$, so
$$
g_\delta'(t) \leq C \delta^{1-s} .
$$
Therefore
\begin{align*}
\frac{d}{d t}
H_{\Sigma_{th}}^s(p_t)
\Big|_{t=0}
& = 2
\lim_{\delta \to 0}
\Bigg[
\int_{\partial E}
\frac{h(x)-h(p)}{|x - p|^{N+s}} \eta_\delta(x-p) dx
\\
& \qquad
+
h(p)
\int_{\partial E} \frac{1-\langle \nu(x) ,\nu(p) \rangle }{|x-p|^{N+s}}
\eta_\delta(x-p) dx
\Bigg] .
\end{align*}
Letting $\delta\to 0$ we find \eqref{der mean C}.
\end{proof}

\medskip
\noindent
\begin{proof}[Proof of Corollary~\ref{coro entire stable}]
The same argument as in the proof of Proposition~\ref{prop derivative H} shows that if $F:\Sigma \to \R^N$ is a smooth bounded vector field and we let
$E_{t}$ be the set whose boundary $\Sigma_t = \pp E_{t} $ is parametrized as
$$
\pp E_{th} =  \{ x+ t F(x)\ /\ x\in \pp E \},
$$
with exterior normal vector close to $\nu$, then
$$
\frac{d}{d t}
H_{\Sigma_t}^s(p_t)
\Big|_{t=0}
= 2 \JJ_\Sigma^s[ \langle F , \nu\rangle] (p) ,
$$
where $p_t = p + t F(p)$.
Taking as $F(x) = e_N = (0,\ldots,0,1)$ we conclude that $w = \langle \nu,e_N \rangle$ is a positive function satisfying
$$
\JJ_\Sigma^s[ w ] (x) =0
 \quad\text{for all } x \in \Sigma.
$$
More explicitly
\begin{align}
\label{20a}
\text{p.v.}
\int_{\Sigma}
\frac{w(y)-w(x)}{|y-x|^{N+s}} dy
+
w(x) A(x) = 0 \quad\text{for all } x \in \Sigma,
\end{align}
where
$$
A(x) = \int_\Sigma
\frac{\langle \nu (x)- \nu(y), \nu (x) \rangle }{|x-y|^{N+s}} dy.
$$

As in the classical setting we can show that $\Sigma$ is stable in the sense that \eqref{linearly stable} holds.
Let $\phi \in C_0^\infty(\Sigma)$ and
observe that
\begin{align*}
\frac12
\int_{\Sigma}\int_{\Sigma}
\frac{(\phi(x) - \phi(y))^2}{|x-y|^{N+s} } d x d y
=
\int_{\Sigma}\int_{\Sigma}
\frac{(\phi(x) - \phi(y) ) \phi(x)}{|x-y|^{N+s} } d x d y .
\end{align*}
Write $\phi = w \psi$ with $\psi \in C_0^\infty(\Sigma)$. Then
\begin{align}
\nonumber
\int_{\Sigma}\int_{\Sigma}
\frac{(\phi(x) - \phi(y) ) \phi(x)}{|x-y|^{N+s} } d x d y
& =
\int_{\Sigma}\int_{\Sigma}
\frac{(w(x) - w(y) ) w(x) \psi(x)^2}{|x-y|^{N+s} } d x d y
\\
\label{term2}
& \quad +
\int_{\Sigma}\int_{\Sigma}
\frac{(\psi(x) - \psi(y) ) w(x) w(y) \psi(x)}{|x-y|^{N+s} } d x d y .
\end{align}
Multiplying \eqref{20a} by $w \psi^2$ and integrating we get
\begin{align}
\label{term3}
\int_{\Sigma}\int_{\Sigma}
\frac{(w(x) - w(y) ) w(x) \psi(x)^2}{|x-y|^{N+s} } d x d y
=
\int_{\Sigma}A(x) w(x)^2 \psi(x)^2 dx
=
\int_{\Sigma} A(x) \phi(x)^2 dx .
\end{align}
For the second term in \eqref{term2} we observe that
\begin{align}
\label{term4}
\int_{\Sigma}\int_{\Sigma}
\frac{(\psi(x) - \psi(y) ) w(x) w(y) \psi(x)}{|x-y|^{N+s} } d x d y
 & =
\frac12
\int_{\Sigma}\int_{\Sigma}
\frac{(\psi(x) - \psi(y) )^2 w(x) w(y) }{|x-y|^{N+s} } d x d y .
\end{align}
Therefore, combining \eqref{term2}, \eqref{term3}, \eqref{term4} we obtain
\begin{align*}
\frac12
\int_{\Sigma}\int_{\Sigma}
\frac{(\phi(x) - \phi(y))^2}{|x-y|^{N+s} } d x d y
& =
\int_{\Sigma}A(x)\phi(x)^2 dx
\\
\nonumber
& \quad +
\frac12
\int_{\Sigma}\int_{\Sigma}
\frac{(\psi(x) - \psi(y) )^2 w(x) w(y) }{|x-y|^{N+s} } d x d y .
\end{align*}
and tis shows \eqref{linearly stable}.
\end{proof}

\section{Graph representation}
\label{app}

Let
$r,\theta$ be polar coordinates for $x\in\R^2$, i.e.
$
x=(r\cos\theta,r\sin\theta) $.
Then we define
$\hat r = \frac{x}{r} = (\cos\theta,\sin\theta)^T$,
$\hat \theta = (-\sin\theta,\cos\theta)^T$.
Given a point $X\in \Sigma_0$, $X = (x,F_\ve(x))$ we let
$\Pi_1(X)$, $\Pi_2(X)$ and $\nu_{\Sigma_0}(X)$ be tangent and normal vector to $\Sigma_0$ at $X$ as defined in \eqref{pi1 pi2}, \eqref{nu sigma0} and let $\Pi = [\Pi_1,\Pi_2]$.
Then we consider coordinates $t=(t_1,t_2)$ and $t_3$ defined by
$$
(t_1,t_2,t_3) \mapsto
\Pi_1(X) t_1 + \Pi_2(X)t_2 + \nu_{\Sigma_0}(X) t_3.
$$
Let
$$
R_X = \delta |X|
$$
where $\delta>0$ is a small fixed constant.

Given $h$ on $\Sigma_0$ with $\|h\|_* \leq \sigma_0 \ve^{\frac12}$, we can represent $\partial E_h$ near $X_h = X + \nu_{\Sigma_0}(X) h(X)$ as
$$
\Pi(X)t+\nu_{\Sigma_0}(X) g(t) , \quad |t -t_0(X)|\leq 2 R_X
$$
where $g$ is of class $C^{2,\alpha}$ in the ball $B_{2R_X}(t_0(X))$,
with $t_0=t_0(X)$ such that $\Pi(X) t_0$ is the orthogonal projection of $X$ onto the plane generated by $\Pi_1(X)$, $\Pi_2(X)$.
We call $G_X$ the operator defined by
$$
g_h = G_X(h) .
$$

To get the correct dependence of the various functions on $|X|$, let $r_0 = |x|$.
Let us change variables
\begin{align}
\label{scaled vars}
y = x + r_0 B \bar y,\quad t = r_0 \bar t, \quad g = r_0 \bar g
\end{align}
where the $2\times2$ matrix $B$ is given by
$$
B = [\hat r ,\hat \theta]
$$
(and depends on $X$),
so that the equation takes the form
\begin{align*}
0=
\Pi(X) \bar t + \bar g \nu_{\Sigma_0}(X) -
\left[
\begin{matrix}
\frac1{r_0} x + B \bar y\\ r_0^{-1} F_\ve(x + r_0 B \bar y)
\end{matrix}
\right]
- \frac{1}{r_0} h(x + r_0 B \bar y) \nu_{\Sigma_0}(x + r_0 B \bar y) .
\end{align*}
To simplify notation we will omit the bars in $\bar t$, $\bar y$, $\bar g$ and let $\Phi=(y,g)$.

We search for a function $\Phi(t) = (y(t),g(t))$, $y(t)\in\R^2$, $g(t)\in\R$ that solves
\begin{align}
\label{eq graph}
\FF(\Phi,X,h)=0
\end{align}
where
\begin{align*}
\mathcal F(\Phi,X,h)(t)
& =
 \Pi(X)  t +  g(t) \nu_{\Sigma_0}(X) -
\left[
\begin{matrix}
\frac1{r_0} x
+ B y(t)\\ r_0^{-1} F_\ve(x+r_0 B  y(t))
\end{matrix}
\right]
\\
&\qquad
- \frac{1}{r_0} h(x + r_0  B y(t)) \nu_{\Sigma_0}(x + r_0  B y(t) ) .
\end{align*}

We search for functions $y,g$ defined in a ball $B_{\delta_0}(t_0(X)) $, where $\delta_0>$ is some small fixed number. By shifting $t$ to $t-t_0(X)$ we will assume $t_0(X)=0$.

Let $X$ be a Banach space of functions over $\overline B_{\delta_0}(0) \subset \R^2$ with values in $\R^3$. We will take later $X$ either $C^1$, $C^2$ or $C^{2,\alpha}$.
Let $ B_{\delta_1}(\Phi_0) \subset X$ be the open ball of radius $\delta_1>0$ centered at the function
$$
\Phi_0 = (y_0,0)
$$
where $y_0(t) = t$. Note that $\FF(\cdot,X,h)$ maps $\overline B_{\delta_1}(\Phi_0) $ into $X$.
We intend to show that if $\delta_0$ is fixed small, $\delta_1$ is small depending on $\ve$, and $\|h\|_*$, then there is a unique solution $\Phi \in \overline B_{\delta_1}(\Phi_0) $ of $\FF(\Phi,X,h)= 0$.

For this we need to construct a bounded left inverse for $D_\Phi \FF(\Phi_0,X,h)$. We have, for $\Phi = (y,g)$
\begin{align}
\label{D F}
& D_\Phi \FF(\Phi,X,h) \\
\nonumber
& =
\left[
\begin{matrix}
-B_{1,1} -D_{y_1} \nu_{\Sigma_0}^{(1)} h  - \nu_{\Sigma_0}^{(1)}D_{y_1} h
&
-B_{1,2}
-D_{y_2} \nu_{\Sigma_0}^{(1)} h
- \nu_{\Sigma_0}^{(1)}D_{y_2} h
& \nu_{\Sigma_0}^{(1)}(X)
\\
-B_{2,1}
-D_{y_1} \nu_{\Sigma_0}^{(2)} h
- \nu_{\Sigma_0}^{(2)}D_{y_1} h
&
-B_{2,2}
-D_{y_2} \nu_{\Sigma_0}^{(2)} h
- \nu_{\Sigma_0}^{(2)}D_{y_2} h
& \nu_{\Sigma_0}^{(2)}(X)
\\
- D_{y_1} F_\ve -D_{y_1} \nu_{\Sigma_0}^{(3)} h  - \nu_{\Sigma_0}^{(3)}D_{y_1} h
&
-D_{y_2} F_\ve -D_{y_2} \nu_{\Sigma_0}^{(3)} h  - \nu_{\Sigma_0}^{(3)}D_{y_2} h
&
\nu_{\Sigma_0}^{(3)}(X)
\end{matrix}
\right] ,
\end{align}
where  $h$, $\nu_{\Sigma_0}$, $F_\ve$ are evaluated at $x+ r_0 B y(t) $ when it is not explicitly written to depend on $X$ (third column). We write  $\nu_{\Sigma_0}^{(i)}$ the $i$-th component $\nu_{\Sigma_0}$.

We take
$$A=
\left[
\begin{matrix}
- B^{-1}
&
0
\\
0
&
1
\end{matrix}
\right]
$$
as a simple approximation of the inverse of $D_\Phi \FF(\Phi_0,X,h) $. We claim that
\begin{align}
\label{F lipschitz}
\| A( \FF(\Phi_1,X,h) - \FF(\Phi_2,X,h) ) - (\Phi_1- \Phi_2) \|_X \leq L \|\Phi_1 - \Phi_2\|_X
\end{align}
for $\Phi_1,\Phi_2 \in \overline B_{\delta_1}(\Phi_0)$,
where $0<L<1$ and that
\begin{align}
\label{F small}
\| A \FF(\Phi_0,X,h) \|_X \leq (1-L)\delta_1.
\end{align}
With \eqref{F lipschitz}, \eqref{F small} we conclude from the contraction mapping principle,  applied to
\begin{align}
\label{oper T}
T(\Phi) = \Phi - A \FF(\Phi,X,h)
\end{align}
that  there is a unique $\Phi \in \overline B_{\delta_1}(\Phi_0)$ such that $\FF(\Phi,X,h)=0$.


To prove  estimates \eqref{F lipschitz}, \eqref{F small} we always assume $\|h\|_*\leq \sigma_0 \ve^{\frac12}$.

We consider first the case $r_0\geq \delta \ve^{-\frac12}|\log\ve|$.
Let us proceed with  \eqref{F lipschitz} and $\| \ \|_X = \| \ \|_{C^1}$.
Let $\Phi_1 = (y_1,g_1)$,  $\Phi_2 = (y_2,g_2) \in \overline B_{\delta_1}(\Phi_0)$. Then we claim that
\begin{align}
\label{est F1}
\| A( \FF(\Phi_1,X,h) - \FF(\Phi_2,X,h) ) - (\Phi_1- \Phi_2) \|_{C^1} \leq \ve^{\frac12}\|\Phi_1-\Phi_2\|_{C^1} .
\end{align}
Indeed
\begin{align*}
A( \FF(\Phi_1,X,h) - \FF(\Phi_2,X,h) ) - (\Phi_1- \Phi_2) = D_1 + D_2 + D_3 .
\end{align*}
We estimate the norm of
\begin{align*}
D_1= ( g_1-g_2 ) \left( A \nu_{\Sigma_h}  -  e_3 \right)  ,
\end{align*}
where $e_3 = (0,0,1)$.
By Corollary~\ref{coro prop F} $|A \nu_{\Sigma_h}  -  e_3|\leq C\ve^{\frac12}$ so
$$
\|D_1\|_{C^1} \leq C \ve^{\frac12} \|\Phi_1-\Phi_1\|_{C^1}.
$$
Next,
$$
D_2 =
-\left[
\begin{matrix}
0
\\
r_0^{-1}( F_\ve( x + r_0 B y_1(t)) - F_\ve( x + r_0 B y_2(t)))
\end{matrix}
\right]
$$
and using Corollary~\ref{coro prop F}
$$
\|D_2\|_{C^1}\leq C \ve^{\frac12} \|\Phi_1-\Phi_2\|_{C^1}.
$$
Finally
\begin{align*}
D_3 = - \frac{1}{r_0}\left(
h(x + r_0 B y_1(t)) \nu_{\Sigma_0}(x + r_0 B y_1(t)) -
h(x + r_0 B y_2(t)) \nu_{\Sigma_0}(x + r_0 B y_2(t))
\right)
\end{align*}
so
\begin{align*}
\sup_{|t|\leq\delta_0} |D_3|\leq C \ve^{\frac12}\|\Phi_1-\Phi_2\|_{C^1}.
\end{align*}
We write the derivative as
\begin{align*}
D_t D_3 & = ( D h(x + r_0 B y_1(t))  - D h(x + r_0 B y_2(t)) ) D y_1(t)\\
&\quad+  D h(x + r_0 B y_2(t)) B( D y_1(t)-D y_2(t)).
\end{align*}
Since $\|h\|_*\leq \sigma_0\ve^{\frac12}$ and  $\| \ \|_*$ is weighted $C^{2,\alpha}$ norm we have
\begin{align*}
& \sup_{|t|\leq \delta_0 }\left| ( D h(x + r_0 B y_1(t))  - D h(x + r_0 B y_2(t)) ) D y_1(t)\right|
\\
&\leq \|D^2 h\|_{L^\infty} \|y_1-y_2\|_{C^1} \|y_1\|_{C^1}
\leq \|h\|_*  \|y_1-y_2\|_{C^1} \leq C \ve^{\frac12}  \|\Phi_1-\Phi_2\|_{C^1} .
\end{align*}
The other term in  $D_t D_3$  is estimated as
\begin{align*}
\sup_{|t|\leq \delta_0} |  D h(x + r_0 B y_2(t)) | (B(Dy_1(t) - Dy_2(t))| \leq C \ve^{\frac12} \|\Phi_1-\Phi_2\|_{C^1}.
\end{align*}
Therefore
$$
\|D_3\|_{C^1} \leq C \ve^{\frac12} \|\Phi_1-\Phi_2\|_{C^1},
$$
and this proves \eqref{est F1}.

Regarding \eqref{F small}, we have
\begin{align*}
& A \FF(\Phi_0,X,h)   \\
& =  - \frac{1}{r_0} A \left[\begin{matrix} 0 \\ F_\ve(x_0)-F_\ve(x_0+r_0  B t)\end{matrix}\right] + A\left( \Pi(X) - \left[\begin{matrix} B \\ 0 \end{matrix} \right] \right)t
\\
&\quad\quad -\frac{1}{r_0} h(x_0 + r_0 B t) \nu_{\Sigma_0}(x_0 + r_0 B t) ,
\end{align*}
and we see that
$$
\| A \FF(\Phi_0,X_0,h)  \|_{C^1}\leq C \ve^{\frac12}.
$$
Then \eqref{F lipschitz}, \eqref{F small} hold with $C^1$ norm and $\delta_1 = C \ve^{\frac12}$. We conclude that there is a unique $\Phi$ with $\|\Phi-\Phi_0\|_{C^1(\overline B_{\delta_0}(0))} \leq C \ve^{\frac12}$ such that $\FF(\Phi,X,h)=0$.

We can get also estimates for $\Phi$ in $C^{2,\alpha}$.
For this we claim that
for $\Phi_1,\Phi_2 \in C^{2,\alpha}(\overline B_{\delta_0}(0))$:
\begin{align*}
& \| D^2 \{ A( \FF(\Phi_1,X_0,h) - \FF(\Phi_2,X_0,h) ) - (\Phi_1 - \Phi_2) \} \|_{C^0}\\
&\leq C \ve^{\frac12} (
\|\Phi_1\|_{C^1}^2 \|\Phi_1-\Phi_2\|_{C^0}^\alpha + \|\Phi_1\|_{C^1} \|\Phi_1-\Phi_2\|_{C^1} + \|D^2 \Phi_1\|_{C^0} \|\Phi_1-\Phi_2\|_{C^0} \\
&\qquad + \|D^2( \Phi_1-\Phi_2) \|_{C^0}) .
\end{align*}
Let us consider $\Phi_1, \Phi_2$ with  $\|\Phi_i - \Phi_0\|_{C^1} \leq C \ve^{\frac12}$ so $\|\Phi_i\|_{C^1} \leq C$. Then we can simplify the above estimate to
\begin{align}
\nonumber
& \| D^2 \{ A( \FF(\Phi_1,X_0,h) - \FF(\Phi_2,X_0,h) ) - (\Phi_1 - \Phi_2) \} \|_{C^0}
\\
\label{500}
&\leq C \ve^{\frac12} (\|\Phi_1-\Phi_2\|_{C^1}^\alpha + \| D^2 \Phi_1 \|_{C^0} \|\Phi_1-\Phi_2\|_{C^0} + \| D^2(\Phi_1-\Phi_2)\|_{C^0}) .
\end{align}
In a similar way, assuming $\|\Phi_i\|_{C^1} \leq C$,
\begin{align*}
& [ D^2 \{ A( \FF(\Phi_1,X_0,h) - \FF(\Phi_2,X_0,h) ) - (\Phi_1 - \Phi_2) \} ]_{\alpha,B_{\delta_0}} \\
&\leq C\ve^{\frac12} ( [D^2 (\Phi_1-\Phi_2)]_{\alpha,B_{\delta_0}}  + 1 + \|D^2 \Phi_1 \|_{C^0} + \| D^2( \Phi_1-\Phi_2)\|_{C^0} ) .
\end{align*}
Let $T$ be the operator defined by \eqref{oper T} and $\Phi_k$ the sequence defined by
\begin{align}
\label{501}
\Phi_{k+1} = T(\Phi_k) , \Phi_0= (y_0,0) .
\end{align}
As shown before $\Phi_{k+1}$ is a Cauchy sequence in $\overline B_{\delta_1}(\Phi_0)$ with $C^1$ topology. Using \eqref{500} we get
\begin{align*}
\|D^2 \Phi_{k+1}\|_{C^0} &\leq \|D^2 ( T ( \Phi_k) - T(\Phi_0) )\|_{C^0} + \|D^2 T(\Phi_0)  \|_{C^0} \\
&\leq C \ve^{\frac12} ( \|D^2 \Phi_k\|_{C^0} + 1) .
\end{align*}
Iterating this inequality shows that $\|D^2 \Phi_k\|_{C^0} $ remains bounded as $k\to\infty$.
Similarly
$$
[D^2 T(\Phi_{k+1})]_{\alpha,B_{\delta_0}} \leq C \ve^{\frac12} ( [D^2 T(\Phi_k)]_{\alpha,B_{\delta_0}} +1)
$$
and iterating this shows that $ [D^2 T(\Phi_k)]_{\alpha,B_{\delta_0}}$ remains bounded.
Therefore the fixed point $\Phi$ actually satisfies $\Phi\in C^{2,\alpha}(\overline B_{\delta_0})$. Again using \eqref{500} and \eqref{501} we find actually
$$
\|\Phi - \Phi_0\|_{C^{2,\alpha}}\leq C \ve^{\frac12}.
$$

\begin{proof}[Proof of Lemma~\ref{lemma est G}]
Estimate \eqref{at g} follows from the definition and the mean value formula.

Let us prove \eqref{B g}:
$$
g(z+t_0(X)) - g(t_0(X))-\nabla g(t_0(X))z = \int_0^1(1-\tau) g''(t_0(X)+\tau z ) [z^2] \,d\tau ,
$$
so that
\begin{align*}
|B(g)| & \leq \|g''\|_{L^\infty(B_{2R_X}(t_0(X)))} |z|
\leq \frac{\|g\|_b}{|X|} |z| \quad\text{in } B_{2R_X}(0).
\end{align*}

To prove the estimates for $g_i = DG_X(h)[h_i]$ we give first and expression for this function.
Next we compute $g_i = DG_X(h)[h_i]$. For this we write $h(y,s) = h(y) + s h_i(y)$, and let us write  $\Phi' = \pd{}{s} \Phi$, where $\Phi = (y,g)$.  We use the scaled variables as defined in \eqref{scaled vars}
and find
$$
D_\Phi \FF \Phi' = h_i(x+r_0 B y(t)) \nu_{\Sigma_0}(x+r_0 B y(t))
$$
where $D_\Phi \FF$ is given in \eqref{D F} and is evaluated at $\Phi,X,h$. From this formula we get
\begin{align}
\label{form gi}
g_i(t)=h_i(x+r_0 B y(t) ) \frac{m}{D} ,
\end{align}
where
$$
m = m_0 + h m_1
$$
and $D$ is the determinant of $D_\Phi \FF$ and can be written as
$$
D = D_0 + h D_1  + D_{y_1}h D_2 + D_{y_2} h  D_3 .
$$
The functions $D_i$, $m_0$, $m_1$,  have the following expressions:
\begin{align*}
D_0 &= \nu_{\Sigma_0}^{(1)}(X) ( B_{21} D_{y_2}F - B_{22} D_{y_1}F ) -   \nu_{\Sigma_0}^{(2)}(X)  (B_{11} D_{y_2}F - B_{12} D_{y_1} F)+  \nu_{\Sigma_0}^{(3)}(X)
\end{align*}
\begin{align*}
D_1 &=
\nu_{\Sigma_0}^{(1)}(X) \big[ B_{21} D_{y_2} \nu_{\Sigma_0}^{(3)} + D_{y_2} F D_{y_1} \nu_{\Sigma_0}^{(2)} +D_{y_1} \nu_{\Sigma_0}^{(2)} D_{y_2} \nu_{\Sigma_0}^{(3)}
\\
&\qquad  - D_{y_1} F  D_{y_2} \nu_{\Sigma_0}^{(2)} - B_{22}   D_{y_1} \nu_{\Sigma_0}^{(3)}- D_{y_1} \nu_{\Sigma_0}^{(3)} D_{y_2} \nu_{\Sigma_0}^{(2)}
\big]
\\
&\quad -\nu_{\Sigma_0}^{(2)}(X) \big[ B_{11} D_{y_2}  \nu_{\Sigma_0}^{(3)}  +  D_{y_1}  \nu_{\Sigma_0}^{(1)} D_{y_2} F + D_{y_1}  \nu_{\Sigma_0}^{(1)}   D_{y_2}  \nu_{\Sigma_0}^{(3)}
\\
&\qquad- D_{y_1} F D_{y_2}  \nu_{\Sigma_0}^{(1)} -B_{12} D_{y_1}  \nu_{\Sigma_0}^{(3)}-D_{y_1}  \nu_{\Sigma_0}^{(3)} D_{y_2}  \nu_{\Sigma_0}^{(1)}
\big]
\\
&\quad +\nu_{\Sigma_0}^{(3)}(X) \big[ B_{11} D_{y_2}  \nu_{\Sigma_0}^{(2)} + B_{22} D_{y_1}  \nu_{\Sigma_0}^{(1)} + D_{y_1}  \nu_{\Sigma_0}^{(1)}D_{y_2}  \nu_{\Sigma_0}^{(2)}
\\
&\qquad - B_{21} D_{y_2}  \nu_{\Sigma_0}^{(1)} - B_{12} D_{y_1}  \nu_{\Sigma_0}^{(2)} - D_{y_1}  \nu_{\Sigma_0}^{(2)} D_{y_2}  \nu_{\Sigma_0}^{(1)}
\big] ,
\end{align*}
where all functions are evaluated at $x + r_0 B y(t)$ if not explicitly written;
\begin{align*}
D_2 & =   D_{y_2} F ( \nu_{\Sigma_0}^{(1)}(X) \nu_{\Sigma_0}^{(2)} - \nu_{\Sigma_0}^{(2)}(X) \nu_{\Sigma_0}^{(1)})  + B_{22} ( \nu_{\Sigma_0}^{(3)}(X) \nu_{\Sigma_0}^{(1)} - \nu_{\Sigma_0}^{(1)}(X) \nu_{\Sigma_0}^{(3)})
\\
&\quad + B_{12} ( \nu_{\Sigma_0}^{(2)}(X) \nu_{\Sigma_0}^{(3)} - \nu_{\Sigma_0}^{(3)}(X) \nu_{\Sigma_0}^{(2)})
\\
D_3 &=  - D_{y_1} F ( \nu_{\Sigma_0}^{(1)}(X) \nu_{\Sigma_0}^{(2)} - \nu_{\Sigma_0}^{(2)}(X) \nu_{\Sigma_0}^{(1)})  + B_{11} ( \nu_{\Sigma_0}^{(3)}(X) \nu_{\Sigma_0}^{(2)} - \nu_{\Sigma_0}^{(2)}(X) \nu_{\Sigma_0}^{(3)})
\\
&\quad + B_{21} ( \nu_{\Sigma_0}^{(1)}(X) \nu_{\Sigma_0}^{(3)} - \nu_{\Sigma_0}^{(3)}(X) \nu_{\Sigma_0}^{(1)}) .
\end{align*}
For $m_0$, $m_1$ we have a similar expressions
\begin{align*}
m_0 &= \nu_{\Sigma_0}^{(1)} ( B_{21} D_{y_2}F - B_{22} D_{y_1}F ) -   \nu_{\Sigma_0}^{(2)}  (B_{11} D_{y_2}F - B_{12} D_{y_1} F) + \nu_{\Sigma_0}^{(3)}
\end{align*}
\begin{align*}
m_1 &=
\nu_{\Sigma_0}^{(1)} \big[ B_{21} D_{y_2} \nu_{\Sigma_0}^{(3)} + D_{y_2} F D_{y_1} \nu_{\Sigma_0}^{(2)} +D_{y_1} \nu_{\Sigma_0}^{(2)} D_{y_2} \nu_{\Sigma_0}^{(3)}
\\
&\qquad  - D_{y_1} F  D_{y_2} \nu_{\Sigma_0}^{(2)} - B_{22}   D_{y_1} \nu_{\Sigma_0}^{(3)}- D_{y_1} \nu_{\Sigma_0}^{(3)} D_{y_2} \nu_{\Sigma_0}^{(2)}
\big]
\\
&\quad -\nu_{\Sigma_0}^{(2)} \big[ B_{11} D_{y_2}  \nu_{\Sigma_0}^{(3)}  +  D_{y_1}  \nu_{\Sigma_0}^{(1)} D_{y_2} F + D_{y_1}  \nu_{\Sigma_0}^{(1)}   D_{y_2}  \nu_{\Sigma_0}^{(3)}
\\
&\qquad- D_{y_1} F D_{y_2}  \nu_{\Sigma_0}^{(1)} -B_{12} D_{y_1}  \nu_{\Sigma_0}^{(3)}-D_{y_1}  \nu_{\Sigma_0}^{(3)} D_{y_2}  \nu_{\Sigma_0}^{(1)}
\big]
\\
&\quad +\nu_{\Sigma_0}^{(3)} \big[ B_{11} D_{y_2}  \nu_{\Sigma_0}^{(2)} + B_{22} D_{y_1}  \nu_{\Sigma_0}^{(1)} + D_{y_1}  \nu_{\Sigma_0}^{(1)}D_{y_2}  \nu_{\Sigma_0}^{(2)}
\\
&\qquad - B_{21} D_{y_2}  \nu_{\Sigma_0}^{(1)} - B_{12} D_{y_1}  \nu_{\Sigma_0}^{(2)} - D_{y_1}  \nu_{\Sigma_0}^{(2)} D_{y_2}  \nu_{\Sigma_0}^{(1)}
\big] .
\end{align*}
Let us rewrite \eqref{form gi} as
\begin{align}
\label{decomp DG}
g_i = \bar g_{i} +\tilde g_i
\end{align}
where
\begin{align*}
\bar g_{i} &= h_i(x+r_0 B y(t) )
\\
\tilde g_i (t) & = h_i(x+r_0 B y(t) )  \frac{(m_0-D_0) + (m_1-D_1) h - D_{y_1}h D_2 - D_{y_2} h  D_3 }{ D_0 + h D_1  + D_{y_1}h D_2 + D_{y_2} h  D_3} .
\end{align*}
These expressions imply the following estimate (after changing variables back from \eqref{scaled vars}):
\begin{align*}
\|\bar g_i\|_b \leq C \|h_i\|_*
\end{align*}
where $\| \ \|_b$ is the norm \eqref{norm b}.
Therefore
\begin{align*}
|B(\bar g_i)(X,z) | & \leq C \frac{\|h_i\|_*}{|X|} |z| .
\end{align*}
Moreover we can write $\tilde g_i$ as
\begin{align*}
\tilde g_i (t) & = h_i(x+r_0 B y(t) )  Q(X,t,h,D_th)
\end{align*}
where
\begin{align}
\label{def Q}
Q(X,t,h,\xi)=\frac{(m_0-D_0) + (m_1-D_1) h - \xi_1 D_2 - \xi_2   D_3 }{ D_0 + h D_1  + \xi_1 D_2 + \xi_2  D_3} .
\end{align}
Let us use the notation
\begin{align*}
\tilde h(t) & = h(x + r_0 B y(t)) , \quad
\tilde h_i(t) = h_i(x + r_0 B y(t)) ,
\\
\tilde Q(t,\xi ) & = Q(X,t,\tilde h(t),\xi)
\end{align*}
so that
$$
\tilde g_i(t) = \tilde h_i(t) \tilde Q(t,D_t \tilde h(t)).
$$
Observe that $\tilde Q(t_0(X),\xi)=0$, $D_\xi \tilde Q(t_0(X),\xi)=0$.
Then we have
\begin{align*}
|B(\tilde g_i)(X,z)| &=\frac{1}{|z|} \Big| \tilde h_i(z+t_0(X)) \tilde Q(z+t_0(X),D_t \tilde h(z+t_0(X)))
\\
& \qquad - \tilde h_i(t_0(X)) D_t\tilde Q(t_0(X),D_t \tilde h(t_0(X))) z\Big|
\\
&\leq A_1 + A_2 + A_3
\end{align*}
where
\begin{align*}
A_1&=\frac{1}{|z|} \Big|
(\tilde h_i(z+t_0(X)) -\tilde h_i(t_0(X)) )\tilde Q(z+t_0(X),D_t \tilde h(z+t_0(X)))
\Big|
\\
A_2&=\frac{1}{|z|} |\tilde h_i(t_0(X))|
\Big| \tilde Q(z+t_0(X),D_t \tilde h(z+t_0(X)))
- \tilde Q(z+t_0(X),D_t \tilde h(t_0(X)))
\Big|
\\
A_3&=\frac{1}{|z|} |\tilde h_i(t_0(X))|
\Big| \tilde Q(z+t_0(X),D_t \tilde h(t_0(X))) -  D_t \tilde Q(t_0(X),D_t \tilde h(t_0(X))) z\Big| .
\end{align*}
We then have for $z\in B_{2 R_X}(t_0)$
\begin{align*}
A_1 \leq C \|\tilde h_i\|_{B_{2 R_X}(t_0)} |z|
\leq C \|h_i\|_* \frac{|z|}{|X|}.
\end{align*}
For $A_2$
\begin{align*}
A_2&\leq \frac{1}{|z|}
\int_0^1 |D_\xi \tilde Q(z+t_0(X),D_t\tilde h(\tau z + t_0(X)))| |D_{tt} \tilde h(\tau z + t_0(X))| \, d\tau |z|
\\
&\leq C \|h_i\|_* \frac{|z|}{|X|}
\end{align*}
since $|D_\xi\tilde Q( z+t_0(X),D_t\tilde h(\tau z + t_0(X)))|\leq C |z|$ for this range of arguments.
Finally also
\begin{align*}
A_3\leq  C \|h_i\|_* \frac{|z|}{|X|}
\end{align*}
because
\begin{align*}
\big| \tilde Q(z+t_0(X),D_t \tilde h(t_0(X))) -  D_t \tilde Q(t_0(X),D_t \tilde h(t_0(X))) z\big|
\leq \frac{|z|^2}{|X|}
\end{align*}
in this range of argument. This establishes  \eqref{202}.

The estimate \eqref{at g} and \eqref{at gi} are direct since the expression $A_t$ involves only one derivative the function where it is applied to, and we have control of one derivative of $g_i$ directly from  \eqref{decomp DG}.
\end{proof}

\end{document}